\documentclass[12pt, a4paper,reqno]{amsart}

\usepackage[hmargin=32mm, vmargin=27mm, includefoot, twoside]{geometry}

\usepackage[english]{babel}
\usepackage{amsmath,amssymb}
\usepackage{mathtools}
\usepackage{mathrsfs}
\usepackage{dsfont}
\usepackage{amscd}
\usepackage{lmodern}
\usepackage[figurename=Fig.]{caption}
\usepackage{graphicx}
\usepackage{color}
\usepackage[colorlinks,linkcolor={red},citecolor={blue}]{hyperref}
\usepackage[all,cmtip]{xy}
\usepackage{tikz}
\usepackage{imakeidx}

\numberwithin{equation}{section}

\newtheorem{thmintro}{Theorem}

\newtheorem{propintro}[thmintro]{Proposition}
\newtheorem{corintro}[thmintro]{Corollary}
\newtheorem{theorem}{Theorem}[section]

\newtheorem{corollary}[theorem]{Corollary}
\newtheorem{lemma}[theorem]{Lemma}
\newtheorem{prop}[theorem]{Proposition}
\theoremstyle{definition}
\newtheorem{remarkintro}{Remark}
\newtheorem{exampleintro}{Example}
\newtheorem{remark}[theorem]{Remark}
\newtheorem{definition}[theorem]{Definition}
\newtheorem{example}[theorem]{Example}
\newtheorem{algo}{Algorithm}

\newcommand{\NN}{\mathbb{N}}
\newcommand{\RR}{\mathbb{R}}
\newcommand{\ZZ}{\mathbb{Z}}

\newcommand{\CCC}{\mathscr{C}}

\newcommand{\KKK}{\mathcal{K}}
\newcommand{\OOO}{\mathcal{O}}
\newcommand{\SSS}{\mathcal{S}}
\newcommand{\WW}{\mathcal{W}}
\newcommand{\ZZZ}{\mathcal{Z}}

\newcommand{\inv}{^{-1}}

\newcommand{\co}{\colon\thinspace}

\newcommand{\changelocaltocdepth}[1]{%
  \addtocontents{toc}{\protect\setcounter{tocdepth}{#1}}%
  \setcounter{tocdepth}{#1}%
}

\DeclareMathOperator{\Aut}{Aut}
\DeclareMathOperator{\CAT}{CAT(0)}
\DeclareMathOperator{\Ch}{Ch}
\DeclareMathOperator{\CMin}{CombiMin}
\DeclareMathOperator{\Conv}{Conv}
\DeclareMathOperator{\Cox}{Cox}
\DeclareMathOperator{\Cyc}{Cyc}
\DeclareMathOperator{\dist}{d}
\DeclareMathOperator{\dc}{\dist_{\Ch}}
\DeclareMathOperator{\dce}{\dist_{\Ch}^{\Sigma^\eta}}
\DeclareMathOperator{\ext}{ext}
\DeclareMathOperator{\Fix}{Fix}
\DeclareMathOperator{\GL}{GL}
\DeclareMathOperator{\height}{ht}
\DeclareMathOperator{\id}{id}
\DeclareMathOperator{\Isom}{Isom}
\DeclareMathOperator{\Min}{Min}
\DeclareMathOperator{\Nb}{Nb}
\DeclareMathOperator{\op}{op}
\DeclareMathOperator{\Pc}{Pc}
\DeclareMathOperator{\proj}{proj}
\DeclareMathOperator{\Reg}{Reg}
\DeclareMathOperator{\sppan}{span}
\DeclareMathOperator{\Stab}{Stab}
\DeclareMathOperator{\supp}{supp}
\DeclareMathOperator{\tor}{tor}
\DeclareMathOperator{\typ}{typ}
\DeclareMathOperator{\Typ}{Typ}
\DeclareMathOperator{\MSN}{MaxStdNorm}
\DeclareMathOperator{\MN}{MaxNorm}

\DeclareMathOperator{\ICS}{IsCoreSplitting}
\DeclareMathOperator{\Root}{Root}
\DeclareMathOperator{\ARA}{AffineRootAction}
\DeclareMathOperator{\too}{\leftrightharpoons}

\makeindex[title=Index of notions]
\makeindex[options= -s s.ist,columns=1,name=s,title=Index of symbols]

\begin{document}

\renewcommand{\proofname}{{\bf Proof}}

\title[Structure of conjugacy classes in Coxeter groups]{Structure of conjugacy classes in Coxeter groups}

\author[T.~Marquis]{Timoth\'ee \textsc{Marquis}$^*$}
\address{UCLouvain, IRMP-MATH, Chemin du Cyclotron 2, Bte L7.01.02, 1348 Louvain-la-Neuve, Belgium}
\email{timothee.marquis@uclouvain.be}
\thanks{$^*$F.R.S.-FNRS Research Associate; supported in part by the FWO and the F.R.S.-FNRS under the EOS programme (project ID 40007542).}


\begin{abstract}
Cet article fournit une solution définitive au problème de la description des classes de conjugaison dans les groupes de Coxeter arbitraires en termes de permutations cycliques.

Soit $(W,S)$ un système de Coxeter. Une \emph{permutation cyclique} d'un élément $w\in W$ est un conjugué de $w$ de la forme $sws$ pour une réflexion simple $s\in S$ telle que $\ell_S(sws)\leq\ell_S(w)$. La \emph{classe de permutation cyclique} de $w$ est alors l'ensemble des éléments de $W$ qui peuvent être obtenus à partir de $w$ par une suite de permutations cycliques. \'Etant donné un sous-ensemble $K\subseteq S$ tel que $W_K:=\langle K\rangle\subseteq W$ est fini, on appelle aussi deux éléments $w,w'\in W$ \emph{$K$-conjugués} si $w,w'$ normalisent $W_K$ et $w'=w_0(K)ww_0(K)$, où $w_0(K)$ est l'élement le plus long de $W_K$.

Soit $\mathcal O$ une classe de conjugaison dans $W$, et soit $\mathcal O^{\min}$ l'ensemble des éléments de longueur minimale dans $\mathcal O$. Alors $\mathcal O^{\min}$ est la réunion disjointe d'un nombre fini de classes de permutation cyclique $C_1,\dots,C_k$. On définit le \emph{graphe de conjugaison structurel} associé à $\mathcal O$ comme étant le graphe de sommets $C_1,\dots,C_k$, et avec une arête entre les sommets distincts $C_i,C_j$ s'ils contiennent des représentants $u\in C_i$ et $v\in C_j$ tels que $u,v$ sont $K$-conjugués pour un certain $K\subseteq S$.

Dans cet article, nous calculons explicitement le graphe de conjugaison structurel associé à toute classe de conjugaison (éventuellement tordue) dans $W$, et montrons en particulier qu'il est connexe (autrement dit, deux éléments conjugués de $W$ ne diffèrent que par une suite de permutations cycliques et de $K$-conjugaisons). Chemin faisant, nous obtenons plusieurs résultats d'intérêt indépendant, comme une description du centralisateur d'un élément d'ordre infini $w\in W$, ainsi que l'existence de décompositions naturelles de $w$ comme produit d'une `` partie de torsion'' et d'une ``partie rectiligne'', avec des propriétés utiles.
\end{abstract}

\begin{abstract}
This paper gives a definitive solution to the problem of describing conjugacy classes in arbitrary Coxeter groups in terms of cyclic shifts.

Let $(W,S)$ be a Coxeter system. A \emph{cyclic shift} of an element $w\in W$ is a conjugate of $w$ of the form $sws$ for some simple reflection $s\in S$ such that $\ell_S(sws)\leq\ell_S(w)$. The \emph{cyclic shift class} of $w$ is then the set of elements of $W$ that can be obtained from $w$ by a sequence of cyclic shifts. Given a subset $K\subseteq S$ such that $W_K:=\langle K\rangle\subseteq W$ is finite, we also call two elements $w,w'\in W$ \emph{$K$-conjugate} if $w,w'$ normalise $W_K$ and $w'=w_0(K)ww_0(K)$, where $w_0(K)$ is the longest element of $W_K$.

Let $\mathcal O$ be a conjugacy class in $W$, and let $\mathcal O^{\min}$ be the set of elements of minimal length in $\mathcal O$. Then $\mathcal O^{\min}$ is the disjoint union of finitely many cyclic shift classes $C_1,\dots,C_k$. We define the \emph{structural conjugation graph} associated to $\mathcal O$ to be the graph with vertices $C_1,\dots,C_k$, and with an edge between distinct vertices $C_i,C_j$ if they contain representatives $u\in C_i$ and $v\in C_j$ such that $u,v$ are $K$-conjugate for some $K\subseteq S$.

In this paper, we compute explicitely the structural conjugation graph associated to any (possibly twisted) conjugacy class in $W$, and show in particular that it is connected (that is, any two conjugate elements of $W$ differ only by a sequence of cyclic shifts and $K$-conjugations). Along the way, we obtain several results of independent interest, such as a description of the centraliser of an infinite order element $w\in W$, as well as the existence of natural decompositions of $w$ as a product of a ``torsion part'' and of a ``straight part'', with useful properties.
\end{abstract}

\maketitle

\section{Introduction}\label{section:introduction}

\subsection{Historical background}\label{subsection:Historicalbackground}

Let $(W,S)$ be a Coxeter system. A classical result of Tits (\cite{Tit69}) and, independently, Matsumoto (\cite{Mat64}), asserts that any two words on the alphabet $S$ representing the same group element only differ by a sequence of simple ``elementary operations'', namely, \emph{braid relations} and \emph{$ss$-cancellations} (i.e. replacing a subword $(s,s)$ with $s\in S$ by the empty word). This provides a particularly elegant solution to the word problem in Coxeter groups.

From the moment D.~Krammer proved, in 1994, that the conjugay problem in Coxeter groups was solvable as well (\cite{Kra09}), the following question came very naturally (see e.g. \cite[\S 2.17]{Coh94}): is there a way to describe the set of words representing conjugates of a given element $w\in W$ in terms of simple ``elementary operations'', in the same way the set of words representing $w$ can be described in terms of braid relations and $ss$-cancellations ?

To this end, one first has to identify such additional ``elementary operations''. Since the difference between a word and the group element it represents is fully understood thanks to Matsumoto--Tits' Theorem, it will be simpler to define these operations directly at the level of the group elements, rather than the words (i.e. making the operations of braid relations and $ss$-cancellations implicit).

The most natural elementary operation one might consider for the conjugacy problem is certainly that of \emph{cyclic shift}: an element $w'\in W$ is called a {\bf cyclic shift} of an element $w\in W$ with $w\neq w'$ if there exists a reduced expression $w=s_1\dots s_d$ of $w$ such that either $w'=s_2\dots s_ds_1$ or $w'=s_ds_1\dots s_{d-1}$, or equivalently, if $w'=sws$ for some $s\in S$ such that $\ell_S(sws)\leq\ell_S(w)$. We then write $w\to w'$ if $w'$ can be obtained from $w$ by performing a sequence of cyclic shifts.

In some cases, for instance if $w$ is {\bf straight} (namely, $\ell_S(w^n)=n\ell_S(w)$ for all $n\in\NN$), cyclic shifts turn out to be sufficient to describe the conjugacy class $\OOO_w$ of an element $w\in W$ (see \cite{HN14}, \cite{conjCox}). In general, however, this need not hold: for instance, two distinct simple reflections $s,t\in S$ are not conjugate by a sequence of cyclic shifts, but might be nonetheless conjugate. In any case, however, cyclic shifts are sufficient in order to obtain from any $w\in W$ an element of minimal length in its conjugacy class (see \cite{GP93} for $W$ finite, \cite{HN14} for $W$ affine, and \cite[Theorem~A(1)]{conjCox} in general); such elements are called {\bf cyclically reduced}, and we write $\OOO_w^{\min}$ for their set. 

In \cite{GP93}, Geck and Pfeiffer introduced a new elementary operation called \emph{elementary strong conjugation}, and prove that if two cyclically reduced elements $w_1,w_2$ of a finite Coxeter group $W$ are conjugate, then they are conjugate by a sequence of (cyclic shifts and) elementary strong conjugations. This result was later generalised (\cite{GKP00}, \cite{HN12}) to the case of \emph{twisted} conjugacy classes $\OOO_{w,\delta}:=\{x\inv w\delta(x) \ | \ x\in W\}$ ($w\in W$, $\delta\in\Aut(W,S)$), as well as to Coxeter groups $W$ of affine type (\cite{HN14}).

In \cite{conjCox}, we introduced a sharp refinement of elementary strong conjugations, called \emph{elementary tight conjugations} (see \S\ref{subsection:Kconjugation} for precise definitions), and proved that if two cyclically reduced elements $w_1,w_2$ of an arbitrary Coxeter group $W$ are conjugate, then they are conjugate by a sequence of (cyclic shifts and) elementary tight conjugations.

\subsection{Goals of the paper}
Coxeter groups play a fundamental role in a variety of mathematical domains, and understanding their conjugacy classes is crucial for many applications. As illustrated above, the problem of describing conjugacy classes in terms of cyclic shifts (in some sense, the most elementary conjugation operations one might consider in a Coxeter group) has a long history. While the situation was already well understood for conjugacy classes of finite order elements thanks to the works of Geck, Pfeiffer, and He, understanding when two conjugate infinite order elements are in fact conjugate by cyclic shifts remained a complete mystery, even in affine Coxeter groups.

The main goal of the present paper is to complete the picture outlined in \S\ref{subsection:Historicalbackground}, by giving a definitive solution to the problem of describing conjugacy classes in arbitrary Coxeter groups in terms of cyclic shifts. We not only elucidate the conceptual reason for which two conjugate elements do or do not belong to the same cyclic shift class (see Theorem~\ref{thmintro:graphisomorphism} below), but we also provide a completely explicit description of the structure of the conjugacy class of a given element $w\in W$ in terms of cyclic shift classes, requiring only straightforward manipulations at the level of the Coxeter diagrams (see Theorem~\ref{thmintro:indefinite} and its illustration in Example~\ref{exampleintro:IND} for the indefinite case, and Theorem~\ref{thmintro:affine} and its illustration in Example~\ref{exampleintroAFF} for the affine case).

Along the way, we establish several results of independent interest and with a high potential for applications, such as a description of the centraliser of an infinite order element (see Theorem~\ref{thmintro:indefinite}(6,7) for the indefinite case and Theorem~\ref{thmintro:affine}(6,7) for the affine case), as well as natural splittings of infinite order elements with useful properties (see the statement (4) in Theorems~\ref{thmintro:indefinite} and \ref{thmintro:affine}).

\subsection{The structural conjugation graph $\KKK_{\OOO_w}$}

In view of \cite[Theorem~A]{conjCox}, the study of a conjugacy class $\OOO_w$ ($w\in W$) reduces to that of $\OOO_w^{\min}$. On the other hand, $\OOO_w^{\min}$ is the disjoint union of finitely many cyclic shift classes $C_1,\dots,C_k$, where we define the {\bf cyclic shift class} of an element $w$ as the set $$\Cyc(w):=\{w'\in W \ | \ w\to w'\}$$ of elements obtained from $w$ by a sequence of cyclic shifts. We next define an even more precise variant of elementary tight conjugations as follows (see \S\ref{subsection:Kconjugation} for the exact connection between the two notions): for a spherical subset $K\subseteq S$ (i.e. such that $W_K:=\langle K\rangle\subseteq W$ is finite), we call $w,w'\in W$ {\bf $K$-conjugate} if $w,w'$ normalise $W_K$ and $w'=\op_K(w):=w_0(K)ww_0(K)$, where $w_0(K)$ denotes the longest element of $W_K$. In this case, we write $$w\stackrel{K}{\too}w'.$$ Finally, we define the {\bf structural conjugation graph} $\KKK_{\OOO_w}$ associated to $\OOO_w$ as the graph with vertex set $\{C_1,\dots, C_k\}$, and with an edge between $C_i,C_j$ ($i\neq j$) if there exist $u\in C_i$ and $v\in C_j$ such that $u,v$ are $K$-conjugate for some spherical subset $K\subseteq S$ (in which case we also call $C_i,C_j$ {\bf $K$-conjugate}). In this paper, we not only prove that $\KKK_{\OOO_w}$ is connected (i.e. every two elements $w_1,w_2\in\OOO_w^{\min}$ are conjugate by a sequence of cyclic shifts and $K$-conjugations), but we compute the graph $\KKK_{\OOO_w}$ explicitely (see Theorems~\ref{thmintro:graphisomorphism}, \ref{thmintro:indefinite} and \ref{thmintro:affine} below). Note that, replacing $K$-conjugations by elementary tight conjugations in the above construction, one obtains a graph $\KKK_{\OOO_w}^t$, which we call the {\bf tight conjugation graph} associated to $\OOO_w$ (see \S\ref{subsection:SATCG} for precise definitions), with same vertex set but in general more edges, and which we compute as well (see Corollary~\ref{corintro:tightconjugationgraph}).

\subsection{Structure of $\KKK_{\OOO_w}$ for $w$ of finite order}

To identify $\KKK_{\OOO_w}$, we relate it to the following easily computable graph: given a diagram automorphism $\delta\in\Aut(W,S)$ (that is, $\delta\in\Aut(W)$ and $\delta(S)=S$), we let $\KKK_{\delta}=\KKK_{\delta, W}$ denote the graph with vertex set $\SSS_{\delta}$ the set of $\delta$-invariant spherical subsets of $S$, and with an edge between $I,J\in\SSS_{\delta}$ if there exists $K\in\SSS_{\delta}$ containing $I,J$ such that $J=\op_K(I)$; 
in that case, $I$ and $J$ are called {\bf $K$-conjugate}, and we write $$I\stackrel{K}{\too}J.$$ (When $\delta=\id$, some version of this graph for instance appears in \cite[\S 3.1]{Kra09}.) Given a subset $I\in\SSS_{\delta}$, we let $\KKK_{\delta}^0(I)$ denote the connected component of $I$ in $\KKK_{\delta}$.

When $w\in W$ is of finite order, the structure of $\KKK_{\OOO_w}$ can essentially be infered from the work of Geck--Pfeiffer \cite[Chapter~3]{GP00}, Geck--Kim--Pfeiffer \cite{GKP00} and He \cite{He07}. More precisely, denote by $\supp(w)\subseteq S$ the {\bf support} of $w\in W$, namely, the set of $s\in S$ appearing in any reduced expression of $w$. Note that the definition of the sets $\OOO_w,\OOO_w^{\min},\Cyc(w),\supp(w)$ and of the graph $\KKK_{\OOO_w}$ can be easily extended to elements $w\in W\rtimes\Aut(W,S)$ (see \S\ref{section:CSC} and \S\ref{subsection:SATCG} for precise definitions). The following proposition describes the structure of $\KKK_{\OOO_w}$ for $w\in W\rtimes\Aut(W,S)$ of finite order. The above-mentioned references already contain (if not textually, at least in essence) the case where $W$ is irreducible and finite, and the required extra work is to reduce to that case\footnote{In \cite{HePC}, X. He informed me that he generalised the result that ``cuspidal classes never fuse'' in \cite[Theorem 3.2.11]{GP00} to arbitrary Coxeter groups, and that one may also deduce Proposition~\ref{propintro:finiteorder} from this result with Lemma~\ref{lemma:cyclredfiniteordersuppspherical} in this paper. }.

\begin{propintro}\label{propintro:finiteorder}
Let $(W,S)$ be a Coxeter system, and let $w\in W\rtimes\Aut(W,S)$ be cyclically reduced and of finite order. Write $w=w'\delta$ with $w'\in W$ and $\delta\in\Aut(W,S)$. Then there is a graph isomorphism
$$\KKK_{\OOO_w}\to \KKK_{\delta}^0(\supp(w))$$
defined on the vertex set of $\KKK_{\OOO_w}$ by the assignment
$$\Cyc(u)\mapsto \supp(u)\quad\textrm{for any $u\in \OOO_w^{\min}$.}$$
Moreover, if $u,v\in\OOO_w^{\min}$ are such that $\supp(u)$ and $\supp(v)$ are $K$-conjugate for some $\delta$-invariant spherical subset $K\subseteq S$, then $\Cyc(u)$ and $\Cyc(v)$ are $K$-conjugate. 
\end{propintro}

The proof of Proposition~\ref{propintro:finiteorder}, which is essentially a combination of Theorem~\ref{thm:finite} with a twisted version of the Lusztig--Spaltenstein algorithm (\cite[Lemma~2.12]{LS79}, \cite[Proposition~5.5]{Deo82}) is given in \S\ref{subsection:TSCGOAFOE} (see Theorem~\ref{thm:main_finite_order}).

\subsection{Geometric tools to describe $\KKK_{\OOO_w}$}

Before explaining the key idea behind the core results of this paper, which describe the structure of $\OOO_w$ for $w\in W\rtimes\Aut(W,S)$ of infinite order, we introduce some geometric notions, since the proofs of the following three theorems are of geometric nature (as was the case in the earlier works \cite{HN12}, \cite{HN14} and \cite{conjCox}). The proofs use  the Davis complex $X$ of $W$, which is a $\CAT$ cellular complex on which $W$ naturally acts by cellular isometries (it is a metric realisation of the Coxeter complex $\Sigma=\Sigma(W,S)$ of $W$, and $W\rtimes\Aut(W,S)$ can be identified with $\Aut(\Sigma)$ --- see \S\ref{subsection:PCC}--\S\ref{subsection:DC} for more details).
For instance, when $W$ is the affine Coxeter group $$W=\langle s_0,s_1,s_2 \ | \ s_0^2=s_1^2=s_2^2=1=(s_0s_1)^3=(s_0s_2)^3=(s_1s_2)^3\rangle$$ of type $\widetilde{A}_2$, then $X$ is the tessellation of the Euclidean plane by congruent equilateral triangles, together with the usual Euclidean metric, and the simple reflections $s_0,s_1,s_2$ of $W$ act on $X$ as orthogonal reflections across the lines (called \emph{walls}) delimiting a fixed triangle $C_0$ (the \emph{fundamental chamber} of $X$), as pictured on Figure~\ref{figure:intro}. 
\begin{figure}
    \centering
        \includegraphics[trim = 61mm 122mm 48mm 65mm, clip, width=0.75\textwidth]{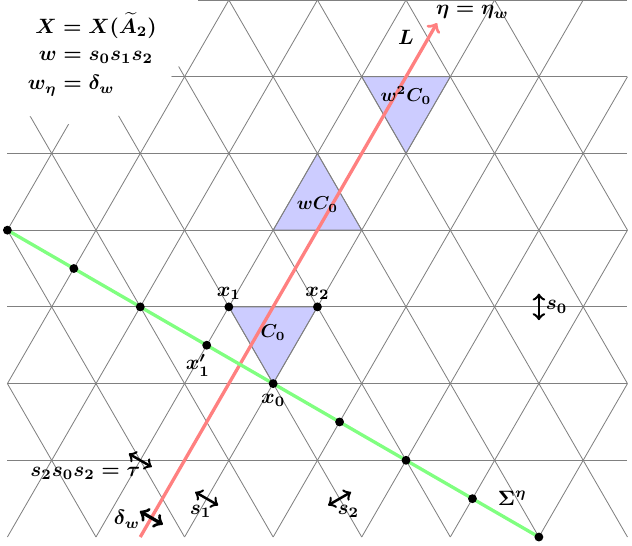}
        \caption{Davis complex of type $\widetilde{A}_2$}
        \label{figure:intro}
\end{figure}
Any infinite order element $w\in W$ possesses at least one \emph{axis}, that is, a geodesic line $L$ on which it acts by translation. In the $\widetilde{A}_2$ example, the element $w=s_0s_1s_2$ is a glide reflection along the axis $L$ pictured on Figure~\ref{figure:intro}, and thus possesses a single axis (namely, $L$). Let $\eta=\eta_w$ be the direction (in the visual boundary $\partial X$ of $X$) determined by some (resp. any) $w$-axis, and consider the set $\WW^\eta$ of walls in the direction $\eta$ (in the example, the lines of the tessellation parallel to $L$). Then the subgroup $W^{\eta}$ generated by the reflections $r_m$ across the walls $m\in\WW^\eta$ is itself a Coxeter group, with set of simple reflections $$S^\eta\subseteq S^W:=\{vsv\inv \ | \ s\in S, \ v\in W\}$$ across the walls delimiting the connected component $C_0^\eta$ of $X\setminus\bigcup_{m\in\WW^\eta}m$ containing $C_0$. Moreover, the tessellation of $X$ by the walls in $\WW^\eta$ induces a cellular complex structure $\Sigma^\eta$ (called the {\bf transversal complex} associated to $\Sigma$ in the direction $\eta$) that can be canonically identified with the Coxeter complex $\Sigma(W^\eta,S^\eta)$. In the example, letting $x_0,x_1,x_2$ denote the vertices of $C_0$ as pictured on Figure~\ref{figure:intro} and choosing $x_0$ as the origin of the Euclidean plane, $\Sigma^\eta$ can be visualised as the ($1$-dimensional) orthogonal complement of the wall $m\in \WW^\eta$ through $x_0$, together with its simplicial structure induced by the traces of the walls in $\WW^\eta$; in other words, $X^\eta$ is a simplicial line and $W^\eta$ is the infinite dihedral Coxeter group, with set of simple reflections $S^\eta=\{s_1,\tau:=s_2s_0s_2\}$ (see Figure~\ref{figure:intro}). We let $$\pi_{\Sigma^\eta}\co\Sigma\to\Sigma^\eta$$ be the corresponding morphism of cellular complexes, mapping a chamber $C$ to $C^\eta$ (see \S\ref{subsection:PTCplx} for precise definitions). In the example, $\pi_{\Sigma^\eta}$ is induced by the orthogonal projection onto $\Sigma^\eta$. Finally, note that the stabiliser $\Aut(\Sigma)_{\eta}$ of $\eta$ in $\Aut(\Sigma)$ (and its subgroup $W_{\eta}:=\Aut(\Sigma)_{\eta}\cap W$) stabilises $\WW^\eta$, and hence acts by cellular automorphisms on $\Sigma^\eta$; we denote by
$$\pi_\eta\co \Aut(\Sigma)_\eta\to\Aut(\Sigma^\eta): w\mapsto w_\eta$$ the corresponding action map. If we identify $W^\eta$ with a subgroup of both $W_\eta$ and $\Aut(\Sigma^\eta)$ (so that $\Aut(\Sigma^\eta)=W^\eta\rtimes\Aut(\Sigma^\eta,C_0^\eta)$, where $\Aut(\Sigma^\eta,C_0^\eta)\approx\Aut(W^\eta,S^\eta)$ is the group of automorphisms of $\Sigma^\eta$ stabilising its fundamental chamber $C_0^\eta$ --- see \S\ref{subsection:PCC}), then $\pi_\eta|_{W^\eta}$ is the identity, whereas $\pi_\eta(W_\eta)$ need not be contained in $W^\eta$. In fact, each $w\in \Aut(\Sigma)_\eta$ determines a unique diagram automorphism $\delta_w\in \Aut(W^\eta,S^\eta)\approx \Aut(\Sigma^\eta,C_0^\eta)$ such that $w_\eta\in W^\eta\delta_w$. In the example, $w_\eta$ stabilises $C_0^\eta$ and acts on $\Sigma^\eta$ as the reflection $\delta_w$ across the midpoint of $C_0^\eta$ (i.e. of the segment joining $x_0$ to the orthogonal projection $x_1'$ of $x_1$ on $\Sigma^\eta$).

\subsection{Key idea to describe $\KKK_{\OOO_w}$}

The key idea to describe the conjugacy class of $w$ is now as follows. We define, as in \cite{conjCox} (see also \cite{HN12}), a parametrisation
$$\pi_w\co\Ch(\Sigma)\to\OOO_w:vC_0\mapsto v\inv wv$$
of the conjugates of $w$ by the set $\Ch(\Sigma)=\{vC_0 \ | \ v\in W\}$ of chambers of $\Sigma$. This allows to translate the combinatorial operations of cyclic shifts and $K$-conjugations on the elements of $\OOO_w$ into geometric operations on the chambers of $\Sigma$ (see \S\ref{subsection:GIOCC} and \S\ref{TSCGOAIOE}). The advantage of this geometric formulation is that these geometric operations on $\Ch(\Sigma)$ can be related, via $\pi_{\Sigma^\eta}$, to the analogous geometric operations on the set $\Ch(\Sigma^\eta)$ of chambers of $\Sigma^\eta$, and, in turn, via $\pi_{w_\eta}$, to the combinatorial operations of cyclic shifts and $K$-conjugations on the elements of the conjugacy class of $w_\eta\in W^\eta\rtimes\Aut(W^\eta,S^\eta)$ in $W^\eta$. But as $w_\eta$ is an element of finite order, its structural conjugation graph is described in Proposition~\ref{propintro:finiteorder}. Translating the situation back at the level of $w$ then yields the desired description of $\OOO_w$. 

More precisely, we associate to each conjugate $\pi_w(C)$ of $w$ ($C\in\Ch(\Sigma)$) the support $$I_w(C):=\supp(\pi_{w_\eta}(C^\eta))\subseteq S^\eta$$ of the corresponding conjugate $\pi_{w_\eta}(C^\eta)$ of $w_\eta$ (that is, $\pi_{w_\eta}(C^\eta)=a\inv w_\eta a$ if $a\in W^\eta$ is such that $C^\eta=aC_0^\eta$). Using Proposition~\ref{propintro:finiteorder} (and under the assumption that $w_\eta$ be cyclically reduced), we then show that this defines a graph isomorphism $$\varphi_w\co \KKK_{\OOO_w}\to \KKK_{\delta_w,W^\eta}^0(I_w)\quad\textrm{where $I_w:=I_w(C_0)=\supp(w_\eta)\subseteq S^\eta$}$$ mapping a cyclic shift class $\Cyc(\pi_w(C))$ to $I_w(C)$, \emph{up to the following subtlety}. Note that for any $v$ in the centraliser $\ZZZ_W(w)$ of $w$ in $W$ and any $C\in\Ch(\Sigma)$, the conjugates $\pi_w(C)$ and $\pi_w(vC)$ of $w$ are the same (more precisely, $\pi_w$ induces a bijection $\Ch(\Sigma)/\ZZZ_W(w)\to \OOO_w$). However, $I_w(C)$ and $I_w(vC)$ might be different: one easily checks (see Lemma~\ref{lemma:IwvC_sigmaIwC}) that $I_w(vC)=\delta_v(I_w(C))$. To ensure the above map $\varphi_w$ is well-defined, we then identify vertices of $\KKK_{\delta_w,W^\eta}$ that are in a same $\Xi_w$-orbit, where
$$\Xi_w:=\{\delta_u \ | \ u\in\ZZZ_W(w)\}\subseteq \Xi_\eta:=\{\delta_u \ | \ u\in W_\eta\}\subseteq\Aut(W^\eta,S^\eta).$$
In other words, we let $\varphi_w$ take value in the quotient graph $\KKK_{\delta_w,W^\eta}^0(I_w)/\Xi_w$ (we will see that $\Xi_w$ commutes with $\delta_w$, and hence indeed permutes the vertices of $\KKK_{\delta_w,W^\eta}$).

\subsection{Structure of $\KKK_{\OOO_w}$ for $w$ of infinite order}

Before formally stating the result we have just alluded to, we make a straightforward reduction step. By \cite[Theorem~A(1)]{conjCox}, the cyclic shift class of $w\in W\rtimes\Aut(W,S)$ contains an element of $\OOO_w^{\min}$ (\cite[Theorem~A(1)]{conjCox} is actually stated for $w\in W$, but as we will see in Remark~\ref{rem:TheoremAconjCoxAutSigma}, this remains true for elements of $W\rtimes\Aut(W,S)$). This thus reduces the study of $\OOO_w$ to that of $\OOO_w^{\min}$, and we may then as well assume that $w$ is cyclically reduced. However, as we will see in Theorem~\ref{thmintro:affine} below, it will be useful to state the following theorem without actually assuming that $w$ is cyclically reduced, but rather by only assuming that $w_\eta$ is cyclically reduced (we will see in Proposition~\ref{prop:corresp_Minsets2} that $w_\eta$ is cyclically reduced as soon as $w$ is). As $\Cyc(w)\not\subseteq\OOO_w^{\min}$ when $w$ is not cyclically reduced (and hence $\Cyc(w)$ is not a vertex of $\KKK_{\OOO_w}$), we then set
$$\Cyc_{\min}(w):=\Cyc(w)\cap\OOO_w^{\min}.$$

\begin{thmintro}\label{thmintro:graphisomorphism}
Let $(W,S)$ be an infinite Coxeter system. Let $w\in W\rtimes\Aut(W,S)$ be of infinite order, and set $\eta:=\eta_w$. Assume that $w_\eta$ is cyclically reduced. Then there is a graph isomorphism $$\varphi_w\co\KKK_{\OOO_w}\stackrel{\cong}{\longrightarrow} \KKK_{\delta_w}^0(I_w)/\Xi_w$$
mapping $\Cyc_{\min}(w)$ to the class $[I_w]$ of $I_w$.

Moreover, if $I_w=J_0\stackrel{L_1}{\too}J_1\stackrel{L_2}{\too}\dots\stackrel{L_m}{\too}J_m$ is a path in $\KKK_{\delta_w}^0(I_w)$, then setting $u_i:=w_0(L_1)w_0(L_2)\dots w_0(L_i)\in W^{\eta}$ and $w'_i:=u_i\inv wu_i\in W$ for each $i=0,\dots,m$, the following assertions hold:
\begin{enumerate}
\item
$\varphi_{w}\inv([J_i])=\Cyc_{\min}(w'_i)$ for each $i=1,\dots,m$.
\item
There exist a spherical subset $K_i\subseteq S$ and $a_i\in W$ of minimal length in $a_iW_{K_i}$ with $W^{\eta}_{L_i}=a_iW_{K_i}a_i\inv$ such that $w_{i-1}:=a_i\inv w'_{i-1} a_i$ is cyclically reduced for each $i=1,\dots,m$. 
\item
For any $a_i,K_i$ and $w_i$ as in (2), $\varphi_{w}\inv([J_i])=\Cyc(w_i)$ for each $i$, and 
$$w\to  w_0\stackrel{K_1}{\too}\op_{K_1}(w_0)\to w_1 \stackrel{K_2}{\too}\dots \to w_{m-1}\stackrel{K_m}{\too}\op_{K_m}(w_{m-1})\leftarrow w'_m.$$
\end{enumerate}
\end{thmintro}

Note that the second part of Theorem~\ref{thmintro:graphisomorphism} allows to compute, starting from $w$, a representative $w_i$ of each cyclic shift class $C_i\subseteq\OOO_w^{\min}$, as well as a sequence of cyclic shifts and $K$-conjugations to reach $w_i$ from $w$ (with a minimal number of $K$-conjugations). Note also that Theorem~\ref{thmintro:graphisomorphism} yields an upper bound depending only on $(W,S)$ on the number of $K$-conjugations needed to reach any $u\in\OOO_w^{\min}$ from $w$ using cyclic shifts and $K$-conjugations. The proof of Theorem~\ref{thmintro:graphisomorphism} is contained in \S\ref{section:RTTPOTFH}.

\subsection{Computation of $\KKK_{\OOO_w}$ for $w\in W$}
We now turn to an explicit computation of the structural conjugation graph $\KKK_{\OOO_w}$, this time restricting to elements $w\in W$ of infinite order. Again, to simplify the statements below, we can make a straightforward reduction step. Recall that a \emph{parabolic subgroup} of $W$ (resp. $W^\eta$) is a conjugate of a \emph{standard parabolic subgroup}, namely, a subgroup of the form $W_K:=\langle K\rangle\subseteq W$ for some $K\subseteq S$ (resp. $W^\eta_L:=\langle L\rangle\subseteq W^\eta$ for some $L\subseteq S^\eta$). For any $w\in W$, there is a smallest parabolic subgroup $\Pc(w)$ containing $w$, called its \emph{parabolic closure}. In order to study $\KKK_{\OOO_w}$, one can easily reduce to the case where $\Pc(w)$ is irreducible (see \S\ref{subsection:PCD} for precise definitions). Moreover, we may safely assume as before that $w$ is cyclically reduced. In this case, $\Pc(w)$ is standard (see e.g. \cite[Proposition~4.2]{CF10}), and any cyclically reduced conjugate of $w$ is conjugate to $w$ in $\Pc(w)$ (see \cite[\S 3.1]{Kra09}). In other words, there is no loss of generality in assuming that $W$ is irreducible and $\Pc(w)=W$.

In view of Theorem~\ref{thmintro:graphisomorphism}, to compute the structural conjugation graph $\KKK_{\OOO_w}$, it now remains to compute the Coxeter system $(W^\eta,S^\eta)$ (more specifically, its Coxeter diagram), the parameters $\delta_w\in\Aut(W^\eta,S^\eta)$ and $I_w\subseteq S^\eta$ associated to $w$, as well as the subgroup $\Xi_w\subseteq \Aut(W^\eta,S^\eta)$. To this end, we will need to deal separately with the case where $W$ is of irreducible affine type, and where $W$ is of irreducible indefinite (i.e. non-affine) type. Along the way, we shall obtain several results of independent interest, including a description of $\ZZZ_W(w)$, as well as natural ways to decompose $w$ as a product of a ``straight'' part and of a ``torsion'' part (see \S\ref{section:TPSOAE}).

\subsection{Computation of $\KKK_{\OOO_w}$, indefinite case}
We start by stating the theorem dealing with the indefinite case, and refer to Remark~\ref{remarkintro:IND} below for any terminology left unexplained, as well as for comments on the meaning of each statement.

\begin{thmintro}\label{thmintro:indefinite}
Let $(W,S)$ be an irreducible Coxeter system of indefinite type. Let $w\in W$ with $\Pc(w)=W$, and set $\eta:=\eta_w$. Then the following assertions hold:
\begin{enumerate}
\item $P_w^{\max}:=W^\eta$ is the largest spherical parabolic subgroup normalised by $w$.
\item
If $w$ is cyclically reduced, there exists $a_w\in W$ of minimal length in $W^\eta a_w$ such that $v:=a_w\inv wa_w$ is \emph{standard}, that is, $v$ is cyclically reduced and the parabolic subgroup $P_v^{\max}=a_w\inv P_w^{\max} a_w$ is standard. 

\noindent
Moreover, $v\in\Cyc(w)$ and $S^{\eta_v}=a_w\inv S^\eta a_w\subseteq S$ for any such $a_w$.
\item
If $w$ is standard, the restriction of $\pi_{\Sigma^{\eta}}$ to the standard $S^{\eta}$-residue is a cellular isomorphism onto $\Sigma^{\eta}$. Moreover, if $I_w=J_0\stackrel{K_1}{\too}J_1\stackrel{K_2}{\too}\dots\stackrel{K_m}{\too}J_m$ is a path in $\KKK_{\delta_w}^0(I_w)$, then $w=w_0\stackrel{K_1}{\too}w_1\stackrel{K_2}{\too}\dots\stackrel{K_m}{\too}w_m$, where $w_i:=\op_{K_i}(w_{i-1})$ and $\varphi_{w}\inv([J_i])=\Cyc(w_i)$ for each $i=1,\dots,m$.
\item
There exist a unique atomic element $w_c\in W$ and unique $n\geq 1$ and $a\in P_w^{\max}$ such that $w=aw_c^n$. 
\item
$\delta_w=\delta_{w_c}^n$ where $\delta_{w_c}$ is the diagram automorphism $S^\eta\to S^\eta: s\mapsto w_csw_c\inv$, and $I_w=\supp(a\delta_w)\subseteq S^\eta$. Moreover, $w_\eta=a\delta_{w_c}^n$.
\item
Every element $u\in\ZZZ_W(w)$ has the form $u=bw_c^m$ for some $m\in\ZZ$ and $b\in P_w^{\max}$. Moreover, the map $$\ZZZ_W(w)\to\ZZ:u=bw_c^m\mapsto m$$ is a group morphism, with kernel the centraliser of $w_\eta$ in $W^\eta$, and image $n_w\ZZ$ for some $n_w\geq 1$ dividing $\min\{n'\geq 1 \ | \ \delta_{w_c}^{n'}(a)=\delta_{w_c}^n(a)\}$. 
\\
Finally, there is an exact sequence
$$1\to\langle w_c^m\rangle\times\ZZZ_{W^\eta}(w_\eta)\to\ZZZ_W(w)\to\Xi_w\to 1,$$
where $m\in\NN$ is the order of $\delta_{w_c}$.
\item
$\Xi_w$ is cyclic and generated by $\delta_{w_c}^{n_w}$.
\end{enumerate}
\end{thmintro}

\begin{remarkintro}\label{remarkintro:IND}
Here are a few comments on the statements of Theorem~\ref{thmintro:indefinite}, following the same numbering. 

(1) The existence of such a parabolic subgroup $P_w^{\max}$ is specific to the indefinite case: in the $\widetilde{A}_2$ example (see Figure~\ref{figure:intro}), the element $w^2=(s_0s_1s_2)^2$ is a translation in the direction $\eta$, and hence normalises each of the spherical parabolic subgroups $\langle r_m\rangle$ with $m\in\WW^\eta$.

(2) This shows that, up to modifying $w$ within its cyclic shift class, we may assume that $W^\eta=P_w^{\max}$ is standard and that $S^\eta\subseteq S$.

(3) Note that $S^\eta\subseteq S$ by (2): it thus makes sense to talk about the \emph{standard $S^\eta$-residue} $R_{S^\eta}\subseteq\Sigma$ (see \S\ref{subsection:PCC} for definitions). This statement then allows to identify $\Sigma^\eta$ with the subcomplex $R_{S^\eta}$ of $\Sigma$, and also provides a more precise version of the second statement of Theorem~\ref{thmintro:graphisomorphism} in this case.

(4) We call $w$ {\bf atomic} if it is indivisible (i.e. $w$ has no decomposition $w=u^n$ with $u\in W$ and $n\geq 2$) and $P_w^{\max}$-reduced (see Definition~\ref{definition:Preduced}). We call the decomposition $w=aw_c^n$ provided by (4) the {\bf core splitting} of $w$, and the atomic element $w_c$ the {\bf core} of $w$ (see \S\ref{section:TPSOAE} and \S\ref{subsection:TCSOAEInd} for more details on and further properties of core splittings).

(5) This allows to compute $\delta_w$ and $I_w$ from the core splitting of $w$.

(6) This provides a precise description of $\ZZZ_W(w)$, going beyond \cite[Corollary~6.3.10]{Kra09}: computing $\ZZZ_W(w)$ boils down to computing the core of $w$, the centraliser $\ZZZ_{W^\eta}(w_\eta):=\{x\in W^\eta \ | \ x\inv w_\eta x=w_\eta\}$ of the (finite order) element $w_\eta$ (given in (5)) in the finite group $W^\eta$, and the natural number $n_w$ which we call the {\bf centraliser degree} of $w$. Details on how to compute $n_w$ are given in Remark~\ref{remark:computingnw}.

(7) This allows to compute $\Xi_w$ from the core and centraliser degree of $w$.
\end{remarkintro}

We now illustrate Theorems~\ref{thmintro:graphisomorphism} and \ref{thmintro:indefinite} with the following example (see \S\ref{subsection:ExIND} for more examples). A more detailed version of this example, which includes justifications for all the claims made here without a proof, can be found in Example~\ref{example:A105}.

\begin{exampleintro}\label{exampleintro:IND}
Consider the Coxeter group $W$ with Coxeter diagram indexed by $S=\{s_i  \ | \ 1\leq i\leq 7\}$ and pictured on Figure~\ref{figure:ExampleintroIND}(a).

\begin{figure}
    \centering
        \includegraphics[trim = 8mm 207mm 19mm 30mm, clip, width=\textwidth]{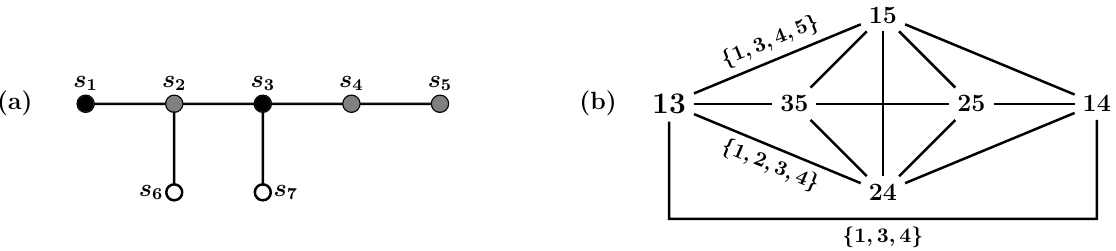}
        \caption{Example~\ref{exampleintro:IND}}
        \label{figure:ExampleintroIND}
\end{figure}

Set $J_1:=S\setminus\{s_6\}$ and $J_2:=S\setminus\{s_7\}$. Then the element $x:=w_0(J_1)w_0(J_2)\in W$ is cyclically reduced and satisfies $\Pc(x)=W$. Moreover, $x$ normalises $I:=S\setminus\{s_6,s_7\}$: more precisely, the conjugation map $\delta_x\co W_I\to W_I:u\mapsto xux\inv$ coincides with $\op_I$. This implies that $P_x^{\max}=W_I$ (and $S^\eta=I$, where $\eta:=\eta_x$), and one checks that $x$ is atomic.

Set $w:=s_1s_3x^2$. Then again $P_w^{\max}=W_I$ (i.e. $\eta_w=\eta$) and $w$ is cyclically reduced. The core of $w$ is $w_c=x$. In particular, $\delta_w=\delta_x^2=\id$ and $I_w =\supp(s_1s_3)=\{s_1,s_3\}\subseteq I$. Since $w_0(I)x\in\ZZZ_W(w)$, we have $n_w=1$. The group $\Xi_w$ is thus generated by $\delta_x=\op_I\in\Aut(W_I,I)$. 

The graph $\KKK_{\delta_{w}}^0(I_{w})$ has $6$ vertices $I_{ij}:=\{s_i,s_j\}$, pictured on Figure~\ref{figure:ExampleintroIND}(b) (where $I_{ij}$ is simply written $ij$, and where we put a label $K\subseteq I$ on some of the edges $\{I_{ij},I_{i'j'}\}$, indicating that $I_{ij},I_{i'j'}$ are $K$-conjugate). The quotient graph $\KKK_{\delta_{w}}^0(I_{w})/\Xi_w$ is then the complete graph on the $4$ vertices $[I_{13}]=[I_{35}],[I_{14}]=[I_{25}],[I_{15}],[I_{24}]$. Moreover, letting $w_{ij}$ denote for each $(i,j)\in\{(1,5),(2,4),(1,4)\}$ the $K_{ij}$-conjugate of $w_{13}:=w$, where $K_{ij}$ is the label of the edge from $I_{13}$ to $I_{ij}$ on Figure~\ref{figure:ExampleintroIND}(b), we have $\varphi_w\inv([I_{ij}])=\Cyc(w_{ij})$ for each such $(i,j)$, and hence $$\OOO_{w}^{\min}=\Cyc(w_{13})\sqcup \Cyc(w_{15})\sqcup\Cyc(w_{24})\sqcup\Cyc(w_{14}).$$
\end{exampleintro}

\subsection{Computation of $\KKK_{\OOO_w}$, affine case}

We next state the theorem dealing with the affine case, that is, when $W$ is the affine Weyl group associated to one of the Dynkin diagrams $\Gamma_S$ pictured on Figure~\ref{figure:TableAFF} in \S\ref{subsection:ACDP}. We let the simple reflections in $S=\{s_0,s_1,\dots,s_{\ell}\}$ be numbered as in that figure, and choose, as in the $\widetilde{A}_2=A_2^{(1)}$ example from Figure~\ref{figure:intro}, the vertex $x_0$ of $C_0$ of type $s_0$ (i.e. fixed by $S\setminus\{s_0\}$) as the origin of the Euclidean space $X$. Thus, $W$ decomposes as a semi-direct product $W=W_{x_0}\ltimes T_0$, where $W_{x_0}$ is the fixer of $x_0$ in $W$, and where $T_0\cong\ZZ^{\ell}$ is the translation subgroup of $W$. As before, we refer to Remark~\ref{remarkintro:AFF} below for any terminology left unexplained, as well as for comments on the meaning of each statement. Note that, while Theorem~\ref{thmintro:affine} contains seven statements so as to draw a parallel with Theorem~\ref{thmintro:indefinite} (and to show how the graph $\KKK_{\delta_w}^0(I_w)/\Xi_w$ in Theorem~\ref{thmintro:graphisomorphism} can be computed), many of these statements are either folklore or can be infered from earlier work (see Remark~\ref{remarkintro:AFF} for more details), and the core content of Theorem~\ref{thmintro:affine} (which occupies the bulk of its proof) is the statement (7) (and its use in (6)).

\begin{thmintro}\label{thmintro:affine}
Let $(W,S)$ be an irreducible Coxeter system of affine type. Let $w\in W$ be of infinite order, and set $\eta:=\eta_w$. Let $x_0$ be the vertex of $C_0$ of type $s_0$, and set $P_w^{\infty}:=W_{x_0}\cap W^{\eta}$. Then the following assertions hold:
\begin{enumerate}
\item $W^\eta=\prod_{i=1}^rW_i$ for some parabolic subgroups $W_1,\dots,W_r\subseteq W^\eta$ of irreducible affine type, where $r$ is the number of components of $P_w^{\infty}$.
\item
Write $P_w^{\infty}=a_wW_Ia_w\inv$ for some $I\subseteq S\setminus\{s_0\}$ and some $a_w\in W_{x_0}$ of minimal length in $a_wW_I$. If $w$ is cyclically reduced, then $v:=a_w\inv wa_w$ is \emph{standard}, that is, $v_{\eta_v}$ is cyclically reduced and $P_v^{\infty}:=W_{x_0}\cap W^{\eta_v}$ is standard. Moreover, $\Cyc(w)=\Cyc_{\min}(v)$ and $S^{\eta_v}=a_w\inv S^\eta a_w$.
\item
If $P_w^{\infty}$ is standard, so that $P_w^{\infty}=W_{I_\eta}$ for some $I_\eta\subseteq S$, and if $I_1,\dots,I_r$ are the components of $I$, then $I_\eta^{\ext}:=S^\eta$ has components $I_1^{\ext},\dots,I_r^{\ext}$ where $$I_i^{\ext}:=I_i\cup\{\tau_i:=r_{\delta-\theta_{I_i}}\},$$ and $\Gamma_{I_i^{\ext}}$ is the Dynkin diagram extending $\Gamma_{I_i}$. 
\item
There exist a unique element $w_{\tor}\in P_w^{\min}:=\Fix_W(\Min(w))\subseteq W^\eta$ and a unique $P_w^{\min}$-reduced element $w_{\infty}\in W$ such that $w=w_{\tor}w_{\infty}$.

Moreover, if $w=ut$ is the standard splitting of $w$ with elliptic part $u$ in the sense of \cite[3.4]{BMC15}, then $w_{\tor}=u$ and $w_{\infty}=t$ if and only if $u,t\in W$.
\item
If $w$ is standard, then $\delta_w$ is the unique automorphism in $\prod_{i=1}^r\Aut(\Gamma_{I_i^{\ext}})$ such that $\delta_w(s)=w_{\infty}sw_{\infty}\inv$ for all $s\in I_\eta$, and $I_w$ is the smallest $\delta_w$-invariant subset of $I_\eta^{\ext}$ containing $\supp_{I_{\eta}^{\ext}}(w_{\tor})$. Moreover, $w_\eta=w_{\tor}\delta_{w}$.
\item
$\ZZZ_W(w)=\pi_{\eta}\inv(\ZZZ_{W^\eta}(w_\eta)\rtimes\Xi_w)$. Moreover, there is an exact sequence
$$1\to T_\eta\rtimes \ZZZ_{W^\eta}(w_\eta)\to\ZZZ_W(w)\to\Xi_w\to 1,$$
where $T_\eta:=\{t_h\in T_0 \ | \ P_w^{\infty}.h=h\}$. 
\item
$\Xi_\eta$ is an abelian subgroup of $\prod_{i=1}^r\Aut(\Gamma_{I_i^{\ext}})$, and is computed explicitely in Theorem~\ref{thm:mainthmXi}. Moreover, if $w_\eta$ is cyclically reduced, the quotient graphs $\KKK_{\delta_w}^0(I_w)/\Xi_w$ and $\KKK_{\delta_w}^0(I_w)/\Xi_\eta$ coincide, and 
$$\Xi_w=\{\sigma\in\Xi_\eta \ | \ \textrm{$\sigma(I_w)$ is a vertex of $\KKK_{\delta_w}^0(I_w)$}\}.$$
\end{enumerate}
\end{thmintro}

\begin{remarkintro}\label{remarkintro:AFF}
The observations (1) and (3) are well-known (see e.g. \cite[\S 3,4]{HN14} or \cite[\S 4]{HN15} for a similar point of view), and the first statement of (2) can for instance be found in \cite[Proposition~4.5]{HN15}. A statement similar to (4) can be found in \cite[Proposition~2.7]{HN14}, and for (5), the existence of a conjugate $x$ of $w_{\infty}$ such that $\delta_w(s)=xsx\inv$ for all $s\in I_\eta$ follows from \cite[Proposition~3.2]{HN14}. Finally, the statement (6) in itself is an elementary observation, whose relevance is given by the explicit computation of $\Xi_w$ in (7).

Here are a few additional comments on the statements of Theorem~\ref{thmintro:affine}, following the same numbering. 

(1) The group $P_w^{\infty}$ coincides with the fixer in $W$ of the geodesic ray from $x_0$ to $\eta$, and is thus in particular a (spherical) parabolic subgroup. It thus makes sense to talk about the components of $P_w^{\infty}$ (see \S\ref{subsection:PCD}).

(2) This shows that, up to modifying $w$ within its cyclic shift class, we may assume that $P_w^{\infty}$ is standard and $w_\eta$ is cyclically reduced.

(3) This allows to compute the Dynkin diagram of $(W^\eta,S^\eta)$: each $W_{I_i}$ is of irreducible finite type, and is the Weyl group of a reduced root system $\Delta_i$. The associated Dynkin diagram $\Gamma_{I_i}$ (say of type $Y_i$) can then be extended to one of the affine Dynkin diagrams $\Gamma_{I_i^{\ext}}$ (of type $Y_i^{(1)}$) pictured on Figure~\ref{figure:TableAFF}, by adding a simple reflection $\tau_i$ to $I_i$. Similarly, the group $W$ is the affine Weyl group associated to the root system $\Delta$ of an affine Kac--Moody algebra. Letting $\delta\in\Delta$ denote the unique positive imaginary root of $\Delta$ of minimal height, and $\theta_{I_i}$ denote the highest root of $\Delta_i$, the reflection $\tau_i$, as an element of $S^\eta\subseteq S^W$, coincides with the reflection associated to the root $\delta-\theta_{I_i}\in\Delta$ (see \S\ref{subsection:ACDP} for precise definitions).

Note that, as in the $\widetilde{A}_2$ example from Figure~\ref{figure:intro}, one can realise $\Sigma^\eta$ (and even the associated Davis complex) as a subspace of $X$ (see Proposition~\ref{prop:WetaSetaSigmaetaAff}).

(4) By \cite[3.4]{BMC15}, the element $w$ has a standard decomposition, in the group $\Isom(X)$ of affine isometries of $X$, as the product $w=ut$ of a finite order element $u\in\Isom(X)$ with a translation $t\in\Isom(X)$ such that the fixed-point set $\Fix(u)$ of $u$ in $X$ coincides with $\Min(w):=\{x\in X \ | \ \textrm{$\dist(x,wx)$ is minimal}\}$, where $\dist$ is the usual Euclidean distance on $X$ (see \S\ref{subsection:TSSIWOAE}). However, $u,t$ need not belong to $W$: for instance, in the $\widetilde{A}_2$ example from Figure~\ref{figure:intro}, $w=ut$ with $u$ the orthogonal reflection across $L$ and $t$ a translation along $L$. 

We then introduce in (4) the {\bf standard splitting in $W$} of $w$, denoted $w=w_{\tor}w_{\infty}$, which in some sense is the best possible approximation of the standard splitting $w=ut$, but this time requiring that the factors $w_{\tor},w_{\infty}$ belong to $W$ (see \S\ref{section:TPSOAE} and \S\ref{subsection:TSSIWOAE} for more details on and further properties of this splitting).

(5) This allows to compute $\delta_w$ and $I_w$ from the standard splitting in $W$ of $w$.

(6) This provides a precise description of $\ZZZ_W(w)$: computing $\ZZZ_W(w)$ boils down to computing the centraliser $\ZZZ_{W^\eta}(w_\eta)$ of the (finite order) element $w_\eta$ (given in (5)) in the group $W^\eta$, as well as the subgroup $\Xi_w$ (given in (7)). Here, $t_h\in T_0$ denotes the translation of vector $h\in X$. See Proposition~\ref{prop:centraliserAFF} for more details.

(7) The explicit computation of $\Xi_\eta\subseteq\Aut(W^\eta,S^\eta)$ is achieved in Theorem~\ref{thm:mainthmXi}; the theorem is stated, for notational simplicity, under the assumption that $P_w^{\infty}=W_{I_\eta}$ is standard (the adaptation to the general case is detailed in Remark~\ref{remark:Xietaabelian}), and the notation for the elements of $\Xi_\eta$ used in the theorem are introduced in Definition~\ref{definition:sigmaj}. 
The statement (7) allows to compute $\KKK_{\OOO_w}$ using $\Xi_\eta$ instead of $\Xi_w$, and as a byproduct allows to compute $\Xi_w$ as well.
\end{remarkintro}

We now illustrate Theorems~\ref{thmintro:graphisomorphism} and \ref{thmintro:affine} with the following example (see \S\ref{subsection:ExAFF} for more examples). A more detailed version of this example, which includes justifications for all the claims made here without a proof, can be found in Example~\ref{example:A53A55}.
\begin{exampleintro}\label{exampleintroAFF}
Consider the Coxeter group $W$ of affine type $D_7^{(1)}$, with set of simple reflections $S=\{s_i \ | \ 0\leq i\leq 7\}$ as on Figure~\ref{figure:ExampleintroAff}.

\begin{figure}
    \centering
     \includegraphics[trim = 8mm 210mm 88mm 36mm, clip, width=0.8\textwidth]{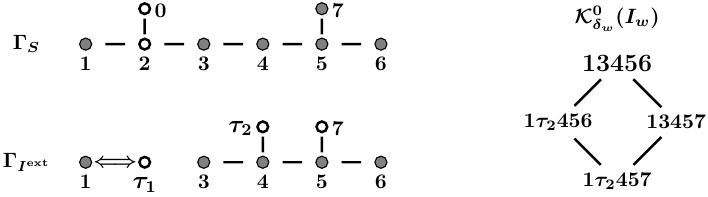}
        \caption{Example~\ref{exampleintroAFF}}
        \label{figure:ExampleintroAff}
\end{figure}

Let $\theta:=\theta_{S\setminus\{s_0\}}$ be the highest root associated to $S\setminus\{s_0\}$. Then $t:=s_0r_{\theta}$ is a translation centralising $I:=S\setminus\{s_0,s_2\}$, and hence $P_t^{\infty}=W_I$. Set $w:=ut$, where $u:=s_1s_3s_4s_5s_6\in W_I$. Then again $P_w^{\infty}=W_I$ (i.e. $\eta_w=\eta_t=:\eta$), and $w$ is cyclically reduced. 

Let $I_1=\{s_1\}$ and $I_2=\{s_3,s_4,s_5,s_6,s_7\}$ be the components of $I$. Then $S^\eta=I^{\ext}$ has two components $I_1^{\ext}=I_1\cup\{\tau_1=r_{\delta-\theta_{I_1}}\}$ and $I_2^{\ext}=I_2\cup\{\tau_2=r_{\delta-\theta_{I_2}}\}$, respectively of type $A_1^{(1)}$ and $D_5^{(1)}$ (see Figure~\ref{figure:ExampleintroAff}). 

The decomposition $w=ut$ is the standard splitting of $w$ (both in $W$ and in $\Isom(X)$), with $w_{\tor}=u$ and $w_{\infty}=t$. In particular, $\delta_w=\id$ and $I_w=\supp(u)=\{s_1,s_3,s_4,s_5,s_6\}$.

Denote, as in Definition~\ref{definition:sigmaj}, by $\sigma_1$ the diagram automorphism of $\Gamma_{I_1^{\ext}}$ of order $2$ permuting $\tau_1$ and $s_1$, and by $\sigma_2$ the diagram automorphism of $\Gamma_{I_2^{\ext}}$ of order $2$ which permutes nontrivially each of the sets $\{\tau_2,s_3\}$ and $\{s_6,s_7\}$. Then the group $\Xi_{\eta}$ is generated by $\sigma_1\sigma_2$ (see Theorem~\ref{thm:mainthmXi}(D1)).

The graph $\KKK_{\delta_w}^0(I_w)$ has $4$ vertices, pictured on Figure~\ref{figure:ExampleintroAff} (with the same notational convention as on Figure~\ref{figure:ExampleintroIND}). Since $\sigma_1\sigma_2(I_w)$ is not one of them, $\Xi_w=\{1\}$. Hence $\KKK_{\delta_w}^0(I_w)/\Xi_w$ coincides with $\KKK_{\delta_w}^0(I_w)$, and $\ZZZ_W(w)=T_\eta\rtimes \ZZZ_{W^\eta}(w_\eta)$, where $W^\eta=\langle I^{\ext}\rangle$ and $w_\eta=u$. 

Representatives for each of the four cyclic shift classes in $\OOO_w^{\min}$ (as well as a sequence of cyclic shifts and $K$-conjugations to reach them from $w$) are computed explicitely in Example~\ref{example:A53A55}.
\end{exampleintro}

\subsection{Tight conjugation graph}

To conclude, we obtain a version of Theorem~\ref{thmintro:graphisomorphism} for the tight conjugation graph $\KKK^t_{\OOO_w}$. For this, call a path $I=I_0,I_1,\dots,I_k=J$ in $\KKK_{\delta_w}$ {\bf spherical} if there exist $K_1,\dots,K_k\in\SSS_{\delta_w}$ with $\bigcup_{i=1}^kK_i\in \SSS_{\delta_w}$ such that $$I=I_0\stackrel{K_1}{\too}I_1\stackrel{K_2}{\too}\dots\stackrel{K_k}{\too}I_k=J.$$
We then let $\overline{\KKK}_{\delta_w}$ denote the graph with vertex set $\SSS_{\delta_w}$ and with an edge between $I,J\in\SSS_{\delta_w}$ if they are connected by a spherical path in $\KKK_{\delta_w}$. As before, we also let $\overline{\KKK}_{\delta_w}^0(I_w)$ denote the connected component of $\overline{\KKK}_{\delta_w}$ containing $I_w$.

\begin{corintro}\label{corintro:tightconjugationgraph}
Let $(W,S)$ be an infinite Coxeter system. Let $w\in W$ with $\Pc(w)=W$, and assume that $w_\eta$ is cyclically reduced. Then there is a graph isomorphism $$\overline{\varphi}_w\co\KKK^t_{\OOO_w}\stackrel{\cong}{\longrightarrow} \overline{\KKK}_{\delta_w}^0(I_w)/\Xi_w$$
mapping $\Cyc_{\min}(w)$ to the class $[I_w]$ of $I_w$. Moreover,
\begin{enumerate}
\item
if $(W,S)$ is of irreducible indefinite type, then $\KKK^t_{\OOO_w}$ is a complete graph;
\item
if $(W,S)$ is of irreducible affine type, then the diameter of $\KKK^t_{\OOO_w}$ is bounded above by a constant depending only on $(W,S)$;
\item
if $(W,S)$ is of affine type $A_{4n+1}^{(1)}$, there is an element $v\in W$ with $\Pc(v)=W$ such that $\KKK^t_{\OOO_{v}}$ has $2n+1$ vertices and diameter $n$. 
\end{enumerate}
\end{corintro}

\begin{remarkintro}
Corollary~\ref{corintro:tightconjugationgraph}(1) implies that, if $W$ is of indefinite type, one can pass from any $w_1\in \OOO_w^{\min}$ to any $w_2\in\OOO_w^{\min}$ using (cyclic shifts and) at most one elementary tight conjugation. When $W$ is of affine type, however, Corollary~\ref{corintro:tightconjugationgraph}(3) shows that no such uniform (across all irreducible affine Coxeter groups $W$) upper bound on the number of required elementary tight conjugations exists.
\end{remarkintro}

The structure of the paper is outlined in the table of contents below.

\subsubsection*{Acknowledgement}
I would like to thank Xuhua He for helpful discussions on twisted conjugacy classes of finite order elements, and for pointing out to me several references.  I would also like to thank the referee for their careful reading of the paper and useful comments.

\tableofcontents


\changelocaltocdepth{2}

\section{Preliminaries}\label{section:Prelim}

In this section, we review some basic concepts and properties related to Coxeter groups.
Basics on Coxeter groups, diagrams and complexes (see \S\ref{subsection:PCG}--\ref{subsection:PCC} below) can be found in \cite[Chapters~1--5]{BrownAbr}.

\subsection{Coxeter groups}\label{subsection:PCG}
Let $(W,S)$ be a Coxeter system of finite rank, that is, $W$ is a group with a distinguished finite subset $S$ of involutions admitting a presentation of the form
$$W=\langle S \ | \ s^2=1, \ (st)^{m_{st}}=1 \ \textrm{for all distinct $s,t\in S$}\rangle,$$
for some $m_{st}\in\NN_{\geq 2}\cup\{\infty\}$ (when $m_{st}=\infty$, the relation $(st)^{m_{st}}=1$ is omitted in the above presentation). The {\bf Coxeter matrix} associated to $(W,S)$ is the matrix $M=(m_{st})_{s,t\in S}$, where we set $m_{ss}:=1$ for all $s\in S$. The elements of $S$ are called {\bf simple reflections} and their conjugates {\bf reflections}; the set of reflections is denoted $S^W:=\{wsw\inv \ | \ s\in S, \ w\in W\}$.

Let $\ell=\ell_S\co W\to\NN$ denote the word metric on $W$ with respect to $S$: $\ell(w)$ is the least number $n\in\NN$ such that $w=s_1\dots s_n$ for some $s_i\in S$. Such a decomposition $w=s_1\dots s_n$ of $w\in W$ with $n=\ell(w)$ is called {\bf reduced}. The {\bf support} $\supp(w)=\supp_S(w)$\index[s]{supp@$\supp(w)$, $\supp_S(w)$ (support of $w$)} of $w\in W$ is the subset $J$ of elements of $S$ appearing in some (equivalently, any) reduced decomposition of $w$. 

A {\bf standard parabolic subgroup} is a subgroup of $W$ of the form $W_I:=\langle I\rangle$ for some $I\subseteq S$; the couple $(W_I,I)$ is then itself a Coxeter system. Conjugates of standard parabolic subgroups are called {\bf parabolic subgroups}. Intersections of parabolic subgroups are again parabolic subgroups. In particular, for any subset $H\subseteq W$, there is a smallest parabolic subgroup $\Pc(H)$ containing $H$ (the intersection of all parabolic subgroups containing $H$), called the {\bf parabolic closure}\index[s]{Pc@$\Pc(H)$ (parabolic closure of $H$)} of $H$; if $H=\{w\}$ for some $w\in W$, we simply write $\Pc(w):=\Pc(\{w\})$.


\subsection{Coxeter diagrams}\label{subsection:PCD}

The {\bf Coxeter diagram} of the Coxeter system $(W,S)$ is the labelled graph\index[s]{GammaSCox@$\Gamma_S^{\Cox}$ (Coxeter diagram)} $\Gamma_S^{\Cox}$ with vertex set $S$ and an edge labelled $m_{st}$ between $s$ and $t$ for each distinct $s,t\in S$ with $m_{st}>2$ (when $m_{st}=3$, the label $m_{st}$ is usually omitted). We will often identify a subset $J$ of $S$ with the full subgraph $\Gamma_J$ of $\Gamma_S^{\Cox}$ with vertex set $J$. 

The {\bf irreducible components} (or just {\bf components}) of a subset $J\subseteq S$ are the subsets $J_1,\dots,J_r$ such that $\Gamma_{J_1},\dots,\Gamma_{J_r}$ are the connected components of $\Gamma_J\subseteq\Gamma_S^{\Cox}$; in that case, $W_J$ decomposes as a direct product $W_J=W_{J_1}\times\dots\times W_{J_r}$. The subset $J$ (resp. the Coxeter group $W_J$) is {\bf irreducible} if $\Gamma_J$ is connected.

Irreducible finite Coxeter groups (also called Coxeter groups of {\bf finite type}) are classified, and are of one of the types $A_n$ ($n\geq 1$), $B_n=C_n$ ($n\geq 2$), $D_n$ ($n\geq 4$), $E_6$, $E_7$, $E_8$, $F_4$, $G_2$, $H_3$, $H_4$, or $I_2(m)$ ($m=5$ or $m\geq 7$), whose associated Coxeter diagram can for instance be found in \cite[\S 1.5.6]{BrownAbr} (see also \S\ref{subsubsection:DCDOFT}). We call a subset $J\subseteq S$ (resp. the parabolic subgroups conjugate to $W_J$) {\bf spherical} if all its components are of finite type.

Another important class of Coxeter groups are those of irreducible {\bf affine type}, namely, of one of the types $A_{\ell}^{(1)}$ ($\ell\geq 1$), $B_{\ell}^{(1)}$ ($\ell \geq 3$), $C_{\ell}^{(1)}$ ($\ell\geq 2$), $D_{\ell}^{(1)}$ ($\ell\geq 4$), $E_6^{(1)}$, $E_7^{(1)}$, $E_8^{(1)}$, $F_4^{(1)}$, or $G_2^{(1)}$ (see \S\ref{subsection:ACDP} for more details). All other irreducible Coxeter groups are said to be of {\bf indefinite type}.

We denote by $\Aut(W,S)$ the group of automorphisms of $(W,S)$, that is, of automorphisms $\delta\in\Aut(W)$ such that $\delta(S)=S$. Such automorphisms naturally induce automorphisms of $\Gamma_S^{\Cox}$ and are called {\bf diagram automorphisms} of $W$. 

Given $\delta\in\Aut(W,S)$ and $w\in W$, we define the {\bf $\delta$-support} of $w$ as the smallest $\delta$-invariant subset $\supp_{\delta}(w):=\bigcup_{n\geq 0}\delta^n(\supp(w))$ of $S$ containing $\supp(w)$. Alternatively, we shall write $\supp(w\delta):=\supp_{\delta}(w)$ for the {\bf support}\index[s]{supp@$\supp(w)$, $\supp_S(w)$ (support of $w$)} of the element $w\delta$ of $W\rtimes\langle\delta\rangle$. We also define the {\bf $\delta$-conjugation} by an element $x\in W$ as the automorphism $W\to W:w\mapsto x\inv w\delta(x)$.


\subsection{Coxeter complexes}\label{subsection:PCC}

The {\bf Coxeter complex} $\Sigma=\Sigma(W,S)$ associated to the Coxeter system $(W,S)$ is the poset $(\{wW_I \ | \ w\in W, \ I\subseteq S\},\supseteq)$ of left cosets of standard parabolic subgroups of $W$, ordered by the opposite of the inclusion relation (called the {\bf face relation}). It is a simplicial complex. If $S_1,\dots,S_n$ are the components of $S$, then $\Sigma$ is naturally isomorphic to the product of the simplicial complexes $\Sigma(W_{S_i},S_i)$ ($1\leq i\leq n$). In particular, each simplex $\sigma$ is of the form $wW_I$ for some $w\in W$ and $I\subseteq S$, and decomposes as a product $w_1W_{I_1}\times\dots\times w_nW_{I_n}$ of simplices (where $w=w_1\dots w_n$ is the decomposition of $w$ in $W=W_{S_1}\times\dots\times W_{S_n}$ and $I_i\subseteq S_i$). The {\bf type} $\typ(\sigma)=\typ_{\Sigma}(\sigma)$ of $\sigma$ is then the subset $I=\sqcup_{i=1}^nI_i$ of $S$.  

The group $W$ acts by left translation on $\Sigma$. This action is type-preserving, and the stabiliser of the simplex $wW_I$ coincides with the parabolic subgroup $wW_Iw\inv$. The maximal simplices of $\Sigma$ (i.e. those of the form $wW_{\varnothing}=\{w\}$ for some $w\in W$) are called {\bf chambers}; the chamber $C_0:=\{1_W\}$ is called the {\bf fundamental chamber} of $\Sigma$. The set $\Ch(\Sigma)$ of chambers of $\Sigma$ can thus be $W$-equivariantly identified with the $0$-skeleton of the Cayley graph $\mathrm{Cay}(W,S)$ of $(W,S)$. For $s\in S$, we then call two chambers $C,D\in\Ch(\Sigma)$ {\bf $s$-adjacent} if they are $s$-adjacent in $\mathrm{Cay}(W,S)$ (and {\bf adjacent} if they are $s$-adjacent for some $s\in S$). In other words, $C$ and $D$ are $s$-adjacent if $\{C,D\}=\{wC_0,wsC_0\}$ for some $w\in W$ or, alternatively, if $C$ and $D$ intersect in a {\bf panel} (that is, a codimension $1$ simplex) of type $s$.

The group $\Aut(\Sigma)$ of (simplicial) automorphisms of $\Sigma$ decomposes as a semidirect product $\Aut(\Sigma)=\Aut_0(\Sigma)\rtimes\Aut(\Sigma,C_0)$, where $\Aut_0(\Sigma)$ is the group of type-preserving automorphisms of $\Sigma$ and coincides with $W$, and where $\Aut(\Sigma,C_0)$ is the group of automorphisms of $\Sigma$ stabilising $C_0$ and is naturally isomorphic to $\Aut(W,S)$ (as an automorphism $\delta\in\Aut(\Sigma,C_0)$ permutes the panels of $C_0$, and hence also their types). We will always identify $\Aut(\Sigma,C_0)$ and $\Aut(W,S)$, so that $\Aut(\Sigma)=W\rtimes\Aut(W,S)$.

A {\bf gallery} (from $D_0$ to $D_k$) in $\Sigma$ is a sequence $\Gamma=(D_0,D_1,\dots,D_k)$ of chambers $D_i\in\Ch(\Sigma)$ such that for each $i\in\{1,\dots,k\}$, the chamber $D_{i-1}$ is $s_i$-adjacent to $D_i$ for some $s_i\in S$. The number $\ell(\Gamma):=k$ is called the {\bf length} of $\Gamma$ and the tuple $\typ(\Gamma):=(s_1,\dots,s_k)\in S^k$ the {\bf type} of $\Gamma$. If $\typ(\Gamma)\subseteq J^k$ for some $J\subseteq S$, we call $\Gamma$ a {\bf $J$-gallery}. The gallery $\Gamma$ between $D_0$ and $D_k$ is {\bf minimal} if it is of minimal length among all galleries from $D_0$ to $D_k$. The {\bf chamber distance} $\dc(C,D)$ between $C,D\in\Ch(\Sigma)$ is the length of a minimal gallery between $C$ and $D$; thus, if $C=wC_0$ and $D=vC_0$ for some $w,v\in W$, then $\dc(C,D)=\ell_S(v\inv w)$. Similarly, if $u\in\Aut(\Sigma)$, say $u=w\delta$ for some $w\in W$ and $\delta\in\Aut(\Sigma,C_0)$, we set
$$\ell_S(u):=\dc(C_0,uC_0)=\dc(C_0,wC_0)=\ell_S(w).$$

Each pair $(C,D)$ of distinct adjacent chambers determines two subcomplexes $\Phi(C,D):=\{E\in\Ch(\Sigma) \ | \ \dc(E,C)<\dc(E,D)\}$ and $\Phi(D,C)$ of $\Sigma$, called (opposite) {\bf half-spaces}, and the intersection of two opposite half-spaces is called a {\bf wall}. The set of walls of $\Sigma$ is denoted $\WW=\WW(\Sigma)$. It is in bijection with $S^W$, each wall $m$ being the fixed-point set in $\Sigma$ of a unique reflection $r_m\in S^W$: if $\Phi(wC_0,wsC_0)$ ($w\in W$, $s\in S$) is a half-space associated to $m$, then $r_m=wsw\inv$. Two chambers $C,D$ are {\bf separated} by a wall $m\in\WW$ if they lie in different half-spaces associated to $m$. We denote by $\WW(C,D)$\index[s]{WCD@$\WW(C,D)$ (set of walls separating $C$ from $D$)} the set of walls separating the chambers $C$ and $D$. Thus, $\dc(C,D)=|\WW(C,D)|$. We also say that a gallery $\Gamma=(D_0,D_1,\dots,D_k)$ from $D_0$ to $D_k$ {\bf crosses} a wall $m$ if $m$ separates $D_{i-1}$ from $D_i$ for some $i$; if $\Gamma$ is minimal, then $\Gamma$ crosses $m$ if and only if $m\in\WW(D_0,D_k)$. Finally, two walls are {\bf parallel} if their intersection is empty.

A {\bf residue} in $\Sigma$ is the set of chambers\index[s]{Rsigma@$R_{\sigma}$ (residue corresponding to the simplex $\sigma$)} $R=R_{\sigma}\subseteq \Ch(\Sigma)$ (or, when convenient, the underlying subcomplex of $\Sigma$) that contain a given simplex $\sigma$ of $\Sigma$. Alternatively, if $\sigma$ is of type $J\subseteq S$ and $C$ is any chamber containing $\sigma$ (i.e. having $\sigma$ as a face), then $R$ coincides with the set $R_J(C)$\index[s]{RJC@$R_J(C)$ ($J$-residue containing $C$)} of chambers connected to $C$ by a $J$-gallery. In this case, $R$ is called a {\bf $J$-residue} and $J=\typ(R)$ is its {\bf type}. We then call $R$ {\bf spherical} if $J$ is spherical. The residue $R_J:=R_J(C_0)$ is called {\bf standard}. We call a wall $m\in\WW$ a {\bf wall of the residue} $R$ (or of $\Stab_W(R)$) if $m$ separates two chambers of $R$ (equivalently, if $r_m\in\Stab_W(R)$). 

If $R$ is a residue and $C\in\Ch(\Sigma)$, there is a unique chamber $D\in R$ at minimal (chamber) distance from $C$; it is callled the {\bf projection} of $C$ on $R$ and is denoted $\proj_R(C)$. Alternatively, $\proj_R(C)$ is the unique chamber $D\in R$ such that $C$ and $D$ are not separated by any wall of $R$. It enjoys the following {\bf gate property}: 
$$\dc(C,E)=\dc(C,\proj_{R}(C))+\dc(\proj_R(C),E)\quad\textrm{for all chambers $E\in R$.}$$
In particular, if $wC_0\in R$ for some $w\in \Aut(\Sigma)$, then $\proj_R(C_0)=wC_0$ if and only if $w$ is of minimal length in $\Stab_W(R)w$. Note also that $w\proj_R(C)=\proj_{wR}(wC)$ for any $w\in \Aut(\Sigma)$.

Two residues $R,R'$ are {\bf parallel} if $\proj_R|_{R'}$ is a bijection from $R'$ to $R$ (with inverse $\proj_{R'}|_{R}$). In that case, $\dc(C,\proj_{R}(C))=:\dc(R,R')$ is independent of the choice of chamber $C\in R'$. Equivalently, $R,R'$ are parallel if they have the same set of walls, or else if $\Stab_W(R)=\Stab_W(R')$.


\subsection{Davis complex}\label{subsection:DC}
Basics on Davis complexes can be found in \cite{Davis} (see also \cite[Example~12.43]{BrownAbr} and \cite{Nos11}).

The {\bf Davis complex} $X=|\Sigma(W,S)|_{\CAT}$ of the Coxeter system $(W,S)$ (also called the {\bf Davis realisation} of $\Sigma$) is a proper complete $\CAT$ cellular complex on which $\Aut(\Sigma)$ acts by cellular isometries. It can be constructed as follows. Let $\Sigma^{(1)}$ be the barycentric subdivision of $\Sigma$, that is, $\Sigma^{(1)}$ is the simplicial complex with vertex set the set of simplices of $\Sigma$, and with simplices the flags of simplices of $\Sigma$. Let $\Sigma^{(1)}_s$ denote the simplicial subcomplex of $\Sigma^{(1)}$ with vertex set the set of {\bf spherical simplices} of $\Sigma$, that is, the set of simplices of $\Sigma$ with finite stabiliser in $W$ (for instance, chambers and panels are spherical simplices). Then $X$ is the standard geometric realisation of $\Sigma^{(1)}_s$ (hence a cellular subcomplex of the barycentric subdivision of the geometric realisation of $\Sigma$), together with a suitably defined locally Euclidean $\CAT$ metric $\dist\co X\times X\to\RR$.

In this paper, we will identify a spherical simplex $wW_J$ of $\Sigma$ with the closed convex subset of $X$ which is the union of the realisations of all its faces: in other words, we identify $wW_J$ with the (realisation of the) union of all flags of spherical simplices whose upper bound (for the face relation) is $wW_J$ --- see also Example~\ref{example:Daviscomplex}. Each point $x\in X$ is then contained in a unique minimal spherical simplex of $\Sigma$, called its {\bf support}, and denoted $\supp(x)$. For any $x\in X$, we will then write $R_x$\index[s]{Rx@$R_x$ (residue corresponding to the point $x\in X$)} for the residue $R_{\supp(x)}$, that is, for the set of chambers containing $x$ (equivalently, $\supp(x)$). Note that, by construction, the residue $R_x$ is spherical. 

The $\Aut(\Sigma)$-action on $X$ is induced by the $\Aut(\Sigma)$-action on $\Sigma^{(1)}_s$. Note that the stabiliser $\Stab_W(x)=\Stab_W(\supp(x))$ of any point $x\in X$ is a spherical (i.e. finite) parabolic subgroup of $W$. 

We also identify walls (and half-spaces) of $\Sigma$ with their realisation in $X$: in other words, the walls in $X$ are the fixed-point sets in $X$ of the reflections of $W$. The set $\WW$ of walls is {\bf locally finite}, in the sense that any $x\in X$ has a neighbourhood meeting only finitely many walls (see e.g. \cite[\S 3.1]{Nos11}).
For any wall $m\in\WW$, the set $X\setminus m$ has two connected components (namely, the two half-spaces associated to $m$), which are both convex (see \cite[Lemma~2.3.1]{Nos11}). The wall $m$ itself is convex in the following strong sense: if $m$ intersects a geodesic segment (or line) $L$ in at least two points, it entirely contains $L$ (see \cite[Lemma~2.2.6]{Nos11}).

\begin{figure}
    \centering
        \includegraphics[trim = 35mm 200mm 41mm 15mm, clip, width=\textwidth]{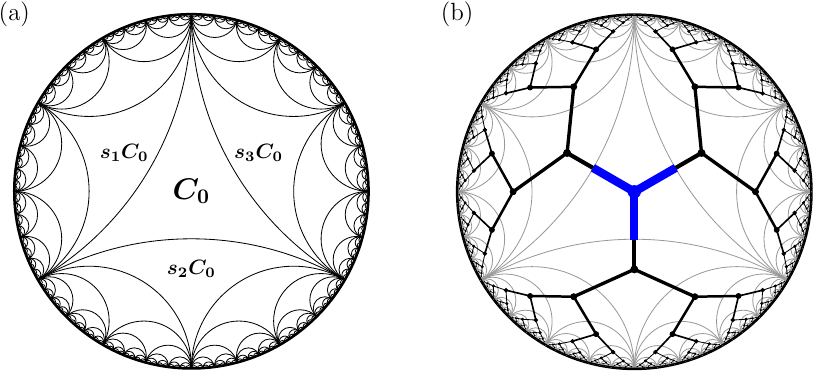}
        \caption{Example~\ref{example:Daviscomplex}: Coxeter complex versus Davis complex}
        \label{figure:Daviscomplex}
\end{figure}

\begin{example}\label{example:Daviscomplex}
Consider the Coxeter group $W=\langle s_1,s_2,s_3 \ | \ s_1^2=s_2^2=s_3^2=1\rangle$. Its Coxeter complex $\Sigma$ is the tessellation of the hyperbolic disk by congruent equilateral ideal triangles, as pictured on Figure~\ref{figure:Daviscomplex}(a). The spherical simplices of $\Sigma$ are the chambers (i.e. triangles) and panels (i.e. edges), and hence $\Sigma^{(1)}_s$ is the simplicial tree pictured on Figure~\ref{figure:Daviscomplex}(b), with vertices the barycenters of all triangles and edges of $\Sigma$. The Davis complex is then a suitable metric realisation of this tree. We identify the fundamental chamber $C_0$ (pictured on Figure~\ref{figure:Daviscomplex}(a)) with the thick blue tripod pictured on Figure~\ref{figure:Daviscomplex}(b), that is, with the intersection of all (closed) half-spaces containing the barycenter of $C_0$. The walls are in this case just the barycenters of the edges, and half-spaces are half-trees.
\end{example}


\subsection{Visual boundary and actions on \texorpdfstring{$\CAT$-spaces}{CAT(0) spaces}}\label{subsection:VBAAOCS}
Basics on $\CAT$-spaces can be found in \cite{BHCAT0}. 

Given two points $x,y$ of the Davis complex $X=|\Sigma(W,S)|_{\CAT}$, we denote by $[x,y]$ the unique geodesic segment from $x$ to $y$. 

A {\bf geodesic ray} $r\subseteq X$ is an isometrically embedded copy of $\RR_+$ in $X$. Two geodesic rays $r,r'$ are {\bf equivalent} if they are at bounded Hausdorff distance from one another, that is, if they lie in a tubular neighbourhood of one another. The {\bf visual boundary} $\partial X$ of $X$ is the set of equivalence classes of geodesic rays in $X$. If $x\in X$ and $\eta\in\partial X$, we write $[x,\eta)$ for the unique geodesic ray $r$ from $x$ in the direction $\eta$ (i.e., such that $r\in\eta$). Note that the $\Aut(\Sigma)$-action on $X$ induces an $\Aut(\Sigma)$-action on $\partial X$.

For $w\in \Aut(\Sigma)$, let 
$$|w|:=\inf\{\dist(x,wx) \ | \ x\in X\}$$
denote the {\bf translation length} of $w$, and let
$$\Min(w):=\{x\in X \ | \ \dist(x,wx)=|w|\}\subseteq X$$
denote its {\bf minimal displacement set}. By a classical result of M.~Bridson (\cite{Bridson}), $\Min(w)$ is a nonempty closed convex subset of $X$ for all $w\in \Aut(\Sigma)$. More precisely, if $w$ has finite order, then $|w|=0$ and $\Min(w)$ is the fixed-point set of $w$: this follows from the fact (see \cite[Theorem~11.26]{BrownAbr}) that any nonempty bounded subset $C\subseteq X$ (such as the $\langle w\rangle$-orbit of some $x\in X$) has a unique {\bf circumcenter} $x_C\in X$ (with $x_C\in C$ if $C$ is convex, see \cite[Theorem~11.27]{BrownAbr}), whose definition (\cite[Definition~11.25]{BrownAbr})) only depends on the metric $\dist$. On the other hand, if $w$ has infinite order, then $|w|>0$ (for otherwise $w$ would fix a point $x\in X$, contradicting the fact that point-stabilisers are finite), and $\Min(w)$ is the union of all $w$-axes, where a {\bf $w$-axis} is a geodesic line stabilised by $w$ (on which $w$ then acts as a translation of step $|w|$). 

If $Z$ is a nonempty closed convex subset of $X$, then for any $x\in X$ there is a unique $y\in Z$ minimising the distance $\dist(x,y)$, called the {\bf $\CAT$-projection} of $x$ on $Z$. As a consequence, every nonempty closed convex subset of $X$ that is stabilised by $w$ intersects $\Min(w)$ nontrivially (cf. \cite[II.6.2(4)]{BHCAT0}).

Recall from \S\ref{subsection:DC} that if $w$ has infinite order and $m\in\WW$ is a wall intersecting a $w$-axis $L$, then either $L\subseteq m$ or $L\cap m$ is a single point; in the latter case, the wall $m$ is called\index{essential@$w$-essential} {\bf $w$-essential}, and intersects any $w$-axis in a single point (see e.g. \cite[Lemma~2.5]{openKM}). Note that a $w$-essential wall always exists (see for instance \cite[Lemma~2.7]{openKM}). We further call a point $x\in\Min(w)$ {\bf $w$-essential} if it does not lie on any $w$-essential wall, or equivalently, if $\Fix_W(x)=\Fix_W(L)$ where $L$ is the $w$-axis through $x$. In this case, we call the (spherical) residue $R_x$ a {\bf $w$-residue}\index{residue@$w$-residue}. Note that the walls of $R_x$ are precisely the walls containing $L$, and hence $w$ normalises $\Stab_W(R_x)$ (equivalently, $R_x$ and $wR_x$ are parallel residues).

Finally, for $w$ of infinite order, we let $\eta_w\in\partial X$ denote the direction of its axes: for any $w$-axis $L$ and any $x\in L$, the geodesic ray based at $x$, contained in $L$ and containing $wx$ is a representative for $\eta_w$.


\subsection{Transversal complex}\label{subsection:PTCplx}
 
Fix a direction $\eta\in\partial X$ of the Davis complex $X=|\Sigma(W,S)|_{\CAT}$. 
Let $\WW^\eta$\index[s]{Weta1@$\WW^\eta$ (set of walls in the direction $\eta$)} denote the set of walls $m\in\WW$ in the direction $\eta$, namely, such that $[x,\eta)\subseteq m$ for some (equivalently, any) $x\in m$. Let\index[s]{Weta2@$W^{\eta}$ (subgroup generated by the reflections wrt a wall in $\WW^\eta$)} $W^{\eta}\subseteq W$ be the reflection subgroup of $W$ generated by the reflections $r_m$ with $m\in \WW^{\eta}$. Given a chamber $C\in\Ch(\Sigma)$, we write $C(\eta)$\index[s]{Ceta1@$C(\eta)$ (connected component of $X\setminus \bigcup_{m\in\WW^{\eta}}m$ containing $C$)} for the connected component of $X\setminus \bigcup_{m\in\WW^{\eta}}m$ containing $C$.

By a classical result of Deodhar (see \cite{Deo89}), $(W^{\eta},S^{\eta})$ is itself a Coxeter system, where\index[s]{Seta@$S^{\eta}$ (simple reflections of $W^\eta$)} $S^{\eta}\subseteq S^W$ is the (finite) set of reflections $r_m$ whose wall $m$ delimits $C_0(\eta)$. The poset $\Sigma^{\eta}$ obtained from the tessellation of $X$ by the walls in $\WW^{\eta}$ is then naturally isomorphic to $\Sigma(W^{\eta},S^{\eta})$, and we call it the {\bf transversal complex}\index{Transversal complex}\index[s]{Sigmaeta@$\Sigma^{\eta}$ (transversal complex to $\Sigma$ in the direction $\eta$)} associated to $\Sigma$ in the direction $\eta$ (see also \cite[Section~5]{CL11} and \cite[Appendix~A]{geohyperbolic}). Its set of walls is $\WW^\eta$. For a chamber $C\in\Ch(\Sigma)$, we will write $C^\eta$\index[s]{Ceta2@$C^\eta$ (chamber of $\Sigma^\eta$ corresponding to $C$)} for the chamber $C(\eta)$ of $\Sigma^\eta$ viewed as an (abstract) cell of the cellular complex $\Sigma^\eta$, rather than as a subset of $X$. The fundamental chamber of $\Sigma^\eta$ is then $C_0^\eta$. 

The assignment $C\in\Ch(\Sigma)\mapsto C^\eta\in\Ch(\Sigma^\eta)$ induces a surjective morphism of posets\index[s]{piSigmaeta@$\pi_{\Sigma^\eta}$ (natural poset morphism from $\Sigma$ to $\Sigma^\eta$)} $\pi_{\Sigma^\eta}\co \Sigma\to\Sigma^{\eta}$. The stabiliser\index[s]{AutSigmaeta@$\Aut(\Sigma)_{\eta}$ (stabiliser of $\eta$ in $\Aut(\Sigma)$)}\index[s]{Weta3@$W_{\eta}$ (stabiliser of $\eta$ in $W$)} $\Aut(\Sigma)_{\eta}:=\Stab_{\Aut(\Sigma)}(\eta)$ of $\eta$ in $\Aut(\Sigma)$ (and its subgroup $W_{\eta}:=\Stab_W(\eta)=\Aut(\Sigma)_{\eta}\cap W$) stabilises $\WW^\eta$, and hence acts by cellular automorphisms on $\Sigma^\eta$. We let\index[s]{pieta@$\pi_{\eta}$ (transversal action map of $W_\eta$ on $\Sigma^\eta$)}\index[s]{weta@$w_{\eta}$ (transversal action of $w$ on $\Sigma^\eta$)} $$\pi_{\eta}\co \Aut(\Sigma)_{\eta}\to \Aut(\Sigma^{\eta}):w\mapsto w_{\eta}$$ denote the corresponding action map. Thus, viewing $W^\eta$ as a subgroup of both $W_\eta$ and $\Aut(\Sigma^\eta)$, the map $\pi_{\eta}$ is the identity on $W^\eta$, and
\begin{equation}\label{eqn:compatibility_pisigmapieta}
\pi_{\Sigma^\eta}(wC)=w_\eta C^\eta\quad\textrm{for all $w\in \Aut(\Sigma)_\eta$ and $C\in\Ch(\Sigma)$.}
\end{equation}

Note that 
\begin{equation}\label{eqn:replacementR1}
wvw\inv=w_{\eta}vw_{\eta}\inv\quad\textrm{for all $w\in \Aut(\Sigma)_\eta$ and $v\in W^\eta$.}
\end{equation}
Indeed, if $x\in W^\eta$ is such that $x\inv w_{\eta}C_0^\eta=C_0^\eta$, then $x\inv w$ stabilises $C_0(\eta)$ and hence permutes the walls delimiting $C_0(\eta)$. In other words, $x\inv w$ normalises $S^\eta$, and its conjugation action on $S^\eta$ is the same as that of $x\inv w_{\eta}=\pi_{\eta}(x\inv w)$. Therefore, $wsw\inv=x(x\inv wsw\inv x)x\inv=x(x\inv w_{\eta}sw_{\eta}\inv x)x\inv=w_{\eta}sw_{\eta}\inv$ for all $s\in S^\eta$, yielding (\ref{eqn:replacementR1}).

Note also that for any two chambers $C,D\in\Ch(\Sigma)$, the chamber distance\index[s]{dchSigmaeta@$\dce$ (chamber distance in $\Sigma^\eta$)} $\dce(C^\eta,D^\eta)$ between $C^\eta$ and $D^\eta$ in $\Sigma^\eta$ can be computed as
\begin{equation}\label{eqn:dce}
\dce(C^\eta,D^\eta)=|\WW(C,D)\cap\WW^\eta|.
\end{equation}

Finally, note that if $v\in \Aut(\Sigma)$, the orbit map $X\to X:x\mapsto v\inv x$ induces a cellular isomorphism $\phi_v\co\Sigma^\eta\to\Sigma^{v\inv\eta}$ mapping the chamber $C^\eta$ to the chamber $(v\inv C)^{v\inv\eta}$, that is, such that 
\begin{equation}
\pi_{\Sigma^{v\inv\eta}}(C)=\phi_v\pi_{\Sigma^\eta}(vC)\quad\textrm{for all $C\in\Ch(\Sigma)$.}
\end{equation}
In other words, if we view $\Sigma^\eta$, $\Sigma^{v\inv\eta}$ as tessellations of $X$, we have 
\begin{equation}
\Sigma^{v\inv\eta}=v\inv\Sigma^\eta.
\end{equation}
In particular, if $\pi_{\Sigma^\eta}(vC_0)=C_0^\eta$, then $\phi_v(C_0^\eta)=C_0^{v\inv\eta}$ (or equivalently, $C_0(v\inv\eta)=v\inv C_0(\eta)$), and hence
\begin{equation}\label{eqn:autominvariance}
S^{v\inv\eta}=v\inv S^\eta v\quad\textrm{for all $v\in \Aut(\Sigma)$ with $\pi_{\Sigma^\eta}(vC_0)=C_0^\eta$.} 
\end{equation}


\section{Cyclic shift classes}\label{section:CSC}

Throughout this section, we fix a Coxeter system $(W,S)$. The purpose of this short section is to introduce the basic notations and terminology related to cyclic shift classes. 

\begin{definition}
For $v\in W$, we denote by\index[s]{kappav@$\kappa_v$ (conjugation by $v$)} $\kappa_v\co W\to W:w\mapsto v\inv wv$ the conjugation by $v$.
\end{definition}

\begin{definition}
Given $w\in \Aut(\Sigma)$, we denote by $\OOO_w:=\{v\inv wv \ | \ v\in W\}$ its conjugacy class and\index[s]{Ow@$\OOO_w$ (conjugacy class of $w$)}\index[s]{Owmin@ $\OOO^{\min}$ (cyclically reduced elements of $\OOO$)} by $\OOO_w^{\min}\subseteq\OOO_w$ the set of conjugates of $w$ of minimal length. We say that $w,w'\in \Aut(\Sigma)$ are {\bf conjugate in $W$} if $w'\in\OOO_w$.

We emphasise that $\OOO_w$ is in general distinct from the conjugacy class of $w$ in $\Aut(\Sigma)$, which coincides with $\bigcup_{\delta\in\Aut(W,S)}\delta\OOO_w\delta\inv$.
\end{definition}

\begin{definition}
Let $w\in\Aut(\Sigma)$. We call $v\in\Aut(\Sigma)$ a {\bf cyclic shift}\index{Cyclic shift} of $w$ if $\ell_S(v)\leq\ell_S(w)$ and $v=sws$ for some $s\in S$; in this case, we write\index[s]{-00@$\stackrel{s}{\to}$ (cyclic shift by $s$)} $w\stackrel{s}{\to}v$. For $v\in\Aut(\Sigma)$, we further write\index[s]{-01@$\to$ (sequence of cyclic shifts)} $w\to v$ if $w=w_0\stackrel{s_1}{\to}w_1\stackrel{s_2}{\to}\dots\stackrel{s_k}{\to}w_k=v$ for some $w_i\in\Aut(\Sigma)$ and $s_i\in S$.  

We define the {\bf cyclic shift class}\index{Cyclic shift class}\index[s]{Cycw@$\Cyc(w)$ (cyclic shift class of $w$)} of $w$ as the set
$$\Cyc(w):=\{v\in\Aut(\Sigma) \ | \ w\to v\}\subseteq\OOO_w$$
of elements $v\in \Aut(\Sigma)$ that can be obtained from $w$ by a sequence of cyclic shifts. We also let $\Cyc_{\min}(w)$\index[s]{Cycwmin@$\Cyc_{\min}(w)$ (cyclically reduced elements of $\Cyc(w)$)} denote the set of elements of minimal length in $\Cyc(w)$. If $w\in W$, \cite[Theorem~A(1)]{conjCox} implies that
\begin{equation}\label{eqn:cycminw}
\Cyc_{\min}(w)=\Cyc(w)\cap\OOO_w^{\min}
\end{equation}
(as noticed in Remark~\ref{rem:TheoremAconjCoxAutSigma}, this in fact also holds for $w\in\Aut(\Sigma)$). We call $w$ {\bf cyclically reduced}\index{Cyclically reduced} if $w\in \OOO_w^{\min}$.

We emphasise that the notion of ``cyclic shift'' used here (and which is better suited for our purposes) is slightly different from the one used for instance in \cite{GP00} or in \cite{HN12} (but it is the same as in \cite{conjCox}), as we do not require that cyclic shifts have the same length. In particular, $\Cyc(w)$ is an equivalence class if and only if $w$ is cyclically reduced.
\end{definition}

\begin{remark}
The name ``cyclic shift'' comes from the following equivalent definition: if $w,v\in W$ are distinct and $s\in S$, then $w\stackrel{s}{\to}v$ if and only if there is a reduced expression $w=s_1\dots s_k$ ($s_i\in S$) of $w$ such that $v\in\{s_2\dots s_ks_1,s_ks_1\dots s_{k-1}\}$ (see \cite[Lemma~3.5(1)]{conjCox}).

Moreover, if $w\in W$, the equality (\ref{eqn:cycminw}) implies that $w$ is cyclically reduced if and only if its length cannot be decreased by a sequence of cyclic shifts.
\end{remark}

\begin{lemma}\label{lemma:cyclicshit_samesupport}
Let $w\in \Aut(\Sigma)$. If $s\in S$ is such that $\ell_S(sws)\leq\ell_S(w)$, then $\supp(sws)\subseteq\supp(w)$. In particular, if $w\to v$ for some $v\in\Aut(\Sigma)$, then $\supp(v)\subseteq\supp(w)$, and $\supp(v)=\supp(w)$ in case $\ell_S(v)=\ell_S(w)$.
\end{lemma}
\begin{proof}
Write $w=u\delta$ with $u\in W$ and $\delta\in\Aut(W,S)$. Then $sws=sut\delta$, where $t:=\delta(s)\in S$. By assumption, $\ell_S(sut)\leq\ell_S(u)$, and we may assume that $sut\neq u$ (that is, $sws\neq s$). Then \cite[Condition (F) page 79]{BrownAbr} implies that either $\ell_S(su)<\ell_S(u)$ or $\ell_S(ut)<\ell_S(u)$. The exchange condition (see \cite[Condition (E) page 79]{BrownAbr}) then implies that $u$ has a reduced expression that either starts with $s$ or ends with $t$, and hence that $\supp(w)\supseteq\{s,t\}$. The lemma follows.
\end{proof}

\begin{remark}\label{remark:wandudelta}
Let $w\in\Aut(\Sigma)$. Write $w=u\delta$ with $u\in W$ and $\delta\in\Aut(W,S)$. Then for any $x\in W$, we have $x\inv w x=x\inv u\delta(x)\cdot \delta$. In particular, $w'=u'\delta'$ (with $u'\in W$ and $\delta'\in\Aut(W,S)$) and $w=u\delta$ are conjugate if and only if $\delta=\delta'$ and $u,u'$ are $\delta$-conjugate (see \S\ref{subsection:PCD}). Similarly, the notation $u\to_{\delta}u'$ used in other papers, such as \cite[\S 3.1]{He07}, is equivalent to $u\delta\to u'\delta$, and we set $\Cyc_{\delta}(u):=\{u'\in W \ | \ u\to_{\delta}u'\}=\Cyc(u\delta)\delta\inv$. We also recall from \S\ref{subsection:PCD} that we set $\supp(w):=\supp_{\delta}(u)=\bigcup_{n\in\NN}\delta^n(\supp(u))$.
\end{remark}


\section{Elements of finite order}

Throughout this section, we fix a Coxeter system $(W,S)$. The purpose of this section is to establish the main ingredients for the proof of Proposition~\ref{propintro:finiteorder}.

\subsection{Preliminaries}\label{subsection:EOFO:Prelim}
In this subsection, we recall some basic facts, mainly concerning finite Coxeter groups (see e.g. \cite[Chapter~1 and \S 2.5]{BrownAbr} or \cite{Deo82}).

\subsubsection{Dynkin/Coxeter diagrams of finite type}\label{subsubsection:DCDOFT}

\begin{figure}[!htb]
    \centering
       \begin{minipage}{\textwidth}
        \centering
        \includegraphics[trim = 20mm 134mm 42mm 40mm, clip, width=\textwidth]{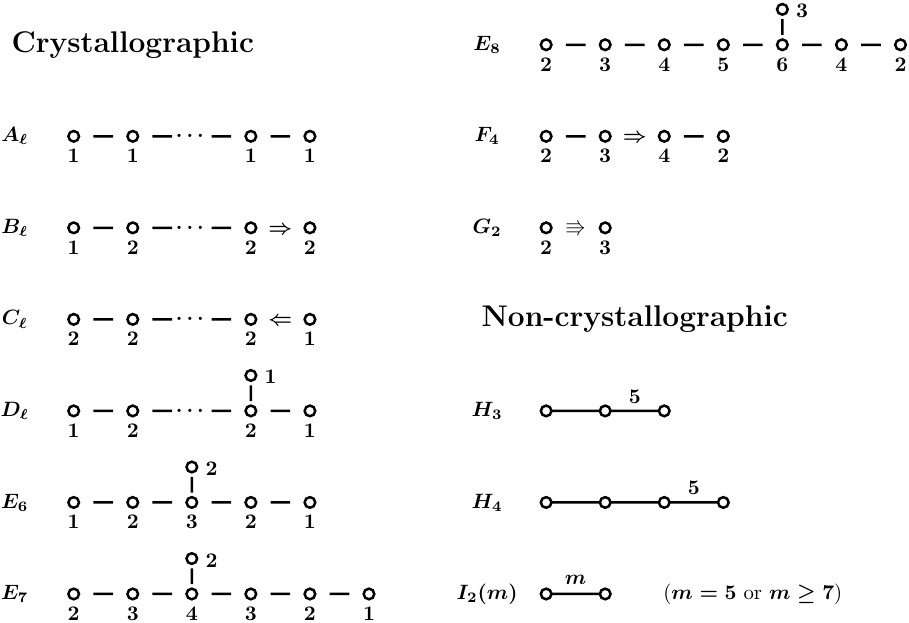}
          \end{minipage}
		\caption{Dynkin/Coxeter diagrams of finite type.}
		\label{figure:TableFIN}
\end{figure}

If the Coxeter matrix $(m_{st})_{s,t\in S}$ of $(W,S)$ satisfies $m_{st}\in\{2,3,4,6,\infty\}$ for all $s,t\in S$, then $W$ is said to be of {\bf crystallographic} type. If $W$ is of crystallographic finite type, then $W$ is the Weyl group of an irreducible reduced root system in the sense of \cite[Chapter~6]{Bourbaki}, and is encoded by its Dynkin diagram\index[s]{GammaS@$\Gamma_S$ (Dynkin diagram)} $\Gamma_S$ (see \cite[Ch.~6, \S4, nr~2]{Bourbaki}), of one of the types $A_{\ell}$ ($\ell\geq 1$), $B_{\ell}$ ($\ell\geq 3$), $C_{\ell}$ ($\ell\geq 2$), $D_{\ell}$ ($\ell\geq 4$), $E_6$, $E_7$, $E_8$, $F_4$, and $G_2$, as pictured on Figure~\ref{figure:TableFIN} (if $\Gamma_S$ is of type $X_{\ell}$, it has $\ell$ vertices). Note that the Coxeter diagram $\Gamma_S^{\Cox}$ can be obtained from $\Gamma_S$ by replacing each double edge (resp. triple edge) with an edge labelled $4$ (resp. $6$). The remaining possible Coxeter diagrams of finite type (of type $H_3$, $H_4$, and $I_2(m)$ for $m=5$ or $m\geq 7$) are also pictured on Figure~\ref{figure:TableFIN}.

\subsubsection{The canonical linear representation}

Let $V$ be a real vector space with basis $\{e_s \ | \ s\in S\}$ in bijection with $S$. Endow $V$ with the symmetric bilinear form $B\co V\times V\to\RR$ defined by $B(e_s,e_t)=-\cos(\pi/m_{st})$ for $s,t\in S$, where $(m_{st})_{s,t\in S}$ is the Coxeter matrix associated to $(W,S)$.  For each $s\in S$, consider the orthogonal (with respect to $B$) reflection $$\rho_s\co V\to V:v\mapsto v-2B(e_s,v)e_s.$$
Then there is an injective group morphism $\rho\co W\to\GL(V)$, called the {\bf canonical linear representation} of $W$, defined by the assignment $s\mapsto\rho_s$ for each $s\in S$. For $w\in W$ and $v\in V$, we will also simply write $w(v):=\rho(w)v$.

Let $\Pi:=\{e_s \ | \ s\in S\}$ be the set of {\bf simple roots}, and $\Phi:=W(\Pi)$ be the set of {\bf roots} of $V$. Denoting by $\Phi^+:=\Phi\cap \sum_{s\in S}\RR_{\geq 0}e_s$ and $\Phi^-:=-\Phi^+$ the set of {\bf positive} and {\bf negative} roots, we have $\Phi=\Phi^+\cup\Phi^-$. For $\alpha\in\Phi$, we will then write $\alpha>0$ or $\alpha<0$, depending on whether $\alpha$ is positive or negative. For $I\subseteq S$, we set $\Pi_I:=\{e_s \ | \ s\in I\}$, $\Phi_I:=\Phi\cap \sum_{i\in I}\RR e_i=W_I(\Pi_I)$, and $\Phi_I^+:=\Phi_I\cap\Phi^+$.

\begin{lemma}\label{lemma:CLRbasicfacts}
Let $v,v'\in W$ and $s\in S$. Then:
\begin{enumerate}
\item
$v(e_s)>0\iff \ell(vs)>\ell(v)$.
\item
$\Phi(v):=\{\alpha\in\Phi^+ \ | \ v(\alpha)<0\}$ is finite, and if $\Phi(v)=\Phi(v')$ then $v=v'$.
\end{enumerate}
\end{lemma}

\subsubsection{The Coxeter complex of a finite Coxeter group}
Assume that $W$ is finite. Then the bilinear form $B$ is positive definite. For each $s\in S$, consider the $\rho_s$-fixed hyperplane $H_s:=\{v\in V \ | \ B(e_s,v)=0\}$ in $V$. Consider also the simplicial cone $C:=\{v\in V \ | \ B(e_s,v)>0 \ \forall s\in S\}$ in $V$. Finally, let $\Sigma(W,V)$ be the cellular complex induced by the tessellation of $V$ by the hyperplanes in $\mathcal H:=\{wH_s \ | \ w\in W, s\in S\}$ (see e.g. \cite[\S 1.4,1.5]{BrownAbr}). Then there is a natural $W$-equivariant cellular isomorphism $\Sigma(W,V)\to\Sigma(W,S)$ mapping $C$ to $C_0$ (and identifying $\mathcal H$ with the set of walls of $\Sigma$). 

Two chambers $D,E$ of $\Sigma$ are called {\bf opposite} if $E=-D$ in $\Sigma(W,V)$. Note that if $W$ is arbitrary (not necessarily finite), then a residue $R$ of $\Sigma$ of spherical type $I\subseteq S$ is isomorphic to the Coxeter complex $\Sigma(W_I,I)$, and it thus also makes sense to call two chambers $D,E$ of $R$ {\bf opposite} (in $R$).

\subsubsection{Opposition in finite Coxeter groups}\label{subsubsection:OIFCG}

Let $I\subseteq S$ be a spherical subset. Then $W_I$ has a unique element of maximal length, which we denote \index[s]{w0K@$w_0(K)$ (longest element of $W_K$)} $w_0(I)$. It has order $2$, and normalises $I$: the map\index[s]{opK@$\op_K$ (opposition map in $W_K$)} $\op_I\co W_I\to W_I:x\mapsto w_0(I)xw_0(I)$ restricts to a diagram automorphism $\op_I\co I\to I$ of $(W_I,I)$. More precisely, in the canonical linear representation of $W$, we have 
\begin{equation}
w_0(I)(e_s)=-e_{\op_I(s)}\quad\textrm{for all $s\in I$.}
\end{equation}
Note that $\op_I\co I\to I$ preserves each component of $I$, and hence can be computed from the following lemma.

\begin{lemma}[{\cite[\S 5.7.4]{BrownAbr}}]\label{lemma:oppositionfinite}
Assume that $(W,S)$ is of irreducible finite type $X$. Then $\op_S\co S\to S$ is nontrivial if and only if $X=A_{\ell}$ ($\ell\geq 2$), or $X=D_{\ell}$ with $\ell$ odd, or $X=E_6$, or $X=I_2(\ell)$ with $\ell$ odd. In those cases, $\op_S$ is the unique nontrivial automorphism of $\Gamma_S^{\Cox}$.
\end{lemma}

Geometrically, if $R_I=R_I(C_0)\subseteq \Sigma$ is the standard residue of type $I$, then $w_0(I)C_0$ is the unique chamber opposite $C_0$ in $R_I$.

\subsubsection{Normaliser of a parabolic subgroup}
For a subset $I\subseteq S$, we denote by $N_W(W_I)$ the normaliser of $W_I$ in $W$, and we set\index[s]{NI@$N_I$ (subgroup of the normaliser of $W_I$ in $W$)} $N_I:=\{w\in W \ | \ w\Pi_I=\Pi_I\}$. 

\begin{lemma}\label{lemma:NWWINIWI}
Let $I\subseteq S$. Then $N_W(W_I)=W_I\rtimes N_I$. Moreover,
$$\ell_S(w_In_I)=\ell_S(w_I)+\ell_S(n_I)\quad\textrm{for all $w_I\in W_I$ and $n_I\in N_I$.}$$
\end{lemma}
\begin{proof}
This follows from \cite[Lemma~5.2]{Lus77} (see also \cite[Proposition~3.1.9]{Kra09}).
\end{proof}

\begin{remark}\label{remark:NWWINIWI_AutSigma}
Lemma~\ref{lemma:NWWINIWI} remains valid if we replace $W$ by $\Aut(\Sigma)$ and $N_I$ by\index[s]{NI2@$\widetilde{N}_I$ (subgroup of the normaliser of $W_I$ in $\Aut(\Sigma)$)}  $$\widetilde{N}_I:=\{w\in \Aut(\Sigma) \ | \ w\Pi_I=\Pi_I\}.$$ 
Indeed, for $w\in\Aut(\Sigma)$ normalising $W_I$, we let $w_I\in W_I$ be such that $w_IC_0=\proj_{R_I}(wC_0)$ and set $n_I:=w_I\inv w$. Then $n_I\in N_{\Aut(\Sigma)}(W_I)$ is of minimal length in $W_In_I$ by the gate property, so that \cite[Proposition~3.1.6]{Kra09} implies that $n_I\in \widetilde{N}_I$ (i.e. writing $n_I=x\delta$ with $x\in W$ and $\delta\in\Aut(W,S)$, we have $xW_{\delta(I)}x\inv=W_I$ and hence $x\delta(\Pi_I)=x\Pi_{\delta(I)}=\Pi_I$ by \emph{loc. cit.}). The conclusions of Lemma~\ref{lemma:NWWINIWI} then follow from the gate property.
\end{remark}

\begin{lemma}\label{lemma:Kra09Prop316}
Let $I,J\subseteq S$ and $a\in W$ be of minimal length in $aW_I$ and such that $aW_Ia\inv=W_J$. Then $a\Pi_I=\Pi_J$. In particular, $aIa\inv=J$ and $a$ is of minimal length in $W_JaW_I$.
\end{lemma}
\begin{proof}
See \cite[Proposition~3.1.6]{Kra09}.
\end{proof}


\subsection{Conjugacy classes in finite Coxeter groups}

In this subsection, we formulate consequences of the properties (P1), (P2) and (P3) of conjugacy classes in finite Coxeter groups described in \cite[Theorem~7.5]{He07} (see also \cite[Theorem~3.2.7]{GP00} for the untwisted case). Remark~\ref{remark:wandudelta} serves as a dictionary between our notations and the ones of \cite{He07}.

\begin{definition}
Let $w\in\Aut(\Sigma)$. Its conjugacy class $\OOO_w$ in $W$ is called\index{Cuspidal} {\bf cuspidal} if $\supp(v)=S$ for every $v\in \OOO_w^{\min}$. Alternatively, writing $w=u\delta$ with $u\in W$ and $\delta\in\Aut(W,S)$, we call the\index{deltaconj@$\delta$-conjugacy class} {\bf $\delta$-conjugacy class} $\OOO=\{x\inv u\delta(x) \ | \ x\in W\}$ of $u$ in $W$ {\bf cuspidal} if $\OOO\delta=\OOO_w$ is cuspidal, that is, if $\supp_{\delta}(v)=S$ for every $v\in\OOO^{\min}$, where $\OOO^{\min}$ is the set of minimal length elements in $\OOO$.
\end{definition}

\begin{prop}\label{prop:He07}
Assume that $W$ is finite. Let $w\in \Aut(\Sigma)$ be cyclically reduced and such that $\supp(w)=S$. Then the following assertions hold:
\begin{enumerate}
\item
$\OOO_w$ is cuspidal.
\item
$\OOO_w^{\min}=\Cyc(w)$.
\end{enumerate}
\end{prop}
\begin{proof}
In the language and notations of \cite{He07} and \cite{GKP00}, the lemma can be reformulated as follows.
Let $\delta\in\Aut(W,S)$ and $w\in W$ be such that $\supp_{\delta}(w)=S$ and $w$ is of minimal length in its $\delta$-conjugacy class $\OOO$. Then:
\begin{enumerate}
\item
$\OOO$ is cuspidal, that is, $\supp_{\delta}(v)=S$ for any $v\in \OOO^{\min}$.
\item
$\OOO^{\min}$ coincides with $\Cyc_{\delta}(w)=\{v\in W \ | \ w\to_{\delta}v\}$.
\end{enumerate}

If $W$ is irreducible, this follows from (P1) and (P2) in \cite[Theorem~7.5]{He07}. Let now $W$ be arbitrary. Reasoning componentwise, there is no loss of generality in assuming that $S$ coincides with the $\langle\delta\rangle$-orbit of a component $J$ of $S$, say $S=\bigcup_{i=0}^{r-1}\delta^i(J)$ with $r\geq 1$ minimal. Let $v\in \OOO^{\min}$. We have to show that $\supp_{\delta}(v)=S$  and that $w\to_{\delta}v$. 

Up to performing ($\delta$-)cyclic shifts, we may assume by \cite[Lemma~2.7(a)]{GKP00} that $w,v\in W_J$ (see Lemma~\ref{lemma:cyclicshit_samesupport}). By assumption, $\supp(w)=J$, and hence also $\supp_{\delta^r}(w)=J$. Since the $\delta^r$-conjugacy class of $w$ in $W_J$ is contained in the $\delta$-conjugacy class of $w$ in $W$ (i.e. $x\inv w\delta^r(x)=y\inv w\delta(y)$ with $y:=x\delta(x)\dots\delta^{r-1}(x)$, for any $x\in W_J$), certainly $w$ is of minimal length in its $\delta^r$-conjugacy class in $W_J$. As $w,v$ are $\delta^r$-conjugate in $W_J$ by \cite[Lemma~2.7(b)]{GKP00}, the irreducible case (applied to $(W_J,\delta^r)$) implies that $\supp_{\delta^r}(v)=J$ and that $w\to_{\delta^r}v$. Hence, $\supp_{\delta}(v)=S$, and $w\to_{\delta}v$ by \cite[Lemma~2.7(c)]{GKP00}, as desired.
\end{proof}

Contrary to (P1) and (P2), the property (P3) in \cite[Theorem~7.5]{He07} does not generalise to arbitrary (not necessarily irreducible) finite Coxeter groups. The following proposition, which will be sufficient for our purpose, deals with an important particular case of (P3), valid in arbitrary finite Coxeter groups.

\begin{prop}\label{prop:useP3}
Assume that $W$ is finite. Let $\delta,\sigma\in\Aut(W,S)$ be commuting diagram automorphisms. Assume that $\sigma$ stabilises every component $I$ of $S$ such that one of the following holds:
\begin{enumerate}
\item[(C1)] $I$ is not of type $A_m$ for some $m\geq 1$;
\item[(C2)] $\delta^r|_I\neq\id$, where $r\geq 1$ is minimal such that $\delta^r(I)=I$.
\end{enumerate}
Assume, moreover, that $\sigma$ is the identity on each component $I$ of $S$ of type $F_4$.

\noindent
Let $w\in W$ be such that the $\delta$-conjugacy class of $w$ in $W$ is cuspidal. Then $\sigma(w\delta)$ and $w\delta$ are conjugate in $W$.
\end{prop}
\begin{proof}
(1) Assume first that $W$ is irreducible. If $|S|=2$, say $S=\{s,t\}$, then either $\sigma\in\{\id,\delta\}$, in which case $\sigma(w\delta)=w\delta$ (if $\sigma=\id$) or $\sigma(w\delta)=w\inv\cdot w\delta \cdot w$ (if $\sigma=\delta$), or else $\delta=\id$ and $\sigma$ permutes $s$ and $t$, in which case $w$ is of the form $(st)^m$ or $(ts)^m$ for some $m\in\NN$ (because its conjugacy class is cuspidal) and $\sigma(w)=sws$.

Assume now that $|S|\geq 3$. In this case, the proposition follows from \cite[Theorem~7.5(P3)]{He07}: indeed, in the notations of \emph{loc. cit.}, we have $l_{i,\delta}(w)=l_{i,\delta}(\sigma(w))$ for each $i\in S$, as $\sigma(s)$ and $s$ are conjugate for every $s\in S$ (i.e. if $\sigma(s)\neq s$, then $W$ is not of type $F_4$ by assumption, and hence $\sigma(s)$ and $s$ are joined by an edge-path with odd labels in $\Gamma_S^{\Cox}$). Similarly, viewing $w,\delta,\sigma$ as elements of $\GL(V)$ as in \cite[\S 7.1]{He07} (so that $\sigma$ is a permutation matrix), we have $$p_{\sigma(w),\delta}(q)=\det(q\id-\sigma w\delta \sigma\inv)=\det(\sigma(q\id-w\delta)\sigma\inv)=\det(q\id-w\delta)=p_{w,\delta}(q),$$
as desired.

(2) We now deal with the general case. Reasoning componentwise, there is no loss of generality in assuming that $S$ coincides with the $\langle\delta,\sigma\rangle$-orbit of a component $J$ of $S$. Let $r\geq 1$ be minimal such that $\delta^r(J)=J$. Then the components $\delta^i(J)$ ($0\leq i<r$) of $S$ are cyclically permuted by $\delta$. 

(2.1) Assume first that $\sigma(J)=\delta^j(J)$ for some $j\in\{0,\dots,r-1\}$. Since $\delta,\sigma$ commute, $\sigma$ then stabilises $\{\delta^i(J) \ | \ 0\leq i<r\}$, and hence $S=\bigcup_{0\leq i<r}\delta^i(J)$ by assumption. Up to conjugating $w\delta$ in $W$, there is then no loss of generality in assuming that $w\in W_J$ (see \cite[Lemma~2.7(a)]{GKP00}). 

If $j=0$, that is, $\sigma(J)=J$, we are done by (1): indeed, since the $\delta^r$-conjugacy class of $w$ in $W_J$ is cuspidal (i.e. $\supp(w\delta^r)\supseteq\supp(w)=J$), (1) implies that $\sigma(w)=x\inv w\delta^r(x)$ for some $x\in W_J$, and hence $\sigma(w\delta)=y\inv w\delta y$ with $y:=x\delta(x)\dots\delta^{r-1}(x)$.

Suppose now that $\sigma(J)\neq J$. By (C1), $J$ is of type $A_m$ for some $m\geq 1$. In particular, $\delta^{-j}\sigma|_{J}\in\{\id,\op_J\}$.
 
If $\sigma|_J=\delta^{j}|_J$, then $\sigma(w\delta)=\delta^j(w)\delta=a\inv w\delta a$, where $a:=w\delta(w)\dots\delta^{j-1}(w)$ and we are done. On the other hand, if $\sigma|_J=\delta^{j}|_J\op_J$, then the same argument implies that $\sigma(w\delta)=\delta^j(\op_J(w))\delta$ is conjugate to $\op_J(w)\delta$. It then remains to see that $\op_J(w)$ is $\delta$-conjugate to $w$. But this follows as before from the fact that $\op_J(w)=w_0(J) ww_0(J)$ is $\delta^r$-conjugate to $w$ in $W_J$.

(2.2) Assume now that $\sigma(J)\notin\{\delta^i(J) \ | \ 0\leq i<r\}$. Let $s\geq 1$ be minimal such that $\sigma^s(J)=J$ (thus, $s\geq 2$). Then by assumption, $S=\coprod_{0\leq i<r}^{0\leq j<s}\sigma^j\delta^i(J)$.

As in (2.1), up to conjugating $w\delta$ in $W$, there is no loss of generality in assuming that $w=w_0\dots w_{s-1}\in \prod_{0\leq j<s}W_{\sigma^j(J)}$ with $w_j\in W_{\sigma^j(J)}$. By assumption, the conjugacy class of $w_j$ in $W_{\sigma^j(J)}$ is cuspidal for each $j$ (otherwise, there is some $a\in W_{\sigma^j(J)}$ such that $\supp(a\inv w_ja)\subsetneq \sigma^j(J)$, and hence $a\inv w a=a\inv w_j a\cdot\prod_{i\neq j}w_i$ has $\delta$-support properly contained in $S$, contradicting the assumption that the $\delta$-conjugacy class of $w$ is cuspidal).

It remains to show that $\sigma(w_{j-1})$ and $w_{j}$ are $\delta$-conjugate in $\prod_{0\leq i<r}W_{\delta^i\sigma^j(J)}$ for each $j\in\{0,\dots,s-1\}$ (where we set $w_{-1}:=w_{s-1}$), for if $\sigma(w_{j-1})=a_j\inv w_{j}\delta(a_j)$ with $a_j\in \prod_{0\leq i<r} W_{\delta^i\sigma^j(J)}$ for each $j$, then $\sigma(w)=a\inv w\delta(a)$ with $a:=\prod_{0\leq j<s}a_j$. As in (2.1), it is sufficient to check that $\sigma(w_{j-1})$ and $w_{j}$ are $\delta^r$-conjugate in $W_{\sigma^{j}(J)}$. 

As we have seen above, the conjugacy class of $w_{j}$ (and of $\sigma(w_{j-1})$) in $W_{\sigma^{j}(J)}$ is cuspidal. Moreover, since $\sigma$ does not stabilise $\sigma^j(J)$, (C1) and (C2) imply that $\sigma^j(J)$ is of type $A_m$ for some $m\geq 1$ and that $\delta^r|_{\sigma^j(J)}=\id$. As there is only one cuspidal conjugacy class in $W_{\sigma^{j}(J)}$ by \cite[Example~3.1.16]{GP00}, the proposition follows. 
\end{proof}

\begin{remark}
The technical assumptions of Proposition~\ref{prop:useP3} cannot be removed:
\begin{enumerate}
\item
The condition (C1) is necessary, because if a component $I$ of $W$ is not of type $A_m$, then $W_I$  in general contains at least two distinct cuspidal conjugacy classes (cf. \cite[3.1.16]{GP00}), say $\OOO_{v}$ and $\OOO_{v'}$. Hence if $\delta=\id$ and $\sigma$ permutes two components $I_1,I_2$ of type $I$ (say $W=W_{I_1}\times W_{I_2}$), and if $v_i,v'_i$ are the elements of $W_{I_i}$ corresponding to $v,v'$ respectively ($i=1,2$), then the conjugacy class of $w:=v_1v'_2\in W$ is cuspidal, but $\sigma(w)=v_2v_1'$ is not conjugate to $w$.
\item
The condition (C2) is necessary for the same reasons as (C1), because if a component $I$ of $W$ is of type $A_m$ but $\delta(I)=I$ and $\delta_{|I}\neq\id$, then $W_I$ contains at least two distinct cuspidal $\delta_{|I}$-conjugacy classes (cf. \cite[7.14]{He07}).
\item
The condition on components of type $F_4$ is necessary, because if $W$ is of type $F_4$ (and $\delta=\id$), then $W$ contains a cuspidal conjugacy class $\OOO_w$ such that, in the notations of the proof, $\ell_{s_1,\delta}(w)\neq\ell_{s_3,\delta}(w)$ (where the vertices of $F_4$ are labelled as usual $s_1,s_2,s_3,s_4$ from left to right) --- see \cite[Table~B.3 on p.407]{GP00}. Hence if $\sigma\neq \id$, the conjugacy criterion mentioned in the proof implies that $\sigma(w)$ and $w$ are not conjugate.
\end{enumerate}
\end{remark}


\subsection{Distinguishing cyclic shift classes}

The following lemma is a reformulation of \cite[Proposition~2.4.1]{CH16}. It is stated in \emph{loc. cit.} for finite groups, but the same proof, which we repeat here for the convenience of the reader, holds for an arbitrary $W$.
\begin{lemma}\label{lemma:241He}
Let $\delta\in\Aut(W,S)$. Let $w\in W$ be such that $w\delta$ is cyclically reduced, and set $J:=\supp(w\delta)$. Let $x\in W$ and set $J':=\supp(xw\delta x\inv)$. Then $x=x'x_1$ for some $x'\in W_{J'}$ and some $x_1\in W$ of minimal length in $W_{J'}x_1W_J$ and such that $\delta(x_1)=x_1$ and $x_1\Pi_J\subseteq\Pi_ {J'}$. Moreover, if $xw\delta x\inv$ is cyclically reduced, then $x_1\Pi_J=\Pi_{J'}$.
\end{lemma}
\begin{proof}
Write $x=y'x_1y$ with $y'\in W_{J'}$, $y\in W_{J}$ and $x_1$ the unique element of minimal length in $W_{J'}x_1W_J$ (see \cite[Proposition~2.23]{BrownAbr}). Set $w':=yw\delta(y)\inv\in W_J$. Thus $x_1w'\delta x_1\inv=(y')\inv xw\delta x\inv y'$ has support $J''$ contained in $J'$ (and $J''=J'$ if $xw\delta x\inv$ is cyclically reduced by Proposition~\ref{prop:He07}(1)). Hence $x_1w'\in x_1 W_J\subseteq W_{J'}x_1W_J$ and $x_1w'\in W_{J'}\delta(x_1)\subseteq W_{J'}\delta(x_1)W_J$. We deduce that $W_{J'}x_1W_J=W_{J'}\delta(x_1)W_J$, and hence that $\delta(x_1)=x_1$ (because $x_1,\delta(x_1)$ are of minimal length in their $(W_{J'},W_J)$-double coset). In particular, $w'\in W_J\cap x_1\inv W_{J'}x_1=W_{J\cap x_1\inv J'x_1}$ (see \cite[Lemma~2.25]{BrownAbr}). Since $w'\delta$ has support $J$ by Proposition~\ref{prop:He07}(1) (and $J\cap x_1\inv J'x_1$ is $\delta$-stable), this implies that $J\cap x_1\inv J'x_1=J$, that is, $x_1 Jx_1\inv \subseteq J'$. Therefore, $J''=\supp(x_1w'\delta x_1\inv)=x_1\supp(w'\delta)x_1\inv=x_1Jx_1\inv$, and hence $x_1\Pi_J=\Pi_ {J''}$ by Lemma~\ref{lemma:Kra09Prop316}. Moreover, $x\in  W_{J'}x_1$, as desired.
\end{proof}

The following proposition is a generalisation of \cite[Proposition~5.5]{Deo82} (see also \cite[Lemma~2.12]{LS79} and \cite[Lemma~5]{Ho80} for the case of finite Coxeter groups) to the twisted case, and its proof is a straightforward adaptation of the proof given in \emph{loc. cit.}.
\begin{prop}\label{prop:Deodhar_twisted}
Let $\delta\in\Aut(W,S)$. Let $J,K$ be $\delta$-invariant spherical subsets of $S$. Suppose there exists $w\in W$ with $\delta(w)=w$, of minimal length in $W_KwW_J$ and such that $w\Pi_J=\Pi_K$.  Then there exists a sequence $J=J_0,J_1,\dots,J_k=K$ of $\delta$-invariant subsets $J_i$ of $S$, and elements $s_1,\dots,s_k\in S$ such that the following assertions hold for $i\in\{1,\dots,k\}$:
\begin{enumerate}
\item
$s_i\notin J_{i-1}$ and either $J_i=J_{i-1}$ or $J_{i-1}\cup J_i=J_{i-1}\cup T_i$, where $T_i:=\{\delta^n(s_i) \ | \ n\in\NN \}$;
\item
$J_{i-1}\cup T_i$ is spherical; we then set $\nu_i:=w_0(J_{i-1}\cup T_i)w_0(J_{i-1})$;
\item
$\Pi_{J_i}=\nu_i\Pi_{J_{i-1}}$;
\item
$w=\nu_k\dots\nu_2\nu_1$ and $\ell(w)=\sum_{i=1}^k\ell(\nu_i)$.
\end{enumerate}
In particular, we have $J_i=\op_{J_{i-1}\cup J_i}(J_{i-1})$ for all $i\in\{1,\dots,k\}$.
\end{prop}
\begin{proof}
We prove the proposition by induction on $\ell(w)$. If $\ell(w)=0$, there is nothing to prove. Assume now that $\ell(w)\geq 1$. Let $s_1\in S$ with $\ell(ws_1)<\ell(w)$. Since $\delta(w)=w$, we then have $\ell(wt)<\ell(w)$ for each $t\in T:=T_1:=\{\delta^n(s_1) \ | \ n\in\NN\}$. Note that $J\cap T=\varnothing$, as $w$ is of minimal length in $wW_J$. 

Let $v\in W$ be the unique element of minimal length in $wW_{J\cup T}$. Thus, $w=vw'$ for some $w'\in W_{J\cup T}$, and we have $\delta(v)=v$ and $\delta(w')=w'$. 

For each $s\in J\cup T$, write $w'(e_s)=\sum_{s'\in J\cup T}\lambda_{s,s'}e_{s'}\in\Phi$ for some $\lambda_{s,s'}\in\RR$. Since $\ell(w's)=\ell(ws)-\ell(v)$ and $\ell(w')=\ell(w)-\ell(v)$, we have (see Lemma~\ref{lemma:CLRbasicfacts}(1))
\begin{equation}\label{eqn:gjkhyfsljl}
w'(e_s)>0\iff \ell(w's)>\ell(w')\iff \ell(ws)>\ell(w)\iff s\in J.
\end{equation}
Since $w\Pi_J\subseteq \Pi$ and $w(e_s)=\sum_{s'\in J\cup T}\lambda_{s,s'}v(e_{s'})$ with $v(e_{s'})>0$ for each $s'\in J\cup T$ (because $\ell(vs')>\ell(v)$), we deduce that $w'\Pi_J\subseteq \Pi_{J\cup T}$. 

In particular, since $w'(e_s)<0$ for each $s\in T$ by (\ref{eqn:gjkhyfsljl}), we have $w'(\alpha)<0$ for every $\alpha\in\Phi_{J\cup T}^+\setminus\Phi_J$ (i.e. if $s\in T$, then $e_s\notin\mathrm{span}_{\RR}\Pi_J$ and hence $w'(e_s)\notin \mathrm{span}_{\RR}w'\Pi_J$). This shows that $\Phi_{J\cup T}^+\setminus\Phi_J\subseteq\Phi_{J\cup T}(w')$ is finite (see Lemma~\ref{lemma:CLRbasicfacts}(2)), and hence $\Phi_{J\cup T}$ is finite, that is, $J\cup T$ is spherical.

Set $\nu_1:=w_0(J\cup T)w_0(J)$ and let $J_1\subseteq S$ with $\Pi_{J_1}=\nu_1\Pi_J$. Note that $\delta(\nu_1)=\nu_1$ as $J$ and $T$ are $\delta$-invariant, and hence $J_1$ is $\delta$-invariant. As we have just seen, $w'(\alpha)<0$ for every root $\alpha\in\Phi^+_{J\cup T}\setminus \Phi_J$, and hence in particular for every $\alpha\in\Phi^+_{J\cup T}$ such that $\nu_1(\alpha)<0$. In other words, $\Phi_{J\cup T}(\nu_1)\subseteq\Phi_{J\cup T}(w')$. Since $w'(e_s)>0$ for $s\in J$ and $\nu_1(\alpha)<0$ for every $\alpha\in\Phi^+_{J\cup T}\setminus \Phi_J$ (because $w_0(J)(\alpha)>0$ by a similar argument as above), we also have $\Phi_{J\cup T}(w')\subseteq\Phi_{J\cup T}(\nu_1)$. Hence $\Phi_{J\cup T}(\nu_1)=\Phi_{J\cup T}(w')$, so that $w'=\nu_1$ by Lemma~\ref{lemma:CLRbasicfacts}(2).

Thus, $w=v\nu_1$ with $\ell(w)=\ell(v)+\ell(\nu_1)$, and $v\in W$ satisfies $\delta(v)=v$ and $v\Pi_{J_1}=\Pi_K$. Note also that either $J_1=J_0$, or $J_0\cup J_1=J_0\cup T$ (because $J_1$ is $\delta$-invariant). Finally, note that $\nu_1\neq 1$ (because $\nu_1(e_s)<0$ for $s\in T$ and the canonical linear representation of $W_{J\cup T}$ is faithful), and hence $\ell(v)<\ell(w)$. We may then apply the induction hypothesis to $v$, yielding the proposition.
\end{proof}

Before proving the main result of this section (Theorem~\ref{thm:finite} below), we need two additional technical lemmas.

\begin{lemma}\label{lemma:C1C2intermediate}
Let $\delta\in\Aut(W,S)$. Let $J,K$ be $\delta$-invariant spherical subsets of $S$. Suppose there exists $x\in W$ with $\delta(x)=x$, of minimal length in $W_KxW_J$ and such that $x\Pi_J=\Pi_K$. Let $I$ be a component of $K$. 
\begin{enumerate}
\item
If $I$ is not of type $A_m$ for some $m\geq 1$, then $\kappa_x(I)\cap I\neq\varnothing$. If, moreover, $I$ is not a subset of type $D_5$ inside a subset of $S$ of type $E_6$, then $\kappa_x(I)=I$.
\item
If $\delta(I)=I$ but $\delta|_I\neq \id$, then $\kappa_x(I)=I$.
\item
If $I$ is of type $F_4$, then $\kappa_x|_I=\id$.
\end{enumerate}
\end{lemma}
\begin{proof}
By Proposition~\ref{prop:Deodhar_twisted}, there is a sequence $K=K_0,K_1,\dots,K_k=J=\kappa_x(K)$ of subsets $K_i$ of $S$ such that $K_{i-1}\cup K_i$ is spherical and $K_i=\op_{K_{i-1}\cup K_i}(K_{i-1})$ for each $i\in\{1,\dots,k\}$. On the other hand, if $L$ is a spherical subset of $S$ containing $I$, then Lemma~\ref{lemma:oppositionfinite} implies that $\op_L(I)=I$ when $I$ is not of type $D_5$ inside a subset of type $E_6$ nor of type $A_m$, that $\op_L(I)\cap I$ contains the subset of $I$ of type $D_4$ when $I$ is of type $D_5$ inside a subset of type $E_6$, and that $\op_L|_I=\id$ when $I$ is of type $F_4$. The statements (1) and (3) follow.

We now prove (2). Note that $\delta$ and $\kappa_x$ commute, as $\delta(x)=x$. Assume for a contradiction that $\delta(I)=I$, that $\delta|_I\neq \id$ (in particular, $|I|\geq 2$), and that $\kappa_x(I)\neq I$. By Proposition~\ref{prop:Deodhar_twisted} (applied to $\delta:=\delta$, $J:=I$, $K:=\kappa_x(I)$ and $w:=x\inv$), there exists a $\delta$-invariant subset $I'$ of $S$ distinct from $I$ such that $I\cup I'$ is irreducible and spherical, and such that $I'=\op_{I\cup I'}(I)$. In particular, the Coxeter diagram $\Gamma_{I\cup I'}^{\Cox}$ is of one of the types $A_{\ell}$, $D_{\ell}$ ($\ell$ odd) and $E_6$, and $\op_{I\cup I'}$ is the only nontrivial automorphism of $\Gamma_{I\cup I'}^{\Cox}$ (see Lemma~\ref{lemma:oppositionfinite}). Since $\delta|_{I}\neq\id$ by assumption, we then have $\delta|_{I\cup I'}=\op_{I\cup I'}$ and hence $\delta(I)=I'\neq I$, yielding the desired contradiction.
\end{proof}

\begin{lemma}\label{lemma:cyclredfiniteordersuppspherical}
Let $w\in\Aut(\Sigma)$ be cyclically reduced and of finite order. Then $\supp(w)$ is a spherical subset of $S$.
\end{lemma}
\begin{proof}
Write $w=u\delta$ with $u\in W$ and $\delta\in\Aut(W,S)$, and set $J:=\supp(w)$. Since $w$ has finite order, it fixes a point $x\in X$ (see \S\ref{subsection:VBAAOCS}). Let $aW_I$ ($a\in W$, $I\subseteq S$ spherical) be the support $\supp(x)$ of $x$ (see \S\ref{subsection:DC}). Thus, $aW_I=w\cdot aW_I=u\delta(a)W_{\delta(I)}$, so that $\delta(I)=I$ and $a\inv u\delta(a)\in W_I$, that is, $a\inv wa\in W_I\delta$. In particular, $J':=\supp(a\inv wa)\subseteq I$. By Lemma~\ref{lemma:241He}, we can write $a\inv=a'a_1$ for some $a'\in W_{J'}$ and some $a_1\in W$ such that $a_1\Pi_J\subseteq\Pi_{J'}$. But then $a_1W_Ja_1\inv\subseteq W_{J'}\subseteq W_I$, and since $W_I$ is finite, $W_J$ is finite as well, yielding the lemma.
\end{proof}

\begin{theorem}\label{thm:finite}
Let $(W,S)$ be a Coxeter system, and let $w,w'\in \Aut(\Sigma)$ be cyclically reduced and of finite order. If $w,w'$ are conjugate, then $\Cyc(w)=\Cyc(w')$ if and only if $\supp(w)=\supp(w')$. 
\end{theorem}
\begin{proof}
The forward implication follows from Lemma~\ref{lemma:cyclicshit_samesupport}. Assume now that $\supp(w)=\supp(w')$.

Let $x\in W$ be such that $w'=xwx\inv$, and set $J:=\supp(w)=\supp(w')$. Note that $W_J$ is finite by Lemma~\ref{lemma:cyclredfiniteordersuppspherical}. Let $\delta\in\Aut(W,S)$ and $u,v\in W$ be such that $w=u\delta$ and $w'=v\delta$. By Lemma~\ref{lemma:241He}, we can write $x=x'x_1$ with $x'\in W_J$ and $x_1\in W$ of minimal length in $W_Jx_1W_J$, such that $\delta(x_1)=x_1$ and $x_1\Pi_J=\Pi_J$. Let $\sigma:=\kappa_{x_1\inv}|_J\co W_J\to W_J:z\mapsto x_1zx_1\inv$, so that $\sigma(u)=(x')\inv v \delta(x')$. 

Since $\delta(x_1)=x_1$, the diagram automorphisms $\delta|_{W_J}$ and $\sigma$ of $W_J$ commute. Moreover, by Lemma~\ref{lemma:C1C2intermediate}(1,3) (applied to $\delta:=\id$, $J=K:=J$ and $x:=x_1$), $\sigma$ stabilises every component $I$ of $J$ that is not of type $A_m$ for some $m\geq 1$ (note that $\sigma(I)\cap I\neq\varnothing\Rightarrow\sigma(I)=I$ as $I,\sigma(I)$ are components of $J$), and is the identity on each component of $J$ of type $F_4$. Similarly, $\sigma$ stabilises every component $I$ of $J$ such that $\delta^r|_I\neq\id$, where $r\geq 1$ is minimal so that $\delta^r(I)=I$, as follows from Lemma~\ref{lemma:C1C2intermediate}(2) (applied to $\delta:=\delta^r$, $J=K:=J$ and $x:=x_1$).

We may thus apply Proposition~\ref{prop:useP3} (with $W:=W_J$) to conclude that $u,v$ are $\delta$-conjugate in $W_J$ (note that the $\delta$-conjugacy class of $u$ in $W_J$ is cuspidal by Proposition~\ref{prop:He07}(1)), and hence that $\Cyc(w)=\Cyc(w')$ by Proposition~\ref{prop:He07}(2).
\end{proof}

\begin{corollary}\label{corollary:1stronglyconjugate}
Let $w,w'\in \Aut(\Sigma)$ be cyclically reduced, conjugate, and of finite order. Then there exists $a\in W$ with $\ell(a\inv w)=\ell(a)+\ell(w)=\ell(wa)$ such that $\Cyc(a\inv wa)=\Cyc(w')$.
\end{corollary}
\begin{proof}
Let $x\in W$ be such that $w'=xwx\inv$, and set $J:=\supp(w)$ and $J':=\supp(w')$. By Lemma~\ref{lemma:241He}, we can write $x=x'x_1$ with $x'\in W_{J'}$ and $x_1\in W$ of minimal length in $W_{J'}x_1W_J$ and such that $x_1\Pi_J=\Pi_{J'}$. In particular, $$\ell(x_1w)=\ell(x_1)+\ell(w)=\ell(x_1)+\ell(w\inv)=\ell(x_1w\inv)=\ell(wx_1\inv)$$
and $\ell(x_1wx_1\inv)=\ell(w)$. Hence $x_1wx_1\inv =(x')\inv w'x'$ and $w'$ are cyclically reduced and have the same support $J'$, and therefore belong to the same cyclic shift class by Theorem~\ref{thm:finite}. We may thus set $a:=x_1\inv$.
\end{proof}


\section{The structural conjugation graph}
Throughout this section, we fix a Coxeter system $(W,S)$.

\subsection{\texorpdfstring{$K$-conjugation}{K-conjugation}}\label{subsection:Kconjugation}

In this subsection, we introduce an ``elementary conjugation operation'', which may be thought of as a refinement of the tight conjugation operation introduced in \cite[Definition~3.4]{conjCox}. We start by recalling the latter notion.

\begin{definition}
Two elements $u,v\in W$ are called\index{Elementarily tightly conjugate} {\bf elementarily tightly conjugate} if $\ell(u)=\ell(v)$ and one of the following holds:
\begin{enumerate}
\item
$v$ is a cyclic shift of $u$;
\item
there exist a spherical subset $K\subseteq S$ and an element $x\in W_K$ such that $u,v$ normalise $W_K$ and $v=x\inv ux$ with either $\ell(x\inv u)=\ell(x)+\ell(u)$ or $\ell(ux)=\ell(u)+\ell(x)$.
\end{enumerate}
In particular, if $u,v\in W$ are cyclically reduced and in different cyclic shift classes, then they are elementarily tightly conjugate if and only if the above condition (2) holds.
\end{definition}

Here is the announced conjugation operation.

\begin{definition}
Let $u,v\in\Aut(\Sigma)$ be conjugate elements. Given a spherical subset $K\subseteq S$, we call $u,v$ {\bf $K$-conjugate}\index{K-conjugation@$K$-conjugation} if $u,v$ normalise $W_K$ and $v=w_0(K)uw_0(K)$. In this case, we write\index[s]{-03@ $\stackrel{K}{\too}$ ($K$-conjugation)} $u\stackrel{K}{\too}v$.
\end{definition}

\begin{remark}
Note that there might be several subsets $K$ for which $u,v\in\Aut(\Sigma)$ are $K$-conjugate. And in general, there is no unique minimal $K$ for which they are: consider for instance the Coxeter group of type $A_3^{(1)}$, with $S=\{s_0,s_1,s_2,s_3\}$ labelled as on Figure~\ref{figure:TableAFF}. Set $K_1=\{s_0,s_1,s_2\}$ and $K_2=\{s_0,s_3,s_2\}$. Then $u=s_0$ and $v=s_2$ are both $K_1$- and $K_2$-conjugate, but they are not $(K_1\cap K_2)$-conjugate.
\end{remark}

Before comparing the two notions in Lemma~\ref{lemma:Kconjugate_tightlyconjugate} below, we make a few observations.

\begin{lemma}\label{lemma:Kconjugatesame length}
If $u,v\in\Aut(\Sigma)$ are $K$-conjugate for some $K\subseteq S$, then $\ell(u)=\ell(v)$. 
\end{lemma}
\begin{proof}
By assumption, $u$ normalises $W_K$ and $v=w_0(K)uw_0(K)$. Write $u=u_Kn_K$ with $u_K\in W_K$ and $n_K\in\widetilde{N}_K$ (see Remark~\ref{remark:NWWINIWI_AutSigma}). As $n_K$ normalises $K$, it commutes with $w_0(K)$, and hence $v=w_0(K)uw_0(K)=\op_K(u_K)\cdot n_K$ has length
$\ell(v)=\ell(\op_K(u_K))+\ell(n_K)=\ell(u_K)+\ell(n_K)=\ell(u)$, as desired.
\end{proof}

\begin{lemma}\label{lemma:nKdeltanK}
Let $K\subseteq S$ and $n_K\in N_K$. Let also $\delta_{n_K}\in\Aut(W_K,K)$ be defined by $\delta_{n_K}(a):=n_Kan_K\inv$ for all $a\in W_K$. Let $a,b\in W_K$.
\begin{enumerate}
\item
If $an_K$ is cyclically reduced in $W$, then $a\delta_{n_K}$ is cyclically reduced in $W_K$. 
\item
If $a\delta_{n_K}\to b\delta_{n_K}$ in $W_K$, then $an_K\to bn_K$ in $W$.
\item
If $b\delta_{n_K}=w_0(L)a\delta_{n_K}w_0(L)$ for some $\delta_{n_K}$-invariant subset $L\subseteq K$, then $bn_K=w_0(L)an_Kw_0(L)$.
\end{enumerate}
\end{lemma}
\begin{proof}
(1) Assume that $a\delta_{n_K}$ is not cyclically reduced, and let $x\in W_K$ be such that $\ell(x\inv a\delta_{n_K}x)=\ell(x\inv a\delta_{n_K}(x))<\ell(a)$. Then $x\inv an_Kx=x\inv a\delta_{n_K}(x)\cdot n_K$ has length $\ell(x\inv a\delta_{n_K}(x))+\ell(n_K)<\ell(a)+\ell(n_K)=\ell(an_K)$, and hence $an_K$ is not cyclically reduced.

(2) This follows from the fact that if $a\delta_{n_K}\stackrel{s}{\to} sa\delta_{n_K}s=sa\delta_{n_K}(s)\delta_{n_K}$ for some $s\in K$, then $\ell(sa\delta_{n_K}(s)\cdot n_K)=\ell(sa\delta_{n_K}(s))+\ell(n_K)\leq \ell(a)+\ell(n_K)=\ell(an_K)$, and hence $a\cdot n_K\stackrel{s}{\to}san_Ks=sa\delta_{n_K}(s)\cdot n_K$.

(3) This follows from the fact that $n_K$ normalises $W_L$ and hence commutes with $w_0(L)$.
\end{proof}

\begin{lemma}\label{lemma:Kconjugate_tightlyconjugate}
Let $\OOO$ be a conjugacy class in $W$. Let $u,v\in\OOO^{\min}$ be such that $u$ normalises $W_K$ for some spherical subset $K\subseteq S$, and $v=x\inv ux$ for some $x\in W_K$. Write $u=u_Kn_K$ with $n_K\in N_K$ and $u_K\in W_K$, and set $I:=\supp(u_K\delta_{n_K})\subseteq K$, where $\delta_{n_K}\in\Aut(W_K,K)$ is defined by $\delta_{n_K}(a):=n_K an_K\inv$ for $a\in W_K$. Then:
\begin{enumerate}
\item
there exists $v'\in\OOO^{\min}$ with $\Cyc(v')=\Cyc(v)$ such that $u,v'$ are elementarily tightly conjugate.
\item
if $u,v$ are $K$-conjugate, one can choose $v'=w_0(J)vw_0(J)$ where $J:=\op_K(I)\subseteq K$.
\end{enumerate}
\end{lemma}
\begin{proof}
Note that $v=x\inv u_K\delta_{n_K}(x)\cdot n_K$. Since $u_K\delta_{n_K}$ and $x\inv u_K\delta_{n_K} x$ are cyclically reduced in $W_K$ by Lemma~\ref{lemma:nKdeltanK}(1), Corollary~\ref{corollary:1stronglyconjugate} yields some $a\in W_K$ such that $x\inv u_K\delta_{n_K} x\to a\inv u_K\delta_{n_K}a$ and $\ell(a\inv u_K)=\ell(a)+\ell(u_K)$. In particular, $v\to v':=a\inv ua=a\inv u_K\delta_{n_K}(a)n_K$ by Lemma~\ref{lemma:nKdeltanK}(2), and $\ell(a\inv u)=\ell(a\inv u_K)+\ell(n_K)=\ell(a)+\ell(u_K)+\ell(n_K)=\ell(a)+\ell(u)$, yielding (1).

Assume now that $x=w_0(K)$ (equivalently, $u,v$ are $K$-conjugate), so that $J=\supp(x\inv u_K\delta_{n_K}x)$. Note from the proof of Corollary~\ref{corollary:1stronglyconjugate} that $a$ is the unique element of minimal length in $W_IxW_J$. Hence $a=w_0(I)w_0(K)=w_0(K)w_0(J)$ (as $\ell(w_0(I)w_0(K))=\ell(w_0(K))-\ell(w_0(I))$), and  $v'=w_0(J)w_0(K)uw_0(K)w_0(J)=w_0(J)vw_0(J)$, yielding (2).
\end{proof}


\subsection{Structural and tight conjugation graphs}\label{subsection:SATCG}

We next introduce the main protagonists of Proposition~\ref{propintro:finiteorder} and Theorem~\ref{thmintro:graphisomorphism}.

\begin{definition}
Let $w\in\Aut(\Sigma)$, which we write as $w=w'\delta$ with $w'\in W$ and $\delta\in\Aut(W,S)$, and let $\OOO:=\OOO_w=\{x\inv w x \ | \ x\in W\}\subseteq W\delta$ be its conjugacy class. Then $\OOO^{\min}$ is the disjoint union of finitely many cyclic shift classes $C_1,\dots,C_k$ (of cyclically reduced conjugates of $w$). 

We define the\index{Structural conjugation graph} {\bf structural conjugation graph} associated to $\OOO$ as the graph\index[s]{KO@$\KKK_{\OOO}$ (structural conjugation graph associated to $\OOO$)} $\KKK_{\OOO}$ with vertex set $\{C_i \ | \ 1\leq i\leq k\}$, and with an edge between $C_i$ and $C_j$ ($i\neq j$) if there exist $u\in C_i$ and $v\in C_j$ and a spherical subset $K\subseteq S$ (assumed to be $\delta$-invariant if $w$ has finite order) such that $u,v$ are $K$-conjugate. We also call $C_i$ and $C_j$ {\bf $K$-conjugate}\index{K-conjugation@$K$-conjugation} if there exist $u\in C_i$ and $v\in C_j$ such that $u,v$ are $K$-conjugate.

When $\OOO$ is a conjugacy class in $W$ (i.e. when $w\in W$), we also consider the\index{Tight conjugation graph} {\bf tight conjugation graph} $\KKK^{\mathrm{t}}_{\OOO}$\index[s]{KtO@$\KKK^{\mathrm{t}}_{\OOO}$ (tight conjugation graph associated to $\OOO$)} associated to $\OOO$, with same vertex set $\{C_i \ | \ 1\leq i\leq k\}$ as $\KKK_{\OOO}$, and with an edge between $C_i$ and $C_j$ ($i\neq j$) if there exist $u\in C_i$ and $v\in C_j$ that are elementarily tightly conjugate.
\end{definition}

\begin{definition}\label{definition:Kdelta}
Let $\delta\in\Aut(W,S)$. Let $\SSS_{\delta}$\index[s]{Sdelta@$\SSS_{\delta}$ (set of $\delta$-invariant spherical subsets of $S$)} denote the set of $\delta$-invariant spherical subsets of $S$. Given $K\in\SSS_{\delta}$, we call two subsets $I,J\in\SSS_{\delta}$ {\bf $K$-conjugate}\index{K-conjugation@$K$-conjugation} if $K\supseteq I$ and $J=\op_{K}(I)$ (in which case $K\supseteq J$ and $I=\op_{K}(J)$). In this case, we write\index[s]{-03@ $\stackrel{K}{\too}$ ($K$-conjugation)} $I\stackrel{K}{\too}J$.

We define the graph\index[s]{Kdelta@ $\KKK_{\delta}$, $\KKK_{\delta,W}$ (graph with vertex set $\SSS_{\delta}$)} $\KKK_{\delta}=\KKK_{\delta,W}$ with vertex set $\SSS_{\delta}$, and with an edge between $I,J\in\SSS_{\delta}$ if $I,J$ are $K$-conjugate for some $K\in\SSS_{\delta}$.
For $I\in\SSS_{\delta}$, we also let\index[s]{Kdelta0I@$\KKK_{\delta}^0(I)$ (connected component of $I$ in $\KKK_{\delta}$)} $\KKK_{\delta}^0(I)$ denote the connected component of $I$ in $\KKK_{\delta}$. We call a path $I=I_0,I_1,\dots,I_k=J$ in $\KKK_{\delta}$ {\bf spherical}\index{Spherical path} if there exist $K_1,\dots,K_k\in\SSS_{\delta}$ with $\bigcup_{i=1}^kK_i\in \SSS_{\delta}$ such that $$I=I_0\stackrel{K_1}{\too}I_1\stackrel{K_2}{\too}\dots\stackrel{K_k}{\too}I_k=J.$$

Finally, we let\index[s]{Kdeltabar@ $\overline{\KKK}_{\delta}$, $\overline{\KKK}_{\delta,W}$ (version of $\KKK_{\delta}$ with more edges)} $\overline{\KKK}_{\delta}=\overline{\KKK}_{\delta,W}$ denote the graph with vertex set $\SSS_{\delta}$ and with an edge between $I,J\in\SSS_{\delta}$ if they are connected by a spherical path in $\KKK_{\delta}$. We also let\index[s]{Kdelta0Ibar@$\overline{\KKK}_{\delta}^0(I)$ (connected component of $I$ in $\overline{\KKK}_{\delta}$)} $\overline{\KKK}_{\delta}^0(I)$ denote the connected component of $I\in\SSS_{\delta}$ in $\overline{\KKK}_{\delta}$. In other words, $\overline{\KKK}_{\delta}^0(I)$ is obtained from $\KKK_{\delta}^0(I)$ by declaring that each set of vertices of a spherical path in $\KKK_{\delta}^0(I)$ forms a clique.
\end{definition}


\subsection{The structural conjugation graph of a finite order element}\label{subsection:TSCGOAFOE}

We are now ready to prove Proposition~\ref{propintro:finiteorder}.

\begin{theorem}\label{thm:main_finite_order}
Let $(W,S)$ be a Coxeter system, and let $w\in\Aut(\Sigma)$ be cyclically reduced and of finite order. Write $w=w'\delta$ with $w'\in W$ and $\delta\in\Aut(W,S)$. Then there is graph isomorphism
$$\varphi_w\co\KKK_{\OOO_w}\to \KKK_{\delta}^0(\supp(w))$$
defined on the vertex set of $\KKK_{\OOO_w}$ by the assignment
$$\Cyc(u)\mapsto \supp(u)\quad\textrm{for any $u\in \OOO_w^{\min}$.}$$
Moreover, if $u,v\in\OOO_w^{\min}$ are such that $\supp(u)$ and $\supp(v)$ are $K$-conjugate for some $K\in\SSS_{\delta}$, then $\Cyc(u)$ and $\Cyc(v)$ are $K$-conjugate.
\end{theorem}
\begin{proof}
Note that, in view of Theorem~\ref{thm:finite}, the assignment $\Cyc(u)\mapsto \supp(u)$ for $u\in \OOO_w^{\min}$ yields a well-defined injective map $\varphi_w$ from the vertex set of $\KKK_{\OOO_w}$ to the vertex set of $\KKK_{\delta}$.

We now show that $\varphi_w$ maps an edge of $\KKK_{\OOO_w}$ to an edge of $\KKK_{\delta}$. Let $u,v\in\OOO_w^{\min}$ with $\Cyc(u)\neq \Cyc(v)$ be $K$-conjugate for some $K\in\SSS_{\delta}$, so that $u,v$ normalise $W_K$ and $v=w_0(K)uw_0(K)$. Write $u=u_1\delta$ with $u_1\in W$. Then $W_K=uW_Ku\inv=u_1W_{\delta(K)}u_1\inv=u_1W_Ku_1\inv$, that is, $u_1\in N_W(W_K)$. Write $u_1=u_Kn_K$ with $u_K\in W_K$ and $n_K\in N_K$ (see Lemma~\ref{lemma:NWWINIWI}). Since $n_K\Pi_K=\Pi_K$ (and $n_K$ is of minimal length in $W_Kn_KW_K=W_Kn_K$), Proposition~\ref{prop:Deodhar_twisted}(4) (applied to $\delta:=\id$) implies that either $n_K=1$ or $\supp(n_K)\supseteq K$ (i.e. if $K':=K\cup\{s\}$ is spherical for some $s\in S\setminus K$, then $\supp(w_0(K')w_0(K))=K'\supseteq K$, as $w_0(K')w_0(K)$ maps every root in $\Phi^+_{K'}\setminus\Phi_K$ to a negative root). But in the latter case, $u$ and $v=\op_K(u_K)n_K$ have same support $\supp(n_K)$, and hence belong to the same cyclic shift class by Theorem~\ref{thm:finite}, a contradiction. Thus $n_K=1$, and hence $\supp(u)\subseteq K$ and $\supp(v)=\op_K(\supp(u))$, so that $\supp(u)$ and $\supp(v)$ are connected by an edge in $\KKK_{\delta}$.

Note next that if there is a path $\gamma$ in $\KKK_{\delta}$ starting at $\supp(w)$, say $\supp(w)=I_0,I_1,\dots,I_k$ where $I_i\in\SSS_{\delta}$ is such that $I_i=\op_{K_i}(I_{i-1})$ for some $K_i\in\SSS_{\delta}$ containing $I_{i-1}$, then the sequence $\Cyc(w)=\Cyc(v_0),\Cyc(v_1),\dots,\Cyc(v_k)$ defined by $v_0:=w$ and $v_i:=w_0(K_i)v_{i-1}w_0(K_i)$ for $i\in\{1,\dots,k\}$ corresponds to a path in $\KKK_{\OOO_w}$ mapped to $\gamma$ under $\varphi_w$. This shows that the image of $\varphi_w$ contains $\KKK_{\delta}^0(\supp(w))$, that $\varphi_w\inv$ maps an edge of $\KKK_{\delta}^0(\supp(w))$ to an edge of $\KKK_{\OOO_w}$, and that the second statement of the theorem holds.

It now remains to see that $\KKK_{\OOO_w}$ is connected. Let $u,v\in\OOO_w^{\min}$, and set $I_u:=\supp(u)$ and $I_v:=\supp(v)$. Let $x\in W$ be such that $v=xux\inv$. By Lemma~\ref{lemma:241He}, we can write $x=x'x_1$ for some $x'\in W_{I_v}$ and some $x_1\in W$ of minimal length in $W_{I_v}x_1W_{I_u}$ such that $\delta(x_1)=x_1$ and $x_1\Pi_{I_u}=\Pi_{I_v}$. Note that $v':=(x')\inv vx'$ and $v$ have same support $I_v$ (and are cyclically reduced), and hence $\Cyc(v')=\Cyc(v)$ by Theorem~\ref{thm:finite}. On the other hand, Proposition~\ref{prop:Deodhar_twisted} yields a sequence $I_u=I_0,I_1,\dots,I_k=I_v$ of $\delta$-invariant spherical subsets of $S$ such that 
\begin{enumerate}
\item
$I_i=\op_{K_i}(I_{i-1})$ for all $i\in\{1,\dots,k\}$, where $K_i:=I_{i-1}\cup I_i$ is spherical;
\item 
setting $u_0:=u$ and $u_i:=w_0(K_i)w_0(I_{i-1})u_{i-1}w_0(I_{i-1})w_0(K_i)$ for $i\in\{1,\dots,k\}$, we have $u_k=v'$.
\end{enumerate}
Note that $\Cyc(u_{i-1})$ and $\Cyc(u_i)$ are $K_i$-conjugate for each $i$, as $\Cyc(u_{i-1})=\Cyc(w_0(I_{i-1})u_{i-1}w_0(I_{i-1}))$ by Theorem~\ref{thm:finite} and as $w_0(I_{i-1})u_{i-1}w_0(I_{i-1})$ and $u_i$ are $K_i$-conjugate by construction. Hence $\Cyc(u)$ and $\Cyc(v)$ are connected by a path in $\KKK_{\OOO_w}$, as desired.
\end{proof}

\begin{corollary}\label{corollary:finiteorderconjugation}
Let $(W,S)$ be a Coxeter system, and let $u,v\in\Aut(\Sigma)$ be cyclically reduced and of finite order. Write $u=u'\delta$ and $v=v'\delta'$ with $u',v'\in W$ and $\delta,\delta'\in\Aut(W,S)$, and set $I_u:=\supp(u)$ and $I_v:=\supp(v)$. Then $u$ and $v$ are conjugate in $W$ if and only if the following assertions hold:
\begin{enumerate}
\item
$\delta=\delta'$;
\item
there exists some $x\in W$ with $\delta(x)=x$, of minimal length in $W_{I_v}xW_{I_u}$, such that $x\Pi_{I_u}=\Pi_{I_v}$;
\item
$\Cyc(u)=\Cyc(\kappa_x(v))$ for some $x$ as in (2).
\end{enumerate}
\end{corollary}
\begin{proof}
This readily follows from Theorem~\ref{thm:main_finite_order} (and Lemma~\ref{lemma:241He}).
\end{proof}


\section{Combinatorial minimal displacement sets}
Throughout this section, we fix a Coxeter system $(W,S)$. The purpose of this section is to provide the  geometric tools necessary for the proof of Theorem~\ref{thmintro:graphisomorphism}.


\subsection{Geometric interpretation of cyclic shifts}\label{subsection:GIOCC}

\begin{definition}\label{definition:piw}
Let $w\in \Aut(\Sigma)$. In order to give a geometric reformulation of the properties of $\OOO_w$ and $\OOO_w^{\min}$, we define a parametrisation of $\OOO_w$ by the set $\Ch(\Sigma)$ of chambers of $\Sigma=\Sigma(W,S)$ via the (surjective) map\index[s]{piw@$\pi_w$ (parametrisation of conjugates of $w$ by chambers)}
$$\pi_w\co\Ch(\Sigma)\to\OOO_w:vC_0\mapsto v\inv wv\quad\textrm{($v\in W$)}.$$
Note that this parametrisation is unique modulo the centraliser\index[s]{ZWw@$\ZZZ_W(w)$ (centraliser of $w$ in $W$)} $\ZZZ_W(w)$ of $w$ in $W$, that is, $\pi_w(vC_0)=\pi_w(uC_0)\iff v\ZZZ_W(w)=u\ZZZ_W(w)$.
\end{definition}

We first define the geometric counterpart of the set $\OOO_w^{\min}$.

\begin{definition}
Given $w\in \Aut(\Sigma)$, we define its {\bf combinatorial minimal displacement set}\index{Combinatorial minimal displacement set}\index[s]{CombiMin@$\CMin(w)$, $\CMin_{\Sigma}(w)$ (combinatorial minimal displacement set of $w$)}
$$\CMin(w)=\CMin_{\Sigma}(w):=\{C\in\Ch(\Sigma) \ | \ \textrm{$\dc(C,wC)$ is minimal}\}.$$
In other words, $\CMin(w)$ is the inverse image of $\OOO_w^{\min}$ under $\pi_w$.
\end{definition}

\begin{remark}
For any $w\in \Aut(\Sigma)$ and $v\in W$, we have
$$\CMin(v\inv wv)=v\inv\CMin(w)\quad\textrm{and}\quad \Min(v\inv wv)=v\inv \Min(w).$$
\end{remark}

We next define the geometric counterpart of the cyclic shift operations.

\begin{definition}\label{definition:wdecreasinggallery}
Let $w\in\Aut(\Sigma)$. Call a sequence of chambers $\Gamma=(D_0,D_1,\dots,D_k)$ in $\Sigma$\index{decreasing@$w$-decreasing gallery} {\bf $w$-decreasing} if $\dc(D_i,wD_i)\leq\dc(D_{i-1},wD_{i-1})$ for all $i\in\{1,\dots,k\}$.
\end{definition}

\begin{lemma}\label{lemma:wdecresing_cs}
Let $w\in\Aut(\Sigma)$, and let $\Gamma=(D_0,D_1,\dots,D_k)$ be a gallery in $\Sigma$, of type $(s_1,\dots,s_k)\in S^k$. Then the following assertions are equivalent:
\begin{enumerate}
\item
$\Gamma$ is $w$-decreasing.
\item
$\pi_w(D_0)\stackrel{s_1}{\to}\pi_w(D_1)\stackrel{s_2}{\to}\dots \stackrel{s_k}{\to}\pi_w(D_k)$.
\end{enumerate}
\end{lemma}
\begin{proof}
This follows from the fact that $\dc(D,wD)=\ell_S(\pi_w(D))$ for any $D\in\Ch(\Sigma)$.
\end{proof}

\begin{lemma}\label{lemma:cyclicshift_geometric}
Let $w\in\Aut(\Sigma)$. Let $C,D\in\CMin(w)$. Then the following assertions are equivalent:
\begin{enumerate}
\item
$\pi_w(C)\to\pi_w(D)$;
\item
There exist $v\in\ZZZ_W(w)$ and a gallery $\Gamma\subseteq\CMin(w)$ from $C$ to $vD$.
\end{enumerate}
\end{lemma}
\begin{proof}
The implication (2)$\implies$(1) follows from Lemma~\ref{lemma:wdecresing_cs}(1)$\Rightarrow$(2): $\Gamma$ is $w$-decreasing, and hence $\pi_w(C)\to\pi_w(vD)=\pi_w(D)$. 

To prove the converse, it is sufficient to consider the case where $\pi_w(C)\stackrel{s}{\to}\pi_w(D)$ for some $s\in S$. Write $C=aC_0$ and $D=bC_0$ for some $a,b\in W$, and let $E:=asC_0$ be the chamber $s$-adjacent to $C$. By assumption, $\pi_w(D)=b\inv wb=sa\inv was=\pi_w(E)$ and $\ell_S(b\inv wb)\leq\ell_S(a\inv wa)$. Thus, $\dc(E,wE)\leq\dc(C,wC)$, so that $E\in\CMin(w)$, and there exists $v:=asb\inv\in\ZZZ_W(w)$ such that $E=vD$, as desired.
\end{proof}

We conclude this subsection by giving a geometric interpretation of the support of $\pi_w(C)$ for $w\in\Aut(\Sigma)$ and $C\in\Ch(\Sigma)$.

\begin{lemma}\label{lemma:geominterp_supppiwC}
Let $w\in\Aut(\Sigma)$, which we write $w=w'\delta$ for some $w'\in W$ and $\delta\in\Aut(W,S)$. Let $C\in\Ch(\Sigma)$ and $L\subseteq S$. Then $w$ stabilises the residue $R_L(C)$ if and only if $\delta(L)=L$ and $L\supseteq \supp(\pi_w(C))$. In particular, $\supp(\pi_w(C))$ is the type of the smallest $w$-invariant residue containing $C$.
\end{lemma}
\begin{proof}
Set $I:=\supp(\pi_w(C))$. Write $C=aC_0$ with $a\in W$. Note that $w$ stabilises $R_{I}(C)$, as $\pi_{w}(C)=a\inv w a$ stabilises $R_{I}(C_0)=a\inv R_{I}(C)$.

If $\delta(L)=L$, then $w$ maps an $L$-gallery to an $L$-gallery. If, moreover, $L\supseteq I$, then $R_{L}(C)\supseteq R_{I}(C)$, and hence $w$ stabilises $R_{L}(C)$.

Conversely, if $w$ stabilises $R_L(C)$ (which corresponds to the coset $aW_L$ of $W$), then $aW_L=w'\delta\cdot aW_L=w'\delta(a)W_{\delta(L)}$, so that $\delta(L)=L$ and $a\inv w'\delta(a)=a\inv w a \cdot \delta\inv\in W_L$. In particular, $I=\supp(a\inv wa)\subseteq L$, as desired.
\end{proof}


\subsection{Relations between \texorpdfstring{$\CMin(w)$ and $\Min(w)$}{CombiMin(w) and Min(w)}}

The following two lemmas collect properties of $\CMin(w)$ obtained in  \cite{straight} and \cite{conjCox} that we will need in the paper.

\begin{lemma}\label{lemma:prop34}
Let $w\in \Aut(\Sigma)$ be of infinite order. Let also $C\in\CMin(w)$ and $D\in\Ch(\Sigma)$. Then the following assertions hold:
\begin{enumerate}
\item
There is a $w$-decreasing gallery from $D$ to some chamber $E$ intersecting $\Min(w)$. In particular, there is a gallery $\Gamma$ from $C$ to some chamber $E$ intersecting $\Min(w)$ such that $\Gamma\subseteq\CMin(w)$.
\item
$\proj_{R_x}(C)\in\CMin(w)$ for any $w$-essential $x\in\Min(w)$.
\item
If $C$ and $D$ both contain a point of $\Min(w)$ and are not separated by any wall containing a $w$-axis, then there is a $w$-decreasing gallery $\Gamma$ from $D$ to $C$. If, moreover, $D\in\CMin(w)$, then $\Gamma\subseteq\CMin(w)$.
\item
There exists a $w$-decreasing sequence $D=D_0,D_1,\dots,D_k=C$ of chambers such that for each $i\in\{1,\dots,k\}$, either the chambers $D_{i-1},D_i$ are adjacent, or $D_i\in\CMin(w)$ and both $D_{i-1},D_i$ belong to a residue $R_i$ such that $w$ normalises $\Stab_W(R_i)$. If, moreover, $D\in\CMin(w)$, then $D_i\in\CMin(w)$ for all $i$.
\end{enumerate}
\end{lemma}
\begin{proof}
Note first that the second assertion in (1), (3) and (4) follows from the first and the observation that a $w$-decreasing sequence of chambers starting at a chamber in $\CMin(w)$ is entirely contained in $\CMin(w)$.

(1) follows from \cite[Lemma~5.4]{conjCox}. Indeed, up to conjugating $w$, there is no loss of generality in assuming that $D=C_0$. In the notation of \emph{loc. cit.}, there exists a chamber $E\subseteq \CCC_1^w$ intersecting $\Min(w)$ (\cite[Lemma~5.4]{conjCox} is actually stated for $w\in W$, but the proof applies \emph{verbatim} to $w\in\Aut(\Sigma)$). By construction of $\CCC_1^w$ (see \cite[Definition~5.1]{conjCox}), there is a gallery $\Gamma=(D=D_0,D_1,\dots,D_k=E)$ from $D=C_0$ to $E$ such that $D_i$ is the second chamber of a minimal gallery $\Gamma_i$ from $D_{i-1}$ to $w^{\varepsilon_i}D_{i-1}$ ($\varepsilon_i\in\{\pm 1\}$) for each $i\in\{1,\dots,k\}$. Since a gallery $\Gamma_i'$ from $D_i$ to $w^{\varepsilon_i}D_i$ can be obtained by concatenating the galleries $\Gamma_i\setminus (D_{i-1})$ and $(w^{\varepsilon_i}D_{i-1},w^{\varepsilon_i}D_{i})$, we have
$$\dc(D_i,wD_i)=\dc(D_i,w^{\varepsilon_i}D_i)\leq\ell(\Gamma_i')=\ell(\Gamma_i)=\dc(D_{i-1},wD_{i-1}),$$ 
and hence $\Gamma$ is $w$-decreasing.

(2) follows from \cite[Proposition~6.2]{conjCox}. Indeed, since $x$ is $w$-essential, it has an open neighbourhood $Z\subseteq X$ not meeting any $w$-essential wall (recall that $\WW$ is locally finite). Let $L$ be the $w$-axis through $x$. Up to shrinking $Z$, we may assume that $Z\cap L$ is contained in the simplex $\supp(x)$ of $\Sigma$. We may then apply \cite[Proposition~6.2]{conjCox} to the residue $R_x=R_{\supp(x)}$. (Again, \cite[Proposition~6.2]{conjCox} is actually stated for $w\in W$, but the proof applies \emph{verbatim} to $w\in\Aut(\Sigma)$: the only ingredients in the proof of \emph{loc. cit.} are Lemma~\ref{lemma:NWWINIWI} --- which still holds by Remark~\ref{remark:NWWINIWI_AutSigma} --- and \cite[Lemmas~4.3 and 4.1]{straight} --- whose proof also applies \emph{verbatim} to elements of $\Aut(\Sigma)$.)

(3) This can be extracted from the proof of \cite[Proposition~6.3]{conjCox}; we repeat here the argument. Up to conjugating $w$, there is no loss of generality in assuming that $D=C_0$. Let $x\in C_0\cap\Min(w)$ and $y\in C\cap\Min(w)$. Since $\Min(w)$ is convex, we have $[x,y]\subseteq \Min(w)$. Let $\Gamma=(C_0,C_1,\dots,C_k=C)$ be a minimal gallery from $C_0$ to $C$ containing $[x,y]$ (see \cite[Lemma~3.1]{straight}). We will show inductively on $i$ that the chamber $C_i$ belongs to the complex $\CCC_1^w$ from \cite[Definition~5.1]{conjCox}; the claim will then follow as in the proof of (1). For $i=0$, this holds by definition. Assume now that $C_{i-1}$ is a chamber of $\CCC_1^w$ for some $\in\{1,\dots,k\}$, and let $x_i\in [x,y]\cap (C_{i-1}\cap C_{i})$. Thus, $x_i\in\Min(w)$, and $x_i$ belongs to the wall $m$ separating $C_{i-1}$ from $C_{i}$. By assumption, the $w$-axis $L$ through $x_i$ intersects $m$ in a single point (namely, $\{x_i\}$). In particular, there is some $\varepsilon\in\{\pm 1\}$ such that $w^{\varepsilon}C_{i-1}$ and $C_{i-1}$ are separated by $m$. This means that $C_i$ is on a minimal gallery from $C_{i-1}$ to $w^{\varepsilon}C_{i-1}$, and hence $C_i$ is a chamber of $\CCC_1^w$ (see \cite[Definition~5.1 (CM1)]{conjCox}), as desired.

(4) By (1), we may assume that $C,D$ intersect $\Min(w)$. The claim is then precisely the content of the proof of \cite[Proposition~6.3]{conjCox}. Indeed, up to conjugating $w$, there is no loss of generality in assuming that $D=C_0$. Let $v\in W$ with $C=vC_0$. The first paragraph of the proof of \cite[Proposition~6.3]{conjCox} reduces to the case where $C_0$ and $vC_0$ intersect $\Min(w)$ (a reduction step we have already performed). Under the assumption that $vC_0\in\CMin(w)$ (which holds here by assumption), the proof of \cite[Proposition~6.3]{conjCox} concludes (see the last paragraph of the proof) that $vC_0$ belongs to the complex $\CCC^w$ defined in \cite[Definition~5.1]{conjCox} (again, although \cite[Proposition~6.3]{conjCox} is stated for $w\in W$, everything goes through \emph{verbatim} for $w\in\Aut(\Sigma)$). But the chambers of $\CCC^w$ are constructed inductively starting from $C_0$ and adding, for any previously constructed chamber $C'\in\Ch(\CCC^w)$, either a chamber $D'$ adjacent to $C'$ on a minimal gallery from $C'$ to $w^{\pm 1}C'$ (in particular, $(C',D')$ is a $w$-decreasing gallery, see (1)), or a chamber $D'\in\CMin(w)$ belonging together with $C'$ to a common spherical residue $R'$ such that $w$ normalises $\Stab_W(R')$ (see also the steps (I) and (II) in the proof of \cite[Proposition~5.6]{conjCox}). This yields the claim.
\end{proof}

\begin{remark}\label{rem:TheoremAconjCoxAutSigma}
The statement of \cite[Theorem~A(1)]{conjCox} remains valid for elements of $\Aut(\Sigma)$: if $w\in \Aut(\Sigma)$, then there exists an element $w'\in\OOO_w^{\min}$ such that $w\to w'$ (in other words, $\Cyc_{\min}(w)=\Cyc(w)\cap\OOO_w^{\min}$).

Indeed, in view of Remark~\ref{remark:NWWINIWI_AutSigma}, the proof of \cite[Lemma~5.5]{conjCox} applies \emph{verbatim} to elements $w\in\Aut(\Sigma)$. Together with Lemma~\ref{lemma:wdecresing_cs}, these are precisely the two ingredients allowing to reformulate the (geometric) statement of Lemma~\ref{lemma:prop34}(4) into the (combinatorial) statement of \cite[Theorem~A(1) and (2)]{conjCox} (see also \cite[Proposition~5.6]{conjCox}).
\end{remark}

\begin{lemma}\label{lemma:prop34fin}
Let $w\in \Aut(\Sigma)$ be of finite order. Let also $C\in\CMin(w)$. Then the following assertions hold:
\begin{enumerate}
\item
There is a gallery $\Gamma\subseteq\CMin(w)$ from $C$ to some chamber $D$ intersecting $\Min(w)$.
\item
$\proj_{R}(C)\in\CMin(w)$ for any $w$-stable residue $R$.
\end{enumerate}
\end{lemma}
\begin{proof}
(1) This follows from the proof of \cite[Proposition~3.4]{straight}. Indeed, up to conjugating $w$, there is no loss of generality in assuming that $C=C_0$. Let $\CCC_w$ be the subcomplex of $X$ constructed in \cite[\S 3]{straight}: it is the smallest subcomplex of $X$ containing $C_0$ and such that for any chamber $D\subseteq\CCC_w$ and any $\varepsilon\in\{\pm 1\}$, any minimal gallery $\Gamma=(D=D_0,D_1,\dots,D_k=w^{\varepsilon}D)$ from $D$ to $w^{\varepsilon}D$ is entirely contained in $\CCC_w$. Note that for any $i\in\{1,\dots,k\}$, we have a gallery $\Gamma_i=(D_i,D_{i+1},\dots,D_k=w^{\varepsilon}D_0,w^{\varepsilon}D_1,\dots,w^{\varepsilon}D_i)$ from $D_i$ to $w^{\varepsilon}D_i$ of length $\ell(\Gamma)$, and hence $\dc(D_i,wD_i)\leq\dc(D,wD)$. Reasoning inductively, this implies that 
\begin{equation}\label{eqn:lemma3.3}
\dc(D,wD)\leq\dc(C_0,wC_0)\quad\textrm{for any chamber $D$ of $\CCC_w$}.
\end{equation}
The proof of \cite[Proposition~3.4]{straight} then applies \emph{verbatim} (with the reference to Lemma~3.3 of \emph{loc. cit.} replaced by (\ref{eqn:lemma3.3})) and implies that there is a chamber $D$ of $\CCC_w$ intersecting $\Min(w)$. Let $\Gamma$ be any gallery from $C_0=C$ to $D$ contained in $\CCC_w$. Since $C\in\CMin(w)$ by assumption, (\ref{eqn:lemma3.3}) implies that $\Gamma\subseteq\CMin(w)$, as desired.

(2) Since $wR=R$, we have $$\dc(\proj_{R}(C),w\proj_{R}(C))=\dc(\proj_{R}(C),\proj_{R}(wC))\leq\dc(C,wC),$$
where the inequality follows from the fact that projections on residues do not increase the chamber distance (see e.g. \cite[Corollary~5.39]{BrownAbr}). Since $C\in\CMin(w)$, the claim follows.
\end{proof}

Note that Lemma~\ref{lemma:prop34fin}(2) implies that $\Min(w)\subseteq\CMin(w)$ for any $w\in\Aut(\Sigma)$ of finite order (since for any $x\in\Min(w)$, the residue $R=R_x$ is $w$-stable). Proposition~\ref{prop:basicprop_finiteCox} below is a sort of converse to this statement. To prove it, we first need the following technical lemma.

\begin{lemma}\label{lemma:basic_to_solve}
Let $R$ be a residue in $\Sigma$, and $D_1,\dots,D_k\in\Ch(\Sigma)$ be such that $D_i$ and $D_{i+1}$ are only separated by walls of $R$ for each $i\in\{1,\dots,k-1\}$. Let $R'$ be the smallest residue containing $D_1,\dots,D_k$. Then $\Stab_W(R')\subseteq\Stab_W(R)$.
\end{lemma}
\begin{proof}
Without loss of generality, we may assume that $D_1=C_0$. For each $i\in\{1,\dots,k\}$, let $w_i\in W$ be such that $D_i=w_iC_0$, and set $I_i:=\supp(w_i)\subseteq S$. Thus, $D_i$ belongs to the standard residue $R_{I_i}$ of type $I_i$, and $R_i$ is the smallest residue containing $D_1$ and $D_i$.

By assumption, the walls separating $C_0$ from $w_iC_0$ are walls of $R$ (i.e. any wall separating $D_1$ from $D_i$ also separates $D_j$ from $D_{j+1}$ for some $j\in\{1,\dots,i-1\}$). On the other hand, if $w_i=s_1\dots s_d$ is a reduced decomposition for $w_i$ ($s_j\in I_i$), then $w_i=t_d\dots t_1$, where $t_j=(s_1\dots s_{j-1})s_j(s_{j-1}\dots s_1)$ is the reflection across the $j$-th wall crossed by the minimal gallery from $C_0$ to $w_iC_0$ of type $(s_1,\dots,s_d)$. Thus, $t_1,\dots,t_d\in\Stab_W(R)$ and since $\langle t_1,\dots,t_d\rangle=\langle s_1,\dots,s_d\rangle=W_{I_i}$, we deduce that $W_{I_i}\subseteq\Stab_W(R)$.

This shows that $W_I\subseteq \Stab_W(R)$, where $I=\bigcup_{i=1}^kI_i$. Since $R_I$ is the smallest residue containing $D_1,\dots,D_k$, and since $W_I=\Stab_W(R_I)$, the lemma follows.
\end{proof}

\begin{prop}\label{prop:basicprop_finiteCox}
Let $w\in\Aut(\Sigma)$ be of finite order, and let $C\in\CMin(w)$. Then $C$ contains a point $x\in\Min(w)$ with $\Fix_W(x)=\Fix_W(\Min(w))$. Moreover, for any such $x$, the residue $R_x$ is the smallest $w$-invariant residue containing $C$.
\end{prop}
\begin{proof}
Let $k\geq 1$ be the order of $w$. Let $y\in\Min(w)$. Thus, $wR_y=R_y$. Set $C_1:=\proj_{R_y}(C)$, so that $wC_1=\proj_{R_y}(wC)$. Note that $C_1\in\CMin(w)$ by Lemma~\ref{lemma:prop34fin}(2), and hence $\dc(C,wC)=\dc(C_1,wC_1)$. Since no wall of $R_y$ separates $C$ from $C_1$ (resp. $wC_1$ from $wC$), all the walls separating $C_1$ from $wC_1$ also separate $C$ from $wC$. Hence the walls separating $C$ from $wC$ are precisely the walls separating $C_1$ from $wC_1$, and in particular are walls of $R_y$. Setting $D_i:=w^{i-1}C$ for all $i\in\{1,\dots,k\}$ we deduce that the walls separating $D_i$ from $D_{i+1}$ are walls of $R_y$ for each $i$ (as $w$ stabilises $R_y$), and hence the smallest residue $R'$ containing $D_1,\dots,D_k$ (equivalently, containing $w^{\ZZ}C$) satisfies $\Stab_W(R')\subseteq\Stab_W(R_y)$ by Lemma~\ref{lemma:basic_to_solve}. In particular, $R'$ is spherical and $w$-stable.

Let $\sigma$ be the spherical simplex of $\Sigma$ such that $R'=R_{\sigma}$, and hence also $\Fix_W(\sigma)=\Stab_W(R')\subseteq\Fix_W(y)$. Since $w$ stabilises $\sigma$, it fixes its circumcenter $x\in\sigma$ (see \S\ref{subsection:VBAAOCS}), so that $x\in\Min(w)$, and we have $\Fix_W(x)=\Fix_W(\sigma)\subseteq\Fix_W(y)$. Since $x$ is independent of the choice of $y\in\Min(w)$, we deduce that $\Fix_W(x)\subseteq\bigcap_{y\in\Min(w)}\Fix_W(y)=\Fix_W(\Min(w))$, and hence $\Fix_W(x)=\Fix_W(\Min(w))$, proving the first statement.

Finally, note that $R'=R_x$ is the smallest residue containing $w^{\ZZ}C$. Moreover, if $x'\in C\cap\Min(w)$ is such that $\Fix_W(x')=\Fix_W(\Min(w))$, then $R_{x'}=R_x$ as $R_{x'}\supseteq R_x$ (by minimality of $R_x$) and $\Stab_W(R_{x'})=\Fix_W(\Min(w))=\Stab_W(R_x)$. This proves the second statement as well.
\end{proof}

We mention for future reference the following consequence of the above results.

\begin{lemma}\label{lemma:CDwetainvaraintresidue}
Let $w\in\Aut(\Sigma)$ be of finite order, and let $C,D\in\CMin(w)$ be chambers in a spherical residue $R$ such that $w$ normalises $\Stab_W(R)$. Let also $R_C$ and $R_D$ be the smallest $w$-invariant residues containing $C$ and $D$, respectively, and let $R_{CD}$ be the smallest residue containing $R_C\cup R_D$. Then:
\begin{enumerate}
\item
$R_{CD}$ is a $w$-invariant spherical residue.
\item
If $C,D$ are opposite in $R$, then there exists a chamber $D'\in R_D$ such that $C,D'$ are opposite chambers in $R_{CD}$. Moreover, $D'\in\CMin(w)$ and $R_D$ is the smallest $w$-invariant residue containing $D'$.
\end{enumerate}
\end{lemma}
\begin{proof}
(1) By Proposition~\ref{prop:basicprop_finiteCox}, we find $w$-fixed points $x\in C$ and $y\in D$  such that $\Fix_{W}(x)=\Fix_{W}(\Min(w))=\Fix_{W}(y)$, and we have $R_x=R_C$ and $R_y=R_D$.

Let $Z$ be the intersection of the walls of $R$, so that $R$ is a nonempty (because $R$ is spherical) closed convex subset of $X$ stabilised by $w$ (because $w$ normalises $\Stab_W(R)$). Hence $Z$ contains a $w$-fixed point $z$. 

Note that the walls of the residues $R,R_x,R_y$ are all walls of the spherical residue $R_z$: for $R$, this holds because $z\in Z$, while for $R_x,R_y$, this holds because $z\in\Min(w)$ (and because the walls of $R_x,R_y$ contain $\Min(w)$). In particular, for any two chambers $E,E'\in R_x\cup R_y$, the walls separating $E$ from $E'$ are walls of $R_z$: if $E,E'\in R_x$ or $E,E'\in R_y$, this is clear, and if $E\in R_x$ and $E'\in R_y$, then fixing a minimal gallery $\Gamma_{E,C}\subseteq R_x$ from $E$ to $C$, a minimal gallery $\Gamma_{C,D}\subseteq R$ from $C$ to $D$, and a minimal gallery $\Gamma_{D,E'}\subseteq R_y$ from $D$ to $E'$, any wall separating $E$ from $E'$ is crossed by one of the galleries $\Gamma_{E,C}$, $\Gamma_{C,D}$, and $\Gamma_{D,E'}$.

Letting $R_{CD}$ be the smallest residue containing $R_x\cup R_y$, Lemma~\ref{lemma:basic_to_solve} now implies that $\Stab_W(R_{CD})\subseteq\Stab_W(R_z)$. In particular, $R_{CD}$ is spherical. On the other hand, since $R_x\cup R_y$ is $w$-invariant, so is $R_{CD}$. This proves (1).

(2) Assume now that $C,D$ are opposite in $R$. Without loss of generality, we may assume that $C=C_0$. Let $L$ be the type of $R$, so that $D=w_0(L)C$, and let $J,K$ be the types of $R_C,R_D$, respectively. Let $D_1:=\proj_{R_D}(C)$, so that $D_1=vC$ with $v$ the unique element of minimal length in $w_0(L)W_K$. In particular, $\supp(v)\subseteq L$. 

Note that the residues $R_C,R_D$ are parallel, as they have the same set of walls. Since every chamber of $R_D$ is connected to a chamber of $R_C$ by a $\supp(v)$-gallery (see e.g.  \cite[Proposition~21.10(ii)]{MPW15}), the residue $R_{CD}$ is of type $I:=J\cup L=K\cup L$. 

We claim that if $s\in I=K\cup L$ is such that $\ell(vs)>\ell(v)$, then $s\in K$. Indeed, suppose that $s\in L\setminus K$. Then $vsC\notin R_D=vR_K$ and $vsC\in R_L(C)=R$, which respectively imply that $\dc(vsC,w_0(L)C)=\dc(vC,w_0(L)C)+1$ (by the gate property in $R_D$) and that $vsC$ is on a minimal gallery from $C$ to $w_0(L)C$, that is, $\dc(C,vsC)+\dc(vsC,w_0(L)C)=\dc(C,w_0(L)C)$. Since $\dc(C,w_0(L)C)=\dc(C,vC)+\dc(vC,w_0(L)C)$ by the gate property in $R_D$, we then have $$\dc(C,vsC)=\dc(C,w_0(L)C)-\dc(vsC,w_0(L)C)=\dc(C,vC)-1,$$ that is, $\ell(vs)<\ell(v)$, as claimed.

It now follows from \cite[Proposition~21.30]{MPW15} (as well as \cite[Proposition~21.29]{MPW15}) that $R_D$ contains a chamber $D'$ opposite $C$ in $R_{CD}$. Moreover, since $R_D=R_y$, Proposition~\ref{prop:basicprop_finiteCox} implies that $R_D$ is the smallest $w$-invariant residue containing $D'$. Finally, note that $$\dc(D',wD')=\dc(w_0(I)C,ww_0(I)C)=\ell(w_0(I)ww_0(I))=\ell(w)=\dc(C,wC),$$ so that $D'\in\CMin(w)$, yielding (2).
\end{proof}


\subsection{Comparison of galleries in \texorpdfstring{$\Sigma$ and $\Sigma^\eta$}{the Coxeter and transversal complexes}}

Throughout this subsection, we fix an element $w\in \Aut(\Sigma)$ of infinite order, and we set $\eta:=\eta_w\in\partial X$. In this subsection, we establish a correspondence between galleries in $\CMin(w)$ and galleries in $\CMin_{\Sigma^\eta}(w_\eta)$ (see Propositions~\ref{prop:corresp_Minsets2} and \ref{prop:corresp_galleries_CMin2} below), thereby relating cyclic shift classes in $\Aut(\Sigma)$ and in $\Aut(\Sigma^{\eta})$ (see Lemma~\ref{lemma:cyclicshift_geometric}).

We start with an easy observation on $w$-decreasing galleries (see Definition~\ref{definition:wdecreasinggallery}).
\begin{definition}\label{definition:pisigmaeta}
If $\Gamma=(D_0,D_1,\dots,D_k)$ is a gallery in $\Sigma$ from $D_0$ to $D_k$, we let\index[s]{Gammaeta@$\Gamma^\eta$ (gallery in $\Sigma^\eta$ corresponding to $\Gamma$)} $\Gamma^{\eta}$ denote the gallery in $\Sigma^\eta$ from $D_0^\eta$ to $D_k^\eta$ obtained from $\pi_{\Sigma^\eta}(\Gamma)$ by deleting repeated consecutive chambers (i.e. by deleting,  for each $i=1,\dots,k$, the chamber $D_{i}^\eta$ if $D_{i}^\eta=D_{i-1}^\eta$).
\end{definition}

\begin{lemma}\label{lemma:wdecreasing_gallery}
If $\Gamma$ is a $w$-decreasing gallery in $\Sigma$, then $\Gamma^{\eta}$ is a $w_\eta$-decreasing gallery in $\Sigma^\eta$.
\end{lemma}
\begin{proof}
Reasoning inductively on the length of $\Gamma$, we may assume that $\Gamma=(C,D)$ for some $C,D\in\Ch(\Sigma)$. By assumption, $\dc(D,wD)\leq\dc(C,wC)$, and we have to show that $\dce(D^\eta,w_\eta D^\eta)\leq \dce(C^\eta,w_\eta C^\eta)$. Let $m$ be the wall of $\Sigma$ separating $C$ from $D$. If $m\notin\WW^\eta$, then $C^\eta=D^\eta$ and the claim is clear. Assume now that $m\in\WW^\eta$. As the wall $wm$ separating $wC$ from $wD$ also belongs to $\WW^\eta$, we have $\WW(C,wC)\setminus\WW^\eta =\WW(D,wD)\setminus\WW^\eta$. It then follows from (\ref{eqn:dce}) that
\begin{align*}
\dce(D^\eta,w_\eta D^\eta)&=|\WW(D,wD)\cap\WW^\eta|=\dc(D,wD)-|\WW(D,wD)\setminus\WW^\eta|\\
&\leq \dc(C,wC)-|\WW(C,wC)\setminus\WW^\eta|=\dce(C^\eta,w_\eta C^\eta),
\end{align*}
as desired.
\end{proof}

The next technical lemma will be of fundamental importance in establishing a correspondence between galleries in $\Sigma$ and in $\Sigma^\eta$, and will be used repeatedly in the rest of the paper.

\begin{lemma}\label{lemma:corresp_wresidues}
The following assertions hold:
\begin{enumerate}
\item
Let $R$ be a spherical residue of $\Sigma$, and suppose that $\pi_{\Sigma^\eta}|_{R}\co R\to R^\eta$ is a cellular isomorphism onto a spherical residue $R^\eta$ of $\Sigma^\eta$. Then
 $$\pi_{\Sigma^\eta}(\proj_R(C))=\proj_{R^\eta}(C^\eta)\quad\textrm{for all $C\in\Ch(\Sigma)$.}$$ In particular, if $C^\eta\in R^\eta$ and $D:=\proj_R(C)$, then $D^\eta=C^\eta$.
\item
If $x,y\in\Min(w)$ are on a same $w$-axis and $y$ is $w$-essential, then for every chamber $C\in R_x$ we have $\pi_{\Sigma^\eta}(\proj_{R_y}(C))=C^\eta$.
\item
If $R$ is a spherical residue of $\Sigma$ such that $w$ normalises $\Stab_W(R)$, then the restriction of $\pi_{\Sigma^\eta}$ to $R$ is a cellular isomorphism onto a residue $R^\eta$ of $\Sigma^\eta$, and $w_\eta$ normalises $\Stab_{W^\eta}(R^\eta)$. If, moreover, $R$ is a $w$-residue, then $R^\eta$ is stabilised by $w_\eta$.
\item 
If $R^\eta$ is a spherical residue of $\Sigma^\eta$ stabilised by $w_\eta$, then there is a $w$-essential point $x\in\Min(w)$ such that the restriction of $\pi_{\Sigma^\eta}$ to the $w$-residue $R_x$ is a cellular isomorphism onto a $w_\eta$-stable residue $R_x^\eta$ containing $R^\eta$.
\end{enumerate}
\end{lemma}
\begin{proof}
(1) Let $C\in\Ch(\Sigma)$, and set $D:=\proj_{R}(C)$. By assumption, $R$ and $R^\eta$ have the same set of walls (in $\WW^\eta$). Hence for any chamber $E$ of $R$, we have (see (\ref{eqn:dce}))
\begin{align*}
\dce(C^\eta,E^\eta)&=|\WW(C,E)\cap\WW^\eta|=|\WW(C,D)\cap\WW^\eta|+\dc(D,E)\\
&=\dce(C^\eta,D^\eta)+\dce(D^\eta,E^\eta),
\end{align*}
where the second equality follows from the gate property (namely, $\WW(C,E)$ is the disjoint union of $\WW(C,D)$ and $\WW(D,E)\subseteq\WW^\eta$). Therefore, $D^\eta=\proj_{R^{\eta}}(C^\eta)$, as desired.

(2) Assume for a contradiction that $\pi_{\Sigma^\eta}(\proj_{R_y}(C))\neq C^\eta$ for some $C\in R_x$ (so that $x\neq y$). Then there is a wall $m\in\WW^\eta$ separating $C$ from $\proj_{R_y}(C)$. In particular, $m$ intersects the geodesic $[x,y]$, which is contained in a $w$-axis $L$ by assumption. As $m\in\WW^\eta$, it must then contain $L$, and hence be a wall of $R_y$ (because $y$ is $w$-essential), a contradiction.

(3) Let $M$ be the set of walls of $R$. By assumption, $w$ stabilises $M$ and hence also the intersection $Z:=\bigcap_{m\in M}m\subseteq X$. Since $Z$ is a nonempty (as $R$ is spherical) closed convex subset of $X$, it contains a $w$-axis. In particular, $M\subseteq\WW^\eta$. Moreover, if $\sigma\subseteq Z$ is the spherical simplex of $\Sigma$ such that $R=R_{\sigma}$ (so that $M$ is the set of walls containing $\sigma$), then the restriction of $\pi_{\Sigma^\eta}$ to $R$ is a cellular isomorphism onto the residue $R^{\eta}$ of $\Sigma^\eta$ corresponding to the cell $\sigma^\eta:=\pi_{\Sigma^\eta}(\sigma)$; moreover, the set $M$ of walls of $R^\eta$ is stabilised by $w_\eta$, and hence $w_\eta$ normalises $\Stab_{W^\eta}(R^\eta)$. 

Assume now that $R$ is a $w$-residue, and let $x\in\Min(w)$ be $w$-essential such that $R=R_x$. Thus, $\sigma=\supp(x)$, and we have to show that $w_{\eta}$ stabilises $\sigma^\eta$. Note that $y:=wx$ is $w$-essential and that the residues $R_x,R_y$ are parallel. Since $\sigma=\bigcap_{C\in R_x}C$ and $w\sigma=\bigcap_{D\in R_y}D$, we then deduce from (2) that $\pi_{\Sigma^\eta}(w\sigma)=\bigcap_{C^\eta\in R^\eta}C^\eta=\sigma^\eta$. Since $w_\eta\sigma^\eta=\pi_{\Sigma^\eta}(w\sigma)$ by (\ref{eqn:compatibility_pisigmapieta}), the claim follows.

(4) Recall from \S\ref{subsection:PTCplx} the definition of the closed convex subset $C(\eta)$ of $X$, for each $C\in\Ch(\Sigma)$. Let $R^\eta$ be a spherical residue of $\Sigma^\eta$ stabilised by $w_\eta$. Then $w$ stabilises the set $\{C(\eta) \ | \ C^\eta\in R^\eta\}$, and hence also the nonempty closed convex set $Z:=\bigcap_{C^\eta\in R^\eta}C(\eta)$. Hence $Z$ contains a $w$-axis $L$, and we choose a $w$-essential point $x\in L$. By (3), the restriction of $\pi_{\Sigma^\eta}$ to $R_x$ is a cellular isomorphism onto a $w_\eta$-stable residue $R_x^\eta$. Moreover, the set of walls of $R_x^\eta$ coincides with the set of walls containing $x$, and hence contains the set of walls containing $Z$, that is, the set of walls of $R^\eta$. In particular, the spherical simplex $\sigma:=\supp(x)$ of $\Sigma$ is contained in $Z$ (as $\sigma$ and $Z$ are an intersection of half-spaces, see e.g. \cite[\S 3.6.6]{BrownAbr}), and hence $R_x^\eta=\pi_{\Sigma^\eta}(R_{\sigma})\supseteq R^\eta$, as desired.
\end{proof}

\begin{remark}
Lemma~\ref{lemma:corresp_wresidues}(3) implies in particular that $w_\eta$ stabilises a spherical residue of $\Sigma^\eta$, and hence is a finite order automorphism of $\Sigma^\eta$.
\end{remark}

Before proving the two main propositions of this subsection (Propositions~\ref{prop:corresp_Minsets2} and \ref{prop:corresp_galleries_CMin2}), we need one more observation. Recall from the last paragraph of \S\ref{subsection:PCC} the definition of the chamber distance $\dc(R,R')$ between two parallel residues $R,R'$.
\begin{lemma}\label{lemma:dccwcdcrxrwx}
Let $x\in\Min(w)$ be $w$-essential, and let $C\in R_x$. Then $$\dc(C,wC)=\dc(R_{x},R_{wx})+\dist_{\Ch}^{\Sigma^\eta}(C^{\eta},w_\eta C^\eta).$$
\end{lemma}
\begin{proof}
Since $R_x$ is a $w$-residue (hence, $R_x$ and $wR_x=R_{wx}$ are parallel), the gate property implies that $\dc(C,wC)=\dc(R_x,R_{wx})+\dc(\proj_{R_{wx}}(C),wC)$. On the other hand, Lemma~\ref{lemma:corresp_wresidues}(3) implies that $\pi_{\Sigma^\eta}|_{R_{wx}}$ is a cellular isomorphism onto its image, which maps $\proj_{R_{wx}}(C)$ to $C^\eta$ (by Lemma~\ref{lemma:corresp_wresidues}(2)) and $wC$ to $w_\eta C^\eta$ (by (\ref{eqn:compatibility_pisigmapieta})). Hence, $\dc(\proj_{R_{wx}}(C),wC)=\dist_{\Ch}^{\Sigma^\eta}(C^{\eta},w_\eta C^\eta)$, yielding the lemma.
\end{proof}

\begin{prop}\label{prop:corresp_Minsets2}
We have
$\pi_{\Sigma^\eta}(\CMin(w))= \CMin_{\Sigma^\eta}(w_\eta)$. 
\end{prop}
\begin{proof}
(1) Let $C\in\CMin(w)$, and let us show that $C^\eta\in \CMin_{\Sigma^\eta}(w_\eta)$. By Lemma~\ref{lemma:prop34}(1), there is a gallery $\Gamma\subseteq\CMin(w)$ from $C$ to some chamber $C'\in\CMin(w)$ intersecting $\Min(w)$. Up to replacing $C$ by $C'$, we may then assume by Lemma~\ref{lemma:wdecreasing_gallery} that $C$ contains a point $x\in\Min(w)$. On the other hand, if $L_x$ is the $w$-axis through $x$ and $\sigma$ is the support of a $w$-essential point $x'\in L_x$ such that $x,x'\in\sigma$, then $C$ and $C'':=\proj_{R_{\sigma}}(C)$ (which belongs to $\CMin(w)$ by Lemma~\ref{lemma:prop34}(2)) are not separated by any wall in $\WW^\eta$ (for such a wall would contain $x$, and hence also $L_x$, and would thus be a wall of $R_{\sigma}$, a contradiction). Thus, $C^\eta=(C'')^{\eta}$, and up to replacing $C$ by $C''$ (and $x$ by $x'$), we may assume that $x$ is $w$-essential.

In particular, $R_x$ is a $w$-residue, and by Lemma~\ref{lemma:corresp_wresidues}(3), the restriction of $\pi_{\Sigma^\eta}$ to $R_x$ is a cellular isomorphism onto a residue $R_x^\eta$ of $\Sigma^\eta$ stabilised by $w_\eta$. By Lemma~\ref{lemma:prop34fin}(2), we find a chamber $D\in R_x$ such that $D^\eta\in R_x^\eta\cap\CMin_{\Sigma^\eta}(w_\eta)$ (by taking for $D^\eta$ the projection on $R_x^\eta$ of any chamber in $\CMin_{\Sigma^\eta}(w_\eta)$). It then follows from Lemma~\ref{lemma:dccwcdcrxrwx} (applied to $C$ and $D$) that
\begin{align*}
\dist_{\Ch}^{\Sigma^\eta}(C^{\eta},w_\eta C^\eta)&=\dc(C,wC)-\dc(R_x,R_{wx})\\
&\leq \dc(D,wD)-\dc(R_x,R_{wx})=\dist_{\Ch}^{\Sigma^\eta}(D^{\eta},w_\eta D^\eta),
\end{align*}
so that $\dist_{\Ch}^{\Sigma^\eta}(C^{\eta},w_\eta C^\eta)=\dist_{\Ch}^{\Sigma^\eta}(D^{\eta},w_\eta D^\eta)$ and $C^\eta\in \CMin_{\Sigma^\eta}(w_\eta)$, as desired.

(2) Conversely, assume that $C^\eta\in \CMin_{\Sigma^\eta}(w_\eta)$ for some $C\in\Ch(\Sigma)$, and let us show that there is some $D\in\CMin(w)$ with $D^\eta=C^\eta$. By Proposition~\ref{prop:basicprop_finiteCox}, $C^\eta$ contains a $w_{\eta}$-fixed point (in the Davis complex of $W^\eta$), and hence belongs to a spherical residue $R^\eta$ of $\Sigma^\eta$ stabilised by $w_\eta$. Up to enlarging $R^\eta$, we may assume by Lemma~\ref{lemma:corresp_wresidues}(4) that there exists a $w$-essential point $x\in\Min(w)$ such that the $w$-residue $R_x$ is mapped isomorphically onto $R^\eta$ by $\pi_{\Sigma^\eta}$. By Lemma~\ref{lemma:prop34}(2), there is a chamber $D_1\in R_x\cap\CMin(w)$. Let $D$ be the unique chamber of $R_x$ with $D^\eta=C^\eta$. Lemma~\ref{lemma:dccwcdcrxrwx} (applied to $D$ and $D_1$) then yields
\begin{align*}
\dc(D,wD)&=\dc(R_x,R_{wx})+\dist_{\Ch}^{\Sigma^\eta}(D^{\eta},w_\eta D^\eta)\\
&\leq \dc(R_x,R_{wx})+\dist_{\Ch}^{\Sigma^\eta}(D_1^{\eta},w_\eta D_1^\eta)=\dc(D_1,wD_1),
\end{align*}
and hence $D\in\CMin(w)$, as desired.
\end{proof}

\begin{prop}\label{prop:corresp_galleries_CMin2}
Let $C,D\in\CMin(w)$. Then the following assertions are equivalent:
\begin{enumerate}
\item
There exists a gallery $\Gamma\subseteq\CMin(w)$ from $C$ to $D$.
\item
There exists a gallery $\Gamma^{\eta}\subseteq\CMin_{\Sigma^\eta}(w_\eta)$ from $C^\eta$ to $D^\eta$.
\end{enumerate}
\end{prop}
\begin{proof}
If (1) holds, then the gallery $\Gamma^{\eta}$ from $C^\eta$ to $D^\eta$ (see Definition~\ref{definition:pisigmaeta}) is contained in $\CMin_{\Sigma^\eta}(w_\eta)$ by Proposition~\ref{prop:corresp_Minsets2}, yielding (2).

Conversely, assume that (2) holds. By Lemma~\ref{lemma:prop34}(1), there exist galleries $\Gamma_C\subseteq \CMin(w)$ from $C$ to some $C'\in\CMin(w)$ and $\Gamma_D\subseteq \CMin(w)$ from $D$ to some $D'\in\CMin(w)$, with $C',D'$ intersecting $\Min(w)$. As we have just seen, this yields corresponding galleries in $\CMin_{\Sigma^\eta}(w_\eta)$ from $C^\eta$ to $(C')^\eta$ and from $D^\eta$ to $(D')^\eta$. To prove (1), we may thus assume without loss of generality that $C$ and $D$ intersect $\Min(w)$. 

By Proposition~\ref{prop:corresp_Minsets2}, we may write $\Gamma^{\eta}=(C^{\eta}=D_0^\eta,D_1^\eta,\dots,D_k^\eta=D^\eta)$ for some $D_i\in\CMin(w)$ (with $D_0:=C$ and $D_k:=D$). By Proposition~\ref{prop:basicprop_finiteCox}, there exists for each $i\in\{0,\dots,k\}$ some $y_{i}\in D_i^\eta\cap\Min(w_\eta)$. 

Let $i\in\{1,\dots,k\}$. As the union of the (closed) chambers $D_{i-1}^\eta$ and $D_i^\eta$ in the Davis complex of $W^\eta$ is convex with respect to the $\mathrm{CAT}(0)$ metric (as it is an intersection of half-spaces by \cite[Theorem~3.131(i)$\Leftrightarrow$(ii)]{BrownAbr}), the geodesic $[y_{i-1},y_i]\subseteq \Min(w_\eta)$ intersects $D_{i-1}^\eta\cap D_i^{\eta}$. In particular, $D_{i-1}^{\eta}$ and $D_i^{\eta}$ are contained in a common spherical residue $R_i^\eta$ of $\Sigma^\eta$ stabilised by $w_\eta$. Up to enlarging $R_i^\eta$, Lemma~\ref{lemma:corresp_wresidues}(4) then yields a $w$-essential point $x_i\in\Min(w)$ such that $R_{x_i}$ is a $w$-residue mapped isomorphically onto $R_i^\eta$ by $\pi_{\Sigma^\eta}$.

\begin{figure}
    \centering
        \includegraphics[trim = 51mm 165mm 20mm 45mm, clip, width=\textwidth]{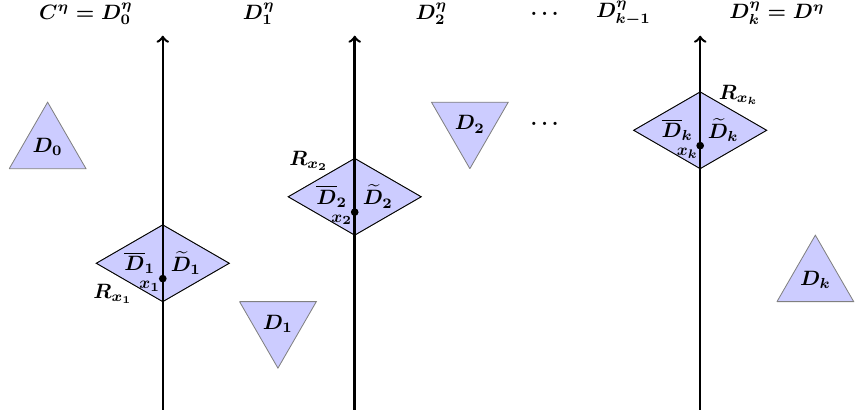}
        \caption{Proposition~\ref{prop:corresp_galleries_CMin2}}
        \label{figure:Correspgalleries}
\end{figure}

By Lemma~\ref{lemma:prop34}(2), the chambers $\overline{D}_{i}:=\proj_{R_{x_i}}(D_{i-1})$ and $\widetilde{D}_i:=\proj_{R_{x_i}}(D_i)$ belong to $\CMin(w)$ (see Figure~\ref{figure:Correspgalleries}). Note that $\overline{D}_i^\eta=D_{i-1}^\eta$ and $\widetilde{D}_i^\eta=D_i^\eta$ by Lemma~\ref{lemma:corresp_wresidues}(1). In particular, $\widetilde{D}_i$ and $\overline{D}_{i+1}$ are not separated by any wall in $\WW^\eta$, so that Lemma~\ref{lemma:prop34}(3) yields for each $i\in\{1,\dots,k-1\}$ a gallery $\Gamma_i\subseteq\CMin(w)$ from $\widetilde{D}_i$ to $\overline{D}_{i+1}$. Similarly, as $C=D_0$ and $D=D_k$ intersect $\Min(w)$ by assumption, Lemma~\ref{lemma:prop34}(3) also yields a gallery $\Gamma_0\subseteq\CMin(w)$ from $C$ to $\overline{D}_1$ and a gallery $\Gamma_k\subseteq\CMin(w)$ from $\widetilde{D}_k$ to $D$. Moreover, since the chambers $D_{i-1}^\eta$ and $D_i^\eta$ of $R_i^\eta$ are adjacent in $\Sigma^\eta$, the chambers $\overline{D}_i$ and $\widetilde{D}_i$ of $R_{x_i}$ are adjacent in $\Sigma$ for each $i\in\{1,\dots,k\}$. We thus obtain the desired gallery $\Gamma\subseteq\CMin(w)$ from $C$ to $D$ by concatenating the galleries $\Gamma_0,\dots,\Gamma_k$, yielding (1).
\end{proof}

We conclude this subsection with the following useful consequence of the above results.
\begin{prop}\label{prop:wdecreasingfromCtoD}
Let $C\in\Ch(\Sigma)$ and $D\in\CMin(w)$ be such that $C^\eta=D^\eta$. Then there exists a $w$-decreasing gallery from $C$ to $D$. In particular, $\pi_w(C)\to\pi_w(D)$.
\end{prop}
\begin{proof}
By Lemma~\ref{lemma:prop34}(1), there exists a $w$-decreasing gallery $\Gamma$ from $C$ to some $D_1\in\Ch(\Sigma)$ containing a point $x\in\Min(w)$. As $C^\eta=D^\eta\in\CMin_{\Sigma^\eta}(w_\eta)$ by Proposition~\ref{prop:corresp_Minsets2}, the gallery $\Gamma^{\eta}$ from $C^\eta$ to $D_1^\eta$ is contained in $\CMin_{\Sigma^\eta}(w_\eta)$ by Lemma~\ref{lemma:wdecreasing_gallery}. Hence there exists by Proposition~\ref{prop:corresp_Minsets2} some $D_2\in\CMin(w)$ with $D_1^\eta=D_2^\eta$. Let $L$ be the $w$-axis through $x$, and let $\sigma$ be the support of a $w$-essential point $y\in L$ such that $x,y\in\sigma$. By Lemma~\ref{lemma:prop34}(2), the chamber $D_3:=\proj_{R_{\sigma}}(D_2)$ belongs to $\CMin(w)$ and contains $x\in\Min(w)$. Since $D_1^\eta=D_2^\eta$, the chambers $D_1$ and $D_2$ are on the same side of every wall of $R_{\sigma}=R_y$, and hence $D_3=\proj_{R_y}(D_1)$. Therefore, $D_3^\eta=D_1^\eta$ by Lemma~\ref{lemma:corresp_wresidues}(2). Hence there exists by Lemma~\ref{lemma:prop34}(3) a $w$-decreasing gallery $\Gamma'$ from $D_1$ to $D_3$. Finally, Proposition~\ref{prop:corresp_galleries_CMin2}(2)$\Rightarrow$(1) applied to the chambers $D_3$ and $D$ (note that $\Gamma^{\eta}$ connects $D_3^\eta$ to $D^\eta$) yields a gallery $\Gamma''\subseteq\CMin(w)$ from $D_3$ to $D$ (in particular, $\Gamma''$ is $w$-decreasing). The desired $w$-decreasing gallery from $C$ to $D$ is then the concatenation of $\Gamma$, $\Gamma'$ and $\Gamma''$.

The last statement now follows from Lemma~\ref{lemma:wdecresing_cs}.
\end{proof}


\subsection{Regular points}
Recall from Proposition~\ref{prop:basicprop_finiteCox} that if $w\in \Aut(\Sigma)$ has finite order, then there exists $x\in\Min(w)$ such that $\Fix_W(x)=\Fix_W(\Min(w))$. The purpose of this subsection is to show that this also holds for $w$ of infinite order.

\begin{definition}\label{definition:regularpoints}
For $w\in \Aut(\Sigma)$, we call a point $x\in\Min(w)$ {\bf $w$-regular}\index{regular@$w$-regular point} if $\Fix_W(x)=\Fix_W(\Min(w))$, and we let\index[s]{RegXw@$\Reg_X(w)$ (set of $w$-regular points in $X$)} $\Reg_X(w)\subseteq X$ denote the set of $w$-regular points.
\end{definition}

\begin{lemma}\label{lemma:generic_axis}
Let $w\in \Aut(\Sigma)$ be of infinite order. Then $\Reg_X(w)$ is dense in $\Min(w)$.
\end{lemma}
\begin{proof}
Let $x\in\Min(w)$, and let $\varepsilon>0$. We have to show that the open ball $B_{\varepsilon}(x)$ centered at $x$ and of radius $\varepsilon$ in $X$ contains a $w$-regular point. We claim that any $y\in\Min(w)\cap B_{\varepsilon}(x)$ such that $P_y:=\Fix_W(y)$ is minimal in $\{\Fix_W(z) \ | \ z\in\Min(w)\cap B_{\varepsilon}(x)\}$ is $w$-regular. Up to changing $x$ within $B_{\varepsilon}(x)$, we may assume that $P_x$ is minimal in $\{P_z \ | \ z\in\Min(w)\cap B_{\varepsilon}(x)\}$. Assume for a contradiction that $P_x\neq P_w:=\Fix_W(\Min(w))$. Note that $P_x\supseteq P_w$.

Let $L_1$ be the $w$-axis through $x$. As $B_{\varepsilon}(x)$ contains a $w$-essential point of $L_1$, we have $P_x=\Fix_W(L_1)$. Since $P_x\neq P_w$, there exists a $w$-axis $L_2$ with $\Fix_W(L_2)\neq \Fix_W(L_1)$ (because $P_w=\bigcap_L\Fix_W(L)$, where the intersection runs over the set of $w$-axes $L$). For each $w$-axis $L$, let $\WW_L$ denote the set of walls containing $L$; set for short $\WW_i:=\WW_{L_i}$ for $i=1,2$.

Since $\Fix_W(L_2)\neq \Fix_W(L_1)$, we have $L_1\neq L_2$. By \cite[Theorem~II.2.13]{BHCAT0}, the convex hull $C:=\Conv(L_1,L_2)$ of $L_1\cup L_2$ is contained in $\Min(w)$ and is isometric to a flat (Euclidean) strip $[z_1,z_2]\times\RR$ (with $z_1\in L_1$ and $z_2\in L_2$).

Let $C_{\varepsilon}\subseteq C$ denote the intersection of $C$ with an $\varepsilon$-neighbourhood of $L_1$. Let $L,L'$ be two distinct $w$-axes different from $L_1,L_2$ and contained in $C_{\varepsilon}$, and such that $\WW_L=\WW_{L'}$; such $L,L'$ exist, because there are only finitely many possibilities for $\WW_L$ as $L$ varies amongst the (infinitely many) $w$-axes contained in $C_{\varepsilon}$. Note first that any wall $m\in\WW_1\cap\WW_2$ contains $C_{\varepsilon}$, and hence also $L$. In other words, $\WW_1\cap\WW_2\subseteq \WW_L$. Conversely, if $m\in\WW_L=\WW_{L'}$, then for any $z_1\in L_1$ and $z_2\in L_2$, the geodesic $[z_1,z_2]$ intersects $L$ and $L'$, and hence the wall $m$ in at least two points, and is therefore entirely contained in $m$. This shows that $\WW_L\subseteq \WW_1\cap\WW_2$, and hence that $\WW_L=\WW_1\cap\WW_2$. 

Let now $y\in L\cap B_{\varepsilon}(x)$ be $w$-essential. Then $\Fix_W(y)=\Fix_W(L)=\Fix_W(L_1)\cap\Fix_W(L_2)\subsetneq P_x$, contradicting the minimality of $P_x$, as desired.
\end{proof}


\section{Proof of Theorem~\ref{thmintro:graphisomorphism}}\label{section:RTTPOTFH}

Throughout this section, we fix a Coxeter system $(W,S)$, as well as an element $w\in \Aut(\Sigma)$ of infinite order, and we set $\eta:=\eta_w\in\partial X$. Recall from \S\ref{subsection:PTCplx} the definitions of the morphism of complexes $\pi_{\Sigma^\eta}\co\Sigma\to\Sigma^\eta$ and of the action map $\pi_{\eta}\co \Aut(\Sigma)_{\eta}\to\Aut(\Sigma^\eta):v\mapsto v_{\eta}$.

\begin{definition}\label{definition:Xieta}
Each element $u\in \Aut(\Sigma)_{\eta}$ induces a diagram automorphism\index[s]{deltau@$\delta_u$ (transversal diagram automorphism induced by $u$)} $\delta_u\in\Aut(\Sigma^\eta,C_0^\eta)=\Aut(W^\eta,S^\eta)$ of the transversal complex $\Sigma^\eta$, defined by $$\delta_u:=x\inv u_{\eta},$$ where $x$ is the unique element of $W^\eta$ such that $u_\eta C_0^\eta=xC_0^\eta$. In other words, $\delta_u$ is the unique element of $\Aut(W^\eta,S^\eta)$ such that $u_\eta\in W^\eta\delta_u\subseteq W^\eta\rtimes\Aut(W^\eta,S^\eta)=\Aut(\Sigma^\eta)$. 

We define the subgroups\index[s]{Xieta@$\Xi_\eta$ (set of $\delta_u$ with $u\in W_\eta$)}\index[s]{Xiw@$\Xi_w$ (set of $\delta_u$ with $u\in \ZZZ_W(w)$)}
$$\Xi_{\eta}:=\{\delta_u \ | \ u\in W_\eta\}\subseteq\Aut(W^\eta,S^\eta)\quad\textrm{and}\quad \Xi_w:=\{\delta_u \ | \ u\in\ZZZ_W(w)\}\subseteq\Xi_{\eta}$$
of all diagram automorphisms of $(W^\eta,S^\eta)$ induced by elements of $W_\eta$ (resp. $\ZZZ_W(w)$). In other words, $\Xi_\eta=\pi_{\eta}(W_\eta)\cap\Aut(\Sigma^\eta,C_0^\eta)\cong \pi_{\eta}(W_\eta)/W^\eta$ and $\Xi_w=\pi_{\eta}(\ZZZ_W(w)\cdot W^\eta)\cap\Aut(\Sigma^\eta,C_0^\eta)$.
\end{definition}

The goal of this section is to prove Theorem~\ref{thmintro:graphisomorphism}. In \S\ref{subsection:DCCC}, we show how the group $\Xi_w$ allows to distinguish between cyclic shift classes in $\OOO_w^{\min}$. We then prove Theorem~\ref{thmintro:graphisomorphism} in \S\ref{TSCGOAIOE}. We conclude by establishing the analogous statement for the tight conjugation graph in \S\ref{subsection:TTCGOAIOE}.

In the next sections (\S\ref{section:ICG} and \S\ref{section:ACG}), we will identify and study the Coxeter system $(W^\eta,S^\eta)$ and its associated complex $\Sigma^\eta$, as well as the group $\Xi_w$, dealing with the indefinite and affine cases separately.


\subsection{Distinguishing cyclic shift classes}\label{subsection:DCCC}

We first introduce a useful notation for the ``support at infinity'' of the conjugates of $w$. Recall from Definition~\ref{definition:piw} the parametrisation $\pi_{w_\eta}$ of the conjugates of $w_\eta$ in $W^\eta$ by chambers of $\Sigma^\eta$.

\begin{definition}
For a chamber $C\in\Ch(\Sigma)$, set\index[s]{IwC@$I_w(C)$ (transversal support corresponding to $\pi_w(C)$)} $I_w(C):=\supp_{S^\eta}(\pi_{w_{\eta}}(C^\eta))\subseteq S^\eta$. Set also for short\index[s]{IwC0@$I_w$ (same as $I_w(C_0)$)} $I_w:=I_w(C_0)=\supp_{S^\eta}(w_\eta)$.
\end{definition}

The following lemma gives a geometric interpretation of $I_w(C)$.

\begin{lemma}\label{lemma:I_w(C)geometric}
Let $C\in\Ch(\Sigma)$, and let $L\subseteq S^\eta$. Then $w_{\eta}$ stabilises the residue $R_L(C^\eta)$ of $\Sigma^\eta$ if and only if $\delta_w(L)=L$ and $L\supseteq I_w(C)$. In particular, $I_w(C)$ is the type of the smallest $w_\eta$-invariant residue containing $C^\eta$.
\end{lemma}
\begin{proof}
This readily follows from Lemma~\ref{lemma:geominterp_supppiwC} applied to $(W,S):=(W^\eta,S^\eta)$ and $w:=w_\eta$.
\end{proof}

Here are a few properties of the sets $I_w(C)$.

\begin{lemma}\label{lemma:IwvC_sigmaIwC}
Let $C\in\CMin(w)$.
\begin{enumerate}
\item
If $v\in \ZZZ_W(w)$, then $I_w(vC)=\delta_{v}(I_w(C))$.
\item
If $\delta\in\Xi_w$, then $\delta(I_w(C))=I_w(vC)$ for some $v\in\ZZZ_W(w)$.
\end{enumerate}
\end{lemma}
\begin{proof}
(1) Let $u\in W^\eta$ be such that $v_\eta C^\eta=uC^\eta$, and let $a\in W^\eta$ with $C^\eta=aC_0^\eta$. Then $a\inv v_\eta aC_0^\eta=a\inv uaC_0^\eta$, and hence $\delta_v=\delta_{a\inv va}=(a\inv ua)\inv a\inv v_\eta a=a\inv u\inv v_\eta a$. Since $v_\eta$ and $w_\eta$ commute, we then have
\begin{align*}
I_w(vC)&=\supp_{S^\eta}(\pi_{w_\eta}(uaC_0^\eta))=\supp_{S^\eta}(a\inv u\inv w_\eta ua)\\
&=\supp_{S^\eta}(a\inv u\inv v_{\eta}w_\eta v_{\eta}\inv ua)=\supp_{S^\eta}(\delta_{v} a\inv w_\eta a\delta_{v}\inv)\\
&=\delta_{v}(\supp_{S^\eta}(a\inv w_\eta a))=\delta_{v}(\supp_{S^\eta}(\pi_{w_\eta}(C^\eta)))=\delta_{v}(I_w(C)),
\end{align*}
as desired.

(2) Writing $\delta=\delta_{v}$ for some $v\in\ZZZ_W(w)$, we have $\delta(I_w(C))=I_w(vC)$ by (1).
\end{proof}

\begin{lemma}\label{lemma:IwCconstantgalleries}
Let $C,D\in\CMin(w)$ be such that $C^\eta,D^\eta$ are connected by a gallery $\Gamma^\eta\subseteq\CMin_{\Sigma^\eta}(w_\eta)$. Then $I_w(C)=I_w(D)$.
\end{lemma}
\begin{proof}
By assumption, the elements $\pi_{w_\eta}(C^\eta)$ and $\pi_{w_\eta}(D^\eta)$ are cyclically reduced, and belong to the same cyclic shift class in $(W^\eta,S^\eta)$ by Lemma~\ref{lemma:cyclicshift_geometric}(2)$\Rightarrow$(1). In particular, they have the same support (see Lemma~\ref{lemma:cyclicshit_samesupport}), that is, $I_w(C)=I_w(D)$.
\end{proof}

We are now ready to show how $\Xi_w$ allows to distinguish between the cyclic shift classes in $\OOO_w^{\min}$.

\begin{prop}\label{prop:R2R3distinguishccc}
Let $C,D\in\CMin(w)$. Then the following assertions are equivalent:
\begin{enumerate}
\item
$\pi_w(C)$ and $\pi_w(D)$ are in the same cyclic shift class.
\item
there exist $v\in\ZZZ_W(w)$ and a gallery $\Gamma^\eta\subseteq\CMin_{\Sigma^\eta}(w_{\eta})$ from $D^\eta$ to $v_\eta C^\eta$. 
\item
$I_w(D)=\sigma(I_w(C))$ for some $\sigma\in\Xi_w$.
\end{enumerate}
\end{prop}
\begin{proof}
(1)$\Leftrightarrow$(2): By Lemma~\ref{lemma:cyclicshift_geometric}, (1) holds if and only if there exist $v\in\ZZZ_W(w)$ and a gallery $\Gamma\subseteq\CMin(w)$ from $D$ to $vC$. This latter condition is equivalent to (2) by Proposition~\ref{prop:corresp_galleries_CMin2} (recall that $\pi_{\Sigma^\eta}(vC)=v_\eta C^\eta$ by (\ref{eqn:compatibility_pisigmapieta})).

\smallskip

(2)$\Leftrightarrow$(3): Assume that (2) holds. Then Lemma~\ref{lemma:IwCconstantgalleries} implies that $I_w(D)=I_w(vC)$. As $I_w(vC)=\sigma(I_w(C))$ for some $\sigma\in\Xi_w$ by Lemma~\ref{lemma:IwvC_sigmaIwC}(1), (3) follows.

Assume, conversely, that (3) holds, and let $\sigma\in \Xi_w$ be such that $I_w(D)=\sigma(I_w(C))$. By Lemma~\ref{lemma:IwvC_sigmaIwC}(2), there exists $v\in\ZZZ_W(w)$ such that $I_w(D)=I_w(vC)$. Hence Theorem~\ref{thm:finite} implies that $\pi_{w_\eta}(D^\eta)$ and $\pi_{w_\eta}(v_\eta C^\eta)$ are in the same cyclic shift class of $(W^\eta,S^\eta)$. Lemma~\ref{lemma:cyclicshift_geometric} then yields some $\overline{v}$ in\index[s]{ZWetaweta@$\ZZZ_{W^\eta}(w_\eta)$ (centraliser of $w_\eta$ in $W^\eta$)} $$\ZZZ_{W^\eta}(w_\eta)=\{v\in W^\eta \ | \ vw_\eta v\inv=w_\eta\}$$ and a gallery $\Gamma^{\eta}\subseteq\CMin_{\Sigma^\eta}(w_{\eta})$ from $D^\eta$ to $\overline{v}v_\eta C^\eta$. Note that  $\ZZZ_{W^\eta}(w_\eta)\subseteq \ZZZ_W(w)$ by (\ref{eqn:replacementR1}). Hence $\overline{v}v\in\ZZZ_W(w)$, which implies (2), as desired.
\end{proof}


\subsection{The structural conjugation graph of an infinite order element}\label{TSCGOAIOE}
The goal of this subsection is to prove Theorem~\ref{thmintro:graphisomorphism} (see Theorem~\ref{theorem=R1R2implythmA} below). The existence of a well-defined map $$\varphi_w\co\KKK_{\OOO_w}\to \KKK_{\delta_w}^0(I_w)/\Xi_w: \Cyc(\pi_w(C))\mapsto I_w(C)\quad\textrm{for all $C\in\CMin(w)$}$$ at the level of the vertex sets will follow from Proposition~\ref{prop:R2R3distinguishccc}, and we now check in a series of lemmas and propositions that $\varphi_w$ is in fact a graph isomorphism.

First, we give a geometric reformulation of $K$-conjugation.

\begin{lemma}\label{lemma:Ktightlyconjuagte_geometric}
Let $K\subseteq S$ be spherical. Let $C,D\in\Ch(\Sigma)$. Then the following assertions are equivalent:
\begin{enumerate}
\item
$\pi_w(C)$ and $\pi_w(D)$ are $K$-conjugate.
\item
There exists $v\in\ZZZ_W(w)$ such that $C$ and $vD$ are opposite chambers in a residue $R$ of type $K$ such that $w$ normalises $\Stab_W(R)$.
\end{enumerate}
\end{lemma}
\begin{proof}
Let $R=R_K(C)$ be the $K$-residue containing $C$. Write $C=aC_0$ and $D=bC_0$ for some $a,b\in W$. 
If (1) holds, then $w$ normalises $aW_Ka\inv =\Stab_W(R)$, and $b\inv wb=w_0(K) a\inv wa w_0(K)$. Hence, $v:=aw_0(K)b\inv\in\ZZZ_W(w)$ and $vD=aw_0(K)C_0$ is the chamber opposite $C$ in $R=aR_K$, so that (2) holds. 

Conversely, if (2) holds, then $\pi_w(C)=a\inv wa$ normalises $W_K=\Stab_W(a\inv R)$ and $\pi_w(D)=\pi_w(vD)=\pi_w(aw_0(K)C_0)=w_0(K)\pi_w(C)w_0(K)$, yielding (1).
\end{proof}

Next, we check that $\varphi_w\inv$ preserves edges.

\begin{prop}\label{prop:phisurj_and_phiinvedgeedge}
Let $C\in\CMin(w)$ and $L\subseteq S^\eta$ be spherical with $\delta_w(L)=L$ and $I_w(C)\subseteq L$. Write $C^\eta=aC_0^\eta$ with $a\in W^\eta$. Then the following assertions hold:
\begin{enumerate}
\item
There exist a spherical subset $K\subseteq S$ and $b\in W$ of minimal length in $bW_K$ such that $W^\eta_L=bW_Kb\inv$ and $abC_0\in\CMin(w)$.
\item
If $b\in W$ and $K\subseteq S$ are as in (1) and $C_1:=abC_0$, then $C_1^\eta=C^\eta$ and $\Cyc(\pi_w(C))=\Cyc(\pi_w(C_1))$. Moreover, the chamber $D$ opposite $C_1$ in $R_K(C_1)$ belongs to $\CMin(w)$, and $\pi_w(C_1)$ and $\pi_w(D)$ are $K$-conjugate. Finally, $D^\eta$ is  opposite $C^\eta$ in $R_L(C^\eta)$ and $I_w(D)=\op_L(I_w(C))$. 
\end{enumerate}
\end{prop}
\begin{proof}
Let $R_L(C^\eta)\subseteq\Sigma^\eta$ be the residue of type $L$ containing $C^\eta$.
By Lemma~\ref{lemma:I_w(C)geometric}, $w_\eta$ stabilises $R_L(C^\eta)$. By Lemma~\ref{lemma:corresp_wresidues}(4), we then find a $w$-essential point $x\in\Min(w)$ such that the restriction of $\pi_{\Sigma^\eta}$ to the $w$-residue $R_x$ is a cellular isomorphism onto the $w_\eta$-stable residue $R_{\overline{L}}(C^\eta)$ for some $\overline{L}\subseteq S^\eta$ containing $L$. Let $R'$ be the residue contained in $R_x$ and mapped isomorphically to $R_L(C^\eta)$ by $\pi_{\Sigma^\eta}$.

(1) Lemma~\ref{lemma:prop34}(2) implies that the chamber $C_1:=\proj_{R_x}(C)\in R_x$ belongs to $\CMin(w)$. Moreover, $C^\eta=C_1^\eta\in\CMin(w_\eta)$ by Lemma~\ref{lemma:corresp_wresidues}(1) and Proposition~\ref{prop:corresp_Minsets2}. In particular, $C_1$ is the unique chamber of $R_x$ mapped to $C^\eta$ by $\pi_{\Sigma^\eta}$, and hence $C_1\in R'$. Let $b\in W$ be such that $a\inv C_1=bC_0$, and let $K\subseteq S$ be the type of $R'$, so that $R'=abR_K$. Then $abW_K(ab)\inv=\Stab_W(R')=\Stab_{W^\eta}(R_L(C^\eta))=aW^{\eta}_La\inv$. Moreover, as $aC_0^\eta=C^\eta=C_1^\eta$, the chambers $aC_0$ and $C_1$ are not separated by any wall of $R'$, that is, $C_1=\proj_{R'}(aC_0)$. Hence $bC_0=a\inv C_1=\proj_{bR_K}(C_0)$, so that $b$ is of minimal length in $bW_K$. This proves (1).

(2) Let $b\in W$ and $K\subseteq S$ be as in (1). Let $R:=R_K(C_1)\subseteq\Sigma$ be the residue of type $K$ containing $C_1:=abC_0$. Since $a\inv C_1=bC_0=\proj_{bR_K}(C_0)$ by assumption, $C_1=\proj_R(aC_0)$ and hence $C_1^\eta=aC_0^\eta=C^\eta$ by Lemma~\ref{lemma:corresp_wresidues}(1). Moreover, as $\Stab_W(R)=abW_K(ab)\inv=aW^\eta_La\inv=\Stab_{W^\eta}(R_L(C^\eta))$, the residues $R$ and $R_L(C^\eta)$ have the same set of walls, and hence the restriction of $\pi_{\Sigma^\eta}$ to $R$ is a cellular isomorphism $R\stackrel{\cong}{\to} R_L(C^\eta)$. 

Since $C_1\in\CMin(w)$, Proposition~\ref{prop:corresp_galleries_CMin2} implies that $C$ and $C_1$ are connected by a gallery contained in $\CMin(w)$. In particular, $\Cyc(\pi_w(C))=\Cyc(\pi_w(C_1))$ by Lemma~\ref{lemma:cyclicshift_geometric}. 

Let $D\in R$ be the chamber opposite $C_1$ in $R$. Then $\pi_w(D)$ and $\pi_w(C_1)$ are $K$-conjugate by Lemma~\ref{lemma:Ktightlyconjuagte_geometric}. In particular, $D\in\CMin(w)$ by Lemma~\ref{lemma:Kconjugatesame length}. Note also that $D^\eta$ is the chamber opposite $C^\eta$ in $R_L(C^\eta)$, and hence $I_w(D)=\supp_{S^\eta}(\pi_{w_\eta}(D^\eta))=\supp_{S^\eta}(w_0(L)\pi_{w_\eta}(C^\eta)w_0(L))=\op_L(I_w(C))$. This proves (2).
\end{proof}

Recall from \S\ref{subsection:SATCG} the definition of the graphs $\KKK_{\OOO_w}$ and $\KKK_{\delta_w}=\KKK_{\delta_w,W^{\eta}}$. 

\begin{corollary}\label{corollary:phisurj_and_phiinvedgeedge}
Let $C\in\CMin(w)$ and let $J$ be a vertex of $\KKK_{\delta_w}$ such that $I_w(C)$ and $J$ are connected by an edge in $\KKK_{\delta_w}$. Then there exists $D\in\CMin(w)$ such that $J=I_w(D)$, and such that $\Cyc(\pi_w(C))$ and $\Cyc(\pi_w(D))$ are connected by an edge in $\KKK_{\OOO_w}$.
\end{corollary}
\begin{proof}
By assumption, there exists a $\delta_w$-invariant spherical subset $L\subseteq S^\eta$ containing $I_w(C)$ such that $J=\op_L(I_w(C))$. Hence the claim follows from Proposition~\ref{prop:phisurj_and_phiinvedgeedge}(2).
\end{proof}

Next, we show that $\KKK_{\OOO_w}$ is connected; together with the fact, proved below, that $\varphi_w$ preserves edges, this will imply that the image of $\varphi_w$ is the desired connected component of $\KKK_{\delta_w}$. Recall from Definition~\ref{definition:Kdelta} the definition of spherical paths in $\KKK_{\delta_w}$.

\begin{lemma}\label{lemma:elemtightconjtosphericalpath}
Let $R$ be a spherical residue such that $w$ normalises $\Stab_W(R)$, and let $C,D\in R\cap\CMin(w)$. Then $I_w(C)$ and $I_w(D)$ are connected by a spherical path in $\KKK_{\delta_w}$.
\end{lemma}
\begin{proof}
By Lemma~\ref{lemma:corresp_wresidues}(3), the restriction of $\pi_{\Sigma^\eta}$ to $R$ is a cellular isomorphism onto a residue $R^\eta$ of $\Sigma^\eta$, and $w_\eta$ normalises $\Stab_{W^\eta}(R^\eta)$. Since $C^\eta,D^\eta\in R^\eta\cap\CMin(w_\eta)$ by Proposition~\ref{prop:corresp_Minsets2}, Lemma~\ref{lemma:CDwetainvaraintresidue}(1) implies that there is a $w_\eta$-invariant spherical residue $\overline{R}^\eta$ containing $C^\eta,D^\eta$.

Write $C^\eta=aC_0^\eta$ with $a\in W^\eta$, so that $\Stab_{W^\eta}(\overline{R}^\eta)=aW^\eta_La\inv$ for some spherical subset $L\subseteq S^\eta$. Let also $x\in W^\eta_{L}$ such that $D^\eta=axC_0^\eta$, so that $v:=\pi_{w_\eta}(D^\eta)=x\inv ux$, where $u:=\pi_{w_\eta}(C^\eta)=a\inv w_\eta a\in W^\eta_L$. We have to show that $I_w(C)=\supp(u)$ is connected to $I_w(D)=\supp(v)$ by a spherical path in $\KKK_{\delta_w}$. But as $u,v$ are conjugate in $W^\eta_L$, this follows from Theorem~\ref{thm:main_finite_order} applied with $W:=W^\eta_L$.
\end{proof}

\begin{prop}\label{prop:KOwconnected}
Let $C,D\in\CMin(w)$. Then $I_w(C)$ and $I_w(D)$ are connected by a path in $\KKK_{\delta_w}$. In particular, $\KKK_{\OOO_w}$ is connected.
\end{prop}
\begin{proof}
If $C,D$ are adjacent, then $I_w(C)=I_w(D)$ by Lemma~\ref{lemma:IwCconstantgalleries} (and Proposition~\ref{prop:corresp_galleries_CMin2}). And if $C,D$ belong to a spherical residue $R$ such that $w$ normalises $\Stab_W(R)$, then $I_w(C)$ and $I_w(D)$ are connected by a path in $\KKK_{\delta_w}$ by Lemma~\ref{lemma:elemtightconjtosphericalpath}. Hence the first statement follows from Lemma~\ref{lemma:prop34}(4). 

Together with Corollary~\ref{corollary:phisurj_and_phiinvedgeedge}, this implies that there exists a chamber $D_1\in\CMin(w)$ with $I_w(D_1)=I_w(D)$ such that $\Cyc(\pi_w(C))$ and $\Cyc(\pi_w(D_1))$ are connected by a path in $\KKK_{\OOO_w}$. Since $\Cyc(\pi_w(D_1))=\Cyc(\pi_w(D))$ by Proposition~\ref{prop:R2R3distinguishccc}(3)$\Rightarrow$(1), the second statement follows as well.
\end{proof}

Finally, we show that $\varphi_w$ preserves edges.

\begin{lemma}\label{lemma:phiedgeedge}
Let $C,D\in\CMin(w)$ be such that $\pi_w(C)$ and $\pi_w(D)$ are $K$-conjugate for some spherical subset $K$ of $S$. Then there exist a $\delta_w$-stable spherical subset $L\subseteq S^\eta$ containing $I_w(C)$ and an automorphism $\sigma\in\Xi_w$ such that $\sigma(I_w(D))=\op_L(I_w(C))$.
\end{lemma}
\begin{proof}
Let $R:=R_K(C)$ be the $K$-residue containing $C$. By Lemma~\ref{lemma:Ktightlyconjuagte_geometric}, there exists some $v\in\ZZZ_W(w)$ such that $C$ and $D_1:=vD\in\CMin(w)$ are opposite chambers of $R$, and $w$ normalises $\Stab_W(R)$. By Lemma~\ref{lemma:corresp_wresidues}(3), the restriction of $\pi_{\Sigma^\eta}$ to $R$ is a cellular isomorphism onto a residue $R^{\eta}$ of $\Sigma^\eta$, and $w_\eta$ normalises $\Stab_{W^\eta}(R^\eta)$. In particular, the chambers $C^\eta,D_1^\eta$ are opposite in $R^\eta$, and belong to $\CMin_{\Sigma^\eta}(w_\eta)$ by Proposition~\ref{prop:corresp_Minsets2}. Moreover, by Lemma~\ref{lemma:I_w(C)geometric}, the residues $R_{I_w(C)}(C^\eta)$ and $R_{I_w(D_1)}(D_1^\eta)$ of $\Sigma^\eta$ are the smallest $w_\eta$-invariant residues containing $C^\eta$ and $D_1^\eta$, respectively. Denoting by $\overline{R}^\eta$ the smallest residue of $\Sigma^\eta$ containing $R_{I_w(C)}(C^\eta)$ and $R_{I_w(D_1)}(D_1^\eta)$, it then follows from Lemma~\ref{lemma:CDwetainvaraintresidue} (and Proposition~\ref{prop:corresp_Minsets2}) that $\overline{R}^\eta$ is a $w_\eta$-invariant spherical residue, and that $C^\eta$ and $D_2^\eta$ are opposite chambers in $\overline{R}^\eta$ for some $D_2\in\CMin(w)$ such that $D_2^\eta\in R_{I_w(D_1)}(D_1^\eta)$. Moreover, $R_{I_w(D_1)}(D_1^\eta)$ is the smallest $w_\eta$-invariant residue containing $D_2^\eta$, and hence $I_w(D_1)=I_w(D_2)$ by Lemma~\ref{lemma:I_w(C)geometric}.

Let $L\subseteq S^\eta$ be the type of $\overline{R}^\eta$, so that $\overline{R}^\eta=R_L(C^\eta)$. Then $\delta_w(L)=L$ and $L\supseteq I_w(C)$ by Lemma~\ref{lemma:I_w(C)geometric}. Let $a\in W^\eta$ be such that $C^\eta=aC_0^\eta$. Since $C^\eta$ and $D_2^\eta$ are opposite chambers of $R_L(C^\eta)$, we then have $$\pi_{w_\eta}(D_2^\eta)=\pi_{w_\eta}(aw_0(L)C_0^\eta)=w_0(L)\pi_{w_\eta}(C^\eta)w_0(L),$$
so that $I_w(D_1)=I_w(D_2)=\op_L(I_w(C))$. Since $I_w(D_1)=\sigma(I_w(D))$ for some $\sigma\in\Xi_w$ by Lemma~\ref{lemma:IwvC_sigmaIwC}(1), the lemma follows.
\end{proof}

\begin{definition}\label{definition:quotientgraph}
Recall from \S\ref{subsection:SATCG} the definition of the graphs $\KKK_{\OOO_w}$ and $\KKK_{\delta_w}=\KKK_{\delta_w,W^{\eta}}$. Note that every element of $\Xi_w$ commutes with $\delta_w$, as follows from Lemma~\ref{lemma:R2} below (applied to $(W^\eta,S^\eta)$). Hence the group $\Xi_w$ acts by graph automorphisms on $\KKK_{\delta_w}$. Let\index[s]{Kdeltaw0IwmodXiw@$\KKK_{\delta_w}^0(I_w)/\Xi_w$ (quotient graph of $\KKK_{\delta_w}^0(I_w)$ by the action of $\Xi_w$)} $\KKK_{\delta_w}^0(I_w)/\Xi_w$ denote the corresponding quotient graph of the connected component $\KKK_{\delta_w}^0(I_w)$ of $I_w$ in $\KKK_{\delta_w}$, namely, the graph with vertex set the equivalence classes\index[s]{-04@$[I]$ (equivalence class of $I$ in the quotient $\KKK_{\delta_w}^0(I_w)/\Xi_w$)} $[I]$ of vertices $I$ of $\KKK_{\delta_w}^0(I_w)$ (where the vertices $I,J$ of $\KKK_{\delta_w}^0(I_w)$ are in the same class if they belong to the same $\Xi_{w}$-orbit), and with an edge between $[I]$ and $[J]$ if there exist $I'\in [I]$ and $J'\in [J]$ connected by an edge in $\KKK_{\delta_w}^0(I_w)$. 

In the sequel, we will make the notational abuse of identifying vertices $I$ of $\KKK_{\delta_w}^0(I_w)$ with their equivalence class $[I]$ in $\KKK_{\delta_w}^0(I_w)/\Xi_w$.
\end{definition}

\begin{lemma}\label{lemma:R2}
Assume that $u_1\delta_1,u_2\delta_2\in\Aut(\Sigma)$ commute  for some $u_i\in W$ and $\delta_i\in\Aut(W,S)$. Then $\delta_1$ and $\delta_2$ commute.
\end{lemma}
\begin{proof}
For each $s\in S$, let $\sigma_s$ be the panel of $C_0$ of type $s$. By assumption, $u_1\delta_1(u_2)C_0=u_1\delta_1u_2\delta_2 C_0=u_2\delta_2u_1\delta_1C_0=u_2\delta_2(u_1)C_0$ and hence $u_1\delta_1(u_2)=u_2\delta_2(u_1)$. Similarly, for all $s\in S$, 
$$u_1\delta_1(u_2)\sigma_{\delta_1\delta_2(s)}=u_1\delta_1u_2\delta_2 \sigma_s=u_2\delta_2u_1\delta_1\sigma_s=u_2\delta_2(u_1)\sigma_{\delta_2\delta_1(s)},$$
so that $\sigma_{\delta_1\delta_2(s)}=\sigma_{\delta_2\delta_1(s)}$ and hence $\delta_1\delta_2(s)=\delta_2\delta_1(s)$, as desired.
\end{proof}

\begin{remark}\label{remark:wetacyclredIwinit}
If $w_\eta$ is cyclically reduced, then there exists $C\in\CMin(w)$ such that $C_0^\eta=C^\eta$ by Proposition~\ref{prop:corresp_Minsets2}, and hence such that $I_w=I_w(C)$.
\end{remark}

Before proving the main result of this subsection, we clarify the relationship between the sets $I_w(C)$ and the vertices of $\KKK_{\delta_w}$. 
\begin{lemma}\label{lemma:vertexset_IwCMinw}
Assume that $w_\eta$ is cyclically reduced. Then
$$I_w(\CMin(w)):=\{I_w(C) \ | \ C\in\CMin(w)\}$$
coincides with the vertex set of $\KKK_{\delta_w}^0(I_w)$.

In particular, $\Xi_w$ stabilises the vertex set of $\KKK_{\delta_w}^0(I_w)$.
\end{lemma}
\begin{proof}
By Corollary~\ref{corollary:phisurj_and_phiinvedgeedge}, every vertex of $\KKK^0_{\delta_w}(I_w)$ is in $I_w(\CMin(w))$. The converse inclusion follows from Proposition~\ref{prop:KOwconnected}. The second statement then follows from Lemma~\ref{lemma:IwvC_sigmaIwC}(2).
\end{proof}

\begin{theorem}\label{theorem=R1R2implythmA}
Let $(W,S)$ be a Coxeter system. Let $w\in \Aut(\Sigma)$ be of infinite order, and set $\eta:=\eta_w\in\partial X$. Assume that $w_\eta$ is cyclically reduced. Then there is a graph isomorphism $$\varphi_w\co\KKK_{\OOO_w}\to \KKK_{\delta_w}^0(I_w)/\Xi_w$$
defined on the vertex set of $\KKK_{\OOO_w}$ by the assignment
$$\Cyc(\pi_w(C))\mapsto I_w(C)\quad\textrm{for all $C\in\CMin(w)$}.$$ 
\end{theorem}
\begin{proof}
Note first that the assignment $\Cyc(\pi_w(C))\mapsto I_w(C)$ ($C\in\CMin(w)$) yields a well-defined map $\varphi$ from the vertex set of $\KKK_{\OOO_w}$ to the vertex set of $\KKK_{\delta_w}/\Xi_w$. Indeed, this amounts to show that if $\pi_w(C)$ and $\pi_w(D)$ are in the same cyclic shift class, then $I_w(D)=\sigma(I_w(C))$ for some $\sigma\in\Xi_w$. But this follows from Proposition~\ref{prop:R2R3distinguishccc}(1)$\Rightarrow$(3). 

By Proposition~\ref{prop:R2R3distinguishccc}(3)$\Rightarrow$(1), the map $\varphi$ is injective. Moreover, $\varphi$ maps an edge of $\KKK_{\OOO_w}$ to an edge of $\KKK_{\delta_w}/\Xi_w$ by Lemma~\ref{lemma:phiedgeedge}. In particular, since $\KKK_{\OOO_w}$ is connected by Proposition~\ref{prop:KOwconnected}, the image of $\varphi$ is contained in the vertex set of $\KKK_{\delta_w}^0(I_w)/\Xi_w$ (see Remark~\ref{remark:wetacyclredIwinit}). We then conclude from Corollary~\ref{corollary:phisurj_and_phiinvedgeedge} that $\varphi$ corestricts to a graph isomorphism $\KKK_{\OOO_w}\to \KKK_{\delta_w}^0(I_w)/\Xi_w$. 
\end{proof}

Here is the proof of the second part of Theorem~\ref{thmintro:graphisomorphism}.

\begin{prop}\label{prop:detailsofoperations}
Assume that $w_\eta$ is cyclically reduced, and let $\varphi_w$ be as in Theorem~\ref{theorem=R1R2implythmA}. Let $I_w=J_0\stackrel{L_1}{\too}J_1\stackrel{L_2}{\too}\dots\stackrel{L_m}{\too}J_m$ be a path in $\KKK_{\delta_w}^0(I_w)$, and set $u_i:=w_0(L_1)w_0(L_2)\dots w_0(L_i)\in W^\eta$ and $w'_i:=u_i\inv wu_i$ for each $i=0,\dots,m$, where $u_0:=1$. Then the following assertions hold:
\begin{enumerate}
\item
$\varphi_{w}\inv([J_i])=\Cyc_{\min}(w'_i)$ for each $i=1,\dots,m$.
\item
There exist a spherical subset $K_i\subseteq S$ and $a_i\in W$ of minimal length in $a_iW_{K_i}$ with $W^\eta_{L_i}=a_iW_{K_i}a_i\inv$ such that $w_{i-1}:=a_i\inv w'_{i-1} a_i$ is cyclically reduced for each $i=1,\dots,m$. 
\item
For any $a_i,K_i$ and $w_i$ as in (2), $\varphi_{w}\inv([J_i])=\Cyc(w_i)$ for each $i$, and 
$$w\to  w_0\stackrel{K_1}{\too}\op_{K_1}(w_0)\to w_1 \stackrel{K_2}{\too}\dots \to w_{m-1}\stackrel{K_m}{\too}\op_{K_m}(w_{m-1})\leftarrow w'_m.$$
\end{enumerate}
\end{prop}
\begin{proof}
Set $C_i:=u_iC_0$ for each $i=1,\dots,m$, so that $w'_i=\pi_w(C_i)$. We construct inductively a sequence of chambers $D_0,D_1,\dots,D_m\in\CMin(w)$ with $D_i^\eta=C_i^\eta$ and $I_w(D_i)=J_i$ (equivalently, $\varphi_{w}\inv([J_i])=\Cyc(\pi_w(D_i))$) for each $i=0,\dots,m$, as follows. Let $D_0\in\CMin(w)$ with $D_0^\eta=C_0^\eta$ (see Proposition~\ref{prop:corresp_Minsets2}), so that $I_w(D_0)=I_w=J_0$ and $w\to\pi_w(D_0)$ by Proposition~\ref{prop:wdecreasingfromCtoD}. Suppose that we have already constructed the chambers $D_0,\dots,D_{i-1}$ for some $i\geq 1$.

Since $D_{i-1}^\eta=C_{i-1}^\eta=u_{i-1}C_0^\eta$,  Proposition~\ref{prop:phisurj_and_phiinvedgeedge} (applied with $C:=D_{i-1}$ and $L:=L_i$) yields some spherical subset $K_i\subseteq S$ and some $a_i\in W$ of minimal length in $a_iW_{K_i}$ such that $W^\eta_{L_i}=a_iW_{K_i}a_i\inv$ and $D_{i-1}':=u_{i-1}a_iC_0\in\CMin(w)$ (equivalently, $w_{i-1}:=a_i\inv w'_{i-1} a_i=\pi_w(D_{i-1}')$ is cyclically reduced). 

Moreover, for any such $a_i,K_i$, we have $$\pi_w(D_{i-1})\to\pi_w(D_{i-1}')=w_{i-1},$$ and the chamber $D_i\in\CMin(w)$ opposite $D_{i-1}'$ in $R_{K_i}(D_{i-1}')$ satisfies $I_w(D_i)=\op_{L_i}(J_{i-1})=J_i$ and is such that $$\pi_w(D_{i-1}')=w_{i-1}\stackrel{K_i}{\too}\pi_w(D_i)=\op_{K_i}(w_{i-1}).$$ Finally, $D_i^\eta$ is opposite $D_{i-1}^\eta$ in $R_{L_i}(D_{i-1}^\eta)$, that is, $D_i^\eta=u_{i-1}w_0(L_i)C_0^\eta=u_iC_0^\eta=C_i^\eta$, thus completing the induction step.

Since $D_i^\eta=C_i^\eta$ for each $i$, we have $\Cyc(\pi_w(D_i))=\Cyc_{\min}(\pi_w(C_i))$ by Proposition~\ref{prop:wdecreasingfromCtoD}, proving (1). The statements (2) and (3) now easily follow: since $w=\pi_w(C_0)\to\pi_w(D_0)\to\pi_w(D_0')$, the sequence of conjugates of $w$ in (3) corresponds under $\pi_w$ to the sequence of chambers $C_0,D_0',D_1,D_1',\dots,D_{m-1}',D_m,C_m$.
\end{proof}


\subsection{The tight conjugation graph of an infinite order element}\label{subsection:TTCGOAIOE}

We conclude this section by proving an analogue of Theorem~\ref{theorem=R1R2implythmA} for the tight conjugation graph associated to $\OOO_w$, where $w\in W$ has infinite order.

We start with an analogue of Lemma~\ref{lemma:Ktightlyconjuagte_geometric}.
\begin{lemma}\label{lemma:ETC_geometric}
Let $C,D\in\CMin(w)$ with $\Cyc(\pi_w(C))\neq\Cyc(\pi_w(D))$. Then the following assertions hold.
\begin{enumerate}
\item
If $\pi_w(C)$ and $\pi_w(D)$ are elementarily tightly conjugate, there exists $z\in\ZZZ_W(w)$ such that $C,zD$ belong to a spherical residue $R$ such that $w$ normalises $\Stab_W(R)$.
\item
If $C,D$ belong to a spherical residue $R$ such that $w$ normalises $\Stab_W(R)$, then there exists $D'\in\CMin(w)$ with $\Cyc(\pi_w(D'))=\Cyc(\pi_w(D))$ such that $\pi_w(C)$ and $\pi_w(D')$ are elementarily tightly conjugate.
\end{enumerate}
\end{lemma}
\begin{proof}
(1) Assume that $u:=\pi_w(C)$ and $v:=\pi_w(D)$ are elementarily tightly conjugate. Let $K\subseteq S$ be spherical such that $u$ normalises $W_K$, and let $x\in W_K$ with $v=x\inv ux$. Let $a\in W$ with $C=aC_0$, and set $R:=R_K(C)=aR_K$ and $D':=axC_0\in R$. Then $w=aua\inv$ normalises $\Stab_W(R)=aW_Ka\inv$ and $\pi_w(D')=v=\pi_w(D)$. In particular, $D'=zD$ for some $z\in\ZZZ_W(w)$, yielding (1).

(2) Write $C=aC_0$ and let $K\subseteq S$ be spherical with $R=aR_K$. Then $u:=\pi_w(C)=a\inv wa$ normalises $W_K=a\inv\Stab_W(R)a$, and there exists $x\in W_K$ such that $D=axC_0$. Thus $v:=\pi_w(D)=x\inv ux$, and Lemma~\ref{lemma:Kconjugate_tightlyconjugate}(1) yields some $v'\in W$ with $\Cyc(v')=\Cyc(v)$ (say $v'=\pi_w(D')$ for some $D'\in\CMin(w)$) such that $u,v'$ are elementarily tightly conjugate, yielding (2).
\end{proof}

Recall from \S\ref{subsection:SATCG} the definition of the graphs $\KKK^t_{\OOO_w}$ and $\overline{\KKK}_{\delta_w}=\overline{\KKK}_{\delta_w,W^{\eta}}$.

\begin{theorem}\label{theorem=R1R2implyCorE}
Let $(W,S)$ be a Coxeter system. Let $w\in W$ be of infinite order, and set $\eta:=\eta_w\in\partial X$. Assume that $w_\eta$ is cyclically reduced. Then there is a graph isomorphism $$\KKK^t_{\OOO_w}\to \overline{\KKK}_{\delta_w}^0(I_w)/\Xi_w$$
defined on the vertex set of $\KKK^t_{\OOO_w}$ by the assignment
$$\Cyc(\pi_w(C))\mapsto I_w(C)\quad\textrm{for all $C\in\CMin(w)$}.$$ 
\end{theorem}
\begin{proof}
Let $C,D\in\CMin(w)$ with $\Cyc(\pi_w(C))\neq\Cyc(\pi_w(D))$. In view of Theorem~\ref{theorem=R1R2implythmA}, it is sufficient to prove the following two assertions:
\begin{enumerate}
\item
If $\pi_w(C)$ and $\pi_w(D)$ are elementarily tightly conjugate, then there exists $\sigma\in\Xi_w$ such that $I_w(C)$ and $\sigma(I_w(D))$ are connected by a spherical path in $\KKK_{\delta_w}$. 
\item
If $I_w(C)$ and $I_w(D)$ are connected by a spherical path in $\KKK_{\delta_w}$, then the classes $\Cyc(\pi_w(C))$ and $\Cyc(\pi_w(D))$ have elementarily tightly conjugate representatives.
\end{enumerate}

(1) Assume first that $\pi_w(C)$ and $\pi_w(D)$ are elementarily tightly conjugate. By Lemma~\ref{lemma:ETC_geometric}(1), there exists $z\in\ZZZ_W(w)$ such that $C,zD$ belong to a spherical residue $R$ such that $w$ normalises $\Stab_W(R)$. Since $\pi_w(zD)=\pi_w(D)$ and $I_w(zD)=\sigma(I_w(D))$ for some $\sigma\in\Xi_w$ by Lemma~\ref{lemma:IwvC_sigmaIwC}(1), the claim follows from Lemma~\ref{lemma:elemtightconjtosphericalpath}.

(2) Conversely, assume that $I_w(C)$ and $I_w(D)$ are connected by a spherical path $I_w(C)=I_0\stackrel{L_1}{\too}I_1\stackrel{L_2}{\too}\dots\stackrel{L_k}{\too}I_k=I_w(D)$ in $\KKK_{\delta_w}$, so that $\overline{L}:=\bigcup_{i=1}^kL_i\subseteq S^\eta$ is spherical. Note that $R_{\overline{L}}(C^\eta)$ is $w_\eta$-stable by Lemma~\ref{lemma:I_w(C)geometric}. Lemma~\ref{lemma:corresp_wresidues}(4) then yields a $w$-essential point $x\in\Min(w)$ such that the restriction of $\pi_{\Sigma^\eta}$ to $R_x$ is a cellular isomorphism onto the  $w_\eta$-stable residue $R_L(C^\eta)$ for some spherical subset $L\subseteq S^\eta$ containing $\overline{L}$.

Set $\overline{C}:=\proj_{R_x}(C)$, so that $\overline{C}\in R_x\cap\CMin(w)$ by Lemma~\ref{lemma:prop34}(2), $\overline{C}^\eta=C^\eta$ by Lemma~\ref{lemma:corresp_wresidues}(1), and $\Cyc(\pi_w(C))=\Cyc(\pi_w(\overline{C}))$ by Proposition~\ref{prop:wdecreasingfromCtoD}. We now construct inductively a sequence of chambers $D_0,D_1,\dots,D_k\in R_x\cap\CMin(w)$ such that $I_w(D_i)=I_i$ for each $i=0,\dots,k$. 

We set $D_0:=\overline{C}$, so that $I_w(D_0)=I_w(\overline{C})=I_w(C)$. Suppose that we already constructed $D_{i-1}$ ($i\in\{1,\dots,k\}$), and let us construct $D_i$. Write $D_{i-1}^\eta=aC_0^\eta$ with $a\in W^\eta$ and $D_{i-1}=abC_0$ with $b\in W$. Let $K_i\subseteq S$ be such that the restriction of $\pi_{\Sigma^\eta}$ to $abR_{K_i}=R_{K_i}(D_{i-1})\subseteq R_x$ is a cellular isomorphism onto $R_{L_i}(D_{i-1}^\eta)\subseteq R_L(C^\eta)$. Then $abW_{K_i}(ab)\inv=\Stab_W(R_{K_i}(D_{i-1}))=\Stab_{W^\eta}(R_{L_i}(D_{i-1}^\eta))=aW^\eta_La\inv$. Moreover, $b$ is of minimal length in $bW_{K_i}$, as $bC_0=a\inv D_{i-1}=\proj_{bR_{K_i}}(C_0)$ (because $a\inv D_{i-1}^\eta=C_0^\eta$). Hence Proposition~\ref{prop:phisurj_and_phiinvedgeedge}(2) implies that the chamber $D_i\in R_x\cap\CMin(w)$ opposite $D_{i-1}$ in $R_{K_i}(D_{i-1})$ satisfies $I_w(D_i)=\op_{L_i}(I_w(D_{i-1}))=\op_{L_i}(I_{i-1})=I_i$, as desired. 

We conclude that $D_k\in R_x\cap\CMin(w)$ is such that $I_w(D_k)=I_k=I_w(D)$. In particular, $\Cyc(\pi_w(D_k))=\Cyc(\pi_w(D))$ by Proposition~\ref{prop:R2R3distinguishccc}(3)$\Rightarrow$(1). As $D_0,D_k\in R_x$ and $\Cyc(\pi_w(C))=\Cyc(\pi_w(D_0))$, the statement (2) now follows from Lemma~\ref{lemma:ETC_geometric}(2).
\end{proof}


\section{The \texorpdfstring{$P$-splitting}{P-splitting} of an element}\label{section:TPSOAE}
Throughout this section, we fix a Coxeter system $(W,S)$ and an element $w\in \Aut(\Sigma)$.

In this short section, we introduce certain decompositions of $w$, which may be thought of as analogues of the decomposition of an affine isometry as a product of a rotation and of a translation. These decompositions will be used (and further investigated) in two important special cases in \S\ref{section:ICG} and \S\ref{section:ACG}.

\begin{definition}\label{definition:wparabolicsubgroup}
Recall from \S\ref{subsection:VBAAOCS} the definition of a $w$-essential point for $w$ of infinite order. We extend this terminology to elements $w$ of finite order by calling any $x\in\Min(w)$ a {\bf $w$-essential} point\index{essential@$w$-essential}.

Call a parabolic subgroup $P$ of $W$ a {\bf $w$-parabolic subgroup}\index{parabolicsubgroup@$w$-parabolic subgroup} if there exists a $w$-essential point $x\in\Min(w)$ such that $P=\Fix_W(x)=\Stab_W(R_x)$. Equivalently, for $w$ of infinite order, $P$ is the fixer of a $w$-axis (resp.  the stabiliser of a $w$-residue).
\end{definition}

\begin{example}\label{example:PsplittingPwmin}
By Proposition~\ref{prop:basicprop_finiteCox} and Lemma~\ref{lemma:generic_axis}, the parabolic subgroup\index[s]{Pwmin@$P_w^{\min}$ (pointwise fixer of $\Min(w)$)} $$P_w^{\min}:=\Fix_W(\Min(w))$$ is a $w$-parabolic subgroup: it is in fact the smallest $w$-parabolic subgroup.
\end{example}

Here is another source of examples of $w$-parabolic subgroups.
\begin{lemma}\label{lemma:Pmaxwparabolic}
Let $P$ be a maximal spherical parabolic subgroup normalised by $w$. Then $P$ is a $w$-parabolic subgroup. 
\end{lemma}
\begin{proof}
Let $M$ be the set of walls of $P$, and let $Z:=\bigcap_{m\in M}m\subseteq X$. By assumption, $Z$ is a nonempty closed convex subset of $X$ stabilised by $w$, and hence contains a $w$-essential point $x\in\Min(w)$: this is clear for $w$ of finite order, and when $w$ has infinite order, $Z$ contains a $w$-axis $L$ since a $w$-axis is the convex closure of the $\langle w\rangle$-orbit of any of its points. In particular, the $w$-parabolic subgroup $P':=\Fix_W(x)$ contains $P=\Fix_W(Z)$. By maximality of $P$, we then have $P=P'$.
\end{proof}

\begin{lemma}\label{lemma:waxisawaxis}
Let $P$ be a $w$-parabolic subgroup, and $x\in\Min(w)$ with $P=\Fix_W(x)$. Let $a\in P$. Then $x\in\Min(aw)$. Moreover, if $w$ has infinite order, the $w$-axis through $x$ is also an $aw$-axis.
\end{lemma}
\begin{proof}
If $w$ has finite order, then $awx=ax=x$, as desired. Assume now that $w$ has infinite order. Let $L_x$ be the $w$-axis through $x$. Then for any $a\in P$, we have $(aw)\inv x=w\inv x\in L_x$ and $awx=w\cdot w\inv aw x=wx\in L_x$ (as $w\inv aw\in P$), and hence $L_x$ is also an $aw$-axis (see e.g. \cite[Chapter II, Proposition~1.4(2)]{BHCAT0}).
\end{proof}

\begin{definition}\label{definition:Preduced}
Let $P$ be a spherical parabolic subgroup of $W$. Call an element $v\in \Aut(\Sigma)$ {\bf $P$-reduced}\index{reduced@$P$-reduced} if $C_0$ and $vC_0$ lie on the same side of every wall of $P$. 

For instance, if $P=W_I$ for some $I\subseteq S$, then $v$ is $P$-reduced if and only if $v$ is of minimal length in $W_Iv$.
\end{definition}

\begin{lemma}\label{lemma:PvinvPvreduced}
Let $R$ be a spherical residue such that $w$ normalises $P:=\Stab_W(R)$, and let $v\in W$ be such that $vC_0=\proj_R(C_0)$. Then $w$ is $P$-reduced if and only if $v\inv wv$ is $v\inv Pv$-reduced.
\end{lemma}
\begin{proof}
By assumption, $C_0$ and $vC_0$ lie on the same side of every wall of $P$. Since $w$ stabilises this set of walls, $wC_0$ and $wvC_0$ also lie on the same side of every wall of $P$. Hence $w$ is $P$-reduced if and only if $vC_0$ and $wvC_0$ lie on the same side of every wall of $P$, that is, if and only if $v\inv wv$ is $v\inv Pv$-reduced.
\end{proof}

Recall from \cite{straight} that an element $v\in \Aut(\Sigma)$ is called\index{Straight} {\bf straight} if $\ell(v^n)=n\ell(v)$ for all $n\in\NN$. We now associate to each $w$-parabolic subgroup $P$ a decomposition of $w$ as a product of an element $w_{\tor}\in P$ with a $P$-reduced element $w_{\infty}\in W$, which intuitively can be thought of as the ``torsion part'' and ``straight part'' of $w$ with respect to $P$, respectively. 

\begin{prop}\label{prop:Psplitting}
Let $P$ be a $w$-parabolic subgroup. Then there are uniquely determined elements\index[s]{w0P@$w_{\tor}(P)$ (torsion part of $w$ wrt $P$)}\index[s]{winfP@$w_{\tor}(P)$ (straight part of $w$ wrt $P$)} $w_{\tor}=w_{\tor}(P),w_{\infty}=w_{\infty}(P)\in \Aut(\Sigma)$ with $w=w_{\tor}w_{\infty}$ such that $w_{\tor}\in P$ and
such that $w_{\infty}$ is $P$-reduced. 

Moreover, if $P=\Stab_W(R_x)$ for some $w$-essential point $x\in\Min(w)$, and $v\in W$ is such that $\proj_{R_x}(C_0)=vC_0$, then $v\inv w_{\infty}v$ is straight.
\end{prop}
\begin{proof}
Let $x\in \Min(w)$ be $w$-essential and such that $\Stab_W(R_x)=P$. Let $C:=\proj_{R_x}(C_0)$ and let $a\in P$ with $aC=\proj_{R_x}(wC_0)$. Then $\proj_{R_x}(a\inv wC_0)=C$, and hence $C_0$ and $a\inv wC_0$ lie on the same side of every wall of $P$, that is, $a\inv w$ is $P$-reduced. We may thus set $w_{\tor}:=a$ and $w_{\infty}:=a\inv w$. For the uniqueness statement, note that if $b\inv w$ is $P$-reduced for some $b\in P$, then $C=\proj_{R_x}(b\inv wC_0)=b\inv\proj_{R_x}(wC_0)=b\inv a C$, so that $b=a$.

To prove the last statement, let $v\in W$ be such that $C=\proj_{R_x}(C_0)=vC_0$. If $w$ has finite order, so that $wR_x=R_x$, we have $aC=wC$ and hence $w_{\infty}C=C$. Thus, in that case, $v\inv w_{\infty}v\in\Aut(\Sigma,C_0)$ has length $0$ and is trivially straight. We may thus assume that $w$ has infinite order. Since $w_{\infty}$ is $P$-reduced, Lemma~\ref{lemma:PvinvPvreduced} implies that $C_0$ and $v\inv w_{\infty}vC_0$ lie on the same side of every wall of $R_{v\inv x}$. 
As $v\inv x\in C_0$ and is a $v\inv w_{\infty}v$-essential point by Lemma~\ref{lemma:waxisawaxis} (so that the walls of $R_{v\inv x}$ are those containing the $v\inv w_{\infty}v$-axis through $v\inv x$), \cite[Lemma~4.3]{straight} (which is stated for elements of $W$ but whose proof applies \emph{verbatim} to elements of $\Aut(\Sigma)$) implies that $v\inv w_{\infty}v$ is straight, thus concluding the proof of the proposition.
\end{proof}

\begin{definition}
Let $P$ be a $w$-parabolic subgroup. We call the decomposition $w=w_{\tor}w_{\infty}$ with $w_{\tor}\in P$ provided by Proposition~\ref{prop:Psplitting} the {\bf $P$-splitting}\index{splitting@$P$-splitting} of $w$.
\end{definition}

Here is another way to compute the $P$-splitting of $w$.

\begin{lemma}\label{Psplittingwini}
Let $P$ be a $w$-parabolic subgroup. Let $I\subseteq S$ and $v\in W$ of minimal length in $vW_I$ be such that $P=vW_Iv\inv$. Write $v\inv wv=w_In_I$ with $w_I\in W_I$ and $n_I\in \widetilde{N}_I$, as in Remark~\ref{remark:NWWINIWI_AutSigma}. Then $w_{\tor}(P)=vw_Iv\inv$ and $w_{\infty}(P)=vn_Iv\inv$.
\end{lemma}
\begin{proof}
By construction, $vw_Iv\inv\in P$. On the other hand, $n_I$ is $W_I$-reduced (cf. Remark~\ref{remark:NWWINIWI_AutSigma}). Since $\proj_R(C_0)=vC_0$ by assumption on $v$, where $R:=vR_I$, Lemma~\ref{lemma:PvinvPvreduced} then implies that $vn_Iv\inv$ is $P$-reduced, yielding the lemma.
\end{proof}

$P$-splittings offer a useful criterion to check whether an element $w$ is cyclically reduced (see also Corollary~\ref{corollary:tmuuinW}).

\begin{lemma}\label{lemma:criterioncyclicallyreducedPsplitting}
Let $P$ be a $w$-parabolic subgroup and let $w=w_{\tor}w_{\infty}$ be the $P$-splitting of $w$, with $w_{\tor}\in P$. Let $\delta_{\infty}\co P\to P:x\mapsto w_{\infty}xw_{\infty}\inv$. Then the following assertions hold:
\begin{enumerate}
\item
There exists $v\in W$ such that $v\inv w v$ is cyclically reduced and $v\inv P v$ is standard.
\item
If $P$ is standard, then $w$ is cyclically reduced if and only if $w_{\infty}$ is cyclically reduced and $w_{\tor}\delta_{\infty}$ is cyclically reduced in $P$.
\end{enumerate}
\end{lemma}
\begin{proof}
(1) Let $x\in\Min(w)$ be $w$-essential and such that $P=\Stab_W(R_x)$. Then $R_x$ contains a chamber $C\in\CMin(w)$ by Lemmas~\ref{lemma:prop34}(2) and \ref{lemma:prop34fin}(2), and we may thus choose $v\in W$ such that $C=vC_0$. 

(2) Let $I\subseteq S$ with $P=W_I$, and write $w=w_In_I$ with $w_I\in W_I$ and $n_I\in \widetilde{N}_I$, so that $w_{\tor}=w_I$ and $w_{\infty}=n_I$ by Lemma~\ref{Psplittingwini}. In particular, $\delta=\delta_{\infty}\co W_I\to W_I:x\mapsto n_Ixn_I\inv$ is a diagram automorphism of $(W_I,I)$.

Assume first that $n_I$ is cyclically reduced and that $w_I$ is of minimal length in $\{x\inv w_I\delta(x) \ | \ x\in W_I\}$. By (1), there exists $v\in W$ such that $v\inv w v$ is cyclically reduced and $v\inv W_Iv=W_J$ for some $J\subseteq S$. By Lemma~\ref{lemma:Kra09Prop316}, we can write $v\inv=x_{IJ}v_I$ with $v_I\in W_I$ and $x_{IJ}\in W$ such that $x_{IJ}\Pi_I=\Pi_J$. Let $\delta_{IJ}\co W\to W:u\mapsto x_{IJ}ux_{IJ}\inv$, so that $\delta_{IJ}(W_I)=W_J$ and $\delta_{IJ}(\widetilde{N}_I)=\widetilde{N}_J$. Then 
$$v\inv wv=x_{IJ}v_I w_In_Iv_I\inv x_{IJ}\inv=\delta_{IJ}(v_Iw_I\delta(v_I)\inv)\cdot \delta_{IJ}(n_I).$$ Since $\ell(\delta_{IJ}(v_Iw_I\delta(v_I)\inv))=\ell(v_Iw_I\delta(v_I)\inv)\geq \ell(w_I)$ and $\ell(\delta_{IJ}(n_I))\geq\ell(n_I)$ by assumption, we conclude that $$\ell(v\inv wv)=\ell(\delta_{IJ}(v_Iw_I\delta(v_I)\inv))+\ell(\delta_{IJ}(n_I))\geq\ell(w_I)+\ell(n_I)=\ell(w),$$
and hence that $w$ is cyclically reduced.

Assume, conversely, that $w$ is cyclically reduced. Then certainly $w_I\delta$ is cyclically reduced in $W_I$, for if $\ell(x\inv w_I\delta(x))<\ell(w_I)$ for some $x\in W_I$, then $x\inv wx=x\inv w_I\delta(x)\cdot n_I$ has length $\ell(x\inv w_I\delta(x))+\ell(n_I)<\ell(w_I)+\ell(n_I)=\ell(w)$. 

To show that $n_I$ is cyclically reduced as well, note that $P$ is also an $n_I$-parabolic subgroup by Lemma~\ref{lemma:waxisawaxis}. Hence (1) yields some $v\in W$ such that $v\inv n_I v$ is cyclically reduced and $v\inv W_Iv=W_J$ for some $J\subseteq S$. Let $x_{IJ},v_I,\delta_{IJ}$ be as above. Then 
$$\ell(w_I)+\ell(n_I)=\ell(w)\leq\ell(\delta_{IJ}(w))=\ell(\delta_{IJ}(w_I))+\ell(\delta_{IJ}(n_I))=\ell(w_I)+\ell(\delta_{IJ}(n_I)),$$ 
so that $\ell(n_I)\leq \ell(\delta_{IJ}(n_I))$. On the other hand, as $v\inv n_Iv=\delta_{IJ}(v_I\delta(v_I)\inv)\cdot\delta_{IJ}(n_I)$, we have
$$\ell(v\inv n_Iv)=\ell(v_I\delta(v_I)\inv)+\ell(\delta_{IJ}(n_I))\geq \ell(\delta_{IJ}(n_I)),$$
so that $\ell(n_I)\leq \ell(v\inv n_Iv)$ and $n_I$ is cyclically reduced, as desired.
\end{proof}

To conclude this section, we prove the following criterion, of independent interest, for an element to be straight, analogous to \cite[Theorem~D]{straight}.

\begin{corollary}\label{corollary:Psplittingstraight}
Let $P$ be a $w$-parabolic subgroup, and let $w=w_{\tor}w_{\infty}$ be the $P$-splitting of $w$, with $w_{\tor}\in P$. Then $w$ is straight if and only if $w$ is cyclically reduced and $w_{\tor}=1$.
\end{corollary}
\begin{proof}
Note first that the proofs of \cite[Lemmas~4.1 and 4.2]{straight} apply \emph{verbatim} to elements of $\Aut(\Sigma)$ and not just to elements of $W$.

If $w_{\tor}=1$ and $w=w_{\infty}$ is cyclically reduced, then $w$ has a straight conjugate by Proposition~\ref{prop:Psplitting}, and hence $w$ is itself straight by \cite[Lemma~4.2]{straight}. Conversely, assume that $w$ is straight. Then $w$ is cyclically reduced by \cite[Lemma~4.1]{straight}. Write $P=\Stab_W(R_x)$ for some $w$-essential point $x\in\Min(w)$, and let $v\in W$ be such that $\proj_{R_x}(C_0)=vC_0$, so that $v\inv Pv=W_I$ for some $I\subseteq S$, and $v\inv wv$ normalises $W_I$. Write $v\inv wv=w_In_I$ with $w_I\in W_I$ and $n_I\in \widetilde{N}_I$, as in Remark~\ref{remark:NWWINIWI_AutSigma}. Note that $vC_0\in\CMin(w)$ by Lemmas~\ref{lemma:prop34}(2) and \ref{lemma:prop34fin}(2), and hence $v\inv wv$ is cyclically reduced. In particular, $v\inv wv$ is straight by \cite[Lemma~4.2]{straight}. It then follows from \cite[Lemma~4.1]{straight} that $w_I=1$. Hence $w_{\tor}=1$ by Lemma~\ref{Psplittingwini}, as desired.
\end{proof}


\section{Indefinite Coxeter groups}\label{section:ICG}

Throughout this section, we assume that $(W,S)$ is a Coxeter system of irreducible indefinite type.

In this section, we will start by identifying $(W^\eta,S^\eta)$ and $\Sigma^\eta$ in \S\ref{subsection:TTCGACInd}. We then introduce the \emph{core splitting} of an element $w\in W$ with $\Pc(w)=W$ in \S\ref{subsection:TCSOAEInd}. This will allow us to give in \S\ref{subsection:TCOAEInd} a precise description of the centraliser of $w$, going beyond \cite[Corollary~6.3.10]{Kra09} (which states that $\ZZZ_W(w)$ contains $\langle w\rangle$ as a finite index subgroup --- note that this is in sharp contrast with the case of affine Coxeter groups), and then to compute $\Xi_w$ in \S\ref{subsection:TSXiwInd}. We will then complete the proof of Theorem~\ref{thmintro:indefinite} in \S\ref{subsection:TBACC}. To conclude, we give in \S\ref{subsection:CKASIND} a summary of the steps to be performed in order to compute the structural conjugation graph associated to $\OOO_w$, and we illustrate this recipe on some examples in \S\ref{subsection:ExIND}.


\subsection{The transversal Coxeter group and complex}\label{subsection:TTCGACInd}
In order to identify the transversal Coxeter group, we first make some elementary observations and extract a few lemmas from \cite{openKM}.

\begin{lemma}\label{lemma:rmcommuting}
Let $w\in W$ be of infinite order, and set $\eta:=\eta_w\in\partial X$. Let $m\in\WW$ be a wall of $\Sigma$. Then the following assertions are equivalent:
\begin{enumerate}
\item
$r_m$ commutes with some positive power of $w$;
\item
$m\in\WW^\eta$.
\end{enumerate}
\end{lemma}
\begin{proof}
If $w^nr_mw^{-n}=r_m$ for some $n>0$, then $w^n$ stabilises $m$ and hence $m$ contains a $w^n$-axis, yielding (2). Conversely, assume that $m\in\WW^\eta$, and let $x\in\Min(w)$. Then the geodesic ray $L:=[x,\eta)$ is contained in a tubular neighbourhood of $m$. As $wL\subseteq L$, this implies that $w^nm=m$ for some $n>0$ (because the set of walls is locally finite), or else that $r_m$ commutes with $w^n$, yielding (1).
\end{proof}

\begin{lemma}\label{lemma:PcvnW}
Let $w\in W$ with $\Pc(w)=W$. Then $\Pc(w^n)=W$ for all $n\geq 1$.
\end{lemma}
\begin{proof}
Since $\Pc(w^n)$ has finite index in $\Pc(w)$ by \cite[Lemma~2.4]{openKM}, this follows from the fact that infinite irreducible Coxeter groups have no proper parabolic subgroups of finite index (see \cite[Proposition~2.43]{BrownAbr}).
\end{proof}

Note that if $\Pc(w)=W$ for some $w\in W$, then $w$ has infinite order.
\begin{lemma}[{\cite[Corollary~2.12]{openKM}}]\label{lemma:Cor212}
Let $w\in W$ with $\Pc(w)=W$. Then there is a constant $C$ such that for every pair $m,m'$ of $w$-essential walls with $\dist(m,m')>C$, we have $\Pc(r_m,r_{m'})=\Pc(w)=W$.
\end{lemma}

The following lemma is a variation of the grid lemma (\cite[Lemma~2.8]{openKM}).
\begin{lemma}\label{lemma:grid_lemma}
Let $w\in W$ with $\Pc(w)=W$. Let $\WW_A$ be an infinite family of pairwise parallel $w$-essential walls. Then there is no infinite family $\WW_B$ of pairwise parallel walls intersecting each wall of $\WW_A$. 
\end{lemma}
\begin{proof}
Suppose for a contradiction that such a family $\WW_B$ of walls exists.
Set $A:=\langle r_m \ | \ m\in \WW_A\rangle\subseteq W$ and $B:=\langle r_m \ | \ m\in \WW_B\rangle\subseteq W$. Note that for any large enough finite subset $\WW_A'$ of $\WW_A$, Lemma~\ref{lemma:Cor212} implies that $W=\Pc(w)\subseteq\Pc(\langle r_m \ | \ m\in\WW_A'\rangle)\subseteq\Pc(A)$ and hence that $\Pc(\langle r_m \ | \ m\in\WW_A'\rangle)=\Pc(A)=W$ (and similarly for $B$). We may then apply \cite[Lemma~2.8]{openKM} to conclude that either $\Pc(A)$ and $\Pc(B)$ are of affine type, or $\Pc(A)$ and $\Pc(B)$ centralise one another. As $\Pc(A)=W$ and as $W$ has trivial center by \cite[Proposition~2.73]{BrownAbr}, we deduce that $W=\Pc(A)$ is of affine type, a contradiction.
\end{proof}

Finally, recall that by Selberg's lemma, $W$ contains a torsion-free finite index normal subgroup $W_0$.
\begin{lemma}[{\cite[Lemma~2.6]{openKM}}]\label{lemma:CM13Lemma26}
Let $w\in W_0$. Then for every wall $m$, either $wm=m$ or $wm\cap m=\varnothing$.
\end{lemma}

We are now ready to determine transversal Coxeter groups.

\begin{theorem}\label{thm:Wetafinite}
Let $(W,S)$ be a Coxeter system of irreducible indefinite type, and let $w\in W$ with $\Pc(w)=W$.
Then $W^{\eta_w}$ is the largest spherical parabolic subgroup of $W$ normalised by $w$.
\end{theorem}
\begin{proof}
Set $\eta:=\eta_w$, and assume for a contradiction that $W^\eta$ is infinite. Let $W_0$ be a torsion-free finite index normal subgroup of $W$. Let $h\in W^{\eta}\cap W_0$ be such that $\Pc(h)$ has finite index in $\Pc(W^{\eta})$ (see \cite[Corollary~2.17]{openKM}). In particular, $h$ has infinite order, for otherwise $\Pc(h)$ and $\Pc(W^{\eta})$ would be finite (see e.g. \cite[Proposition~2.87]{BrownAbr}), a contradiction.

Let $m\in\WW^\eta$ be an $h$-essential wall, so that $hm\cap m=\varnothing$ by Lemma~\ref{lemma:CM13Lemma26}. Then the reflections $r_m$ and $r_{hm}=hr_mh\inv$ belong to $W^\eta$, commute with $w^n$ for some $n>0$ by Lemma~\ref{lemma:rmcommuting}, and generate an infinite dihedral group. In particular, setting $u:=r_mr_{hm}\in W^\eta$, the set $\WW_B:=u^{\ZZ}m$ consists of pairwise parallel walls stabilised by $w^n$. Note that $\Pc(w^n)=W$ by Lemma~\ref{lemma:PcvnW}.

Let now $m'$ be a $w^{n}$-essential wall, and let $k>0$ be such that the walls in $\WW_A:=w^{nk\ZZ}m'$ are pairwise parallel (see Lemma~\ref{lemma:CM13Lemma26}) and such that the reflections associated to any two of them generate $\Pc(w^n)$ as a parabolic subgroup (see Lemma~\ref{lemma:Cor212}). Note that any wall in $\WW_{B}$ contains a $w^n$-axis, and hence intersects any wall in $\WW_{A}$ (because the walls in $\WW_A$ are $w^n$-essential, and hence intersect any $w^n$-axis). Lemma~\ref{lemma:grid_lemma} then yields the desired contradiction.

Thus $W^\eta$ is indeed finite. Moreover, $w$ normalises $W^\eta$ (and hence also the parabolic closure of $W^\eta$, which is finite by \cite[Proposition~2.87]{BrownAbr}). Finally, if $P$ is any spherical parabolic subgroup normalised by $w$, then there is some $n>0$ such that $w^n$ centralises $P$, and hence the walls of $P$ belong to $\WW^\eta$ by Lemma~\ref{lemma:rmcommuting}, that is, $P\subseteq W^\eta$. This shows that $W^\eta=\Pc(W^\eta)$ is parabolic, and the largest spherical parabolic subgroup of $W$ normalised by $w$.
\end{proof}

\begin{definition}
For $w\in W$ with $\Pc(w)=W$, we let\index[s]{Pwmax@$P_w^{\max}$ (largest spherical parabolic normalised by $w$)} $P_w^{\max}$ denote the largest spherical parabolic subgroup of $W$ normalised by $w$, that is, $P_w^{\max}=W^{\eta_w}$. 
\end{definition}

\begin{remark}\label{remark:Pwmaxwparabolic}
Let $w\in W$ with $\Pc(w)=W$. Then $P_w^{\max}$ is a $w$-parabolic subgroup by Lemma~\ref{lemma:Pmaxwparabolic} (and hence the largest one).
\end{remark}

\begin{definition}\label{definition:standardelementIND}
We call an element $w\in W$ with $\Pc(w)=W$ {\bf standard}\index{Standard! element of $W$} if $w$ is cyclically reduced and $P_w^{\max}$ is standard.
\end{definition}

\begin{prop}\label{prop:SetainS}
Let $w\in W$ with $\Pc(w)=W$ be cyclically reduced, and set $\eta:=\eta_w\in\partial X$. Then there exists $a_w\in W$ of minimal length in $W^\eta a_w$ such that $v:=a_w\inv wa_w$ is standard. Moreover, for any such $a_w$, the following assertions hold:
\begin{enumerate}
\item
$v\in\Cyc(w)$ and $\pi_{\Sigma^\eta}(a_wC_0)=C_0^\eta$ and $S^{\eta_v}=a_w\inv S^{\eta}a_w$;
\item
$S^{\eta_v}\subseteq S$;
\item
the restriction of $\pi_{\Sigma^\eta}$ to the residue $R_{S^{\eta_v}}(a_wC_0)$ is a cellular isomorphism onto $\Sigma^\eta$.
\end{enumerate}
\end{prop}
\begin{proof}
Let $x\in\Min(w)$ be $w$-essential and such that $P_w^{\max}=\Stab_W(R_x)$ (see Remark~\ref{remark:Pwmaxwparabolic}). Let $a_w\in W$ be such that $a_wC_0=\proj_{R_x}(C_0)$ (so that $a_w$ is of minimal length in $W^\eta a_w$). Since $C_0\in\CMin(w)$ by assumption, $a_wC_0\in\CMin(w)$ by Lemma~\ref{lemma:prop34}(2), that is, $v:=a_w\inv wa_w$ is cyclically reduced. Moreover, $P_v^{\max}=a_w\inv P_w^{\max}a_w=\Stab_W(a_w\inv R_x)$ is standard, yielding the first assertion.

Assume now that $a_w\in W$ is of minimal length in $W^\eta a_w$ and such that $v:=a_w\inv wa_w$ is standard, and let us show (1)--(3). By assumption, $W^\eta=a_w W_Ia_w\inv$ for some $I\subseteq S$, the chamber $C:=a_wC_0$ belongs to $\CMin(w)$, and $C=\proj_R(C_0)$ where $R:=a_wR_I=R_I(C)$ has stabiliser $W^\eta$. 

Since $C$ and $C_0$ are not separated by any wall of $R$ (that is, by any wall in $\WW^\eta$), we have $C_0^\eta=C^\eta$. 
Propositions~\ref{prop:corresp_Minsets2} and \ref{prop:corresp_galleries_CMin2} then imply that $C_0$ and $C$ are connected by a gallery in $\CMin(w)$, and hence that $v=\pi_w(C)\in\Cyc(w)$ by Lemma~\ref{lemma:cyclicshift_geometric}. Moreover, $S^{\eta_v}=S^{a_w\inv\eta}=a_w\inv S^{\eta}a_w$ by (\ref{eqn:autominvariance}), yielding (1).

Since $\Stab_W(R)=W^\eta$ and $\Stab_W(R_I)=W^{\eta_v}$, Lemma~\ref{lemma:corresp_wresidues}(3) implies that $\pi_{\Sigma^\eta}|_{R}\co R\to\Sigma^\eta$ and $\pi_{\Sigma^{\eta_v}}|_{R_I}\co R_I\to\Sigma^{\eta_v}$ are cellular isomorphisms. It then remains for the proof of (2) and (3) to see that $I=S^{\eta_v}$. But by definition of $S^{\eta_v}$, the walls $m$ with $r_m\in S^{\eta_v}$ are precisely the walls of $R_I$ that delimit the chamber $C_0$ in $R_I$, yielding the claim.
\end{proof}

\begin{remark}\label{remark:type_preserving_isomorphism}
Let $w\in W$ with $\Pc(w)=W$ be cyclically reduced. Let $a_w\in W$ and $v:=a_w\inv wa_w$ be as in Proposition~\ref{prop:SetainS}. In the notations of \S\ref{subsection:PTCplx}, we have the following commutative diagram, where all maps are cellular isomorphisms:
\begin{equation*}
\begin{CD}
R_{S^{\eta_v}}(a_wC_0) @>\pi_{\Sigma^{\eta_w}}>> \Sigma^{\eta_w}\\
@V{a_w\inv\cdot}VV @VV{\phi_{a_w}}V\\
R_{S^{\eta_v}}@>\pi_{\Sigma^{\eta_v}}>> \Sigma^{\eta_v}.
\end{CD}
\end{equation*}
Note, moreover, that identifying $R_{S^{\eta_v}}$ with the Coxeter complex $\Sigma(W_{S^{\eta_v}},S^{\eta_v})$, the cellular isomorphism $$\pi_{\Sigma^{\eta_v}}|_{R_{S^{\eta_v}}}\co R_{S^{\eta_v}}\stackrel{\cong}{\longrightarrow} \Sigma^{\eta_v}$$ is type-preserving.
\end{remark}

Proposition~\ref{prop:SetainS} allows to view the subgroup $\pi_{\eta}(W_\eta)$ of $\Aut(\Sigma^\eta)$ as a group of automorphisms of a spherical residue of $\Sigma$, as follows.

\begin{lemma}\label{lemma:pietaWetaasgroupautomresidueSigma}
Let $w\in W$ with $\Pc(w)=W$ be standard. Write $P_w^{\max}=W_I$ for some $I\subseteq S$, and $w=w_In_I$ with $w_I\in W_I$ and $n_I\in N_I$. Let $\delta\in\Aut(W_I,I)=\Aut(R_I,C_0)$ be defined by $\delta(s)=n_Isn_I\inv$ for all $s\in I$. Identifying $\Sigma^{\eta_w}$ with $R_I$ as in Proposition~\ref{prop:SetainS}(3), we have $\delta_w=\delta$ and $w_{\eta_w}=w_I\delta$. In particular, $I_w$ is the smallest $\delta$-invariant subset of $I$ containing $\supp(w_I)$.
\end{lemma}
\begin{proof}
Since $\pi_{\eta}(w)=w_I\pi_{\eta}(n_I)$ and $\pi_{\eta}(n_I)C_0^{\eta}=\pi_{\Sigma^\eta}(n_IC_0)=C_0^\eta$ (where $\eta:=\eta_w$), we have $\delta_w=\pi_{\eta}(n_I)$, so that the lemma follows from the definition of the identifications $R_I\approx \Sigma^\eta$ and $\Aut(W_I,I)=\Aut(R_I,C_0)$.
\end{proof}

The identification of $W^\eta$ allows to give a more precise version of the second statement of Theorem~\ref{thmintro:graphisomorphism} in this setting.

\begin{prop}\label{prop:explicitsequenceopIND}
Let $w\in W$ with $\Pc(w)=W$ be standard. Let $\varphi_w$ be as in Theorem~\ref{theorem=R1R2implythmA}. Let $I_w=J_0\stackrel{K_1}{\too}J_1\stackrel{K_2}{\too}\dots\stackrel{K_m}{\too}J_m$ be a path in $\KKK_{\delta_w}^0(I_w)$. Then $$w=w_0\stackrel{K_1}{\too}w_1\stackrel{K_2}{\too}\dots\stackrel{K_m}{\too}w_m,$$ where $w_i:=\op_{K_i}(w_{i-1})$ and $\varphi_{w}\inv([J_i])=\Cyc(w_i)$ for each $i=1,\dots,m$.
\end{prop}
\begin{proof}
Note that $S^\eta\subseteq S$ by Proposition~\ref{prop:SetainS}(2) (applied with $a_w=1$). Define the chambers $C_1,\dots,C_m\in R_{S^\eta}$ recursively by $C_i:=\op_{R_{K_i}(C_{i-1})}(C_{i-1})$ ($i=1,\dots,m$) and set $w_i:=\pi_w(C_i)$ for each $i$. Since $\pi_{\Sigma^\eta}$ restricts to a type-preserving cellular isomorphism $R_{S^\eta}\to\Sigma^\eta$ (see Remark~\ref{remark:type_preserving_isomorphism}), we have $C_i^\eta=\op_{R_{K_i}(C^\eta_{i-1})}(C^\eta_{i-1})$ for each $i\geq 1$. Reasoning inductively on $i$, we then have $I_w(C_i)=J_i$ for all $i\geq 0$: if $I_w(C_{i-1})=J_{i-1}\subseteq K_i$, then $R_{K_i}(C^\eta_{i-1})$ is $w_\eta$-invariant by Lemma~\ref{lemma:I_w(C)geometric}, and hence contains the smallest $w_\eta$-invariant residue containing $C^\eta_{i-1}$, so that $I_w(C_i)=\op_{K_i}(I_w(C_{i-1}))=\op_{K_i}(J_{i-1})=J_i$ by Lemma~\ref{lemma:I_w(C)geometric}. 

In view of Lemma~\ref{lemma:Ktightlyconjuagte_geometric}(2)$\Rightarrow$(1), it then remains to show that $w$ normalises $\Stab_W(R_{K_i}(C_{i-1}))$ for each $i\geq 1$. But since $R_{K_i}(C_{i-1}^\eta)$ is $w_\eta$-invariant as we saw above, $w_\eta$ (and thus $w$) normalises $\Stab_{W^\eta}(R_{K_i}(C^\eta_{i-1}))=\Stab_W(R_{K_i}(C_{i-1}))$, as desired.
\end{proof}


\subsection{The core splitting of an element}\label{subsection:TCSOAEInd}
In this section, we investigate a more precise version of the $P_w^{\max}$-splitting of an element $w$ (in the sense of \S\ref{section:TPSOAE}), which will allow us to compute $\ZZZ_W(w)$ and $\Xi_w$, as well the parameters $I_w$ and $\delta_w$ attached to $w$. 

We start with two geometric preparatory lemmas.
\begin{lemma}\label{lemma:minwbddnaff}
Let $w\in W$ with $\Pc(w)=W$. Let $L$ be a $w$-axis. Then $\Min(w)$ is contained in a tubular neighbourhood of $L$.
\end{lemma}
\begin{proof}
Let $W_0$ be a torsion-free finite index normal subgroup of $W$. Since for any nonzero $n\in\NN$, we have $\Min(w)\subseteq\Min(w^n)$ and $\Pc(w^n)=\Pc(w)$ (see Lemma~\ref{lemma:PcvnW}), there is no loss of generality in assuming that $w\in W_0$. In particular, for any wall $m\in\WW$, either $wm=m$ or $wm\cap m\neq\varnothing$ by Lemma~\ref{lemma:CM13Lemma26}. 

Assume for a contradiction that $\Min(w)$ is not contained in a tubular neighbourhood of $L$. Since $X$ is proper, \cite[Theorem~II.2.13]{BHCAT0} implies that $\Min(w)$ contains an isometrically embedded copy $Z=L\times \RR_+$ of $\RR\times\RR_+$ (with the Euclidean metric), where $Z$ is the convex hull $\Conv(L\cup r)\subseteq\Min(w)$ of $L$ and of some geodesic ray $r$ based at a point of $L$. 

Note that any wall intersecting $Z$ nontrivially either contains a $w$-axis or is $w$-essential.
Moreover, any wall $m$ intersecting $Z$ nontrivially and not containing $Z$ (in particular, any $w$-essential wall) must intersect $Z$ in either a Euclidean line (which is then a $w$-axis), or a Euclidean half-line (based at a point of $L$): this follows from the fact that $m$ entirely contains every geodesic segment it intersects in at least two points. 

Let $m$ be a $w$-essential wall. Since $w\in W_0$, the walls in $\WW_A:=w^{\ZZ}m$ are pairwise parallel, and hence intersect $Z$ in pairwise parallel (in the Euclidean sense) half-lines based at points of $L$. We claim that there is a second infinite family $\WW_B$ of pairwise parallel walls intersecting each of the walls in $\WW_A$; this will yield the desired contradiction by Lemma~\ref{lemma:grid_lemma}.

If there exists a $w$-essential wall $m'$ intersecting nontrivially some wall in $\WW_A$, then for the same reasons as above, the collection $\WW_B:=w^{\ZZ}m'$ of walls will satisfy the desired properties (as the intersection of the walls in $\WW_A$ and $\WW_B$ with $Z$ will form the desired (Euclidean) grid). Assume now this is not the case. Then the Euclidean ``stripes'' in $Z$ delimited by the walls in $\WW_A$ must be further subdivided into bounded subsets by the traces on $Z$ of an infinite family $\WW'_B$ of walls $m'$ not containing $Z$ and containing each a $w$-axis (this is because the spherical simplices of $\Sigma$ are compact, and hence the set of points of $X$ with a given support must be bounded). By \cite[Lemma~3]{NR03}, $\WW'_B$ contains two parallel walls $m_1,m_2$. Since $m_1,m_2$ contain a $w$-axis and since $w\in W_0$, we have $wm_1=m_1$ and $wm_2=m_2$. In other words, $w$ commutes with the reflections $r_{m_1}$ and $r_{m_2}$. In particular, the walls in $\WW_B:=(r_{m_1}r_{m_2})^\ZZ m_1$ are pairwise parallel and all contain a $w$-axis, and hence intersect all $w$-essential walls (e.g. the walls in $\WW_A$). This completes the proof of the claim, and hence of the lemma.
\end{proof}

\begin{lemma}\label{lemma:circumcenter}
Let $w\in W$ with $\Pc(w)=W$. Then there is a $w$-axis $L$ with $\Fix_W(L)=P_w^{\max}$ that is also a $v$-axis for all $v\in\ZZZ_W(w)$. 
\end{lemma}
\begin{proof}
Since $\ZZZ_W(w)$ normalises $P_w^{\max}=W^{\eta_w}$, the set $Z:=\bigcap_{m\in \WW^{\eta_w}}m$ is a nonempty closed convex $\ZZZ_W(w)$-stable subspace of $X$. In particular, $Z\cap\Min(w)$ is nonempty and is a union of $w$-axes.

By \cite[Theorems~II.6.8(4) and II.2.14]{BHCAT0}, there exists a (convex) CAT(0)-subspace $Y\subseteq Z\cap\Min(w)$ such that $Z\cap\Min(w)$ is isometric to $Y\times\RR$. Let $v\in\ZZZ_W(w)$. Then $v$ stabilises $Z\cap\Min(w)=Y\times\RR$ and its restriction to $Y\times \RR$ is of the form $(v^Y,v^{\RR})$ with $v^Y$ an isometry of $Y$ and $v^{\RR}$ a translation of $\RR$ (see \cite[Theorem~II.6.8(5)]{BHCAT0}). As $Y$ is bounded by Lemma~\ref{lemma:minwbddnaff}, the Bruhat--Tits fixed point theorem (see e.g. \cite[Theorems~11.23 and 11.26]{BrownAbr}) implies that $v^Y$ fixes the unique circumcenter $y$ of $Y$; moreover, $y\in Y$ by  \cite[Theorem~11.27]{BrownAbr}. In other words, the $w$-axis $L\subseteq Z\cap\Min(w)$ through $y$ is also a $v$-axis. Finally, since $\Fix_W(L)\subseteq P_w^{\max}$ (by maximality of $P_w^{\max}$) and $P_w^{\max}=\Fix_W(Z)\subseteq\Fix_W(L)$, we have $\Fix_W(L)=P_w^{\max}$, as desired.
\end{proof}

Recall from Definition~\ref{definition:Preduced} the definition of $P$-reduced elements, for $P$ a spherical parabolic subgroup of $W$. Here are a few useful observations.
\begin{lemma}\label{lemma:Preducedproperties}
Let $w\in W$ with $\Pc(w)=W$, and set $P:=P_w^{\max}$. Then:
\begin{enumerate}
\item
there exists a $w$-essential point $x\in\Min(w)$ such that $\Stab_W(R_x)=P$. For any such $x$, the $w$-axis through $x$ is also an $aw$-axis for any $a\in P$.
\item
if $u\in W$ and $w$ share an axis, then $\Pc(u)=\Pc(w)$ and $P_u^{\max}=P$.
\item
if $u\in W$ is such that either $u=aw$ for some $a\in P$, or $u^n=w$ or $u=w^n$ for some $n\geq 1$, then $\Pc(u)=W$ and $P_u^{\max}=P$.
\item
there is a unique $a\in P$ such that $aw$ is $P$-reduced.
\item
if $w$ is $P$-reduced, then $w^n$ is $P$-reduced for all $n\geq 1$.
\item
if $w$ is $P$-reduced, then every cyclically reduced conjugate of $w$ is straight.
\item
if $w$ is $P$-reduced, there exists an element $u\in\Cyc_{\min}(w)$ that is $P_u^{\max}$-reduced and such that $P_u^{\max}$ is standard.
\end{enumerate}
\end{lemma}
\begin{proof}
(1) This follows from Remark~\ref{remark:Pwmaxwparabolic} and Lemma~\ref{lemma:waxisawaxis}.

(2) By Lemmas~\ref{lemma:Cor212} and \ref{lemma:CM13Lemma26}, there exist $w$-essential walls $m,m'$ such that $\Pc(r_m,r_{m'})=W$. Since $m,m'$ are also $u$-essential, \cite[Lemma~2.7]{openKM} implies that $r_m,r_{m'}\in\Pc(u)$, and hence $\Pc(u)=W$. Moreover, since $\eta_u\in\{\eta_w,\eta_{w\inv}\}$ and $P_{w}^{\max}=P_{w\inv}^{\max}$, we have $P_{u}^{\max}=W^{\eta_u}=W^{\eta_{w}}=P$.

(3) This follows from (1) and (2).

(4) This follows from Proposition~\ref{prop:Psplitting}.

(5) If $w$ is $P$-reduced, then $C_0$ and $wC_0$ are on the same side of any wall of $P$, and hence so are $w^mC_0$ and $w^{m+1}C_0$ for all $m\in\NN$, yielding (5).

(6) Proposition~\ref{prop:Psplitting} implies that $w$ has a straight conjugate, and hence every cyclically reduced conjugate of $w$ is straight by \cite[Lemmas~4.1 and 4.2]{straight}.

(7) Let $w'\in\Cyc_{\min}(w)$, and let $x\in\Min(w')$ be $w'$-essential and such that $\Stab_W(R_x)=P_{w'}^{\max}$ (see (1)). Let $v\in W$ with $vC_0=\proj_{R_x}(C_0)$, so that $R_x=vR_I$ for some $I\subseteq S$, and set $u:=v\inv w'v$. Then $u$ is cyclically reduced by Lemma~\ref{lemma:prop34}(2), and hence $u\in\Cyc(w')=\Cyc_{\min}(w)$ by Proposition~\ref{prop:wdecreasingfromCtoD}. Moreover, $P_u^{\max}=W_I$ is standard. Finally, writing $u=u_In_I$ with $u_I\in W_I$ and $n_I\in N_I$, \cite[Lemma~4.1]{straight} implies that $u_I=1$ as $u$ is straight by (6), and hence $u=n_I$ is $W_I$-reduced, as desired. 
\end{proof}

We next prove a uniqueness statement for ``roots'' of $w$.

\begin{lemma}\label{lemma:uniqueness_atomic}
Let $w\in W$ with $\Pc(w)=W$, and let $u_1,u_2\in W$ be such that $u_1^m=u_2^n=w$ for some $m,n>0$. Then $P_{u_1}^{\max}=P_{u_2}^{\max}=P_w^{\max}=:P$ and there exist a $P$-reduced element $u\in W$, as well as $r_1,r_2\in\NN$ and $a_1,a_2\in P$ such that $u_1=a_1u^{r_1}$ and $u_2=a_2u^{r_2}$. 
\end{lemma}
\begin{proof}
Let $L$ be the $w$-axis provided by Lemma~\ref{lemma:circumcenter}. Thus, $\Fix_W(L)=P$ and $L$ is also a $u_1$-axis and a $u_2$-axis. In particular, $P_{u_1}^{\max}=P_{u_2}^{\max}=P$ by Lemma~\ref{lemma:Preducedproperties}(2).

Let $x\in L$ be $w$-essential (so that $\Fix_W(x)=P$), and choose a parametrisation $r\co\RR\to L$ of $L$ with $r(0)=x$ and $r(mn/d)=wx$, where $d:=\mathrm{gcd}(m,n)$. Thus, $u_1x=r(n/d)$ and $u_2x=r(m/d)$. Let $a,b\in\ZZ$ be such that $am+bn=d$. Then $u':=u_1^bu_2^a$ commutes with $w$ and satisfies $u'x=r(1)$; moreover, $L$ is a $u'$-axis and $P_{u'}^{\max}=P$ by Lemma~\ref{lemma:Preducedproperties}(2).

Let $a\in P=\Fix_W(x)$ be such that $u:=au'$ is $P$-reduced (see Lemma~\ref{lemma:Preducedproperties}(4)). Thus, $L$ is a $u$-axis and $P_u^{\max}=P$ by Lemma~\ref{lemma:Preducedproperties}(1,2). Since $u^{n/d}$ and $u_1$ both map $x\in L$ to $r(n/d)\in L$, there exists $a_1\in P$ such that $u_1=a_1u^{n/d}$. Similarly, there exists $a_2\in P$ such that $u_2=a_2u^{m/d}$. We may thus choose $r_1:=n/d$ and $r_2:=m/d$, as desired.
\end{proof}

\begin{definition}
Call $w\in W$ with $\Pc(w)=W$ {\bf atomic}\index{Atomic} if $w$ is $P_w^{\max}$-reduced and if there is no $P_w^{\max}$-reduced element $u\in W$ such that $w=u^n$ for some $n\geq 2$.

We also call an element $w\in W$ {\bf divisible} if there exist $u\in W$ and $n\geq 2$ such that $w=u^n$, and {\bf indivisible}\index{Indivisible} otherwise.
\end{definition}

\begin{lemma}\label{lemma:atomicsimplerdef}
Let $w\in W$ with $\Pc(w)=W$. Then $w$ is atomic if and only if $w$ is $P_w^{\max}$-reduced and indivisible.
\end{lemma}
\begin{proof}
Set $P:=P_w^{\max}$. Suppose that $w$ is $P$-reduced, and let $u\in W$ and $n\geq 1$ be such that $w=u^n$. We have to show that $w=u_1^n$ for some $P$-reduced element $u_1\in W$. Since $P_u^{\max}=P$ by Lemma~\ref{lemma:Preducedproperties}(3), and since $P$ is a $u$-parabolic subgroup by Lemma~\ref{lemma:Preducedproperties}(1), we can write $u=au_1$ with $a\in P$ and $u_1\in W$ a $P$-reduced element normalising $P$ by Proposition~\ref{prop:Psplitting}. Hence $w=(au_1)^n=a'u_1^n$ for some $a'\in P$. Since $u_1^n$ is $P$-reduced by Lemma~\ref{lemma:Preducedproperties}(5), $a'=1$, and hence $w=u_1^n$, as desired.
\end{proof}

We can now state the main result of this subsection.
\begin{theorem}\label{thm:existence_core}
Let $(W,S)$ be a Coxeter system of irreducible indefinite type, and let $w\in W$ with $\Pc(w)=W$.
Then there is a unique atomic element $u\in W$, and unique $n\in\NN$ and $a\in P_w^{\max}$ such that $w=au^n$. 
\end{theorem}
\begin{proof}
Let $a\in P:=P_w^{\max}$ be such that $a\inv w$ is $P$-reduced (see Lemma~\ref{lemma:Preducedproperties}(4)). Let $n\in\NN$ be maximal such that there exists a $P$-reduced element $u\in W$ with $a\inv w=u^n$. Note that $P^{\max}_u=P$ by Lemma~\ref{lemma:Preducedproperties}(3). Hence $u$ is atomic: otherwise, there is some $P$-reduced $v\in W$ such that $u=v^m$ for some $m\geq 2$, so that $a\inv w=v^{mn}$, contradicting the minimality of $n$. This proves the existence of $u,n,a$.

For the uniqueness statement, suppose that $w=bv^m$ for some atomic $v\in W$, some $m\in\NN$ and $b\in P$ (so that $P_v=P$ and $v^m$ is $P$-reduced by Lemma~\ref{lemma:Preducedproperties}(3,5)). Then $b=a$ by Lemma~\ref{lemma:Preducedproperties}(4) as $a\inv bv^m=a\inv w$ is $P$-reduced. Thus, $u^n=a\inv w=v^m$. Since $u$ and $v$ are $P$-reduced, Lemma~\ref{lemma:uniqueness_atomic} (applied to $w:=a\inv w$) yields some $P$-reduced $\overline{u}\in W$ such that $u=\overline{u}^{r_1}$ and $v=\overline{u}^{r_2}$ for some $r_1,r_2\in\NN$. Since $u$ and $v$ are atomic, we have $r_1=r_2=1$, and hence $u=v$ and $n=m$, as desired.
\end{proof}

\begin{definition}
Define the {\bf core}\index{Core}\index[s]{wc@$w_c$ (core of $w$)} $w_c$ of an element $w\in W$ with $\Pc(w)=W$ to be the unique atomic element $u\in W$ provided by Theorem~\ref{thm:existence_core}, that is, such that $w=au^n$ for some (uniquely determined) $n\geq 1$ and $a\in P_w^{\max}$. We further call the decomposition $w=aw_c^n$ the {\bf core splitting}\index{Core splitting} of $w$.
\end{definition}

\begin{remark}\label{remark:coresplittingPsplitting}
Let $w\in W$ with $\Pc(w)=W$, with core splitting $w=aw_c^n$ with $n\geq 1$ and $a\in P=P_w^{\max}$. Recall from Remark~\ref{remark:Pwmaxwparabolic} that $P$ is a $w$-parabolic subgroup in the sense of Definition~\ref{definition:wparabolicsubgroup}. Hence the core splitting of $w$ is a refinement of its $P$-splitting (see \S\ref{section:TPSOAE}), as $w_{\tor}(P)=a$ and $w_{\infty}(P)=w_c^n$ by Lemma~\ref{lemma:Preducedproperties}(5).
\end{remark}

We now collect a few properties of cores and core splittings, for future reference.

\begin{lemma}\label{lemma:wwmsamecore}
Let $w\in W$ with $\Pc(w)=W$. Then $w$ and $w^m$ have the same core for any $m\geq 1$.
\end{lemma}
\begin{proof}
Note that $\Pc(w^n)=W$ and $P_{w^n}^{\max}=P_w^{\max}=:P$ by Lemma~\ref{lemma:Preducedproperties}(2). Let $w=aw_c^n$ be the core splitting of $w$ ($n\geq 1$, $a\in P$). Then $w^m=bw_c^{mn}$ for some $b\in P$. Since $w_c$ is atomic, this is the core splitting of $w^m$, yielding the lemma.
\end{proof}

\begin{lemma}\label{lemma:awcndeltaeta}
Let $w\in W$ with $\Pc(w)=W$, and set $\eta:=\eta_w$. Let $w=aw_c^n$ be the core splitting of $w$, where $n\geq 1$ and $a\in W^\eta$. Then $\delta_w=\delta_{w_c}^n$ and $w_\eta=a\delta_{w_c}^n$, and $\delta_{w_c}=\pi_{\eta}(w_c)$ is the diagram automorphism $W^\eta\to W^\eta:b\mapsto w_cbw_c\inv$ of $W^\eta$.
\end{lemma}
\begin{proof}
Since $w_c$ is $P$-reduced, we have $\pi_{\Sigma^\eta}(w_cC_0)=C_0^\eta$, so that $\pi_{\eta}(w_c)=\delta_{w_c}$, yielding the last assertion (the fact that $\delta_{w_c}$ is the conjugation by $w_c$ on $W^\eta$ follows from the definition of the semi-direct product $\Aut(\Sigma^\eta)=W^\eta\rtimes\Aut(W^\eta,S^\eta)$). Similarly, since $w_c^n$ is $P$-reduced by Lemma~\ref{lemma:Preducedproperties}(5), we have $\pi_{\eta}(a\inv w)C_0^\eta=C_0^\eta$, so that $\delta_w=\pi_{\eta}(a\inv w)=\pi_{\eta}(w_c^n)=\delta_{w_c}^n$ and $w_\eta=a\delta_{w_c}^n$, as desired.
\end{proof}

\begin{remark}\label{remark:conjCoxA3IND}
Let $w\in W$ with $\Pc(w)=W$, and assume that $w$ is straight. Let $w=aw_c^n$ be the core splitting of $w$, where $n\geq 1$ and $a\in P:=P_w^{\max}$. Since $w_{\tor}(P)=a$ and $w_{\infty}(P)=w_c^n$ by Remark~\ref{remark:coresplittingPsplitting}, Corollary~\ref{corollary:Psplittingstraight} implies that $a=1$. Hence $I_w=\varnothing$ by Lemma~\ref{lemma:awcndeltaeta}, so that $\OOO_w^{\min}=\Cyc(w)$ by Theorem~\ref{thmintro:graphisomorphism}. We thus recover \cite[Theorem~A(3)]{conjCox} for $W$ of irreducible indefinite type.
\end{remark}

\begin{lemma}\label{lemma:computingwc}
Let $w\in W$ with $\Pc(w)=W$ be cyclically reduced, and let $w=aw_c^n$ be its core splitting, with $n\geq 1$ and $a\in P:=P_w^{\max}$. Let $v\in W$ be of minimal length in $Pv$ and such that $v\inv P v$ is standard. Then $u:=v\inv wv$ has core splitting $u=bu_c^n$ with $b:=v\inv av$ and $u_c:=v\inv w_cv$.

Moreover, $v$ can be chosen so that, in addition, $u\in\Cyc(w)$ and $u_c$ is straight, and hence $\ell(u)=\ell(b)+n\ell(u_c)$. In particular, if $w$ is standard, then $n$ divides $\ell(a\inv w)$.
\end{lemma}
\begin{proof}
Write $v\inv P v=W_I$ for some $I\subseteq S$, so that $vC_0=\proj_{vR_I}(C_0)$. Since $P_u^{\max}=v\inv Pv$, Lemma~\ref{lemma:PvinvPvreduced} implies that $v\inv w_cv$ is $P_u^{\max}$-reduced. In particular, it is atomic by Lemma~\ref{lemma:atomicsimplerdef}, as $w_c$ is atomic. As $v\inv av\in P_u^{\max}$, we conclude that $u_c=v\inv w_cv$, yielding the first claim.

Since $P_{w_c}^{\max}=P$ by Lemma~\ref{lemma:Preducedproperties}(3), we find by Lemma~\ref{lemma:Preducedproperties}(1) some $w_c$-essential point $x\in\Min(w_c)$ such that $\Stab_W(R_x)=P$. We now choose $v\in W$ such that $vC_0=\proj_{R_x}(C_0)$, so that $v$ is of minimal length in $Pv$ and $P_u^{\max}=v\inv Pv=\Stab_W(R_{v\inv x})$ is standard (where $u=v\inv wv$). Then $vC_0\in\CMin(w)$ by Lemma~\ref{lemma:prop34}(2), and hence $u\in\Cyc(w)$ by Proposition~\ref{prop:wdecreasingfromCtoD}. Moreover, since $u_c$ is $P_u^{\max}$-reduced and has an axis through $v\inv x\in C_0$, it is straight by \cite[Lemma~4.3]{straight}. Since  $P_u^{\max}=W_I$ for some $I\subseteq S$ and $u_c^n$ is the unique element of minimal length in $W_Iw$, the equality $\ell(u)=\ell(b)+n\ell(u_c)$ follows. Finally, if $w$ is standard, then $w_c^n$ is cyclically reduced by Lemma~\ref{lemma:criterioncyclicallyreducedPsplitting}(2), and hence $n$ divides $\ell(a\inv w)=\ell(w_c^n)=\ell(u_c^n)=n\ell(u_c)$.
\end{proof}

Here is a criterion to determine whether an element is atomic. Call an element $w\in W$ {\bf weakly indivisible}\index{Weakly indivisible} if there is no decomposition $w=u^n$ with $\ell(w)=n\ell(u)$ for some $u\in W$ and $n\geq 2$. 
\begin{lemma}\label{lemma:computingwc2}
Let $w\in W$ with $\Pc(w)=W$ be $P_w^{\max}$-reduced. Then the following assertions are equivalent:
\begin{enumerate}
\item
$w$ is atomic.
\item
Every $u\in\Cyc_{\min}(w)$ that is $P_u^{\max}$-reduced and such that $P_u^{\max}$ is standard is weakly indivisible.
\end{enumerate}
\end{lemma}
\begin{proof}
The implication (1)$\Rightarrow$(2) follows from Lemma~\ref{lemma:atomicsimplerdef}. To show that (2)$\Rightarrow$(1), let $x\in\Cyc_{\min}(w)$ be $P_{x}^{\max}$-reduced, as provided by Lemma~\ref{lemma:Preducedproperties}(7), and suppose that $w$ is not atomic. Then $w$ is divisible, and hence so is $x$. In other words, $x=x_c^n$ for some $n\geq 2$. Lemma~\ref{lemma:computingwc} (applied with $w:=x$) then yields some $v\in W$ such that $u:=v\inv xv\in\Cyc(x)=\Cyc_{\min}(w)$, such that $u$ is $P_u^{\max}$-reduced and $P_u^{\max}$ is standard, and such that $u=u_c^n$ with $\ell(u)=n\ell(u_c)$, that is, such that $u$ is not weakly indivisible.
\end{proof}

Finally, note that Lemma~\ref{lemma:criterioncyclicallyreducedPsplitting} yields the following criterion to check whether $w$ is cyclically reduced.

\begin{lemma}\label{lemma:criterionwcyclredIND}
Let $w\in W$ with $\Pc(w)=W$, with core splitting $w=aw_c^n$. Assume that $P:=P_w^{\max}$ is standard, and let $\delta_{w_c}\co P\to P:x\mapsto w_cxw_c\inv$. If $w_c$ is cyclically reduced and $a\delta_{w_c}^n$ is cyclically reduced in $P$, then $w$ is cyclically reduced.
\end{lemma}
\begin{proof}
If $w_c$ is cyclically reduced, then it is straight by Lemma~\ref{lemma:Preducedproperties}(6), and hence $w_c^n$ is straight (in particular, cyclically reduced) as well. The lemma thus follows from Lemma~\ref{lemma:criterioncyclicallyreducedPsplitting}(2). 
\end{proof}


\subsection{The centraliser of an element}\label{subsection:TCOAEInd}
We next describe the centraliser of $w$, by establishing that every $v\in\ZZZ_W(w)$ has the same core as $w$.

\begin{prop}\label{prop:centraliser_same_core}
Let $w\in W$ with $\Pc(w)=W$, and set $\eta:=\eta_w$. Then:
\begin{enumerate}
\item
For any $v\in\ZZZ_W(w)$, there exist unique $n\in\ZZ$ and $a\in P_w^{\max}$ such that $v=aw_c^n$.
\item
$\ZZZ_{W}(w)\cap W^\eta=\ZZZ_{W^\eta}(w_\eta)$.
\item
The map $\ZZZ_W(w)\to\ZZ:aw_c^n\mapsto n$ $(n\in\ZZ$, $a\in P_w^{\max})$ is a group morphism, with kernel $\ZZZ_{W^{\eta}}(w_{\eta})$.
\item The kernel of the restriction to $\ZZZ_W(w)$ of $\pi_{\eta}$ coincides with $\langle w_c^m\rangle$, where $m\in\NN$ is the order of the automorphism $\delta_{w_c}$. 
\end{enumerate}
\end{prop}
\begin{proof}
(1) As $P_w^{\max}=P_v^{\max}=P_{v\inv}^{\max}$ (see Lemmas~\ref{lemma:circumcenter} and \ref{lemma:Preducedproperties}(2)), this amounts to show that $w$ and $v^{\varepsilon}$ have the same core for some $\varepsilon\in\{\pm 1\}$. By \cite[Corollary~6.3.10]{Kra09}, there exist some $\varepsilon\in\{\pm 1\}$ and some nonzero $r_1,r_2\in\NN$ such that $w^{r_1}=v^{\varepsilon r_2}$. As $w$ and $w^{r_1}$ (resp. $v^{\varepsilon}$ and $v^{\varepsilon r_2}$) have the same core by Lemma~\ref{lemma:wwmsamecore}, (1) follows.

(2) The inclusion $\subseteq$ is clear. Conversely, suppose that $b\in W^\eta$ commutes with $w_\eta$. Let $w=aw_c^n$ ($a\in W^\eta$, $n\geq 1$) be the core splitting of $w$. Then Lemma~\ref{lemma:awcndeltaeta} implies that 
$$bwb\inv =ba\delta^n_{w_c}(b\inv)\cdot w_c^n=bw_\eta b\inv \cdot \delta^{-n}_{w_c}w_c^n= w_\eta \delta^{-n}_{w_c}w_c^n=aw_c^n=w,$$
so that $b\in\ZZZ_W(w)\cap W^\eta$, as desired.

(3) Since $w_c$ normalises $P_w^{\max}=W^\eta$, the assignment $aw_c^n\mapsto n$ indeed defines a group morphism. Moreover, its kernel is $\ZZZ_{W}(w)\cap W^\eta$, so that (3) follows from (2).

(4) Let $v\in \ZZZ_W(w)$. By (1), we can write $v=aw_c^n$ with $n\in\ZZ$ and $a\in W^\eta$. Then $v\in\ker\pi_\eta$ if and only if $v_\eta=a\delta_{w_c}^n=1$ if and only if $a=1$ and $m$ divides $|n|$.
\end{proof}

Note that, as we will see in Example~\ref{example:nw=2} below, the map $\ZZZ_W(w)\to\ZZ$ provided by Proposition~\ref{prop:centraliser_same_core}(3) need not be surjective.

\begin{definition}\label{definition:nw}
Let $w\in W$ with $\Pc(w)=W$. We call the unique natural number $n=n_w\in\NN$ such that the map $\ZZZ_W(w)\to\ZZ$ provided by Proposition~\ref{prop:centraliser_same_core}(3) has image $n\ZZ$ the {\bf centraliser degree}\index{Centraliser degree} of $w$. 
Note that, in general, $n_w$ might be bigger than $1$ (see Example~\ref{example:nw=2}).
\end{definition}

We conclude this subsection by indicating how the centraliser degree of $w$ can be computed.

\begin{lemma}\label{lemma:nw}
Let $w\in W$ with $\Pc(w)=W$, and set $\eta:=\eta_w$. Let $n_w$ be the centraliser degree of $w$. Let $w=aw_c^m$ be the core splitting of $w$, where $m\geq 1$ and $a\in P_w^{\max}$. Set $\delta:=\delta_{w_c}$. Then the following assertions hold:
\begin{enumerate}
\item
Let $n\in\ZZ$. Then $n\in n_w\ZZ$ if and only if $\delta^n(a)$ is $\delta^m$-conjugate to $a$ in $W^\eta$.
\item
$n_w\geq 1$ and $n_w$ divides $m'$ for any $m'\in\NN$ such that $\delta^{m'}(a)=\delta^m(a)$.
\end{enumerate}
\end{lemma}
\begin{proof}
(1) By definition, $n\in n_w\ZZ$ if and only if there exists $b\in P_w^{\max}$ such that $bw_c^n$ commutes with $w=aw_c^m$. As $$bw_c^n w (bw_c^n)\inv= b\delta^n(a) w_c^m b\inv=b\delta^n(a)\delta^m(b)\inv\cdot w_c^m,$$ the assertion (1) follows.

(2) This follows from (1), as $\delta^{m'}(a)=\delta^m(a)=a\inv a\delta^m(a)$ is $\delta^m$-conjugate to $a$ in $W^\eta$.
\end{proof}

\begin{remark}\label{remark:computingnw}
Let $w\in W$ with $\Pc(w)=W$, and set $\eta:=\eta_w$. Assume that $w_\eta$ is cyclically reduced (which occurs for instance if $w$ is cyclically reduced by Proposition~\ref{prop:corresp_Minsets2}), so that the conjugacy class of $w_\eta$ in $W^\eta_{I_w}$ is cuspidal by Proposition~\ref{prop:He07}(1). Recall from Lemma~\ref{lemma:awcndeltaeta} that, if $w=aw_c^m$ is the core splitting of $w$, then $w_\eta=a\delta^m$, where $\delta:=\delta_{w_c}$.

To compute the centraliser degree $n_w$ of $w$, it suffices to determine when a given $n\in\NN$ belongs to $n_w\ZZ$. Note for this the equivalence of the following statements (where (1)$\Leftrightarrow$(2) is Lemma~\ref{lemma:nw}(1), and (2)$\Leftrightarrow$(3) follows from Corollary~\ref{corollary:finiteorderconjugation} and Proposition~\ref{prop:Deodhar_twisted}):
\begin{enumerate}
\item
$n\in n_w\ZZ$.
\item
$\delta^n(a)$ and $a$ are $\delta^m$-conjugate in $W^\eta$.
\item
$\delta^n(I_w)\in\KKK_{\delta^m}^0(I_w)$, and if $x\in W^\eta$ with $\delta^m(x)=x$ is such that $\delta^n(I_w)=xI_wx\inv$, then $\kappa_x\delta^n(a)$ and $a$ are $\delta^m$-conjugate in $W^\eta_{I_w}$.
\end{enumerate}
The condition that $\kappa_x\delta^n(a)$ and $a$ are $\delta^m$-conjugate in $W^\eta_{I_w}$ (with $x$ as in (3)) can be determined by reducing the problem within components of $I_w$ exactly as in the proof of Proposition~\ref{prop:useP3}, and then using \cite[Theorem~7.5(P3)]{He07}. Note that Proposition~\ref{prop:useP3} (applied to $S:=I_w$, $w:=a$, $\delta:=\delta^m$, and $\sigma:=\kappa_x\delta^n$) also provides a powerful and easy-to-check criterion to verify this condition (see also Lemma~\ref{lemma:C1C2intermediate}).
\end{remark}


\subsection{The subgroup \texorpdfstring{$\Xi_w$}{Xiw}}\label{subsection:TSXiwInd}

\begin{theorem}\label{thm:Xiwindef}
Let $w\in W$ with $\Pc(w)=W$. Let $n_w$ be the centraliser degree of $w$. Then $\Xi_w$ is generated by $\delta_{w_c}^{n_w}$. In particular, $\Xi_w$ is cyclic.
\end{theorem}
\begin{proof}
Note first that $\delta_{w_c}^{n_w}\in \Xi_w$ by Lemma~\ref{lemma:awcndeltaeta}, as by definition of $n_w$, there exists $a\in P_w^{\max}$ such that $aw_c^{n_w}\in\ZZZ_W(w)$.

Conversely, let $u\in\Xi_w$, and let $v\in\ZZZ_W(w)$ and $x\in P_w^{\max}$ with $u=x\inv v_\eta$, where $\eta:=\eta_w$. By Proposition~\ref{prop:centraliser_same_core}, there exist $n\in\ZZ$ and $a\in P_w^{\max}$ such that $v=aw_c^{nn_w}$ (and this is the core splitting of $v$ by Lemma~\ref{lemma:Preducedproperties}(3)). Moreover, $v_\eta=a\delta_{w_c}^{nn_w}$ by Lemma~\ref{lemma:awcndeltaeta}. Hence $a\delta_{w_c}^{nn_w}=xu$, that is, $a=x$ and $u=\delta_{w_c}^{nn_w}\in\langle\delta_{w_c}^{n_w}\rangle$, as desired.
\end{proof}

\begin{remark}\label{remark:SESIND}
Let $w\in W$ with $\Pc(w)=W$, and set $\eta:=\eta_w$. In view of Theorem~\ref{thm:Xiwindef}, Proposition~\ref{prop:centraliser_same_core} yields a short exact sequence
\begin{equation*}
1\to \langle w_c^m\rangle\times\ZZZ_{W^{\eta}}(w_{\eta})\to \ZZZ_W(w)\to\Xi_w\to 1,
\end{equation*}
where $m\in\NN$ is the order of $\delta_{w_c}$ (with $w_c$ the core of $w$) and where the map $\ZZZ_W(w)\to\Xi_w:aw_c^n\mapsto\delta_{w_c}^n$ ($n\in\ZZ$, $a\in P_w^{\max}$) corresponds to the restriction to $\ZZZ_W(w)$ of the composite map $W_\eta\to\pi_{\eta}(W_\eta)\to \pi_{\eta}(W_\eta)/W^\eta\cong\Xi_\eta$.
\end{remark}


\subsection{Theorem~\ref{thmintro:indefinite}, concluded}\label{subsection:TBACC}

We are now ready to conclude the proof of Theorem~\ref{thmintro:indefinite}.

\begin{theorem}
Theorem~\ref{thmintro:indefinite} holds.
\end{theorem}
\begin{proof}
Let $w\in W$ with $\Pc(w)=W$, and let us prove the statements (1)--(7) of Theorem~\ref{thmintro:indefinite}.

(1) follows from Theorem~\ref{thm:Wetafinite}. 

(2) follows from Proposition~\ref{prop:SetainS}.

(3) The first statement follows from Proposition~\ref{prop:SetainS} (applied with $a_w=1$), and the second from Proposition~\ref{prop:explicitsequenceopIND}.

(4) follows from Theorem~\ref{thm:existence_core}.

(5) follows from Lemma~\ref{lemma:awcndeltaeta}.

(6) follows from Proposition~\ref{prop:centraliser_same_core}, Lemma~\ref{lemma:nw} and Remark~\ref{remark:SESIND}.

(7) follows from Theorem~\ref{thm:Xiwindef}.
\end{proof}


\subsection{Computing \texorpdfstring{$\KKK_{\OOO_w}$}{the structural conjugation graph}: a summary}\label{subsection:CKASIND}
The purpose of this subsection is to summarise the steps to follow in order to compute, given a Coxeter group $W$ of irreducible indefinite type (given by a Coxeter matrix) and an element $w\in W$ (given by a word on the alphabet $S$), the graph $\KKK_{\OOO_w}$.

\subsubsection{A few algorithms}\label{subsubsection:AFA}
We first outline a few simple algorithms, which can be easily implemented in the algebra package CHEVIE (\cite{GHLMP96}, \cite{Mi15}) of GAP (\cite{GAP4}), and which do not aim at being efficient in terms of algorithmic complexity, but rather at providing a simple way to verify some facts in the examples given in \S\ref{subsection:ExIND} and \S\ref{subsection:ExAFF} (the relevant code used for these examples is available on the author's webpage). These examples, in turn, only serve as illustrations of the results obtained in this paper. 

In this paragraph, we relax our standing assumption that $W$ is of irreducible indefinite type, and let $(W,S)$ be an arbitrary irreducible Coxeter system.

\begin{algo}[Cyclic shift class]\label{algo:cyclicshiftclass}
One can define a function $\boldsymbol{\Cyc}$ taking as input an element $w\in W$ and giving as output the cyclic shift class $\Cyc(w)$, as a list starting with $w$: such an algorithm is for instance described in \cite[Algorithm G page 80]{GP00} (where we have to replace ``$\ell(sys)=\ell(y)$'' by ``$\ell(sys)\leq\ell(y)$'' in step G2 of that algorithm, to account for our slightly different definition of $\Cyc(w)$).

In particular, we can test whether $w$ is cyclically reduced (i.e. of minimal length in $\Cyc(w)$), and find a cyclically reduced conjugate of $w$ (i.e. an element of $\Cyc(w)$ of minimal length). 

We can also test whether $\Pc(w)=W$: first find a cyclically reduced conjugate $w'$ of $w$ (so that $\Pc(w')=W_I$ for some $I\subseteq S$ by \cite[Proposition~4.2]{CF10}), then test whether $\supp(w')=S$ (i.e. whether $I=S$).
\end{algo}

For the next algorithm, we need the following lemma.
\begin{lemma}
Let $w\in W$ with $\Pc(w)=W$. If $J,J'\subseteq S$ are spherical subsets of $S$ such that $w$ normalises $W_J$ and $W_{J'}$, then $J\cup J'$ is spherical. In particular, there is a largest spherical subset $J\subseteq S$ such that $w$ normalises $W_J$.
\end{lemma}
\begin{proof}
Without loss of generality, we may assume that $W$ is infinite. Let $N\in\NN$ be such that $w^N$ centralises $W_{J}$ and $W_{J'}$. Then $w^N$ centralises $W_{J\cup J'}$. Assume for a contradiction that $J\cup J'$ is not spherical, and let $K\subseteq J\cup J'$ be the union of all components of $J\cup J'$ that are not spherical. Then $w^N\in W_{K^\perp}$, where $K^{\perp}:=\{s\in S \ | \ st=ts \ \forall t\in K\}\subseteq S\setminus K$ (see e.g. \cite[Lemma~2.1]{openKM}). In particular, $\Pc(w^N)\neq W$, contradicting Lemma~\ref{lemma:PcvnW}.
\end{proof}

\begin{algo}[Maximal spherical standard parabolic normalised]
One can define a function $\boldsymbol{\MSN}$ taking as input an element $w\in W$ with $\Pc(w)=W$ and giving as output the largest spherical subset $I\subseteq S$ such that $w$ normalises $W_I$, as follows:
\begin{enumerate}
\item
For each $s\in S$, compute the support $I_s\subseteq S$ of $wsw\inv$.
\item
For each $s\in S$, compute the smallest subset $J_s\subseteq S$ containing $s$ such that $w$ normalises $W_{J_s}$, as follows: define $J_1:=\{s\}$ and, inductively, $J_{i+1}:=\bigcup_{t\in J_i}I_t$ for $i\geq 1$, until reaching $J_m$ with $J_m=J_{m+1}$. Then $J_s=J_m$.
\item
$\MSN(w)$ is the union of all $J_s$ ($s\in S$) such that $J_s$ is spherical. 
\end{enumerate}
\end{algo}

\begin{algo}[Conjugate with standard maximal spherical parabolic normalised]\label{algo:maxsphparabnorm}
Assume that $W$ is of indefinite type. One can define a function $\boldsymbol{\MN}$ taking as input a cyclically reduced element $w\in W$ with $\Pc(w)=W$ (as provided by Algorithm~\ref{algo:cyclicshiftclass}) and giving as output an element $v\in\Cyc(w)$ and a spherical subset $I\subseteq S$ such that $P_v^{\max}=W_I$ (see Proposition~\ref{prop:SetainS}), as follows:
\begin{enumerate}
\item 
For each $v\in\Cyc(w)$, compute $I_v:=\MSN(v)$.
\item
Find the first $v\in\Cyc(w)$ such that $I_v$ is maximal, and return $(v,I_v)$.
\end{enumerate}

Note that if $P_w^{\max}$ is standard, say $P_w^{\max}=W_I$ for some $I\subseteq S$, then the function $\MSN(w)$ returns $(w,I)$ (as the list $\Cyc(w)$ starts with $w$). In particular, given $w$ as above and $I\subseteq S$, we can check whether $P_w^{\max}=W_I$ (i.e. whether $\MN(w)=(w,I)$).
\end{algo}

For $W$ of indefinite type, we now describe an algorithm to check whether a given decomposition $w=aw_c^n$ of an element $w\in W$ with $\Pc(w)=W$ is the core decomposition of $w$. We first define the following $\mathrm{Root}$ algorithm. 

\begin{algo}[Weak indivisibility]
One can define a function $\boldsymbol{\Root}$ taking as input an element $w\in W$ and giving as output a couple $(n,x)\in\NN\times W$ with $w=x^n$ and such that $\ell(w)=n\ell(x)$ and $n$ is maximal for these properties, as follows:
\begin{enumerate}
\item
For each positive divider $d$ of $k:=\ell(w)$ (going from the largest to the smallest) and each reduced expression $\boldsymbol{w}=(s_1,\dots,s_k)$ of $w$, check whether $\boldsymbol{w}$ is the concatenation of $k/d$ copies of $(s_1,\dots,s_d)$. 
\item
Return $(k/d,s_1\dots s_d)$ for the first value of $d$ for which the condition holds.
\end{enumerate}
\end{algo}

\begin{algo}[Core splitting]\label{algo:core_splitting}
Assume that $W$ is of indefinite type. One can define a function $\boldsymbol{\ICS}$ taking as input a quadruple $(w,n,a,x)$ where $n\in\NN$ and $a,x\in W$, and where $w$ is a cyclically reduced element $w\in W$ with $\Pc(w)=W$ such that $P_w^{\max}=W_I$ for some $I\subseteq S$ (as provided by Algorithm~\ref{algo:maxsphparabnorm}), and giving as output a Boolean value ``$\mathrm{true}$'' or ``$\mathrm{false}$'', answering the question ``Is $w=ax^n$ the core splitting of $w$, with core $x$ ?'', as follows:
\begin{enumerate}
\item
Check whether $w=ax^n$.
\item
Check whether $a\in W_I$ and $x$ is of minimal length in $W_Ix$.
\item
For each $y\in\Cyc(x)$, check whether $\Root(y)=(1,y)$.
\item
Return $\mathrm{true}$ if (1), (2), (3) hold, and $\mathrm{false}$ otherwise.
\end{enumerate}
Indeed, (2) checks whether $a\in P_w^{\max}$ and $x$ is $P_w^{\max}$-reduced, and (3) determines whether $x=w_c$ by Lemma~\ref{lemma:computingwc2}, as $\Root(y)=(1,y)$ if and only if $y$ is weakly indivisible.
\end{algo}


\subsubsection{Steps to compute $\KKK_{\OOO_w}$}
We now come back to our standing assumption that $(W,S)$ is an irreducible Coxeter system of indefinite type.

Let $w\in W$ with $\Pc(w)=W$. Here are the steps to follow in order to compute $\KKK_{\OOO_w}$:

\begin{enumerate}
\item
Up to modifying $w$ inside $\Cyc(w)$, one can assume that $w$ is standard, that is, $w$ is cyclically reduced and $P_w^{\max}=W_I$, where $I:=S^{\eta_w}\subseteq S$ (see Algorithms~\ref{algo:cyclicshiftclass} and \ref{algo:maxsphparabnorm}).
\item
Compute $n_I\in W$ of minimal length in $W_Iw$, so that $w=w_In_I$ with $w_I:=wn_I\inv\in W_I$ and $n_I\in N_I$.
\item
Compute $\delta_w\in\Aut(W_I,I)$ and $I_w\subseteq I$ using Lemma~\ref{lemma:pietaWetaasgroupautomresidueSigma}: 
\begin{enumerate}
\item
$\delta_w$ is the diagram automorphism $I\to I:s\mapsto n_Isn_I\inv$.
\item
$I_w$ is the smallest $\delta_{w}$-invariant subset of $I$ containing $\supp(w_I)$.
\end{enumerate}
\item
Compute $\KKK_{\delta_w}^0(I_w)=\KKK_{\delta_w,W_I}^0(I_w)$ using Lemma~\ref{lemma:oppositionfinite}.
\item
Compute the core splitting $w=w_Iw_c^n$ of $w$ (i.e. find an atomic element $w_c$ such that $n_I=w_c^n$) with the help of Lemmas~\ref{lemma:computingwc} and \ref{lemma:computingwc2} (see also Algorithm~\ref{algo:core_splitting}).
\item
Compute the centraliser degree $n_w$ of $w$ using Lemma~\ref{lemma:nw} and Remark~\ref{remark:computingnw}.
\item
Compute $\Xi_w$ using Theorem~\ref{thm:Xiwindef}: it is generated by $\delta_{w_c}^{n_w}$.
\item
Compute $\KKK_{\delta_w}^0(I_w)/\Xi_w$, and hence also $\KKK_{\OOO_w}$ by Theorem~\ref{thmintro:graphisomorphism}.
\end{enumerate}


\subsection{Examples}\label{subsection:ExIND}

\begin{remark}\label{remark:Ilessn-2}
Note from \cite[\S 3.1]{Kra09} that if $(W,S)$ is any Coxeter system and $I$ is a maximal spherical subset of $S$, then $N_W(W_I)=W_I$. In particular, if $W$ is infinite and $w\in N_W(W_I)$ for some $w\in W$ with $\Pc(w)=W$, then $|I|\leq |S|-2$.
\end{remark}

\begin{example}
Let $(W,S)$ be a hyperbolic triangle group, that is, $S=\{s,t,u\}$ and $\tfrac{1}{m_{st}}+\tfrac{1}{m_{tu}}+\tfrac{1}{m_{su}}<1$, where $(m_{ab})_{a,b\in S}$ is the Coxeter matrix of $(W,S)$. Let $w\in W$ with $\Pc(w)=W$. Then $\OOO_w^{\min}=\Cyc_{\min}(w)$. 

Indeed, by Proposition~\ref{prop:SetainS}, we may assume that $w$ is standard, so that $P_w^{\max}=W_I$ for some spherical subset $I\subseteq S$. Remark~\ref{remark:Ilessn-2} then implies that $|I|\leq 1$. Hence either $I_w=\varnothing$ or $I_w=I$; in both cases, $\KKK_{\delta_w}^0(I_w)$ has a unique vertex, so that the claim follows from Theorem~\ref{thmintro:graphisomorphism}.
\end{example}

\begin{example}\label{example:nw=2}
Consider the Coxeter group $W$ with Coxeter diagram indexed by $S=\{s_i  \ | \ 1\leq i\leq 5\}$ and pictured on Figure~\ref{figure:CDInd}(a). 

One checks that $x:=s_3s_1s_2s_3s_4s_1s_2s_4s_5s_1s_2s_5\in W$ is cyclically reduced, and hence $\Pc(x)=W$ (see Algorithm~\ref{algo:cyclicshiftclass}). Moreover, a quick computation shows that $x\inv s_1x=s_2$ and $x\inv s_2x=s_1$, so that $x$ normalises $W_I$ with $I:=\{s_1,s_2\}$. In fact, one can check that $P_x^{\max}=W_I$ (see Algorithm~\ref{algo:maxsphparabnorm}), that $x$ is of minimal length in $W_Ix$ and that $x$ is indivisible (see Algorithm~\ref{algo:core_splitting}), so that $x$ is atomic by Lemma~\ref{lemma:atomicsimplerdef}. 

Let $w:=s_1x^2$, so that $P_w^{\max}=W_I$ by Lemma~\ref{lemma:Preducedproperties}(3). Then $w$ is cyclically reduced by Lemma~\ref{lemma:criterionwcyclredIND}, and $w_c=x$. We claim that $n_w=2$. Since $n_w\leq 2$ by Lemma~\ref{lemma:nw}(2), this amounts to prove that $n_w\neq 1$. But as $\delta_{w_c}(s_1)=s_2$ is not conjugate to $s_1$ in $W_I=\langle s_1\rangle\times\langle s_2\rangle$, this follows from Lemma~\ref{lemma:nw}(1) (note that $\delta_{w_c}^2=\id$ on $W_I$).

Since $\delta_w=\delta_{w_c}^2=\id$ and $I_w=\supp(s_1)=\{s_1\}$, the graph $\KKK_{\delta_w}^0(I_w)=\KKK_{\delta_w,W_I}^0(I_w)$ has only one vertex, and hence $\OOO_w^{\min}=\Cyc(w)$.

Let now $w':=s_1x$. As before, $w'$ is cyclically reduced and $P_{w'}^{\max}=W_I$. In this case, $n_{w'}=1$ by Lemma~\ref{lemma:nw}(2) and $w'_c=x$. Moreover, $\delta_{w'}=\delta_{x}$ and $I_{w'}=\supp_{\delta_x}(s_1)=I$, so that again $\OOO_{w'}^{\min}=\Cyc(w')$.
\end{example}

\begin{example}\label{example:A105}
Consider the Coxeter group $W$ with Coxeter diagram indexed by $S=\{s_i  \ | \ 1\leq i\leq 7\}$ and pictured on Figure~\ref{figure:CDInd}(b).

Set $J_1:=S\setminus\{s_6\}$ and $J_2:=S\setminus\{s_7\}$. One checks that $x:=w_0(J_1)w_0(J_2)\in W$ is cyclically reduced, and hence $\Pc(x)=W$ (see Algorithm~\ref{algo:cyclicshiftclass}). Moreover, $x$ normalises $I:=S\setminus\{s_6,s_7\}$, and the map $\delta_x\co W_I\to W_I:u\mapsto xux\inv$ is the unique nontrivial diagram automorphism of $(W_I,I)$. It follows from Remark~\ref{remark:Ilessn-2} that $P_x^{\max}=W_I$, and one easily checks that $x$ is of minimal length in $W_Ix$. Moreover, $x$ is indivisible (see Algorithm~\ref{algo:core_splitting}), and hence atomic by Lemma~\ref{lemma:atomicsimplerdef}. 

Let $w_i:=s_3x^i$ for $i\in\{1,2\}$. Then $P_{w_i}^{\max}=W_I$ by Lemma~\ref{lemma:Preducedproperties}(3) and $(w_i)_c=x$. In particular, $\delta_{w_i}=\delta_x^i$ and $I_{w_i}=\{s_3\}$. Moreover, $w_i$ is cyclically reduced by Lemma~\ref{lemma:criterionwcyclredIND}. Since $\delta_x$ is the conjugation by $w_0(I)$, Lemma~\ref{lemma:nw}(1) further implies that $n_{w_i}=1$, and hence $\Xi_{w_i}=\langle \delta_x\rangle$.

If $w:=w_1$, so that $\delta_{w}=\delta_x$, then $\KKK_{\delta_{w}}^0(I_{w})$ has only one vertex, so that $\OOO_{w}^{\min}=\Cyc(w)$.

If $w:=w_2$, so that $\delta_{w}=\id$, then $\KKK_{\delta_{w}}^0(I_{w})$ is a complete graph with vertices $I_j:=\{s_j\}$ for $j=1,\dots,5$. Hence $\KKK_{\delta_{w}}^0(I_{w})/\Xi_{w}$ is a complete graph with vertices $[I_1],[I_2],[I_3]$. Therefore, $\OOO_{w}^{\min}$ is the union of three distinct cyclic shift classes. Moreover, setting $u_3:=w$, $u_2:=s_2s_3s_2u_3s_2s_3s_2$ and $u_1:=s_1s_2s_1u_2s_1s_2s_1$, so that $$w=u_3\stackrel{\{s_2,s_3\}}{\too} u_2\stackrel{\{s_1,s_2\}}{\too} u_1,$$ we have $\varphi_w\inv([I_i])=\Cyc(u_i)$ for each $i=1,2,3$ by Theorem~\ref{thmintro:indefinite}(3), and hence $\OOO_{w}^{\min}=\sqcup_{i=1}^3\Cyc(u_i)$.

\begin{figure}
    \centering
    \begin{minipage}{0.5\textwidth}
        \centering
        \includegraphics[trim = 10mm 200mm 152mm 30mm, clip, width=0.7\textwidth]{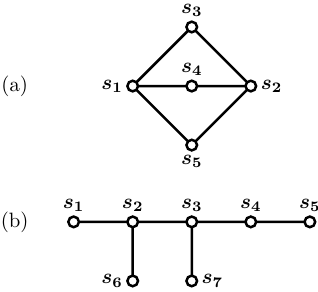}
        \caption{Examples~\ref{example:nw=2} and \ref{example:A105}}
        \label{figure:CDInd}
    \end{minipage}\hfill
    \begin{minipage}{0.5\textwidth}
        \centering
        \includegraphics[trim = 25mm 204mm 101mm 14mm, clip, width=0.95\textwidth]{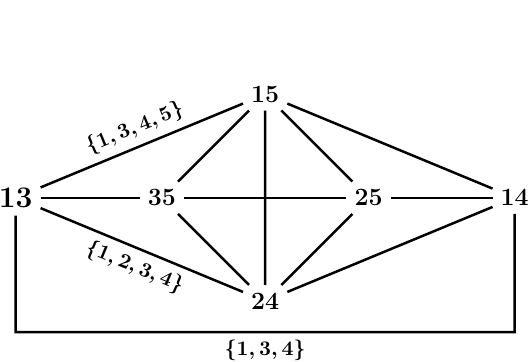} 
        \caption{Example~\ref{example:A105}}
        \label{figure:KKKInd}
    \end{minipage}
\end{figure}

Finally, let $w:=s_1s_3x^2$. As before, $P_w^{\max}=W_I$, $w_c=x$, $w$ is cyclically reduced, $n_w=1$ and $\Xi_w=\langle \delta_x\rangle$. Moreover, $\delta_w=\id$ and $I_w=\{s_1,s_3\}$. 
The graph $\KKK_{\delta_{w}}^0(I_{w})$ has $6$ vertices $I_{ij}:=\{s_i,s_j\}$ with $(i,j)\in\{(1,3),(1,4),(1,5),(2,4),(2,5),(3,5)\}$ and is pictured on Figure~\ref{figure:KKKInd} (where $I_{ij}$ is simply written $ij$, and where we put a label $K\subseteq I$ on some of the edges $\{I_{ij},I_{i'j'}\}$, indicating that $I_{ij},I_{i'j'}$ are $K$-conjugate). The quotient graph $\KKK_{\delta_{w}}^0(I_{w})/\Xi_w$ is then the complete graph on the $4$ vertices $[I_{13}]=[I_{35}],[I_{14}]=[I_{25}],[I_{15}],[I_{24}]$. Moreover, letting $w_{ij}$ denote for each $(i,j)\in\{(1,5),(2,4),(1,4)\}$ the conjugate of $w_{13}:=w$ by $w_0(K_{ij})$, where $K_{ij}$ is the label of the edge from $I_{13}$ to $I_{ij}$ on Figure~\ref{figure:KKKInd}, we have $\varphi_w\inv([I_{ij}])=\Cyc(w_{ij})$ for each such $(i,j)$, and hence $$\OOO_{w}^{\min}=\Cyc(w_{13})\sqcup \Cyc(w_{15})\sqcup\Cyc(w_{24})\sqcup\Cyc(w_{14}).$$
\end{example}


\section{Affine Coxeter groups}\label{section:ACG}

We now turn to the study of affine Coxeter groups $W$. We start with some preliminaries on their construction and basic properties in \S\ref{subsection:ACDP}, and fix the notations to be used for the whole section in \S\ref{subsection:SFTROSSACG}. In \S\ref{subsection:TTCGACAFF}, we identify the transversal Coxeter group and complex associated to a direction $\eta\in\partial X$ (see Proposition~\ref{prop:WetaSetaSigmaetaAff}). We then compute the group $\Xi_\eta$ in \S\ref{subsection:TSXEAff} (see Theorem~\ref{thm:mainthmXi}), and show how it can be related to $\Xi_w$ for $w\in W$ of infinite order with $\eta_w=\eta$ in \S\ref{subsection:DCSCAff} (see Theorem~\ref{thm:XietaXiw}); in particular, we compute $\Xi_w$ (see Proposition~\ref{prop:equivalentdefinitionsXiw}). As a byproduct, we also obtain in \S\ref{subsection:DCSCAff} a description of the centraliser of $w$ in $W$. In \S\ref{subsection:TSSIWOAE}, we investigate in more details the $P$-splitting introduced in \S\ref{section:TPSOAE} when $P=P_w^{\min}$, and show how it can be used to compute $\delta_w\in\Xi_w$ and $I_w\subseteq S^\eta$. We complete the proof of Theorem~\ref{thmintro:affine} and of Corollary~\ref{corintro:tightconjugationgraph} in \S\ref{subsection:TBADACEC}. To conclude, we give in \S\ref{subsection:CKASAFF} a summary of the steps to be performed in order to compute the structural conjugation graph associated to $\OOO_w$, and we illustrate this recipe on some examples in \S\ref{subsection:ExAFF}.

\subsection{Preliminaries}\label{subsection:ACDP}
General references for this subsection are \cite[Chapter~6]{Bourbaki}, \cite[Chapter~10]{BrownAbr} and \cite[Chapters~1,2]{Wei09} (see also \cite[\S 3]{DL11}).


\subsubsection{Root system}\label{subsubsection:RS}
Let $V=V_{\Phi}$ be a Euclidean vector space of dimension $\ell\geq 1$, whose inner product we denote by $\langle\cdot,\cdot\rangle$, and $\Phi\subseteq V$ be a reduced (not necessarily irreducible) root system in $V$ in the sense of \cite[Chapter~6]{Bourbaki}, with {\bf basis} $\Pi=\{\alpha_1,\dots,\alpha_{\ell}\}$ and associated {\bf Weyl group} $W_{\Phi}$. Thus, $\Pi$ is a basis of $V$, and every root $\alpha\in\Phi$ has the form $\alpha=\varepsilon\sum_{i=1}^{\ell}n_i\alpha_i$ for some $\varepsilon\in\{\pm 1\}$ and $n_i\in\NN$. In this case, the integer $\height(\alpha):=\varepsilon\sum_{i=1}^{\ell}n_i$ is called the {\bf height} of $\alpha$.

For each root $\alpha\in\Phi$, let $\alpha^{\vee}:=\tfrac{2\alpha}{\langle\alpha,\alpha\rangle}\in V$ denote its {\bf coroot}, and let $$s_{\alpha}\co V\to V:v\mapsto v-\langle v,\alpha\rangle \alpha^{\vee}$$
be its associated reflection. Then $W_{\Phi}\subseteq\GL(V)$ stabilises $\Phi^{\vee}:=\{\alpha^{\vee} \ | \ \alpha\in\Phi\}$, and is a finite Coxeter group with set of simple reflections $S_{\Pi}:=\{ s_i=s_{\alpha_i} \ | \ 1\leq i\leq\ell\}$. For each $i\neq j$, the numbers $a_{ij}:=\langle\alpha_i^{\vee},\alpha_j\rangle$ are nonpositive integers, and the Coxeter matrix of $W_{\Phi}$ is $(m_{ij})_{1\leq i,j\leq \ell}$ with $m_{ij}=2,3,4,6$ or $\infty$, depending on whether $a_{ij}a_{ji}=0,1,2,3$ or $\geq 4$, respectively.

Let $J_1,\dots,J_r$ be the components of $S_{\Pi}$. Then for each $i\in\{1,\dots,r\}$, the set $\Phi_{J_i}:=\Phi\cap \sum_{j\in J_i}\ZZ\alpha_j$ is an irreducible reduced root system (whose type we denote by $X_i\in\{A_{|J_i|},B_{|J_i|},C_{|J_i|},D_{|J_i|},E_{|J_i|},F_4,G_2\}$) coinciding with the orbit of $\Pi_{J_i}:=\Pi\cap \Phi_{J_i}$ under the action of the irreducible finite Coxeter group $(W_{\Phi})_{J_i}=\langle s_j \ | \ j\in J_i\rangle\subseteq W_{\Phi}$ (also of type $X_i$). Finally, we denote by $\theta_{J_i}$ the unique {\bf highest root} of $\Phi_{J_i}$, and we set $$H_{\Phi}:=\{\theta_{J_i} \ | \ 1\leq i\leq r\}\subseteq\Phi.$$

Note that, if $S_{\Pi}$ is of irreducible (finite) type $X_{\ell}$, then the highest root $\theta_{S_{\Pi}}=\sum_{j=1}^{\ell}a_j\alpha_j$ of $\Phi$ can be recovered from the Dynkin diagram $\Gamma_{S_{\Pi}}$ of type $X_{\ell}$ pictured on Figure~\ref{figure:TableFIN} in \S\ref{subsection:EOFO:Prelim}: the coefficient $a_j$ is written below the vertex of $\Gamma_{S_{\Pi}}$ corresponding to $\alpha_j$ for each $j\in\{1,\dots,\ell\}$ (see e.g. \cite[Chapter~4]{Kac}).


\subsubsection{Affine Weyl group}

Set $\Phi^a:=\Phi\times\ZZ$. To each $(\alpha,k)\in\Phi^a$, we associate an affine hyperplane $H_{\alpha,k}:=\{v\in V \ | \ \langle v,\alpha\rangle =k\}=H_{\alpha,0}+k\alpha^{\vee}/2$ in $V$, as well as the orthogonal reflection $$s_{\alpha,k}\co V\to V:v\mapsto s_{\alpha}(v)+k\alpha^{\vee}$$ with fixed hyperplane $H_{\alpha,k}$. For each $v\in V$, we denote by $t_v\co V\to V$ the {\bf translation} of $V$ given by 
$$t_v(u):=u+v\quad\textrm{for all $u\in V$}.$$
Thus, $s_{\alpha,k}=t_{k\alpha^{\vee}}s_{\alpha}$ for all $(\alpha,k)\in\Phi^a$.

The group $$W:=W_{\Phi^a}:=\langle s_{\alpha,k} \ | \ (\alpha,k)\in\Phi^a\rangle\subseteq\GL(V)$$ is called the {\bf affine Weyl group} of $\Phi$. It is a Coxeter group, with set of simple reflections 
$$S:=S_{\Pi}^a:=S_{\Pi}\cup \{s_{\theta,1} \ | \ \theta\in H_{\Phi}\},$$
and has $r=|H_{\Phi}|$ irreducible components of affine type $X_1^{(1)},\dots,X_r^{(1)}$, respectively determined by the subsets $$S_i:=\{s_j \ | \ j\in J_i\}\cup\{s_{\theta_{J_i},1}\}$$ ($1\leq i\leq r$) of $S$.


\subsubsection{Linear realisation and Dynkin diagram}\label{subsubsection:LRADD}
Let $\widehat{V}:=V\times\RR$, and set $\delta:=(0,-1)\in\widehat{V}$, so that $\widehat{V}=V\oplus \RR\delta$ (identifying $V$ with $V\times\{0\}\subseteq \widehat{V}$). Then $W$ can be viewed as a group of linear transformations of $\widehat{V}$, as follows (see \cite[\S 3.3]{DL11}).

Extend $\langle\cdot,\cdot\rangle$ to a symmetric positive semidefinite bilinear form on $\widehat{V}$ with radical $\RR\delta$. For each $x\in \widehat{V}\setminus\RR\delta$, set as before $x^{\vee}:=\tfrac{2x}{\langle x,x\rangle}\in\widehat{V}$, and define the linear reflection
$$\widehat{s}_x\co \widehat{V}\to\widehat{V}:v\mapsto v-\langle v,x\rangle x^{\vee}.$$
Then the assignment $$s_{\alpha,k}\mapsto \widehat{s}_{\alpha-k\delta}=\widehat{s}_{k\delta-\alpha}\quad\textrm{for $(\alpha,k)\in\Phi^a$}$$ defines an injective morphism $W\to\GL(\widehat{V})$, and we identify $W$ with its image in $\GL(\widehat{V})$. 

Moreover, when $\Phi$ is irreducible (that is, $r=1$), of type $X:=X_1$, then the assignment $(\alpha,k)\mapsto \alpha-k\delta$ identifies $\Phi^a$ with the set of real roots $\Delta^{re}(X^{(1)})$ of the Kac--Moody algebra of untwisted affine type $X^{(1)}$, mapping the roots $$\alpha_0:=-(\theta_{\Pi_S},1)=\delta-\theta_{\Pi_S},\alpha_1,\dots,\alpha_{\ell}\in\Phi^a$$ to the corresponding simple roots $\alpha_0,\alpha_1,\dots,\alpha_{\ell}$ of $\Delta^{re}(X^{(1)})$, and $\delta$ to the positive imaginary root of minimal height $\delta$ of $\Delta^{re}(X^{(1)})$ (see \cite[Chapter~1 and Theorem~5.6 and Proposition~6.3a)]{Kac}).

In this case, the matrix $(a_{ij})_{0\leq i,j\leq\ell}$ (where $a_{ij}:=\langle\alpha_i^{\vee},\alpha_j\rangle$) is the generalised Cartan matrix of affine type $X^{(1)}$ in the sense of \cite[Chapters~1,4]{Kac}, and we define the {\bf Dynkin diagram} associated to $\Phi^a$ as the graph $\Gamma_S$ with vertex set $S=\{s_0:=s_{\alpha_0}=s_{\theta_{\Pi_S},1},s_1,\dots,s_{\ell}\}$, and with $|a_{ij}|$ edges between $s_i$ and $s_j$ ($i\neq j$) if $|a_{ij}|\geq |a_{ji}|$, decorated with an arrow pointing towards $s_i$ in case $|a_{ij}|>1$. We call $X^{(1)}$ the {\bf type} of $\Gamma_S$.

The Dynkin diagram of irreducible type $X^{(1)}$ is pictured on Figure~\ref{figure:TableAFF} (where the vertex $s_i$ is simply labelled $i$). It will sometimes be useful to distinguish between a subset $I\subseteq S$ and the set\index[s]{-06@$\overline{I}$ (set of $\overline{s}$ with $s\in I$)} $\overline{I}:=\{\overline{s} \ | \ s\in I\}\subseteq\{0,1,\dots,\ell\}$ of corresponding indices, where for each $i\in \{0,1,\dots,\ell\}$ we set $\overline{s_i}:=i$\index[s]{-05@$\overline{s}$ (defined by $\overline{s_i}:=i$)} (so that $I=\{s_i \ | \ i\in\overline{I}\}$).

\begin{figure}[!htb]
    \centering
       \begin{minipage}{\textwidth}
        \centering
        \includegraphics[trim = 10mm 160mm 23mm 3mm, clip, width=\textwidth]{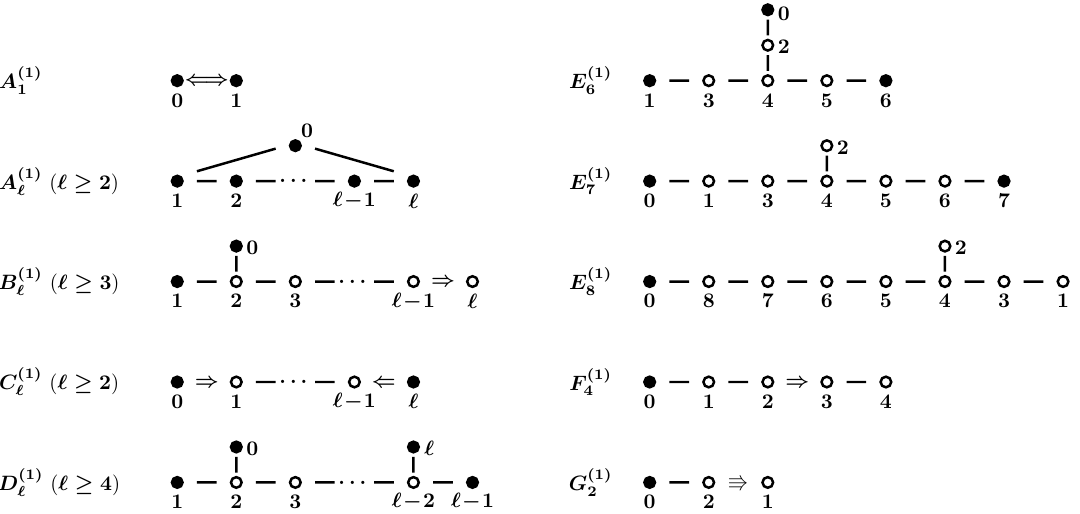}
          \end{minipage}
		\caption{Dynkin diagrams of untwisted affine type.}
		\label{figure:TableAFF}
\end{figure}

Back to a general $\Phi$, we define the {\bf Dynkin diagram} associated to $\Phi^a$ as the graph $\Gamma_S$ obtained as the disjoint union of the Dynkin diagrams $\Gamma_{S_i}$ of irreducible affine type $X_i^{(1)}$ (keeping the notations of the previous paragraphs).

Note that $\Gamma_{S}$ carries a bit more information (given by the arrows) than the Coxeter diagram $\Gamma_{S}^{\Cox}$ of $(W,S)$, but that both diagrams have the same connectedness properties. As in \S\ref{subsection:PCD}, we let $\Gamma_I$ for $I\subseteq S$ denote the subdiagram of $\Gamma_S$ with vertex set $I$, so that $I_1,\dots,I_r$ are the components of $I$ if and only if $\Gamma_{I_1},\dots,\Gamma_{I_r}$ are the connected components of $\Gamma_I\subseteq\Gamma_S$.


\subsubsection{Coxeter and Davis complex}

The group $W$ permutes the affine hyperplanes $H_{\alpha,k}$ ($(\alpha,k)\in\Phi^a$), and hence also the connected components of $V\setminus\bigcup_{(\alpha,k)\in\Phi^a}H_{\alpha,k}$, which are called {\bf alcoves}. We denote by $\Sigma(W,V)$ the cell complex induced from the tessellation of $V$ by the affine hyperplanes $H_{\alpha,k}$ ($(\alpha,k)\in\Phi^a$), equipped with the usual Euclidean metric on $V$, and by $\Aut(\Sigma(W,V))$ the group of isometries of $\Sigma(W,V)$ preserving the cell structure.

If $W$ is irreducible, then the Coxeter complex $\Sigma=\Sigma(W,S)$ (or rather, its subcomplex $\Sigma_s(W,S)$ of spherical simplices, obtained in this case from $\Sigma(W,S)$ by removing the empty simplex) can be identified with the underlying cell complex of $\Sigma(W,V)$. In general, the cell complex $\Sigma(W,V)$ is still isomorphic, as a poset, to $\Sigma_s(W,S)$, but need not be a simplicial complex anymore (see \cite[Proposition~10.13 and paragraph after Definition~10.15]{BrownAbr}). The walls of $\Sigma(W,S)$ then correspond to the affine hyperplanes $H_{\alpha,k}$ ($(\alpha,k)\in\Phi^a$), and the fundamental chamber $C_0$ correspond to the {\bf fundamental alcove}
$$C_0:=\{v\in V \ | \ \langle v,\alpha\rangle> 0 \ \textrm{for all $\alpha\in\Pi$,} \quad \langle v,\theta\rangle<1 \ \textrm{for all $\theta\in H_{\Phi}$}\}$$
delimited by the walls $\{H_{\alpha,0} \ | \ \alpha\in\Pi\}\cup\{H_{\theta,1} \ | \ \theta\in H_{\Phi}\}$ (see \cite[Proposition~1]{DL11}). Moreover, $\Sigma(W,V)$ itself is isometric to (and will be identified with) the Davis complex of $(W,S)$ (see \cite[Example~12.43]{BrownAbr}).


\subsubsection{Extended affine Weyl group and special vertices}\label{subsubsection:EAWGASV}

A point $v\in V$ is called {\bf special} if for each $\alpha\in\Phi$, there is some $k\in\ZZ$ such that $v\in H_{\alpha,k}$ (or, equivalently, if $\langle v,\alpha\rangle\in\ZZ$ for all $\alpha\in\Phi$). The group\index[s]{T@$T$ (translations in $\Aut(\Sigma)$)}
$$T:=T_{\Phi}:=\{t_v \ | \ \textrm{$v$ is special}\}$$
acts transitively on the set of special points of $V$, and coincides with the set of all translations of $V$ belonging to $\Aut(\Sigma)$. 

Note that each special point $v\in V$ is in fact a vertex (i.e. a $0$-dimensional cell) of $\Sigma$. In particular, denoting as before by $S_1,\dots,S_r$ the components of $S$, the vertex $v$ is of the form $w_1W_{S_1\setminus\{t_1\}}\times\dots\times w_rW_{S_r\setminus\{t_r\}}$ for some $w_i\in W_{S_i}$ and $t_i\in S_i$, $1\leq i\leq r$ (see \S\ref{subsection:PCC}). We then define its {\bf cotype} $\overline{\typ}_{\Sigma}(v):=(t_1,\dots,t_r)$. To lighten the notations, we shall always write $\typ_{\Sigma}(v):=\overline{\typ}_{\Sigma}(v)$ for the cotype\index[s]{typSigma@$\typ_{\Sigma}(v)$ (cotype of $v$ in $\Sigma$)} of a vertex $v$ of $\Sigma$ (this is the only use we will make of the notation $\typ_{\Sigma}$ for the rest of the paper, so there will be no confusion possible with the type function from \S\ref{subsection:PCC}). When $(W,S)$ is irreducible (that is, $r=1$), we simply write $\typ_{\Sigma}(v):=t_1\in S$, and we call the vertex $t_1\in S$ of $\Gamma_S$ {\bf special} when $v$ is special; the special vertices of $\Gamma_S$ are the black dots indicated on Figure~\ref{figure:TableAFF}. Alternatively, if $\Gamma_S$ is of irreducible type $X^{(1)}$, then a vertex $s$ of $\Gamma_S$ is special if and only if the Dynkin diagram $\Gamma_{S\setminus\{s\}}$ is of type $X$ (see \cite[Chapter~4]{Kac}); in this case, we call $\Gamma_S$ the Dynkin diagram {\bf extending} $\Gamma_{S\setminus\{s\}}$.

The {\bf extended affine Weyl group} of $\Phi$ (also called the {\bf extended Weyl group} of $W$) is the group\index[s]{Wtilde@$\widetilde{W}$ (extended Weyl group of $W$)} $\widetilde{W}$ generated by $W$ and $T$. It is the direct product of the extended affine Weyl groups $\widetilde{W}_{S_i}$ of $\Phi_{J_i}$ ($1\leq i\leq r$). Moreover, it decomposes as a semidirect product $\widetilde{W}=W_{\Phi}\ltimes T$, and contains $W$ as a normal subgroup (see \cite[Proposition~1]{DL11}). 

We will identify the quotient group $\widetilde{W}/W=\prod_{i=1}^r\widetilde{W}_{S_i}/W_{S_i}$ with the image of $\widetilde{W}\subseteq\Aut(\Sigma)=W\rtimes\Aut(\Sigma,C_0)$ under the quotient map $\Aut(\Sigma)\to\Aut(\Sigma,C_0)=\Aut(W,S)=\Aut(\Gamma_S^{\Cox})=\Aut(\Gamma_S)$. It is given by the following lemma (note that for an irreducible $(W,S)$, any automorphism of $\Gamma_S$ is determined by where it sends the special vertices).

\begin{lemma}\label{lemma:extended_diagram_autom}
Assume that $(W,S)$ is irreducible, of type $X=X_{\ell}^{(1)}$. Writing each element of $\Aut(\Gamma_S)$ as a permutation of the set $\{i\in\overline{S} \ | \ \textrm{$s_i$ is special}\}$, the group $\widetilde{W}/W\subseteq\Aut(\Gamma_S)$ is given as follows:
\begin{enumerate}
\item If $X=A_{\ell}^{(1)}$, then $\widetilde{W}/W=\langle (0,1,2,\dots,\ell)\rangle$.
\item If $X=B_{\ell}^{(1)}$, then $\widetilde{W}/W=\langle (0,1)\rangle$.
\item If $X=C_{\ell}^{(1)}$, then $\widetilde{W}/W=\langle (0,\ell)\rangle$.
\item If $X=D_{\ell}^{(1)}$, then $\widetilde{W}/W=\langle (0,\ell-1,1,\ell)\rangle$ if $\ell$ is odd, and $\widetilde{W}/W=\langle (0,1)(\ell-1,\ell),\ (0,\ell-1)(1,\ell)\rangle$ if $\ell$ is even.
\item If $X=E_6^{(1)}$, then $\widetilde{W}/W=\langle (0,1,6)\rangle$.
\item If $X=E_7^{(1)}$, then $\widetilde{W}/W=\langle (0,7)\rangle$.
\item If $X=E_8^{(1)}$, $F_4^{(1)}$, or $G_2^{(1)}$, then $\widetilde{W}/W$ is trivial.
\end{enumerate}
In particular, for any two special vertices $x,y$ of $\Gamma_S$, there is a unique $\sigma_{x,y}\in\widetilde{W}/W\subseteq\Aut(\Gamma_S)$ mapping $x$ to $y$.
\end{lemma}
\begin{proof}
See \cite[VI, \S4.5(XII)--4.13(XII)]{Bourbaki}.
\end{proof}

Finally, for any special vertex $x\in V$, the stabiliser $W_x$ of $x$ in $W$ can be identified with $W_{\Phi}$, and we have a semidirect decomposition $W=W_x\ltimes T_0$, where\index[s]{T0@$T_0$ (translations in $W$)}
$$T_0:=T\cap W.$$
In particular,
$$\widetilde{W}/W\cong T/T_0.$$
Moreover, $T_0$ acts transitively on the set of special vertices of $\Sigma$ of a given (co-)type (see e.g. \cite[Corollary~1.30]{Wei09}).


\subsection{Setting for the rest of Section~\ref{section:ACG}}\label{subsection:SFTROSSACG}
For the rest of this section, we fix a Coxeter system $(W,S)$ of irreducible affine type $X_{\ell}^{(1)}$, so that $W=W_{\Phi^a}$ and $S=S_{\Pi}^a=\{s_0,s_1,\dots,s_{\ell}\}$ for some irreducible root system $\Phi$ of type $X_{\ell}$ and root basis $\Pi=\{\alpha_1,\dots,\alpha_{\ell}\}$, and we keep all the notations from \S\ref{subsection:ACDP} (including $T:=T_{\Phi}$ and $T_0:=T\cap W$).

In particular, we identify $\Phi^a$ with the set $\Delta^{re}(X^{(1)})=\{\alpha+n\delta \ | \ \alpha\in\Phi, \ n\in\ZZ\}$ of real roots of the Kac--Moody algebra of untwisted affine type $X_{\ell}^{(1)}$, as in \S\ref{subsubsection:LRADD}. Thus, $\delta=\alpha_0+\theta_{S_{\Pi}}$ where $S_{\Pi}=\{s_1,\dots,s_{\ell}\}$, and $\Phi^a=W\cdot\{\alpha_0,\alpha_1,\dots,\alpha_{\ell}\}$ with $s_i\alpha_j=\alpha_j-a_{ij}\alpha_i$ ($i,j\in \overline{S}=\{0,1,\dots,\ell\}$). For $\beta=\sum_{i=0}^{\ell}n_i\alpha_i\in \Delta^{re}(X^{(1)})\cup\ZZ\delta$, we write $$\height(\beta):=\sum_{i=0}^{\ell}n_i\quad\textrm{and}\quad\supp(\beta):=\{\alpha_i \ | \ n_i\neq 0, \ 0\leq i\leq\ell\}$$ for the {\bf height} and {\bf support} of $\beta$, respectively.

Note that if the root $\beta:=(\alpha,k)=\alpha-k\delta\in\Phi^a$ can be written as $\beta=w\alpha_i$ with $w\in W$ and $i\in \overline{S}$, then the reflection $r_m$ of $\Sigma$ across the wall $m=H_{\beta}=H_{\alpha,k}\in\WW$ (see \S\ref{subsection:PCC}) coincides with $s_{\alpha,k}=ws_iw\inv$. We will then also write $$r_{\beta}:=r_m=ws_iw\inv,$$
and we set $$m_{r_{\beta}}:=H_{\beta}\in\WW.$$

Finally, for each $i\in\overline{S}$, we shall denote by\index[s]{xi@$x_i$ (vertex of cotype $s_i$ of the fundamental alcove)} $x_i\in V$ the vertex of cotype $s_i$ of the fundamental chamber (i.e. alcove) $C_0$, and we set\index[s]{mi@$m_i$ (wall fixed by $s_i$)} $m_i:=m_{s_i}=H_{\alpha_i}\in\WW$, so that $C_0$ is delimited by the walls $m_0,\dots,m_{\ell}$, and $\{x_i\}=\bigcap_{j\in\overline{S}\setminus\{i\}}m_j$. Moreover, we choose the vertex $x_0$ as the origin of $V$ (and we identify $V$ with the Davis complex $X$ of $(W,S)$). Since $x_0$ is special, we can also identify $W_{\Phi}$ and $W_{x_0}$.


\subsection{The transversal Coxeter group and complex}\label{subsection:TTCGACAFF}

\begin{definition}\label{definition:Veta}
Let $\eta\in\partial V$. Recall that $\WW^{\eta}$ denotes the set of walls $m\subseteq V$ of $\Sigma$ containing $\eta$ in their visual boundary.  Set\index[s]{Wx0@$\WW_{x_0}$ (walls containing $x_0$)}\index[s]{Wx0eta@$\WW^{\eta}_{x_0}$ (walls of $\WW^{\eta}$ containing $x_0$)} $$\WW_{x_0}:=\{m\in\WW \ | \ x_0\in m\}\quad\textrm{and}\quad \WW^{\eta}_{x_0}:=\WW^\eta\cap\WW_{x_0}.$$ Let also\index[s]{Veta@$V^{\eta}$ (orthogonal complement of $\bigcap_{m\in \WW^{\eta}_{x_0}}m$ in $V$)} $V^{\eta}\subseteq V$ be the orthogonal complement of $\bigcap_{m\in \WW^{\eta}_{x_0}}m$ in $V$, and let\index[s]{piVeta@$\pi_{V^{\eta}}$ (orthogonal projection onto $V^{\eta}$)}\index[s]{xeta@$x^\eta$ (image of $x$ under $\pi_{V^{\eta}}$)} $$\pi_{V^{\eta}}\co V\to V^{\eta}:x\mapsto x^\eta$$ denote the orthogonal projection onto $V^{\eta}$. 

We also let\index[s]{xperpeta@$x^{\perp_\eta}$ (unit vector in the direction $\eta$)} $x^{\perp_\eta}\in [x_0,\eta)\subseteq V$ denote the unit vector in the direction $\eta$. Finally, we let\index[s]{sigmax0eta@$\sigma_{x_0,\eta}$ (spherical simplex containing a germ of $[x_0,\eta)$ at $x_0$)} $\sigma_{x_0,\eta}$ be the (closed) spherical simplex of $\Sigma$ containing the geodesic segment $[x_0,\epsilon x^{\perp_\eta}]\subseteq[x_0,\eta)$ for all sufficiently small $\epsilon>0$, so that the residue $R_{\sigma_{x_0,\eta}}$ of $\Sigma$ has stabiliser $$\Fix_W([x_0,\eta))=W_{x_0}\cap W_{\eta}$$ and set of walls $\WW^{\eta}_{x_0}$.
\end{definition}

\begin{definition}\label{definition:etastandard}
We call $\eta\in\partial V$ {\bf standard}\index{Standard! direction $\eta$} if $\Fix_W([x_0,\eta))$ is a standard parabolic subgroup (equivalently, if the simplex $\sigma_{x_0,\eta}$ is contained in the (closed) fundamental chamber $C_0$). In this case, we let\index[s]{Ieta@$I_{\eta}$ (defined by $\Fix_W([x_0,\eta))=W_{I_\eta}$ in case $\eta$ is standard)} $I_{\eta}\subseteq S$ be such that $\Fix_W([x_0,\eta))=W_{I_\eta}$.
\end{definition}

\begin{lemma}\label{lemma:Veta_forI}
Let $\eta\in\partial V$ be standard, and set $I:=I_{\eta}$. Then:
\begin{enumerate}
\item
$V^{\eta}$ is the orthogonal complement of $\pi_{V^\eta}\inv(x_0)=\bigcap_{i\in \overline{I}}m_i$ in $V$. 
\item
$\pi_{V^\eta}\inv(x_0)$ has dimension $\ell-|I|$ and is spanned by the lines through $x_0,x_i$ for $i\in \{1,\dots,\ell\}\setminus\overline{I}$.
\item
$\dim V^{\eta}=|I|\leq\ell-1$.
\end{enumerate}
\end{lemma}
\begin{proof}
Recalling that $\{m_i \ | \ i\in\overline{I}\}$ is the set of walls of the parabolic subgroup $\Fix_W([x_0,\eta))$ of $W$ (that is, the set of walls in $\WW^\eta_{x_0}$), the statement (1) follows from the definition of $V^\eta$. The statement (2) then follows from the fact that the $\ell$ lines $L_i=\bigcap_{1\leq j\leq\ell}^{j\neq i}m_j$ through $x_0,x_i$ ($i=1,\dots,\ell$) span $V$, and (3) is a consequence of (1) and (2).
\end{proof}

The following proposition, whose content is certainly folklore, describes $(W^\eta,S^\eta)$ and $\Sigma^\eta$ for any $\eta\in\partial V$.

\begin{prop}\label{prop:WetaSetaSigmaetaAff}
Let $\eta\in\partial V$. Then the following assertions hold:
\begin{enumerate}
\item
The set $\Phi_{\eta}:=\{\alpha\in\Phi \ | \ H_{\alpha,0}\in \WW^{\eta}_{x_0}\}=\{\alpha\in\Phi \ | \ \langle \alpha,x^{\perp_\eta}\rangle=0\}$ is a reduced (not necessarily irreducible) root system in $V^\eta$.
\item
For each $(\alpha,k)\in\Phi_\eta^a=\Phi_\eta\times\ZZ$, the affine hyperplane $H_{\alpha,k}=\{v\in V^\eta \ | \ \langle v,\alpha\rangle=k\}$ of $V^\eta$ is the trace on $V^\eta$ of the corresponding hyperplane $H_{\alpha,k}$ of $V$ (we will identify these two hyperplanes in the sequel). 
\item 
Write $\Fix_W([x_0,\eta))=bW_Ib\inv\subseteq W_{\Phi}$ for some subset $I\subseteq S_{\Pi}=S\setminus\{s_0\}$ and some $b\in W_{x_0}$. Then $\Pi^\eta_b:=b\Pi_I=\{b\alpha_i \ | \ i\in \overline{I}\}$ is a root basis of $\Phi_{\eta}$.
\item
Viewing $W^\eta$ as a subgroup of $\GL(V^\eta)$, and identifying the affine reflection $s_{\alpha,k}$ for $(\alpha,k)\in \Phi^a_{\eta}$ with its restriction to $V^\eta$, the couple $(W^\eta,V^\eta)$ coincides, in the notations of \S\ref{subsection:ACDP}, with $(W_{\Phi_{\eta}^a},V_{\Phi_{\eta}})$. 
\item
The poset $\Sigma^\eta$ is the underlying cell complex of $\Sigma(W^\eta,V^\eta)$, with fundamental alcove $C_0^\eta$, and $\Sigma(W^\eta,V^\eta)$ is the Davis realisation of $\Sigma^\eta$.
\item
Let $b,I$ be as in (3), and such that $b$ is of minimal length in $bW_I$. Then $b\inv S^{\eta}b=S^{b\inv\eta}$. Moreover, if $I_1,\dots,I_r$ are the components of $I$, then $S^{b\inv\eta}=I\cup\{s_{\theta_{I_i},1}=r_{\delta-\theta_{I_i}} \ | \ 1\leq i\leq r\}$ and has components $I_1^{\ext},\dots,I_r^{\ext}$, where $I_i^{\ext}:=I_i\cup\{r_{\delta-\theta_{I_i}}\}$. In particular, the Dynkin diagram $\Gamma_{S^{b\inv\eta}}$ associated to $\Phi^a_{b\inv\eta}=b\inv \Phi^a_{\eta}$ has connected components $\Gamma_{I_1^{\ext}},\dots,\Gamma_{I_r^{\ext}}$, and $\Gamma_{I_i^{\ext}}$ is the Dynkin diagram extending $\Gamma_{I_i}$. 
\end{enumerate}
\end{prop}
\begin{proof}
For (1), note that $\Phi_\eta$ is certainly a reduced root system in $\mathrm{span}_{\RR}\Phi_\eta\subseteq V^\eta$. On the other hand, if $b,I$ are as in (3), then $b\Pi_I\subseteq\Phi_{\eta}$ (see the proof of (3) below). Since $V^{b\inv\eta}=b\inv V^{\eta}$ has dimension $|I|$ by  Lemma~\ref{lemma:Veta_forI}(3) (as $\Fix_W([x_0,b\inv\eta))=W_I$) and $b\Pi_I$ is a linearly independent subset of $\Phi_\eta$, we conclude that $\mathrm{span}_{\RR}\Phi_\eta=V^\eta$, as desired.

The statement (2) is a tautology. For (3), note first that, by assumption, $I$ consists of all $s_i\in S\setminus\{s_0\}$ such that $H_{b\alpha_i}=bH_{\alpha_i}$ contains $[x_0,\eta)$, or equivalently, such that $b\alpha_i\in\Phi_{\eta}$. In particular, $b\Pi_I\subseteq\Phi_{\eta}$. Let now $\alpha\in \Phi_{\eta}$. Since $b\Pi$ is a root basis of $\Phi$, we can write $\alpha$ in a unique way as $\alpha=\varepsilon \sum_{i=1}^{\ell}n_i b\alpha_i$ for some $\varepsilon\in\{\pm 1\}$ and $n_i\in\NN$. On the other hand, since $b\inv x^{\perp_\eta}$ belongs to the connected component of $V\setminus\bigcup_{m\in\WW_{x_0}}m$ (i.e. \emph{Weyl chamber} of $\Phi$) delimited by the walls $H_{\alpha_1},\dots,H_{\alpha_{\ell}}$, we have $\langle \alpha_i, b\inv x^{\perp_\eta}\rangle\geq 0$ for all $i\in\{1,\dots,\ell\}$ (see \cite[Ch. VI, \S1 nr 5]{Bourbaki}), and hence also $$\langle b\alpha_i, x^{\perp_\eta}\rangle\geq 0\quad\textrm{for all $i\in \{1,\dots,\ell\}$}.$$ We then deduce from
$$0=\langle \alpha,x^{\perp_\eta}\rangle=\varepsilon\sum_{i=1}^{\ell}n_i\langle b\alpha_i,x^{\perp_\eta}\rangle$$
that $n_i=0$ if $s_i\notin I$, and hence $b\Pi_I$ is indeed a root basis of $\Phi_{\eta}$, yielding (3). 

The statements (4), (5) and (6) are clear from \S\ref{subsection:ACDP}, except for the first claim in (6): to check that $b\inv S^{\eta}b=S^{b\inv\eta}$, it is sufficient to show by (\ref{eqn:autominvariance}) that $\pi_{\Sigma^\eta}(bC_0)=C_0^\eta$. The assumption on $b$ means that $bC_0=\proj_{R_{\sigma_{x_0,\eta}}}(C_0)$. Suppose for a contradiction that a wall $m\in\WW^\eta$ separates $C_0$ from $bC_0$. Then $m$ contains $x_0$, and hence $m\in\WW_{x_0}^\eta$, that is, $m$ is a wall of $R_{\sigma_{x_0,\eta}}$, yielding the desired contradiction.
\end{proof}

\begin{remark}\label{remark:binveta}
Let $\eta\in\partial V$, and let $b,I$ be as in Proposition~\ref{prop:WetaSetaSigmaetaAff}(6). Then $b\inv\eta$ is standard, $I=I_{b\inv\eta}$, and Proposition~\ref{prop:WetaSetaSigmaetaAff}(6) implies that $$S^{b\inv\eta}=b\inv S^\eta b,\quad  W^{b\inv\eta}=b\inv W^\eta b,\quad\textrm{and} \quad\Sigma^{b\inv\eta}=b\inv\Sigma^\eta\subseteq V^{b\inv\eta}=b\inv V^\eta.$$
Moreover, the fundamental chamber $C_0^{b\inv\eta}$ of $\Sigma^{b\inv\eta}=b\inv\Sigma^\eta$ coincides with $b\inv C_0^\eta$, as it is delimited by the walls fixed by the reflections in $S^{b\inv\eta}=b\inv S^\eta b$.
\end{remark}

\begin{definition}\label{definition:tauIextetc}
Let $\eta\in\partial V$ be standard, and let $I_1,\dots,I_r$ be the components of $I_\eta$. For each $i\in\{1,\dots,r\}$, we denote by $\tau_i$\index[s]{taui@$\tau_i$ (added vertex in the affine extension of $I_i$)} the additional vertex in the Dynkin diagram $\Gamma_{I_i^{\ext}}$ extending $\Gamma_{I_i}\subseteq\Gamma_S$, so that\index[s]{Iiext@$I_i^{\ext}$ (affine extension of $I_i$ by $\tau_i$)} $$I_i^{\ext}=I_i\cup\{\tau_i\}.$$ Proposition~\ref{prop:WetaSetaSigmaetaAff}(6) then implies that 
$$\tau_i=r_{\delta-\theta_{I_i}}\quad\textrm{for each $i\in\{1,\dots,r\}$}$$
and that the Dynkin diagram $\Gamma_{S^\eta}$ associated to $\Phi_{\eta}^a$ has connected components $\Gamma_{I_1^{\ext}},\dots,\Gamma_{I_r^{\ext}}$. We will then also write\index[s]{Ietaext@$I_{\eta}^{\ext}$ (affine extension of $I_\eta$)} $$I_{\eta}^{\ext}:=S^\eta=\bigcup_{i=1}^rI_i^{\ext}.$$

As usual, we let $$W^{\eta}_{I_i^{\ext}}:=\langle I_i^{\ext}\rangle\subseteq W^{\eta}$$ denote the standard parabolic subgroup of $W^\eta$ of type $I_i^{\ext}$ ($i=1,\dots,r$), so that $$W^{\eta}=W^{\eta}_{I_1^{\ext}}\times\dots\times W^{\eta}_{I_r^{\ext}}.$$
Finally, recall from \S\ref{subsubsection:EAWGASV} that the extended Weyl group $\widetilde{W}^{\eta}$ of $W^\eta$ is the direct product
$$\widetilde{W}^{\eta}=\widetilde{W}^{\eta}_{I_1^{\ext}}\times\dots\times \widetilde{W}^{\eta}_{I_r^{\ext}}$$ of the extended Weyl groups $\widetilde{W}^{\eta}_{I_i^{\ext}}$ of the irreducible Coxeter groups $W^{\eta}_{I_i^{\ext}}$.
\end{definition}


\subsection{The subgroup \texorpdfstring{$\Xi_\eta$}{Xieta}}\label{subsection:TSXEAff}

Recall from Definition~\ref{definition:Xieta} that, for a given $\eta\in\partial V$, we set
$$\Xi_{\eta}=\pi_{\eta}(W_\eta)\cap\Aut(\Sigma^\eta,C_0^\eta)\approx \pi_{\eta}(W_\eta)/W^\eta\subseteq \Aut(\Sigma^\eta,C_0^\eta)=\Aut(W^\eta,S^\eta),$$
where $W_\eta$ is the stabiliser of $\eta$ in $W$. The purpose of this section is to obtain a complete description of $\Xi_\eta$.

We start with a few observations on $W_\eta$. Recall from \S\ref{subsubsection:EAWGASV} that $W=W_{x_0}\ltimes T_0$ where $T_0$ is the set of translations in $W$, and that $W^\eta_{x_0}$ denotes the stabiliser of $x_0$ in $W^\eta$.

\begin{lemma}\label{lemma:Weta}
Let $\eta\in\partial V$. Then the following assertions hold:
\begin{enumerate}
\item $W_\eta=\langle W^{\eta},T_0\rangle\subseteq W$.
\item $W_\eta=W^{\eta}_{x_0}\ltimes T_0$.
\item The image of $\pi_{\eta}\co W_\eta\to \Aut(\Sigma^{\eta})$ is a subgroup of the extended Weyl group $\widetilde{W}^{\eta}$ of $W^\eta$. 
\item
$\ker\pi_{\eta}=\ZZZ_{W_\eta}(W^\eta)=\{t_v\in T_0 \ | \ v\in (V^\eta)^{\perp}=\bigcap_{m\in\WW_{x_0}^\eta}m\}$.
\item If $w\in W_\eta$, then $\pi_{V^\eta}(wx)=\pi_{\eta}(w)\pi_{V^\eta}(x)$ for all $x\in V$.
\end{enumerate}
\end{lemma}
\begin{proof}
(1) Clearly, $\langle W^{\eta},T_0\rangle$ stabilises $\eta$. Conversely, if $ut\in W=W_{x_0}\ltimes T_0$ stabilises $\eta$ ($u\in W_{x_0}$, $t\in T_0$), then $u$ stabilises $\eta$, and hence $u$ fixes the geodesic ray $[x_0,\eta)$. Since $\Fix_W([x_0,\eta))\subseteq W^\eta$, the claim follows.

(2) This follows from (1) and the fact that $W^\eta\subseteq  W^{\eta}_{x_0}T_0$.

(3) This follows from the fact that $\pi_{\eta}(W^{\eta})=W^{\eta}$ and that $\pi_{\eta}(T_0)$ is a group of translations in $\Aut(\Sigma^{\eta})$.

(4) An element $w\in W_\eta$ belongs to $\ker\pi_{\eta}$ if and ony if it stabilises each of the walls in $\WW^\eta$, or equivalently, if it centralises $W^\eta$. On the other hand, if $w$  stabilises each of the walls in $\WW^\eta$, then in particular it maps every wall in $\WW^\eta$ to a parallel one, and hence $w\in T_0$ by (2). The claim easily follows.

(5) If $w\in W^\eta$ or if $w\in T_0$, the claim is clear. In general, if $w=ut$ with $u\in W^\eta_{x_0}$ and $t\in T_0$ (see (2)), we then have $\pi_{V^\eta}(wx)=\pi_{\eta}(u)\pi_{V^\eta}(tx)=\pi_{\eta}(u)\pi_{\eta}(t)\pi_{V^\eta}(x)=\pi_{\eta}(w)\pi_{V^\eta}(x)$, as desired.
\end{proof}

\begin{lemma}\label{lemma:Xietaabelian}
Let $\eta\in\partial V$ be standard, and let $I_1,\dots,I_r$ be the components of $I_\eta$. Then 
$$\Xi_{\eta}\subseteq \widetilde{W}^{\eta}/W^\eta=\prod_{i=1}^r\widetilde{W}^\eta_{I_i^{\ext}}/W^\eta_{I_i^{\ext}}\subseteq\prod_{i=1}^r\Aut(W^\eta_{I_i^{\ext}},I_i^{\ext})=\prod_{i=1}^r\Aut(\Gamma_{I_i^{\ext}}).$$
In particular, $\Xi_\eta$ is abelian.
\end{lemma}
\begin{proof}
The first assertion follows from Lemma~\ref{lemma:Weta}(3), while the second follows from Lemma~\ref{lemma:extended_diagram_autom}.
\end{proof}

\begin{remark}\label{remark:Xietaabelian}
Let $\eta\in\partial V$, and let $b,I$ be as in Proposition~\ref{prop:WetaSetaSigmaetaAff}(6). Then, as in Remark~\ref{remark:binveta}, $b\inv\eta\in\partial V$ is standard, and 
$$\Xi_{b\inv\eta}\cong \pi_{b\inv\eta}(W_{b\inv\eta})/W^{b\inv\eta}=b\inv \pi_{\eta}(W_{\eta})b/b\inv W^{\eta}b\cong \pi_{\eta}(W_\eta)/W^\eta\cong\Xi_{\eta}.$$
In particular, Lemma~\ref{lemma:Xietaabelian} implies that $\Xi_{\eta}$ is abelian.

More precisely, since $S^{b\inv\eta}=b\inv S^\eta b$, the conjugation map $\kappa_{b}\co S^\eta\to S^{b\inv\eta}$ induces an isomorphism
\begin{equation}
\phi_b\co \Gamma_{S^\eta}\stackrel{\cong}{\to} \Gamma_{S^{b\inv\eta}}
\end{equation}
of the corresponding Dynkin diagrams, and the above isomorphism between $\Xi_{\eta}\subseteq\Aut(\Gamma_{S^\eta})$ and $\Xi_{b\inv\eta}\subseteq\Aut(\Gamma_{S^{b\inv\eta}})$ is given by
\begin{equation}
\widetilde{\phi}_b\co \Xi_{\eta}\stackrel{\cong}{\to}  \Xi_{b\inv\eta}:\delta\mapsto \phi_b\circ\delta\circ\phi_b\inv.
\end{equation}
Alternatively, if we instead view $\Xi_{\eta}$ and $\Xi_{b\inv\eta}$ as subgroups of $\Aut(\Sigma^\eta)$ and $\Aut(\Sigma^{b\inv\eta})$, respectively, then as $\Sigma^{b\inv\eta}=b\inv\Sigma^\eta$ (see Remark~\ref{remark:binveta}), this isomorphism is simply given by
\begin{equation}
\Xi_{\eta}\stackrel{\cong}{\to}  \Xi_{b\inv\eta}:\delta\mapsto b\inv\delta b.
\end{equation}
\end{remark}

\begin{definition}\label{definition:sigmaj}
Let $\eta\in\partial V$ be standard, and let $I_1,\dots,I_r$ be the components of $I:=I_{\eta}\subseteq S\setminus\{s_0\}$. In order to describe the elements of $\Xi_{\eta}\subseteq \prod_{i=1}^r\Aut(\Gamma_{I_i^{\ext}})$, we introduce the following notations, based on the numbering of the vertices of the Dynkin diagram of untwisted affine type $X_{\ell}^{(1)}$ pictured on Figure~\ref{figure:TableAFF}.
\begin{itemize}
\item For each $i\in\{1,\dots,r\}$, we set\index[s]{Ijalpha@$I_j^{\alpha}$ (vertex of $\Gamma_{I_j}$ with smallest index)}\index[s]{Ijomega@$I_j^{\omega}$ (vertex of $\Gamma_{I_j}$ with largest index)}
$$I_j^{\alpha}:=s_{\min\overline{I_j}}\in I_j\quad\textrm{and}\quad I_j^{\omega}:=s_{\max\overline{I_j}}\in I_j,$$
so that $I_j^{\alpha}$ and $I_j^{\omega}$ are the vertices of $\Gamma_{I_j}\subseteq\Gamma_S$ with smallest and largest index, respectively.
\item
For each $i\in\{1,\dots,r\}$ and $s\in I_i$, we let\index[s]{sigmas@$\sigma_s$ (either identity or diagram automorphism of $\Gamma_{I_i^{\ext}}$ mapping $\tau_i$ to $s$)} $$\sigma_s\in \widetilde{W}^{\eta}_{I_i^{\ext}}/W^{\eta}_{I_i^{\ext}}\subseteq \Aut(\Gamma_{I_i^{\ext}})$$ denote the unique diagram automorphism mapping $\tau_i\in I_i^{\ext}$ (cf. Definition~\ref{definition:tauIextetc}) to $s$ in case $s$ is a special vertex of $\Gamma_{I_i^{\ext}}$ (as provided by Lemma~\ref{lemma:extended_diagram_autom}), and denote the identity element of $\Aut(\Gamma_{I_i^{\ext}})$ otherwise. If $s\in S\setminus I$, we set $\sigma_s:=\id\in \widetilde{W}^{\eta}/W^{\eta}$.
\item
Order $I_1,\dots,I_r$ so that $I_1^{\alpha}\leq\dots\leq I_r^{\alpha}$, and assume that $\Gamma_S$ is of {\bf classical type} $X_{\ell}^{(1)}\in\{A_{\ell}^{(1)},B_{\ell}^{(1)},C_{\ell}^{(1)},D_{\ell}^{(1)}\}$. In this case, we set\index[s]{sigmai@$\sigma_i$ (short for $\sigma_{I_i^{\alpha}}$ for classical types)}
$$\sigma_i:=\sigma_{I_i^{\alpha}}\quad\textrm{for all $i\in\{1,\dots,r\}$.}$$
A quick inspection of Figure~\ref{figure:TableAFF} shows that the vertex $I_i^{\alpha}$ of $\Gamma_{I_i^{\ext}}$ must in fact be special for any $i\in\{1,\dots,r\}$, except if $X_{\ell}^{(1)}=C_{\ell}^{(1)}$, $i=r$ and $\{s_{\ell-1},s_{\ell}\}\subseteq I_r$. Thus, $\sigma_i$ is always a nontrivial automorphism, except when $X_{\ell}^{(1)}=C_{\ell}^{(1)}$, $i=r$ and $\{s_{\ell-1},s_{\ell}\}\subseteq I_r$.
\end{itemize}
\end{definition}

Here is the announced description of $\Xi_{\eta}$ for $\eta\in\partial V$ standard (for an arbitrary $\eta\in\partial V$, the description of $\Xi_{\eta}$ then follows from Remark~\ref{remark:Xietaabelian}).

\begin{theorem}\label{thm:mainthmXi}
Let $\eta\in\partial V$ be standard, and let $I_1,\dots,I_r$ be the components of $I:=I_{\eta}$, ordered so that $I_1^{\alpha}\leq\dots\leq I_r^{\alpha}$. Then $\Xi_\eta$ is given, in the notations of Definition~\ref{definition:sigmaj}, as follows, depending on the type $X_{\ell}^{(1)}$ of $\Gamma_S$.
\begin{enumerate}
\item[($A_{\ell}^{(1)}$)]
If $\Gamma_{S\setminus I}$ is of type $A_1^r$, then $\Xi_{\eta}=\langle \sigma_i\inv\sigma_j \ | \ 1\leq i,j\leq r\rangle$ has index $\mathrm{gcd}\{|I_j|+1 \ | \ 1\leq j\leq r\}$ in $\widetilde{W}^{\eta}/W^{\eta}$. Otherwise, $\Xi_{\eta}=\widetilde{W}^{\eta}/W^{\eta}$.
\item[($B_{\ell}^{(1)}$)]
If $\overline{S}\setminus\overline{I}\subseteq 2\NN$, then $\Xi_{\eta}=\langle \sigma_1^2,\sigma_1\sigma_j \ | \ 2\leq j\leq r\rangle$ has index $2$ in $\widetilde{W}^{\eta}/W^{\eta}$. Otherwise, $\Xi_{\eta}=\widetilde{W}^{\eta}/W^{\eta}$.
\item[($C_{\ell}^{(1)}$)]
If $s_{\ell}\in I$, then $\Xi_{\eta}=\langle \sigma_j \ | \ 1\leq j\leq r-1\rangle$ has index $2$ in $\widetilde{W}^{\eta}/W^{\eta}$. Otherwise, $\Xi_{\eta}=\widetilde{W}^{\eta}/W^{\eta}$.
\item[($D_{\ell}^{(1)}$)]
$\bullet$ If $\{s_{\ell-2},s_{\ell-1},s_{\ell}\}\subseteq I$:
\smallskip
\begin{itemize}
\item[(D1)] if $\overline{S}\setminus\overline{I}\subseteq 2\NN$, then $r\geq 2$ and $\Xi_{\eta}=\langle \sigma_1^2,\sigma_1\sigma_j \ | \ 2\leq j\leq r\rangle$ has index $4$ in $\widetilde{W}^{\eta}/W^{\eta}$.
\item[(D2)] otherwise, $\Xi_{\eta}=\langle \sigma_j \ | \ 1\leq j\leq r\rangle$ has index $2$ in $\widetilde{W}^{\eta}/W^{\eta}$.
\end{itemize}
\smallskip
$\bullet$ If $\{s_{\ell-1},s_{\ell}\}\subseteq I$ but $s_{\ell-2}\notin I$:
\smallskip
\begin{itemize}
\item[(D3)] if $\overline{S}\setminus\overline{I}\subseteq 2\NN$, then $r\geq 3$ and $\Xi_{\eta}=\langle \sigma_1^2,\sigma_1\sigma_j,\sigma_1\sigma_{r-1}\sigma_r \ | \ 2\leq j\leq r-2\rangle$ has index $4$ in $\widetilde{W}^{\eta}/W^{\eta}$.
\item[(D4)] otherwise, $r\geq 2$ and $\Xi_{\eta}=\langle \sigma_j,\sigma_{r-1}\sigma_r \ | \ 1\leq j\leq r-2\rangle$ has index $2$ in $\widetilde{W}^{\eta}/W^{\eta}$.
\end{itemize}
\smallskip
$\bullet$ If $\{s_{\ell-1},s_{\ell}\}\not\subseteq I$:
\smallskip
\begin{itemize}
\item[(D5)] if $\ell$ is even and $\overline{S}\setminus(\overline{I}\cup\{\ell-1,\ell\})\subseteq 2\NN$ and $\{s_{\ell-1},s_{\ell}\}\not\subseteq S\setminus I$, then $\Xi_{\eta}=\langle \sigma_1^2,\sigma_1\sigma_j \ | \ 2\leq j\leq r\rangle$ has index $2$ in $\widetilde{W}^{\eta}/W^{\eta}$.
\item[(D6)] otherwise, $\Xi_{\eta}=\widetilde{W}^{\eta}/W^{\eta}$.
\end{itemize}
\item[($E_6^{(1)}$)]
If $I=\{s_1,s_3,s_5,s_6\}\cup I'$ for some $I'\subseteq\{s_2\}$, then $\Xi_{\eta}=\langle\sigma_{s_2},\sigma_{s_3}\sigma_{s_5}\rangle$ has index $2$ in $\widetilde{W}^{\eta}/W^{\eta}$. Otherwise, $\Xi_{\eta}=\widetilde{W}^{\eta}/W^{\eta}$.
\item[($E_7^{(1)}$)]
$\Xi_{\eta}=\widetilde{W}^{\eta}/W^{\eta}$, unless in the following four cases, where $\Xi_{\eta}$ has index $2$ in $\widetilde{W}^{\eta}/W^{\eta}$:
\begin{itemize}
\item[(E1)]
$I=\{s_2,s_5,s_6,s_7\}\cup I'$ for some $I'\subseteq \{s_1,s_3\}$, in which case $\Xi_{\eta}=\langle\sigma_{s_1},\sigma_{s_3},\sigma_{s_2}\sigma_{s_5}\rangle$.
\item[(E2)]
$I=\{s_2,s_3,s_4,s_5,s_6,s_7\}$, in which case $\Xi_{\eta}=\langle \sigma_{s_3}\rangle$.
\item[(E3)]
$I=\{s_1,s_2,s_3,s_4,s_5,s_7\}$, in which case $\Xi_{\eta}=\langle \sigma_{s_2}\sigma_{s_7}\rangle$.
\item[(E4)]
$I=\{s_2,s_3,s_4,s_5,s_7\}$, in which case $\Xi_{\eta}=\langle \sigma_{s_3},\sigma_{s_2}\sigma_{s_7}\rangle$.
\end{itemize}
\item[($E_8^{(1)}$)]
$\Xi_{\eta}=\widetilde{W}^{\eta}/W^{\eta}$.
\item[($F_4^{(1)}$)]
$\Xi_{\eta}=\widetilde{W}^{\eta}/W^{\eta}$.
\item[($G_2^{(1)}$)]
$\Xi_{\eta}=\widetilde{W}^{\eta}/W^{\eta}$.
\end{enumerate}
\end{theorem}

\begin{remark}\label{remark:pietaWetagivenbyXieta}
Given $\eta\in\partial V$, Theorem~\ref{thm:mainthmXi} (and Remark~\ref{remark:Xietaabelian}) allows to compute $\pi_\eta(W_\eta)\subseteq\Aut(\Sigma^\eta)$, as 
\begin{equation}
\pi_\eta(W_\eta)=W^\eta\rtimes \Xi_{\eta}\subseteq\Aut(\Sigma^\eta).
\end{equation}
\end{remark}

The rest of \S\ref{subsection:TSXEAff} is devoted to the proof of Theorem~\ref{thm:mainthmXi}. {\bf Henceforth, and until the end of \S\ref{subsection:TSXEAff}}, we fix a standard $\eta\in\partial V$, and we let $I_1,\dots,I_r$ denote the components of $I:=I_{\eta}\subseteq S\setminus\{s_0\}$. If $\Gamma_S$ is of classical type, we moreover order $I_1,\dots,I_r$ so that $I_1^{\alpha}\leq\dots\leq I_r^{\alpha}$. Finally, note that we may assume $I$ to be nonempty, for otherwise $W^\eta$ is trivial and Theorem~\ref{thm:mainthmXi} holds trivially.

The first step in order to prove Theorem~\ref{thm:mainthmXi} will be to establish a criterion to check that a given element of $\widetilde{W}^{\eta}/W^{\eta}$ does not belong to $\Xi_{\eta}$: this will be achieved in Lemma~\ref{lemma:tobeornotinTYP} below.

\begin{remark}\label{remark:type}
Note that if $x\in V$ is a special vertex of $\Sigma$, then $x^{\eta}=\pi_{V^\eta}(x)\in V^\eta$ is a (special) vertex of $\Sigma^\eta$. Recall from \S\ref{subsubsection:EAWGASV} that we denoted by $$\typ_{\Sigma}(x)\in S\quad\textrm{and}\quad\typ_{\Sigma^\eta}(x^\eta)\in I_1^{\ext}\times\dots\times I_r^{\ext}$$ the cotype of $x$ in $\Sigma$ and of $x^\eta$ in $\Sigma^\eta$, respectively. Recall also from \S\ref{subsection:SFTROSSACG} that $m_j=m_{s_j}\in\WW$ ($j\in \overline{S}$) denotes the wall of $C_0$ not containing the vertex $x_j$ of $C_0$ with $\typ_{\Sigma}(x_j)=s_j$. Moreover, $\{m_s \ | \ s\in S^\eta\}=\{m_j \ | \ j\in \overline{I}\}\cup \{m_{\tau_i} \ | \ 1\leq i\leq r\}$ is the set of walls of $C_0^\eta$.

In particular, the vertex of $C_0^{\eta}$ of cotype $(t_1,\dots,t_r)\in I_1^{\ext}\times\dots\times I_r^{\ext}$ is the unique point $x\in V^{\eta}$ contained in the wall $m_{s}\in\WW^{\eta}$ for all $s\in S^{\eta}\setminus\{t_1,\dots,t_r\}$. For instance, $x_0=x_0^\eta\in V^{\eta}$ is the (special) vertex of $C_0$ of cotype $\typ_{\Sigma}(x_0)=s_0$ in $\Sigma$, and also the vertex of $C_0^{\eta}$ of cotype $\typ_{\Sigma^{\eta}}(x_0)=(\tau_1,\dots,\tau_r)$ in $\Sigma^\eta$, since $x_0\in m_i$ for all $i=1,\dots,\ell$. 
\end{remark}

Recall from \S\ref{subsubsection:OIFCG} the definition of the opposition map $\op_J\co J\to J$ for $J$ a spherical subset of $S$.

\begin{definition}\label{definition:TypSI}
We let\index[s]{TypSI@$\Typ(S,I)$ (smallest subset of $S$ satisfying (TYP0)--(TYP2))} $\Typ(S,I)$ denote the smallest subset $A$ of $S$ satisfying the following three properties:
\begin{enumerate}
\item[(TYP0)] $A\supseteq S\setminus I$.
\item[(TYP1)] 
If $a\in S\setminus(I\cup\{s_0\})$, then the sequence $(a_n)_{n\in\NN}\subseteq S$ defined recursively by $a_0:=s_0$, $a_1:=a$ and $a_{i+1}:=\op_{a_i}(a_{i-1})$ for all $i\geq 1$ is contained in $A$, where we set $\op_a(b):=\op_{S\setminus\{a\}}(b)$.
\item[(TYP2)] If $a,b\in A$ are special, then $\sigma_{a,b}(A)\subseteq A$, where $\sigma_{a,b}$ is the unique automorphism in $\widetilde{W}/W\subseteq\Aut(\Gamma_S)$ mapping $a$ to $b$ (see Lemma~\ref{lemma:extended_diagram_autom}).
\end{enumerate}
\end{definition}

\begin{lemma}\label{lemma:meaning_TypSI}
We have $\typ_{\Sigma}(y)\in \Typ(S,I)$ for any special vertex $y\in \pi_{V^\eta}\inv(x_0)$.
\end{lemma}
\begin{proof}
For an edge $e=\{e_0,e_1\}$ of $\Sigma$ (namely, $e_0,e_1$ are vertices of $\Sigma$ contained in a common $1$-dimensional closed cell $e$ of $\Sigma$), we let $L_e:=\sppan_V(e)$ denote its affine span in $V$, and we let $(e_i)_{i\in\ZZ}\subseteq L_e$ denote the sequence of vertices contained in $L_e$, ordered so that $e_i$ and $e_{i+1}$ belong to a same edge for each $i\in\ZZ$. We then write $\typ_{\Sigma}(L_e):=\{\typ_{\Sigma}(e_i) \ | \ i\in\ZZ\}\subseteq S$ for the set of types of vertices of $L_e$. Note that $$\typ_{\Sigma}(e_{i\pm 1})=\op_{\typ_{\Sigma}(e_i)}(\typ_{\Sigma}(e_{i\mp 1}))$$ for all $i\in\ZZ$ (see \S\ref{subsubsection:OIFCG}). In particular, the sequence $(\typ_{\Sigma}(e_i))_{i\in\ZZ}$ is periodic, and hence $\typ_{\Sigma}(L_e)=\{\typ_{\Sigma}(e_i) \ | \ i\in\NN\}$.

Let $y\in \pi_{V^\eta}\inv(x_0)$ be a special vertex, and let $t\in T$ be such that $tx_0=y$. Write $S\setminus (I\cup\{s_0\})=\{s_{i_1},\dots,s_{i_m}\}$, where $m\geq 1$ and $i_1,\dots,i_m\in\{1,\dots,\ell\}$. Consider the $m$ edges $e^k:=\{e^k_0:=x_0,e^k_1:=x_{i_k}\}$, for $k=1,\dots,m$.
By Lemma~\ref{lemma:Veta_forI}(2), $\pi_{V^\eta}\inv(x_0)$ is spanned by the lines $L_{e^k}$ for $k=1,\dots,m$. Moreover, by \cite[Proposition~1.24]{Wei09}, there exist elements $t_{i_1},\dots,t_{i_m}\in T$ with $t=t_{i_1}\dots t_{i_m}$ and such that $t_{i_k}$ stabilises $L_{e^k}$ for each $k$. 

Note that $\typ_{\Sigma}(L_{e^k})\subseteq\Typ(S,I)$ for each $k\in\{1,\dots,m\}$. Indeed, by (TYP0), $\typ_{\Sigma}(e^k_0)=s_0$ and $\typ_{\Sigma}(e^k_1)=s_{i_k}$ belong to $\Typ(S,I)$. On the other hand, if $i\geq 1$, then $$\typ_{\Sigma}(e^k_{i+1})=\op_{\typ_{\Sigma}(e^k_i)}(\typ_{\Sigma}(e^k_{i-1})),$$ so that the claim follows from the property (TYP1) of $\Typ(S,I)$.

Set $y_0:=x_0$ and, for each $k\in\{1,\dots,m\}$, define recursively $y_k:=t_{i_k}y_{k-1}$ and set $a_k:=\typ_{\Sigma}(y_k)$. Thus, $y_m=y$. We now show, by induction on $k$, that $a_k$ is special and belongs to $\Typ(S,I)$ for all $k$, yielding the lemma. For $k=0$, this holds by (TYP0). Assume now that $a_{k-1}$ is special and belongs to $\Typ(S,I)$ for some $k\in\{1,\dots,m\}$. Then $a_k$ is also special since $y_k=t_{i_k}y_{k-1}$. Moreover, since $t_{i_k}L_{e^k}=L_{e^k}$ and $x_0\in L_{e^k}$, the line $t_{i_1}\dots t_{i_{k}}L_{e^k}=t_{i_1}\dots t_{i_{k-1}}L_{e^k}=L_{t_{i_1}\dots t_{i_{k-1}}e^k}$ contains both $y_{k-1}$ and $y_k$. Hence $$a_k\in\typ_{\Sigma}(L_{t_{i_1}\dots t_{i_{k-1}}e^k})=\sigma_{a_0,a_{k-1}}(\typ_{\Sigma}(L_{e^k})),$$ where $\sigma_{a_0,a_{k-1}}$ is the unique automorphism in $\widetilde{W}/W\subseteq\Aut(\Gamma_S)$ mapping $a_0$ to $a_{k-1}$ (it is induced by $t_{i_1}\dots t_{i_{k-1}}\in\widetilde{W}$, which maps $y_0=x_0$ to $y_{k-1}$). Since $\typ_{\Sigma}(L_{e^k})\subseteq\Typ(S,I)$, this concludes the induction step by the property (TYP2) of $\Typ(S,I)$.
\end{proof}

The second statement of the following lemma will be our main tool to check that a given $\sigma\in\widetilde{W}^{\eta}/W^{\eta}$ does not belong to $\Xi_{\eta}$. 
\begin{lemma}\label{lemma:tobeornotinTYP}
Let $x,y\in V$ be special vertices. 
\begin{enumerate}
\item
If $\typ_{\Sigma}(x)=\typ_{\Sigma}(y)$, then $\typ_{\Sigma^{\eta}}(x^{\eta})$ and $\typ_{\Sigma^{\eta}}(y^\eta)$ are in the same $\Xi_\eta$-orbit.
\item
If $\typ_{\Sigma}(y)\notin\Typ(S,I)$, then $\typ_{\Sigma^{\eta}}(y^{\eta})$ and $\typ_{\Sigma^{\eta}}(x_0)$ are not in the same $\Xi_\eta$-orbit.
\end{enumerate}
\end{lemma}
\begin{proof}
(1) If $\typ_{\Sigma}(x)=\typ_{\Sigma}(y)$, then there exists $t\in T_0\subseteq W_{\eta}$ such that $tx=y$. Hence, $\pi_{\eta}(t)x^\eta=y^\eta$ by Lemma~\ref{lemma:Weta}(5), yielding (1).

(2) Assume that $\typ_{\Sigma^{\eta}}(y^{\eta})$ and $\typ_{\Sigma^{\eta}}(x_0)$ are in the same $\Xi_\eta$-orbit. Then $\pi_{\eta}(T_0)x_0$ contains a (special) vertex of $\Sigma^\eta$ of type $\typ_{\Sigma^\eta}(y^\eta)$ (note that $\Xi_{\eta}\cong \pi_{\eta}(W_\eta)/W^\eta\cong \pi_{\eta}(T_0)/(\pi_{\eta}(T_0)\cap W^\eta)$ by Lemma~\ref{lemma:Weta}(1)), and since $T_0\cap W^\eta=T\cap W^\eta$ is transitive on the set of special vertices of $\Sigma^{\eta}$ of a given type, we find some $t\in T_0$ such that $\pi_{\eta}(t)x_0=y^{\eta}$.  In particular, $x_0=\pi_{\eta}(t\inv)\pi_{V^\eta}(y)=\pi_{V^\eta}(t\inv y)$ by Lemma~\ref{lemma:Weta}(5), that is, $t\inv y\in  \pi_{V^\eta}\inv(x_0)$. Lemma~\ref{lemma:meaning_TypSI} then implies that $\typ_{\Sigma}(y)=\typ_{\Sigma}(t\inv y)\in \Typ(S,I)$, proving (2).
\end{proof}

The second step in order to prove Theorem~\ref{thm:mainthmXi} is to establish a criterion allowing to prove that a given element of $\widetilde{W}^{\eta}/W^{\eta}$ belongs to $\Xi_{\eta}$: this is achieved in Lemma~\ref{lemma:whatsinXi} below. We advise the reader to keep the list of Dynkin diagrams of affine type (Figure~\ref{figure:TableAFF}) at hand until the end of \S\ref{subsection:TSXEAff}.

\begin{definition}
For each $s\in S$, let\index[s]{Nbs@$\Nb(s)$ (neighbours of $s$ in $\Gamma_S$)} $\Nb(s):=\{a\in S\setminus\{s\} \ | \ sa\neq as\}$ denote the set of neighbours of $s$ in $\Gamma_S$. For $J\subseteq S$, we also set\index[s]{NbJ@$\Nb(J)$ (neighbours of $J$ in $\Gamma_S$)} $\Nb(J):=\bigcup_{s\in J}\Nb(s)$. 
\end{definition}

\begin{lemma}\label{lemma:inmi}
Let $a\in S$ be special, and let $b\in S$, which we assume to be non-special if $b\neq a$. Let $\gamma_{ab}:=(a=a_0,a_1,\dots,a_k=b)$ be the unique shortest path in $\Gamma_S$ from $a$ to $b$ (viewed as an ordered subset of $S$), and set $w_{ba}:=a_k\dots a_1a_0\in W$. 
Assume that all edges of $\gamma_{ab}\setminus\{a\}=(a_1,\dots,a_k)$ are simple edges. Assume, moreover, that $k\neq\ell-1$ in case $\Gamma_S$ is of type $C_{\ell}^{(1)}$.

Let $s\in S\setminus \{a,b\}$. Then $w_{ba}x_{\overline{a}}\notin m_s$ if and only if $s\in\Nb(\gamma_{ab})\setminus\gamma_{ab}$.
\end{lemma}
\begin{proof}
We first make the following observations:
\begin{equation}\label{eqn:ximj}
x_i\in m_j\quad\textrm{and}\quad s_jx_i=x_i\quad\textrm{for all $i,j\in\overline{S}$ with $i\neq j$},
\end{equation}
\begin{equation}\label{eqn:stms}
stm_s=m_t\quad\textrm{for all $s,t\in S$ such that $(s,t)$ is a simple edge,}
\end{equation}
\begin{equation}\label{eqn:tms}
tm_s=m_s\quad\textrm{for all $s,t\in S$ such that $(s,t)$ is not an edge,}
\end{equation}
\begin{equation}\label{eqn:backwards}
a_i\dots a_1a_0x_{\overline{a}}\in m_{a_j}\quad\textrm{for all $i,j\in\{1,\dots,k\}$ with $j<i$,}
\end{equation}
\begin{equation}\label{eqn:backwardszero}
a_i\dots a_1a_0x_{\overline{a}}\in m_{a}\quad\textrm{for all $i\in\{1,\dots,k\}$ when $(a,a_1)$ is a simple edge,}
\end{equation}
\begin{equation}\label{eqn:doubletripleedge}
a_i\dots a_1a_0x_{\overline{a}}\notin m_s\quad\textrm{for all $i\in\{0,\dots,k\}$ and $s\in\Nb(a_i)\setminus\{a,a_{i-1}\}$}
\end{equation}
(where $a_{i-1}$ is omitted if $i=0$). Indeed, (\ref{eqn:ximj}), (\ref{eqn:stms}) and (\ref{eqn:tms}) are clear, while (\ref{eqn:backwards}) and (\ref{eqn:backwardszero}) follow from (\ref{eqn:ximj}), as 
\begin{align*}
a_i\dots a_1a_0x_{\overline{a}}\in m_{a_j} &\iff a_{j+1}\dots a_1a_0x_{\overline{a}}\in m_{a_j}\quad\textrm{by (\ref{eqn:tms})}\\
&\iff a_{j-1}\dots a_1a_0x_{\overline{a}}\in m_{a_{j+1}}\quad\textrm{by (\ref{eqn:stms})}\\
&\iff x_{\overline{a}}\in m_{a_{j+1}}\quad\textrm{by (\ref{eqn:tms})}
\end{align*}
(with the convention $a_{j-1}\dots a_1a_0x_{\overline{a}}:=x_{\overline{a}}$ if $j=0$).
We first prove (\ref{eqn:doubletripleedge}) when $(a_i,s)$ is a double or triple edge. Note from Figure~\ref{figure:TableAFF} that there are in this case only 4 possibilities: 

\smallskip

(1) $\Gamma_S$ is of type $B_{\ell}^{(1)}$, $a_i=s_{\ell-1}$ and $s=s_{\ell}$. Up to exchanging $s_0$ and $s_1$, we may moreover assume that $a=s_0$. In that case, we claim that $a_i\dots a_1a_0x_{\overline{a}}\in m_{\tau}$, where $\tau$ is the reflection associated to the root $\delta-\alpha_{\ell}$ (so that $m_{\tau}\cap m_s=\varnothing$, yielding (\ref{eqn:doubletripleedge})). We have $$\delta-\alpha_{\ell}=\alpha_0+\alpha_1+2\sum_{j=2}^{\ell-1}\alpha_j+\alpha_{\ell}=s_{\ell-1}s_{\ell-2}\dots s_3s_2s_0s_1s_2\dots s_{\ell-1}\alpha_{\ell}$$ (see the last paragraph of \S\ref{subsubsection:RS}), so that $m_{\tau}=a_i\dots a_1a_0s_1s_2\dots s_{\ell-1}m_{\ell}$. Hence the claim follows from (\ref{eqn:ximj}).

\smallskip

(2) $\Gamma_S$ is of type $C_{\ell}^{(1)}$ and $i=0$ (note that $k<\ell-1$ by assumption). Assume that $a=s_0$ (the case $a=s_{\ell}$ being symmetric), so that $s=s_1$, and suppose for a contradiction that $s_0x_0\in m_1$. Then $s_0x_0\in\bigcap_{j\in \overline{S}\setminus\{0\}}m_j=\{x_0\}$ by (\ref{eqn:tms}) and (\ref{eqn:ximj}), a contradiction since $s_0x_0\neq x_0$.

\smallskip

(3) $\Gamma_S$ is of type $F_4^{(1)}$, $a=s_0$, $a_i=s_2$ and $s=s_3$. Suppose for a contradiction that $s_2s_1s_0x_0\in m_3$. As $s_2s_1s_0x_0\in m_4$ by (\ref{eqn:tms}), and $s_2s_1s_0x_0\in m_0\cap m_1$ by (\ref{eqn:backwards}) and (\ref{eqn:backwardszero}), we then have $s_2s_1s_0x_0\in \bigcap_{j\in \overline{S}\setminus\{2\}}m_j=\{x_2\}$, a contradiction since $s_2s_1s_0x_0\neq x_2$ (because $x_0$ is special but $x_2$ is not).

\smallskip

(4) $\Gamma_S$ is of type $G_2^{(1)}$, $a=s_0$, $a_i=s_2$ and $s=s_1$. Suppose for a contradiction that $s_2s_0x_0\in m_1$. Then $s_2s_0x_0\in m_1\cap m_0=\{x_2\}$ by (\ref{eqn:backwardszero}), a contradiction since $s_2s_0x_0\neq x_2$ (because $x_0$ is special but $x_2$ is not).

\smallskip

We now prove (\ref{eqn:doubletripleedge}) by induction on $i$ assuming that $(a_i,s)$ is a simple edge. If $i=0$, then $ax_{\overline{a}}=asx_{\overline{a}}\notin m_s\Leftrightarrow x_{\overline{a}}\notin m_a$ by (\ref{eqn:ximj}) and (\ref{eqn:stms}), as desired. Let now $i>0$. Note that $(a_j,s)$ is not an edge for all $j=0,\dots,i-1$ (i.e. since $b\neq a$ in this case, $b$ is non-special and hence $\Gamma_S$ is not of type $A_{\ell}^{(1)}$ and therefore contains no loop). Using  (\ref{eqn:ximj}) and (\ref{eqn:stms}), we then have $$a_i\dots a_1a_0x_{\overline{a}}=a_i\dots a_1a_0sx_{\overline{a}}=a_isa_{i-1}\dots a_1a_0x_{\overline{a}}\notin m_s \iff a_{i-1}\dots a_1a_0x_{\overline{a}}\notin m_{a_i}$$
which holds by induction hypothesis. This completes the proof of (\ref{eqn:doubletripleedge}).

\smallskip

We can now prove the lemma. Let $s\in S\setminus\{a,b\}$. If $s\in\gamma_{ab}$, then $w_{ba}x_{\overline{a}}\in m_s$ by (\ref{eqn:backwards}). If $s\notin\Nb(\gamma_{ab})$, then $w_{ba}x_{\overline{a}}\in m_s$ by (\ref{eqn:tms}). Finally, assume that $s\in\Nb(\gamma_{ab})\setminus\gamma_{ab}$, and let $i\in\{0,\dots,k\}$ be such that $(a_i,s)$ is an edge (note that $i$ is unique). Then $w_{ba}x_{\overline{a}}\notin m_s\Leftrightarrow a_i\dots a_1a_0x_{\overline{a}}\notin m_s$ by (\ref{eqn:tms}), which holds by (\ref{eqn:doubletripleedge}), as desired.
\end{proof}

\begin{lemma}\label{lemma:inC0etaprelim}
Let $i_1,\dots,i_d\in\overline{S}$ be pairwise distinct and such that $(s_{i_1},\dots,s_{i_d})$ is a path in $\Gamma_S$. Set $w:=s_{i_1}\dots s_{i_d}\in W$. Assume that the following conditions hold:
\begin{enumerate}
\item
$s_{i_1}\notin I$;
\item
$s_{i_1}\dots s_{i_{k-1}}\alpha_{i_k}+\theta_{I_j}\neq\delta$ for all $k=1,\dots,d$ and $j=1,\dots,r$.
\end{enumerate}
Then $\pi_{V^\eta}(wx)\in C_0^\eta$ for all $x\in C_0^\eta$.
\end{lemma}
\begin{proof}
Set $w_k:=s_{i_1}\dots s_{i_k}$ for each $k\in\{0,\dots,d\}$ (with $w_0:=1$).
Then the set of walls separating $C_0$ from $wC_0$ (equivalenty, the set of walls crossed by the gallery from $C_0$ to $wC_0$ of type $(s_{i_1},\dots, s_{i_d})$) is given by $\{H_{w_{k-1}\alpha_{i_k}} \ | \ 1\leq k\leq d\}$ (see e.g. \cite[\S 2.1]{BrownAbr}). To prove the lemma, it is sufficient to show that none of these walls is a wall of $C_0^\eta$. In other words, we have to show that $$H_{w_{k-1}\alpha_{i_k}}\notin \{H_{\alpha_j} \ | \ j\in \overline{I}\}\cup \{H_{\delta-\theta_{I_j}} \ | \ 1\leq j\leq r\}\quad\textrm{for all $k\in\{1,\dots,d\}$},$$ or equivalently, that $w_{k-1}\alpha_{i_k}\notin \{\alpha_j \ | \ j\in \overline{I}\}\cup \{\delta-\theta_{I_j} \ | \ 1\leq j\leq r\}$ for all $k\in\{1,\dots,d\}$ (as all the above roots are positive). As $\supp(w_{k-1}\alpha_{i_k})=\{\alpha_{i_1},\dots,\alpha_{i_k}\}\supseteq \{\alpha_{i_1}\}$, this follows from the assumptions (1) and (2).
\end{proof}

\begin{lemma}\label{lemma:heighthighestroot}
The height $h(X)$ of the highest root associated to the Dynkin diagram of finite type $X$ is given as follows: $h(A_n)=n$, $h(B_n)=h(C_n)=2n-1$, $h(D_n)=2n-3$, $h(E_6)=11$, $h(E_7)=17$, $h(E_8)=29$, $h(F_4)=11$, $h(G_2)=5$.
Moreover, $\height(\delta)=h(X_{\ell})+1$.
\end{lemma}
\begin{proof}
As mentioned in \S\ref{subsubsection:RS}, $h(X)$ can be obtained by adding the labels of the vertices of the Dynkin diagram of type $X$ pictured on Figure~\ref{figure:TableFIN}, yielding the first assertion. Since $\Gamma_S$ is of type $X_{\ell}^{(1)}$ and $\delta=\alpha_0+\theta_{S_{\Pi}}$, the second assertion holds as well. 
\end{proof}

\begin{lemma}\label{lemma:inC0eta}
Let $a\in S$ be special, and let $b\in S\setminus I$, which we assume to be non-special if $b\neq a$. Let $\gamma_{ab}:=(a=a_0,a_1,\dots,a_k=b)$ be the unique shortest path in $\Gamma_S$ from $a$ to $b$ (viewed as an ordered subset of $S$), and set $w_{ba}:=a_k\dots a_1a_0\in W$. 
If $b\neq a$, assume, moreover, that we are not in one of the following four cases: 
\begin{enumerate}
\item[(i)]
$\Gamma_S$ is of type $B_{\ell}^{(1)}$ and $b=s_{\ell}$. 
\item[(ii)]
$\Gamma_S$ is of type $C_{\ell}^{(1)}$ and $S\setminus\gamma_{ab}\subseteq I$.
\item[(iii)]
$\Gamma_S$ is of type $F_4^{(1)}$ and $b=s_4$. 
\item[(iv)]
$\Gamma_S$ is of type $G_2^{(1)}$ and $b=s_1$.
\end{enumerate}
Then $\pi_{V^\eta}(w_{ba}x_{\overline{a}})\in C_0^\eta$.
\end{lemma}
\begin{proof}
If $\gamma_{ab}\cup I_j\subsetneq S$ for all $j=1,\dots,r$ (for instance, if $a=b$, as $|I|\leq\ell-1$ by Lemma~\ref{lemma:Veta_forI}(3)), Lemma~\ref{lemma:inC0etaprelim} implies that $\pi_{V^\eta}(w_{ba}x_{\overline{a}})\in C_0^\eta$, as in that case $\supp(a_k\dots a_1\alpha_{\overline{a}}+\theta_{I_j})\subsetneq S=\supp(\delta)$ for all $j=1,\dots,r$.

We may thus assume that $\gamma_{ab}\cup I_{j}=\gamma_{ab}\cup I=S$ for some $j\in\{1,\dots,r\}$ (in particular, $b\neq a$), that $b$ is not special (in particular, $\Gamma_S$ is not of type $A_{\ell}^{(1)}$), and that we are not in one of the cases (i)--(iv). In particular, since $s_0\in\gamma_{ab}$ and $b$ is not special, we must have $a=s_0$. 

To show that $\pi_{V^\eta}(w_{ba}x_{0})\in C_0^\eta$, it is sufficient to check by Lemma~\ref{lemma:inC0etaprelim} that 
\begin{equation}\label{eqn:smallerthandelta}
\height(a_k\dots a_2a_1\alpha_{0}+\theta_{I_j})<\height(\delta)\quad\textrm{for all $j\in\{1,\dots,r\}$.}
\end{equation}

We now prove (\ref{eqn:smallerthandelta}) using Lemma~\ref{lemma:heighthighestroot}. Set for short $n_b:=\height(a_k\dots a_2a_1\alpha_{0})$.

\smallskip

(1) If $\Gamma_S$ is of type $B_{\ell}^{(1)}$, then $b=s_i$ for some $i\in\{2,\dots,\ell-1\}$ by the assumption (i). But then $\gamma_{ab}\cup I_{j}$ cannot contain both $s_1$ and $s_{\ell}$ ($j=1,\dots,r$), a contradiction.

\smallskip

(2) If $\Gamma_S$ is of type $C_{\ell}^{(1)}$, then $S\setminus\gamma_{ab}\not\subseteq I$ by the assumption (ii), and hence $\gamma_{ab}\cup I\neq S$, a contradiction.

\smallskip

(3) If $\Gamma_S$ is of type $D_{\ell}^{(1)}$, then $b=s_i$ for some $i\in\{2,\dots,\ell-2\}$. Hence $\gamma_{ab}\cup I_{j}$ cannot contain both $s_1$ and $s_{\ell}$ ($j=1,\dots,r$), a contradiction.

\smallskip

For the next cases, recall that $I\subseteq S\setminus\{s_0,b\}$ and that $b$ is not special.

\smallskip

(4) If $\Gamma_S$ is of type $E_6^{(1)}$, then $n_b\leq 4$ and $\max_j\height(\theta_{I_j})\leq h(A_5)=5$, so that $n_b+\max_j\height(\theta_{I_j})\leq 9<12=\height(\delta)$.

\smallskip

(5) If $\Gamma_S$ is of type $E_7^{(1)}$, then $n_b\leq 6$ and $\max_j\height(\theta_{I_j})\leq \max\{h(A_6),h(D_6)\}=9$, so that $n_b+\max_j\height(\theta_{I_j})\leq 15<18=\height(\delta)$.

\smallskip

(6) If $\Gamma_S$ is of type $E_8^{(1)}$, then $n_b\leq 8$ and $\max_j\height(\theta_{I_j})\leq \max\{h(A_7),h(D_7),h(E_7)\}=17$, so that $n_b+\max_j\height(\theta_{I_j})\leq 25<30=\height(\delta)$.

\smallskip

(7) If $\Gamma_S$ is of type $F_4^{(1)}$, then $b=s_i$ for some $i\in\{1,2,3\}$ by the assumption (iii). But then $n_b\leq 5$ and $\max_j\height(\theta_{I_j})\leq \max\{h(A_2),h(B_3),h(C_3)\}=5$, so that $n_b+\max_j\height(\theta_{I_j})\leq 10<12=\height(\delta)$. 

\smallskip

(8) If $\Gamma_S$ is of type $G_2^{(1)}$, then $b=s_2$ by the assumption (iv). Hence $n_b=2$ and $\max_j\height(\theta_{I_j})\leq h(A_1)=1$, so that $n_b+\max_j\height(\theta_{I_j})\leq 3<6=\height(\delta)$. 
\end{proof}

Recall from Definition~\ref{definition:sigmaj} the definition of the automorphisms $\sigma_s$.

\begin{lemma}\label{lemma:whatsinXi}
Let $a\in S$ be special, and let $b\in S\setminus I$, which we assume to be non-special if $b\neq a$. Let $\gamma_{ab}:=(a=a_0,a_1,\dots,a_k=b)$ be the unique shortest path in $\Gamma_S$ from $a$ to $b$ (viewed as an ordered subset of $S$). If $b\neq a$, assume, moreover, that we are not in one of the four cases (i)--(iv) from Lemma~\ref{lemma:inC0eta}, that $k\neq\ell-1$ in case $\Gamma_S$ is of type $C_{\ell}^{(1)}$, and that $b\neq s_3$ in case $\Gamma_S$ is of type $F_4^{(1)}$.
\begin{enumerate}
\item
The automorphism $\sigma_s$ is nontrivial for all $s\in I\cap \Nb(\gamma_{ab})\setminus\gamma_{ab}$. 
\item
If $a\notin I$, then $\sigma(a,b):=\prod_{s\in I\cap \Nb(\gamma_{ab})\setminus\gamma_{ab}}\sigma_{s}\in\Xi_\eta$.
\item
If $a\in I$, then $\sigma(a,b):=(\prod_{s\in I\cap \Nb(\gamma_{ab})\setminus\gamma_{ab}}\sigma_{s})\cdot \sigma_a\inv\in\Xi_\eta$.
\end{enumerate}
\end{lemma}
\begin{proof}
Recall from Lemma~\ref{lemma:tobeornotinTYP}(1) that since $x_{\overline{a}}$ and $w_{ba}x_{\overline{a}}$ are two special vertices of $\Sigma$ of the same type (where $w_{ba}:=a_k\dots a_1a_0\in W$), there exists an automorphism $\sigma\in \Xi_{\eta}$ mapping $\typ_{\Sigma^\eta}(x_{\overline{a}}^\eta)$ to $\typ_{\Sigma^\eta}(\pi_{V^\eta}(w_{ba}x_{\overline{a}}))$. 

To compute $\typ_{\Sigma^\eta}(\pi_{V^\eta}(w_{ba}x_{\overline{a}}))$, note that $\pi_{V^\eta}(w_{ba}x_{\overline{a}})$ is a vertex of the fundamental chamber $C_0^{\eta}$ of $\Sigma^{\eta}$ by Lemma~\ref{lemma:inC0eta}. Hence $\pi_{V^\eta}(w_{ba}x_{\overline{a}})$ is of type $(t_1,\dots,t_r)\in I_1^{\ext}\times\dots\times I_r^{\ext}$ if and only if $\pi_{V^\eta}(w_{ba}x_{\overline{a}})$ (or equivalently, $w_{ba}x_{\overline{a}}$) belongs to the wall $m_s$ for all $s\in S^\eta\setminus\{t_1,\dots,t_r\}$ (see Remark~\ref{remark:type}). In other words, if $(t_1,\dots,t_r)$ is the type of $\pi_{V^\eta}(w_{ba}x_{\overline{a}})$ in $\Sigma^{\eta}$, then for each $j\in\{1,\dots,r\}$, one of the following holds:
\begin{itemize}
\item
$w_{ba}x_{\overline{a}}\in m_s$ for all $s\in I_j$: in that case, $t_j=\tau_j$;
\item
$w_{ba}x_{\overline{a}}\notin m_s$ for some $s\in I_j$: in that case, $t_j=s$.
\end{itemize}

On the other hand, note that all edges of $\gamma_{ab}\setminus\{a\}=(a_1,\dots,a_k)$ are simple edges (if $b=a$, this is clear, and if $b\neq a$, this follows from the fact that $b$ is not special, as well as the assumption that we are not in cases (i), (iii) and (iv) from Lemma~\ref{lemma:inC0eta}, and that $b\neq s_3$ in case $\Gamma_S$ is of type $F_4^{(1)}$). Hence Lemma~\ref{lemma:inmi} implies that if $s\in S\setminus\{a,b\}$, then $w_{ba}x_{\overline{a}}\notin m_s$ if and only if $s\in\Nb(\gamma_{ab})\setminus\gamma_{ab}$. Therefore, if $j\in\{1,\dots,r\}$, then $t_j=\tau_j$, unless if $I_j$ contains some $s\in  I\cap \Nb(\gamma_{ab})\setminus\gamma_{ab}$ (which is then the only such element of $I_j$), in which case $t_j=s$. 

If $a\notin I$ (so that $x_{\overline{a}}\in\pi_{V^\eta}\inv(x_0)$ by Lemma~\ref{lemma:Veta_forI}(2)), then $x_{\overline{a}}^\eta=x_0$ has type $(\tau_1,\dots,\tau_r)$ in $\Sigma^{\eta}$, and hence $\sigma$ is the unique automorphism in $\widetilde{W}^\eta/W^\eta$ mapping $(\tau_1,\dots,\tau_r)$ to $(t_1,\dots,t_r)$. In particular, $\sigma_s$ is nontrivial for each $s\in I\cap \Nb(\gamma_{ab})\setminus\gamma_{ab}$ and $\sigma=\prod_{s\in I\cap \Nb(\gamma_{ab})\setminus\gamma_{ab}}\sigma_{s}$, proving (2) and (1) in that case.

Similarly, if $a\in I$ (say $a\in I_1$ up to re-indexing the $I_j$), then $x_{\overline{a}}^\eta$ has type $(a,\tau_2,\dots,\tau_r)=\sigma_a((\tau_1,\dots,\tau_r))$, and hence $\sigma\sigma_a$ is the unique automorphism in $\widetilde{W}^\eta/W^\eta$ mapping $(\tau_1,\dots,\tau_r)$ to $(t_1,\dots,t_r)$. In particular, $\sigma_s$ is nontrivial for each $s\in I\cap \Nb(\gamma_{ab})\setminus\gamma_{ab}$ and $\sigma=(\prod_{s\in I\cap \Nb(\gamma_{ab})\setminus\gamma_{ab}}\sigma_{s})\cdot \sigma_a\inv$, proving (3) and (1).
\end{proof}

We are now ready to prove Theorem~\ref{thm:mainthmXi} case by case. Recall from Definition~\ref{definition:sigmaj} the definition of the automorphisms $\sigma_s$ ($s\in I$) and $\sigma_i$ ($i\in\{1,\dots,r\}$).

\begin{lemma}\label{lemma:lastobs}
Assume that $\Gamma_S$ is of classical type, and let $i\in\{1,\dots,r\}$. Then $\widetilde{W}^\eta_{I_i^{\ext}}/W^\eta_{I_i^{\ext}}=\langle\sigma_i\rangle$, unless $i=r$ and one of the following holds:
\begin{enumerate}
\item
$\Gamma_S$ is of type $C_{\ell}^{(1)}$ and $\{s_{\ell-1},s_{\ell}\}\subseteq I_r$. In that case, $\sigma_i=\id$ and $\Gamma_{I_i^{\ext}}$ is of type $C^{(1)}_{|I_i|}$, so that $\langle\sigma_i\rangle$ has index $2$ in $\widetilde{W}^\eta_{I_i^{\ext}}/W^\eta_{I_i^{\ext}}$.
\item
$\Gamma_S$ is of type $D_{\ell}^{(1)}$ and $\{s_{\ell-3},s_{\ell-2},s_{\ell-1},s_{\ell}\}\subseteq I_r$. In that case, $\Gamma_{I_i^{\ext}}$ is of type $D^{(1)}_{|I_i|}$ and $\langle\sigma_i\rangle$ has index $2$ in $\widetilde{W}^\eta_{I_i^{\ext}}/W^\eta_{I_i^{\ext}}$.
\item
$\Gamma_S$ is of type $D_{\ell}^{(1)}$ and $I_r=\{s_{\ell-2},s_{\ell-1},s_{\ell}\}$. In that case, $\Gamma_{I_i^{\ext}}$ is of type $A^{(1)}_{3}$ and $\langle\sigma_i\rangle$ has index $2$ in $\widetilde{W}^\eta_{I_i^{\ext}}/W^\eta_{I_i^{\ext}}$.
\end{enumerate}
\end{lemma}
\begin{proof}
Note that if $i\neq r$, then $I_i$ is of the form $\{s_m,s_{m+1},\dots,s_{m+n}\}$ (with $n=|I_i|-1$), $\Gamma_{I_i^{\ext}}$ is of type $A_{|I_i|}^{(1)}$, and $\sigma_i$ maps $\tau_i$ to $s_m$. In particular, $\sigma_i$ has order $|I_i|+1$ and generates $\widetilde{W}^\eta_{I_i^{\ext}}/W^\eta_{I_i^{\ext}}$ in this case (see Lemma~\ref{lemma:extended_diagram_autom}(1)). On the other hand, if $i=r$, the same situation occurs, except if $\Gamma_S$ is of type $B_{\ell}^{(1)}$ and $\{s_{\ell-1},s_{\ell}\}\subseteq I_r$ (but in that case, $\Gamma_{I_i^{\ext}}$ is of type $C_2^{(1)}$ if $|I_i|=2$ and $B_{|I_i|}^{(1)}$ otherwise, and $\sigma_i$ is the only nontrivial element of $\widetilde{W}^\eta_{I_i^{\ext}}/W^\eta_{I_i^{\ext}}$), or if one of the cases (1)--(3) described in the statement of the lemma occurs, as can be seen from a quick inspection of Figure~\ref{figure:TableAFF}. The additional statements in (1)--(3) follow from Lemma~\ref{lemma:extended_diagram_autom}.
\end{proof}

\begin{prop}\label{prop:main_ABCD}
Theorem~\ref{thm:mainthmXi} holds for $\Gamma_S$ of type $X\in\{A_{\ell}^{(1)},B_{\ell}^{(1)},C_{\ell}^{(1)},D_{\ell}^{(1)}\}$.
\end{prop}
\begin{proof}
We will use Lemma~\ref{lemma:whatsinXi} repeatedly to conclude, in the notations of that lemma, that $\sigma(a,b)\in\Xi_{\eta}$ for various pairs $(a,b)$ (checking that a given pair $(a,b)$ indeed satisfies the hypotheses of Lemma~\ref{lemma:whatsinXi} is straightforward). We will also use the observations in Lemma~\ref{lemma:lastobs} without further mention.

\smallskip

(1) Assume first that $X=A_{\ell}^{(1)}$. Then for each $j\in\{1,\dots,r\}$, the set $I_j$ is of the form $I_j=\{s_{m},s_{m+1},\dots,s_{m+n}\}$ for some $m\in\{1,\dots,\ell\}$ (where $n=|I_j|-1$), and we set $a_j:=s_{m-1}\in S\setminus I$ and $z_j:=s_{m+n+1}\in S\setminus I$ (with the convention $s_{\ell+1}:=s_0$).

Then $\sigma(a_j,a_j)\in\Xi_{\eta}$ for all $j\in\{1,\dots,r\}$. If $a_{j}=z_{j-1}$ (with the convention $z_0:=z_r$), then $\sigma(a_j,a_j)=\sigma_{j-1}\inv\sigma_j$ (with the convention $\sigma_0:=\sigma_r$); otherwise, $\sigma(a_j,a_j)=\sigma_j$. In particular, if $a_{j_0}\neq z_{j_0-1}$ for some $j_0\in \{1,\dots,r\}$ (that is, if $\Gamma_{S\setminus I}$ is not of type $A_1^r$), then $\sigma_j\in\Xi_{\eta}$ for all $j$ and hence $\Xi_{\eta}=\widetilde{W}^\eta/W^\eta$. Assume now that $\Gamma_{S\setminus I}$ is of type $A_1^r$, so that $$\Xi_{\eta}\supseteq \langle\sigma_{j-1}\inv\sigma_j \ | \ 1\leq j\leq r\rangle=\langle \sigma_i\inv\sigma_j \ | \ 1\leq i,j\leq r\rangle=:\widetilde{\Xi}_{\eta}.$$
Set $d_j:=|I_j|+1$ for each $j$, and $d:=\mathrm{gcd}\{d_j \ | \ 1\leq j\leq r\}$. As $\sigma_j$ has order $d_j$ and $$\widetilde{\Xi}_{\eta}=\Big\{\prod_{j=1}^r\sigma_j^{e_j} \ \Big| \ e_j\in\ZZ, \ \sum_{j=1}^re_j=0\Big\},$$ the group $\widetilde{\Xi}_{\eta}$ has index $d$ in $\widetilde{W}^\eta/W^\eta=\langle \sigma_j \ | \ 1\leq j\leq r\rangle$, with set of left coset representatives $\{\sigma_1^{e} \ | \ 0\leq e< d\}$. It thus remains to check that $\Xi_{\eta}\subseteq\widetilde{\Xi}_{\eta}$, or else that $\sigma_1^{e}\notin \Xi_{\eta}$ for each $e\in\{1,\dots,d-1\}$. 

Fix $e\in\{1,\dots,d-1\}$ (assuming that $d\geq 2$). Since $s_0\notin I$ and $\Gamma_{S\setminus I}$ is of type $A_1^r$, we have $s_1\in I$. Hence $\{s_1,s_2,\dots,s_e\}\subseteq I_1$ because $|I_1|\geq d-1$. As $\sigma_1^e$ is the unique automorphism of $\widetilde{W}^\eta/W^\eta$ mapping $(\tau_1,\dots,\tau_r)=\typ_{\Sigma^{\eta}}(x_0)$ to $(s_e,\tau_2,\dots,\tau_r)=\typ_{\Sigma^\eta}(x_e^\eta)$, it is sufficient to check by Lemma~\ref{lemma:tobeornotinTYP}(2) that $s_e=\typ_{\Sigma}(x_e)\notin \Typ(S,I)$. We claim that the set $A:=\{s_i \ | \ i\in\overline{S}\cap d\NN\}$ satisfies the three properties (TYP0), (TYP1) and (TYP2) from Definition~\ref{definition:TypSI}, and hence contains $\Typ(S,I)$. As $s_e\notin A$, this will imply that $s_e\notin\Typ(S,I)$, as desired.

Since $\Gamma_{S\setminus I}$ is of type $A_1^r$, we have $S\setminus I=\{z_j \ | \ 1\leq j\leq r\}$. Since, moreover, $|I_j|+1$ is a multiple of $d$ for each $j$ (and since $z_r=s_0\in S\setminus I$), we deduce that $\overline{z_j}\in d\NN$ for each $j$, and that $A$ satisfies (TYP0). The fact that $A$ satisfies (TYP1) and (TYP2) is now also clear, since for each $a,b\in A$ we have $\op_a(b)\in A$ (see Lemma~\ref{lemma:oppositionfinite}) and $\sigma_{a,b}(A)\subseteq A$ (in the notations of Definition~\ref{definition:TypSI}).

\smallskip

(2) Assume next that $X=B_{\ell}^{(1)}$. Then for each $j\in\{1,\dots,r\}$, the set $I_j$ is of the form $I_j=\{s_{m},s_{m+1},\dots,s_{m+n}\}$ for some $m\in\{1,\dots,\ell\}$ (where $n=|I_j|-1$), and we set $a_j:=s_{m-1}\in S\setminus I$.

\smallskip

(2a) If $s_1\notin I$, then $\sigma(s_1,a_j)=\sigma_j\in\Xi_{\eta}$ for all $j\in\{1,\dots,r\}$, so that $\Xi_{\eta}=\widetilde{W}^\eta/W^\eta$ in that case. 

\smallskip

Assume now that $s_1\in I$. 

\smallskip

(2b) If $s_2\in I$, then $\sigma(s_0,s_0)=\sigma_1^2\in\Xi_{\eta}$ and $\sigma(s_0,a_j)=\sigma_1\sigma_j\in\Xi_{\eta}$ for all $j\in\{2,\dots,r\}$. If $s_2\notin I$ but $s_3\in I$ (so that $a_2=s_2$), then $\sigma_1$ has order $2$ and $\sigma(s_0,a_j)=\sigma_1\sigma_j\in\Xi_{\eta}$ for all $j\in\{2,\dots,r\}$. And if $s_2,s_3\notin I$, then $\sigma(s_0,s_2)=\sigma_1\in\Xi_{\eta}$ and $\sigma(s_0,a_j)=\sigma_1\sigma_j\in\Xi_{\eta}$ for all $j\in\{2,\dots,r\}$. Thus, in all cases, the group
$$\widetilde{\Xi}_{\eta}:=\langle \sigma_1^2,\sigma_1\sigma_j \ | \ 2\leq j\leq r\rangle$$
is contained in $\Xi_{\eta}$. 

\smallskip

(2c) If there is some $i\in\{2,\dots,\ell-1\}$ such that $s_i,s_{i+1}\in S\setminus I$, then choosing a minimal such $i$ we have $\sigma(s_0,s_i)=\sigma_1\in\Xi_{\eta}$, so that $\Xi_{\eta}=\widetilde{W}^\eta/W^\eta$ in that case. 

\smallskip

Assume now that $\Gamma_{S\setminus (I\cup\{s_0\})}$ has no edge. 

\smallskip

(2d) If $\overline{S}\setminus\overline{I}\not\subseteq 2\NN$, then the order $d_j:=|I_j|+1$ of $\sigma_j$ is odd for some $j\in\{1,\dots,r\}$, and hence $\sigma_1=(\sigma_1^2)^{\tfrac{d_j+1}{2}}\cdot (\sigma_1\sigma_j)^{-d_j}\in\Xi_{\eta}$. Thus, in that case, $\Xi_{\eta}=\widetilde{W}^\eta/W^\eta$.

\smallskip

(2e) Finally, assume that $\overline{S}\setminus\overline{I}\subseteq 2\NN$, and let us show that $\Xi_{\eta}\subseteq \widetilde{\Xi}_{\eta}$. Since $\widetilde{\Xi}_{\eta}$ has index $2$ in $\widetilde{W}^\eta/W^\eta=\langle \sigma_j \ | \ 1\leq j\leq r\rangle$, with set of left coset representatives $\{\sigma_1^e \ | \ 0\leq e\leq 1\}$, it is sufficient to check that $\sigma_1\notin\Xi_{\eta}$. 

As $\sigma_1$ is the unique automorphism of $\widetilde{W}^\eta/W^\eta$ mapping $(\tau_1,\dots,\tau_r)=\typ_{\Sigma^{\eta}}(x_0)$ to $(s_1,\tau_2,\dots,\tau_r)=\typ_{\Sigma^\eta}(x_1^\eta)$, it is sufficient to check by Lemma~\ref{lemma:tobeornotinTYP} that $s_1=\typ_{\Sigma}(x_1)\notin \Typ(S,I)$. We claim that the set $A:=\{s_i \ | \ i\in\overline{S}\cap 2\NN\}$ satisfies the three properties (TYP0), (TYP1) and (TYP2) from Definition~\ref{definition:TypSI}, and hence contains $\Typ(S,I)$. As $s_1\notin A$, this will imply that $s_1\notin\Typ(S,I)$, as desired.

By assumption, $A$ satisfies (TYP0). To see that $A$ satisfies (TYP1), let $a,b\in A$ be distinct, and let us show that $\op_a(b)\in A$. If $a=s_0$, then $\op_a(b)=b\in A$, while $\op_b(a)=a\in A$ because the connected component of $\Gamma_{S\setminus \{b\}}$ containing $a$ is of type $D_n$ with $n$ even (or of type $A_1$ if $b=s_2$) --- see Lemma~\ref{lemma:oppositionfinite}. Similarly, if $a,b\neq s_0$, then $\op_a(b)=b\in A$ (see Lemma~\ref{lemma:oppositionfinite}). Finally, $A$ clearly satisfies (TYP2) since $s_0$ is the only special vertex in $A$.

\smallskip

(3) Assume next that $X=C_{\ell}^{(1)}$. Then for each $j\in\{1,\dots,r\}$, the set $I_j$ is of the form $I_j=\{s_{m},s_{m+1},\dots,s_{m+n}\}$ for some $m\in\{1,\dots,\ell\}$ (where $n=|I_j|-1$), and we set $a_j:=s_{m-1}\in S\setminus I$.

\smallskip

(3a) If $s_{\ell}\notin I$ (so that $\sigma_r\neq\id$ and $\overline{a_r}<\ell-1$), then $\sigma(s_0,a_j)=\sigma_j\in\Xi_{\eta}$ for all $j\in\{1,\dots,r\}$, so that $\Xi_{\eta}=\widetilde{W}^\eta/W^\eta$ in that case. 

\smallskip

Assume now that $s_{\ell}\in I$. 

\smallskip

(3b) Then $\sigma(s_0,a_j)=\sigma_j\in\Xi_{\eta}$ for all $j\in\{1,\dots,r-1\}$, so that $\Xi_{\eta}$ contains $$\widetilde{\Xi}_{\eta}:=\langle \sigma_j \ | \ 1\leq j\leq r-1\rangle.$$
The group $\widetilde{W}^\eta/W^\eta$ coincides with $\widetilde{\Xi}_{\eta}\times\langle\sigma_{s_\ell}\rangle$ (recall that $\sigma_{s_\ell}$ is the unique automorphism of $\Gamma_{I_r^{\ext}}$ mapping $\tau_r$ to $s_{\ell}$), and hence contains $\widetilde{\Xi}_{\eta}$ as a subgroup of index $2$, with set of left coset representatives $\{\sigma_{s_\ell}^{e} \ | \ 0\leq e\leq 1\}$. To show that $\Xi_{\eta}\subseteq \widetilde{\Xi}_{\eta}$, we thus have to show that $\sigma_{s_\ell}\notin\Xi_{\eta}$. 

As $\sigma_{s_\ell}$ is the unique automorphism of $\widetilde{W}^\eta/W^\eta$ mapping $(\tau_1,\dots,\tau_r)=\typ_{\Sigma^{\eta}}(x_0)$ to $(\tau_1,\dots,\tau_{r-1},s_{\ell})=\typ_{\Sigma^\eta}(x_{\ell}^\eta)$, it is sufficient to check by Lemma~\ref{lemma:tobeornotinTYP} that $s_{\ell}=\typ_{\Sigma}(x_{\ell})\notin \Typ(S,I)$. But the set $A:=S\setminus\{s_{\ell}\}$ clearly satisfies the three properties (TYP0), (TYP1) and (TYP2) from Definition~\ref{definition:TypSI}, and hence contains $\Typ(S,I)$. As $s_{\ell}\notin A$, this implies that $s_{\ell}\notin\Typ(S,I)$, as desired.

\smallskip

(4) Assume finally that $X=D_{\ell}^{(1)}$. Then for each $j\in\{1,\dots,r-1\}$, the set $I_j$ is of the form $I_j=\{s_{m},s_{m+1},\dots,s_{m+n}\}$ for some $m\in\{1,\dots,\ell-1\}$ (where $n=|I_j|-1$), and we set $a_j:=s_{m-1}\in S\setminus I$. For $j=r$, we set $a_j:=s_{\ell-2}\in S\setminus I$ if $I_r=\{s_{\ell}\}$, and $a_j:=s_{m-1}\in S\setminus I$ with $m\in\{1,\dots,\ell\}$ the minimal index such that $s_m\in I_r$ if $I_r\neq\{s_{\ell}\}$. 

In order to have uniform notations in the computations below, we further set $\widetilde{\sigma}_j:=\sigma_j$ for all $j\in\{1,\dots,r\}$, except if $\{s_{\ell-1},s_{\ell}\}\subseteq I$ but $s_{\ell-2}\notin I$ (so that $I_{r-1}=\{s_{\ell-1}\}$ and $I_r=\{s_{\ell}\}$), in which case $r\geq 2$ and we set $\widetilde{\sigma}_j:=\sigma_j$ for all $j\in\{1,\dots,r-2\}$ and $\widetilde{\sigma}_{r-1}:=\widetilde{\sigma}_{r}:=\sigma_{r-1}\sigma_r$.

We first make a few observations analogous to (2a)--(2d).

\smallskip

(4a) If $s_1\notin I$, then $\sigma(s_1,a_j)=\widetilde{\sigma}_j\in\Xi_{\eta}$ for all $j\in\{1,\dots,r\}$, and hence $\Xi_{\eta}$ contains the subgroup
$$\widetilde{\Xi}_{\eta,1}:=\langle \widetilde{\sigma}_j \ | \ 1\leq j\leq r\rangle.$$

\smallskip

(4b) If $s_1\in I$, then proceeding as in (2b), we obtain that $\Xi_{\eta}$ contains the subgroup
$$\widetilde{\Xi}_{\eta,2}:=\langle \widetilde{\sigma}_1^2,\widetilde{\sigma}_1\widetilde{\sigma}_j \ | \ 2\leq j\leq r\rangle.$$

\smallskip

(4c) If $s_1\in I$, and there is some $i\in\{2,\dots,\ell-1\}\setminus\{\ell-2\}$ such that $\{s_i,s_{i+1}\}\subseteq S\setminus I$, then $\sigma(s_0,s_i)=\sigma_1=\widetilde{\sigma}_1\in\Xi_{\eta}$, so that $\Xi_{\eta}$ contains $\widetilde{\Xi}_{\eta,1}$ by (4b).

\smallskip

(4d) If $s_1\in I$, and the order $d_j:=|I_j|+1$ of $\sigma_j$ is odd for some $j\in\{1,\dots,r\}$ (in particular, $\widetilde{\sigma}_j=\sigma_j$), then $\widetilde{\sigma}_1=(\widetilde{\sigma}_1^2)^{\tfrac{d_j+1}{2}}\cdot (\widetilde{\sigma}_1\widetilde{\sigma}_j)^{-d_j}\in\Xi_{\eta}$ by (4b), so that $\Xi_{\eta}$ contains $\widetilde{\Xi}_{\eta,1}$ again by (4b).

\smallskip

We now proceed with the proof of (4). Assume first that $\{s_{\ell-1},s_{\ell}\}\subseteq I$. Then $\Gamma_{I_r^{\ext}}$ is of type $A_1^{(1)}$ (if $s_{\ell-2}\notin I$), or of type $A_3^{(1)}$ (if $s_{\ell-2}\in I$ but $s_{\ell-3}\notin I$), or of type $D_{|I_r|}^{(1)}$ (if $\{s_{\ell-2},s_{\ell-3}\}\subseteq I$). In any case, $\widetilde{W}^\eta/W^\eta$ is generated by its index $2$ subgroup $\widetilde{\Xi}_{\eta,1}$ and by the unique automorphism $\sigma_{s_{\ell}}$ of $\Gamma_{I_r^{\ext}}$ mapping $\tau_r$ to $s_{\ell}$.  

\smallskip

(4e) We claim that $\sigma_{s_{\ell}}\notin\Xi_{\eta}$. Indeed, as $\sigma_{s_{\ell}}$ is the unique automorphism of $\widetilde{W}^\eta/W^\eta$ mapping $(\tau_1,\dots,\tau_r)=\typ_{\Sigma^{\eta}}(x_0)$ to $(\tau_1,\dots,\tau_{r-1},s_{\ell})=\typ_{\Sigma^\eta}(x_{\ell}^\eta)$, it is sufficient to check by Lemma~\ref{lemma:tobeornotinTYP} that $s_{\ell}=\typ_{\Sigma}(x_{\ell})\notin \Typ(S,I)$. But the set $A:=S\setminus\{s_{\ell-1},s_{\ell}\}$ clearly satisfies the three properties (TYP0), (TYP1) and (TYP2) from Definition~\ref{definition:TypSI}, and hence contains $\Typ(S,I)$. As $s_{\ell}\notin A$, this implies that $s_{\ell}\notin\Typ(S,I)$, as desired.

\smallskip

(4f) If $s_1\notin I$, then (4a) and (4e) imply that $\Xi_{\eta}=\widetilde{\Xi}_{\eta,1}$ has index $2$ in $\widetilde{W}^\eta/W^\eta$. By (4c) and (4e), the same conclusion holds if $s_1\in I$ and $\Gamma_{S\setminus (I\cup\{s_0\})}$ has an edge. Similarly, (4d) and (4e) also yield the same conclusion if $s_1\in I$ and $\Gamma_{S\setminus (I\cup\{s_0\})}$ has no edge and $\overline{S}\setminus\overline{I}\not\subseteq 2\NN$, since in that case $d_j=|I_j|+1$ is odd for some $j\in\{1,\dots,r\}$. This settles the cases (D2) and (D4) of Theorem~\ref{thm:mainthmXi}.

\smallskip

(4g) Assume now that $\overline{S}\setminus\overline{I}\subseteq 2\NN$ (in particular, $s_1\in I$). Then (4b) implies that $\Xi_{\eta}$ contains $\widetilde{\Xi}_{\eta,2}$. We claim that $\Xi_{\eta}=\widetilde{\Xi}_{\eta,2}$ and has index $4$ in $\widetilde{W}^\eta/W^\eta$ (thus settling the cases (D1) and (D3) of Theorem~\ref{thm:mainthmXi}). In view of (4e), it is sufficient to check that $\sigma_1=\widetilde{\sigma}_1\notin \Xi_{\eta}$. But this follows exactly as in (2e), with $A:=\{s_i \ | \ i\in\overline{S}\cap 2\NN\}\setminus \{s_{\ell-1},s_{\ell}\}$.

\smallskip

To conclude the proof of (4), we are left with the case $\{s_{\ell-1},s_{\ell}\}\not\subseteq I$, which we now investigate. In this case, $\widetilde{\sigma}_j=\sigma_j$ and $\Gamma_{I_j^{\ext}}$ is of type $A_{|I_j|}^{(1)}$ for all $j\in\{1,\dots,r\}$. In particular, $\widetilde{W}^\eta/W^\eta=\widetilde{\Xi}_{\eta,1}$.

\smallskip

(4h) If $s_1\notin I$, then (4a) implies that $\Xi_{\eta}=\widetilde{W}^\eta/W^\eta$. If $s_1\in I$ and either $\Gamma_{S\setminus (I\cup\{s_0,s_{\ell-1},s_{\ell}\})}$ has an edge or $\{s_{\ell-1},s_{\ell}\}\subseteq S\setminus I$, then (4c) also implies that $\Xi_{\eta}=\widetilde{W}^\eta/W^\eta$. Finally, if $s_1\in I$ and $\Gamma_{S\setminus (I\cup\{s_0,s_{\ell-1},s_{\ell}\})}$ has no edge and $\{s_{\ell-1},s_{\ell}\}\not\subseteq S\setminus I$, and either $\ell$ is odd or $\overline{S}\setminus (\overline{I}\cup\{\ell-1,\ell\})\not\subseteq 2\NN$, then $d_j=|I_j|+1$ is odd for some $j\in\{1,\dots,r\}$ (i.e. if all $d_j$ were even, then $\overline{S}\setminus (\overline{I}\cup\{\ell-1,\ell\})\subseteq 2\NN$, and thus $\ell$ would be even, as exactly one of $s_{\ell-1}$ and $s_{\ell}$ belongs to $I$ and as $|I_r|$ would be odd, a contradiction), so that $\Xi_{\eta}=\widetilde{W}^\eta/W^\eta$ by (4d). This settles the case (D6) of Theorem~\ref{thm:mainthmXi}.

\smallskip

(4i) Assume now that $\ell$ is even, that $\{s_{\ell-1},s_{\ell}\}\not\subseteq S\setminus I$, and that $\overline{S}\setminus (\overline{I}\cup\{\ell-1,\ell\})\subseteq 2\NN$ (in particular, $s_1\in I$). Then $\Xi_{\eta}$ contains $\widetilde{\Xi}_{\eta,2}$ by (4b). We claim that $\Xi_{\eta}=\widetilde{\Xi}_{\eta,2}$ and has index $2$ in $\widetilde{W}^\eta/W^\eta$ (thus settling the remaining case (D5) of Theorem~\ref{thm:mainthmXi}). For this, it is sufficient to check that $\sigma_1\notin \Xi_{\eta}$. As in (2e), this can be done by checking that $s_1\notin\Typ(S,I)$, or else that the set $A:=\{s_i \ | \ i\in (\overline{S}\setminus\{\ell-1,\ell\})\cap 2\NN\}\cup \{x\}$ (where $x$ is the only element of $\{s_{\ell-1},s_{\ell}\}\cap (S\setminus I)$) satisfies the properties (TYP0), (TYP1) and (TYP2) from Definition~\ref{definition:TypSI}. By assumption, $A$ satisfies (TYP0). To check (TYP1), let $a,b\in A$ be distinct, and let us show that $\op_a(b)\in A$.  If $a=s_0$ or $a=x$, then $\op_a(b)=b\in A$ (if $b\in\{s_0,x\}$, this uses the fact that $\ell$ is even) and $\op_b(a)=a\in A$ (this uses the fact that $\ell$ is even and the definition of $A$) --- see Lemma~\ref{lemma:oppositionfinite}. Similarly, if $a,b\notin\{s_0,x\}$, then $\op_a(b)=b\in A$ (see Lemma~\ref{lemma:oppositionfinite}). Finally, $A$ satisfies (TYP2), as the unique automorphism of $\widetilde{W}/W$ mapping $s_0$ to $x$ stabilises $\{s_0,x\}$ (this uses the fact that $\ell$ is even --- see Lemma~\ref{lemma:extended_diagram_autom}(4)).

This concludes the proof of the proposition.
\end{proof}

\begin{prop}
Theorem~\ref{thm:mainthmXi} holds for $\Gamma_S$ of type $X\in\{E_6^{(1)},E_7^{(1)},E_8^{(1)}\}$.
\end{prop}
\begin{proof}
Note that $\Gamma_{I_j^{\ext}}$ is of type $A_{|I_j|}^{(1)}$ for all $j\in\{1,\dots,r\}$, unless if $I_j\supseteq\{s_2,s_3,s_4,s_5\}$, in which case $\Gamma_{I_j^{\ext}}$ is either of type $D_{|I_j|}^{(1)}$, or $E_6^{(1)}$, or $E_7^{(1)}$ (recall that $|I|\leq\ell-1$ by Lemma~\ref{lemma:Veta_forI}(3)).

As in the proof of Proposition~\ref{prop:main_ABCD}, we will use Lemma~\ref{lemma:whatsinXi} repeatedly to conclude, in the notations of that lemma, that $\sigma(a,b)\in\Xi_{\eta}$ for various pairs $(a,b)$; we will also use the notation $\gamma_{ab}$ for the unique minimal path in $\Gamma_S$ from $a$ to $b$, and view paths in $\Gamma_S$ as ordered subsets of $S$.

For each $j\in\{1,\dots,r\}$, call $\gamma_j:=\Conv(x_0,I_j)\setminus I_j$ (where $\Conv(x_0,I_j)\subseteq S$ is the convex hull of $\{x_0\}\cup I_j$ in $\Gamma_S$) the \emph{path from $x_0$ to $I_j$}, and let $a_j\in S\setminus I$ be the last vertex of $\gamma_j$. If $\Gamma_{I_j^{\ext}}$ is of type $A_{|I_j|}^{(1)}$, we also let $\sigma_j$ denote the unique automorphism of $\widetilde{W}_{I_j^{\ext}}^{\eta}/W_{I_j^{\ext}}^{\eta}$ mapping $\tau_j$ to the vertex of $I_j$ that is closest to $s_0$ in $\Gamma_S$; hence, in that case, $\widetilde{W}_{I_j^{\ext}}^{\eta}/W_{I_j^{\ext}}^{\eta}=\langle\sigma_j\rangle$ and $\sigma(s_0,a_j)=\sigma_j$. 

We start with a few observations.

\smallskip

(a) If $s_4\in I$ and the irreducible component $I_{i}$ of $I$ containing $s_4$ is not of type $D_{|I_i|}$, then $\Xi_{\eta}=\widetilde{W}^{\eta}/W^{\eta}$. Indeed, in that case, $\sigma(s_0,a_j)\in\Xi_{\eta}$ generates $\widetilde{W}_{I_j^{\ext}}^{\eta}/W_{I_j^{\ext}}^{\eta}$ for each $j\in\{1,\dots,r\}$ (note that $\widetilde{W}_{I_i^{\ext}}^{\eta}/W_{I_i^{\ext}}^{\eta}$ is generated by any nontrivial automorphism when $I_i$ is of type $E_6$ or $E_7$).

\smallskip

(b) If $s_4\notin I$ and if $\{s_2,s_3,s_5\}\setminus\gamma_{s_0s_4}\not\subseteq I$, then $\Xi_{\eta}=\widetilde{W}^{\eta}/W^{\eta}$. Indeed, in that case, $\sigma(s_0,a_j)=\sigma_j\in\Xi_{\eta}$ generates $\widetilde{W}_{I_j^{\ext}}^{\eta}/W_{I_j^{\ext}}^{\eta}$ for each $j\in\{1,\dots,r\}$.

\smallskip

(c) If $s_4\notin I$ and $\{s_2,s_3,s_5\}\setminus\gamma_{s_0s_4}\subseteq I$, and if $\mathrm{gcd}(d_{i_1},d_{i_2})=1$ (where $d_{j}:=|I_j|+1$ and $I_{i_1},I_{i_2}$ are two distinct irreducible components of $I$, each containing one of the two vertices in $\{s_2,s_3,s_5\}\setminus\gamma_{s_0s_4}$), then $\Xi_{\eta}=\widetilde{W}^{\eta}/W^{\eta}$. Indeed, in that case, $\sigma(s_0,a_j)=\sigma_j\in\Xi_{\eta}$ generates $\widetilde{W}_{I_j^{\ext}}^{\eta}/W_{I_j^{\ext}}^{\eta}$ for $j\neq i_1,i_2$, and $\sigma(s_0,s_4)=\sigma_{i_1}\sigma_{i_2}$ generates $\widetilde{W}_{I_{i_1}^{\ext}}^{\eta}/W_{I_{i_1}^{\ext}}^{\eta}\times \widetilde{W}_{I_{i_2}^{\ext}}^{\eta}/W_{I_{i_2}^{\ext}}^{\eta}$ (as $\sigma_{j}$ has order $d_j$).

\smallskip

(d) If $S\setminus (I\cup\{s_0\})$ contains a special vertex $x$, then $\Xi_{\eta}=\widetilde{W}^{\eta}/W^{\eta}$. Indeed, if $s_4\notin I$, we may assume by (b) and (c) that $\{s_2,s_3,s_5\}\setminus\gamma_{s_0s_4}\subseteq I$ and that $\mathrm{gcd}(d_{i_1},d_{i_2})\neq 1$ (keeping the notations of (c)). But in that case, $\sigma(s_0,a_j)=\sigma_j\in\Xi_{\eta}$ for $j\neq i_1,i_2$, and $\sigma(s_0,s_4)=\sigma_{i_1}\sigma_{i_2}\in\Xi_{\eta}$, and either $\sigma(x,s_4)\in\{\sigma_{i_1},\sigma_{i_2}\}\cap\Xi_{\eta}$ (if the unique vertex in $\{s_2,s_3,s_5\}\cap\gamma_{s_0s_4}$ belongs to $S\setminus I$) or $\sigma(x,s_4)\in\{\sigma_{i_1}\sigma_j\inv,\sigma_{i_2}\sigma_j\inv\}\cap\Xi_{\eta}$ (if the unique vertex in $\{s_2,s_3,s_5\}\cap\gamma_{s_0s_4}$ belongs to $I_j$ --- note that $j\neq i_1,i_2$), yielding the claim. On the other hand, if $s_4\in I$, then by (a) we may assume that the irreducible component $I_{i}$ of $I$ containing $s_4$ is of type $D_{|I_i|}$. But then $\sigma(s_0,a_j)=\sigma_j\in\Xi_{\eta}$ for all $j\in\{1,\dots,r\}$, and since $\widetilde{W}_{I_i^{\ext}}^{\eta}/W_{I_i^{\ext}}^{\eta}$ is generated by $\sigma_i$ and $\sigma(x,y)\in\Xi_{\eta}$ (where $y\in S\setminus I$ is the last vertex of the path from $x$ to $I_i$), the claim also follows in that case.

\smallskip

We now proceed with the proof of the proposition, investigating each of the types $E_6^{(1)}$, $E_7^{(1)}$ and $E_8^{(1)}$ separately. Recall that $\sigma_s$ is defined as the identity in $\widetilde{W}^\eta/W^\eta$ if $s\in S\setminus I$ (see Definition~\ref{definition:sigmaj}).

(1) Assume first that $X=E_6^{(1)}$. Suppose that $\Xi_{\eta}\neq\widetilde{W}^{\eta}/W^{\eta}$. Then $\{s_1,s_6\}\subseteq I$ by (d). If $s_4\in I$, then $\{s_2,s_3,s_5\}\subseteq I$ by (a) and hence $|I|=\ell$, a contradiction. Hence $s_4\notin I$, so that $\{s_3,s_5\}\subseteq I$ by (b). Thus, either $I=\{s_1,s_3,s_5,s_6\}$ or $I=\{s_1,s_2,s_3,s_5,s_6\}$. 

\smallskip

(1a) We claim that $\sigma_{s_1}\notin\Xi_{\eta}$. By Lemma~\ref{lemma:tobeornotinTYP}(2), it is sufficient to check that $s_1\notin\Typ(S,I)$. But one readily checks that $A:=\{s_0,s_2,s_4\}$ satisfies (TYP0), (TYP1) and (TYP2), and hence contains $\Typ(S,I)$, as desired.

\smallskip

As $\sigma(s_0,a_j)\in\Xi_{\eta}$ for all $j\in\{1,\dots,r\}$, we deduce from (1a) that $\Xi_{\eta}=\langle\sigma_{s_2},\sigma_{s_3}\sigma_{s_5}\rangle$ has index $2$ in $\widetilde{W}^{\eta}/W^{\eta}=\langle\sigma_{s_2},\sigma_{s_3},\sigma_{s_5}\rangle$. 

\smallskip

(2) Assume next that $X=E_7^{(1)}$. Suppose that $\Xi_{\eta}\neq\widetilde{W}^{\eta}/W^{\eta}$. Then $s_7\in I$ by (d).

\smallskip

(2a) Assume first that $s_4\notin I$. Then $\{s_2,s_5\}\subseteq I$ by (b). If $s_6\notin I$, then as $\sigma(s_0,a_j)\in\Xi_{\eta}$ for all $j\in\{1,\dots,r\}$, we have $\sigma_{s_1},\sigma_{s_3},\sigma_{s_2}\sigma_{s_5},\sigma_{s_7}\in\Xi_{\eta}$, and as $\sigma(s_7,s_4)=\sigma_{s_2}\sigma_{s_3}\sigma_{s_7}\inv\in\Xi_{\eta}$, we deduce that $\Xi_{\eta}=\widetilde{W}^{\eta}/W^{\eta}$, a contradiction. Thus $s_6\in I$, and hence $I=\{s_2,s_5,s_6,s_7\}\cup I'$ for some $I'\subseteq\{s_1,s_3\}$. 

We claim that $\sigma_{s_7}\notin \Xi_{\eta}$. Indeed, by Lemma~\ref{lemma:tobeornotinTYP}(2), it is sufficient to check that $s_7\notin\Typ(S,I)$, or else that $A:=\{s_0,s_1,s_3,s_4,s_6\}$ satisfies (TYP0), (TYP1) and (TYP2). Clearly, $A$ satisfies (TYP0) and (TYP2). Moreover, the sequences defined in (TYP1) associated to $s_1,s_3,s_4,s_6$ are respectively given by $(s_0,s_1,s_0,s_1,\dots)$, $(s_0,s_3,s_1,s_3,s_0,s_3,\dots)$, $(s_0,s_4,s_3,s_6,s_3,s_4,s_0,s_4,\dots)$ and $(s_0,s_6,s_0,s_6,\dots)$, yielding the claim.

As $\sigma(s_0,a_j)\in\Xi_{\eta}$ for all $j\in\{1,\dots,r\}$, we conclude that $\Xi_{\eta}=\langle \sigma_{s_1},\sigma_{s_3},\sigma_{s_2}\sigma_{s_5}\rangle$ has index $2$ in $\widetilde{W}^{\eta}/W^{\eta}=\langle\sigma_{s_1},\sigma_{s_3},\sigma_{s_2},\sigma_{s_5}\rangle$.

\smallskip

(2b) Assume next that $s_4\in I$. Then $\{s_2,s_3,s_5\}\subseteq I$ by (a), and $\{s_1,s_6\}\not\subseteq I$ because $|I|\leq\ell-1$. Hence $I=\{s_2,s_3,s_4,s_5,s_7\}\cup I'$ for some $I'\subsetneq\{s_1,s_6\}$.

We claim that $\sigma_{s_7}\notin \Xi_{\eta}$. Indeed, by Lemma~\ref{lemma:tobeornotinTYP}(2), it is sufficient to check that $s_7\notin\Typ(S,I)$, or else that $A:=\{s_0,s_1,s_6\}$ satisfies (TYP0), (TYP1) and (TYP2), which can be seen from the corresponding argument in (2a).

As $\sigma(s_0,a_j)\in\Xi_{\eta}$ for all $j\in\{1,\dots,r\}$, we deduce the following description of $\Xi_{\eta}$ (see Lemma~\ref{lemma:extended_diagram_autom}): if $I'=\{s_6\}$, then $\Xi_{\eta}=\langle \sigma_{s_3}\rangle$ has index $2$ in $\widetilde{W}^{\eta}/W^{\eta}=\langle\sigma_{s_2},\sigma_{s_3}\rangle$. If $I'=\{s_1\}$, then $\Xi_{\eta}=\langle \sigma_{s_2}\sigma_{s_7}\rangle$ has index $2$ in $\widetilde{W}^{\eta}/W^{\eta}=\langle\sigma_{s_2},\sigma_{s_7}\rangle$. And if $I'=\varnothing$, then $\Xi_{\eta}=\langle \sigma_{s_3},\sigma_{s_2}\sigma_{s_7}\rangle$ has index $2$ in $\widetilde{W}^{\eta}/W^{\eta}=\langle\sigma_{s_2},\sigma_{s_3},\sigma_{s_7}\rangle$.

\smallskip

(3) Assume finally that $X=E_8^{(1)}$. Suppose for a contradiction that $\Xi_{\eta}\neq\widetilde{W}^{\eta}/W^{\eta}$. 

\smallskip

(3a) Assume first that $s_4\in I$. Then $\{s_2,s_3,s_5\}\subseteq I$ by (a). If $s_1\in I$, then $s_6\notin I$ by (a), and since $\sigma(s_0,a_j)\in\Xi_{\eta}$ for all $j\in\{1,\dots,r\}$, we have $\Xi_{\eta}\supseteq\langle \sigma_{s_8},\sigma_{s_7},\sigma_{s_5}\rangle=\widetilde{W}^{\eta}/W^{\eta}$ (see Lemma~\ref{lemma:extended_diagram_autom}(4)), a contradiction. Thus, $s_1\notin I$. Hence $\sigma(s_0,s_1)=\sigma_{s_2}\in\Xi_{\eta}$. As $\sigma(s_0,a_j)\in\Xi_{\eta}$ for all $j\in\{1,\dots,r\}$, we conclude that $\Xi_{\eta}=\widetilde{W}^{\eta}/W^{\eta}$, a contradiction.

\smallskip

(3b) Assume now that $s_4\notin I$. Then $\{s_2,s_3\}\subseteq I$ by (b), and hence $s_1\notin I$ by (c). Then $\sigma(s_0,s_1)=\sigma_{s_2}\in\Xi_{\eta}$ and $\sigma(s_0,s_4)=\sigma_{s_2}\sigma_{s_3}\in\Xi_{\eta}$, so that $\sigma_{s_2},\sigma_{s_3}\in\Xi_{\eta}$. As $\sigma(s_0,a_j)\in\Xi_{\eta}$ for all $j\in\{1,\dots,r\}$, we conclude that $\Xi_{\eta}=\widetilde{W}^{\eta}/W^{\eta}$, again a contradiction.

This concludes the proof of the proposition.
\end{proof}

\begin{prop}
Theorem~\ref{thm:mainthmXi} holds for $\Gamma_S$ of type $X\in\{F_4^{(1)},G_2^{(1)}\}$.
\end{prop}
\begin{proof}
As in the proof of Proposition~\ref{prop:main_ABCD}, we will use Lemma~\ref{lemma:whatsinXi} repeatedly to conclude, in the notations of that lemma, that $\sigma(a,b)\in\Xi_{\eta}$ for various pairs $(a,b)$.

(1) Assume first that $X=G_2^{(1)}$. Then either $I=\{s_1\}$ or $I=\{s_2\}$, and $\widetilde{W}^{\eta}/W^{\eta}$ has order $2$. If $I=\{s_2\}$, then $\sigma(s_0,s_0)=\sigma_{s_2}\in\Xi_{\eta}$, and if $I=\{s_1\}$, then $\sigma(s_0,s_2)=\sigma_{s_1}\in\Xi_{\eta}$. Thus, in both cases, $\Xi_{\eta}=\widetilde{W}^{\eta}/W^{\eta}$.

\smallskip

(2) Assume next that $X=F_4^{(1)}$. If $s_3\in I$, then as $\sigma(s_0,s_i)\in\Xi_{\eta}$ for $i\in\{0,1,2\}$ such that $s_i\notin I$, we have $\Xi_{\eta}=\widetilde{W}^{\eta}/W^{\eta}$. Similarly, if $s_3\notin I$ and $s_4\notin I$, then as $\sigma(s_0,s_i)\in\Xi_{\eta}$ for $i\in\{0,1\}$ such that $s_i\notin I$, we have $\Xi_{\eta}=\widetilde{W}^{\eta}/W^{\eta}$. We may thus assume that $I=\{s_4\}\cup I'$ for some  $I'\subseteq\{s_1,s_2\}$. Let $I_1=\{s_4\}$ and $I_2=I'$ (if $I'$ is nonempty). Then $\Xi_{\eta}$ contains $\widetilde{W}_{I_2^{\ext}}^{\eta}/W_{I_2^{\ext}}^{\eta}$, as $\sigma(s_0,s_i)\in\Xi_{\eta}$ for $i\in\{0,1\}$ such that $s_i\notin I$. To show that $\Xi_{\eta}=\widetilde{W}^{\eta}/W^{\eta}$, it is thus sufficient to see that $\Xi_{\eta}$ contains an automorphism mapping $\tau_1$ to $s_4$. Note for this that $y:=s_3s_2s_1s_0x_0\notin m_4$, as $$s_3s_2s_1s_0x_0=s_3s_4s_2s_1s_0x_0\notin m_4\iff s_2s_1s_0x_0\notin s_4s_3m_4=m_3,$$
which holds by Lemma~\ref{lemma:inmi}. Moreover, $y^\eta\in C_0^\eta$ by Lemma~\ref{lemma:inC0eta}. As in the proof of Lemma~\ref{lemma:whatsinXi}, we then deduce from Lemma~\ref{lemma:tobeornotinTYP}(1) that $\typ_{\Sigma^\eta}(x_0)=(\tau_1,\tau_2)$ and $\typ_{\Sigma^\eta}(y^\eta)=(s_4,a)$ are in the same $\Xi_{\eta}$-orbit for some $a\in I_2^{\ext}$ (where we omit $\tau_2$ and $a$ if $I'=\varnothing$), as desired. 
\end{proof}


\subsection{Relation between \texorpdfstring{$\Xi_w$ and $\Xi_{\eta_w}$}{Xiw and Xieta}, and centraliser of an element}\label{subsection:DCSCAff}

Now that we obtained a complete description of $\Xi_\eta$, we will show that one can replace the subgroup $\Xi_w$ by $\Xi_{\eta}$ in the statement of Theorem~\ref{theorem=R1R2implythmA} (see Theorem~\ref{thm:XietaXiw} below).

We start with an elementary observation.

\begin{lemma}\label{lemma:centraliser}
Let $w\in W$ be of infinite order, and set $\eta:=\eta_w\in\partial V$. Let $v\in W_{\eta}$. Then  $v\in\ZZZ_W(w)$ if and only if $v_\eta\in \ZZZ_{\pi_{\eta}(W_{\eta})}(w_\eta)$.
\end{lemma}
\begin{proof}
The forward implication is clear. Conversely, suppose $v_\eta\in \ZZZ_{\pi_{\eta}(W_{\eta})}(w_\eta)$.  

We first claim that for any vector $h\in V$ (based at the origin $x_0$) and any $u\in W_{x_0}^\eta$, the vector $uh-h$ belongs to $V^\eta$. Indeed, using the formula $u_1u_2h-h=u_1(u_2h-h)+(u_1h-h)$ for $u_1,u_2\in W_{x_0}^\eta$, and the fact that $W_{x_0}^\eta$ stabilises $V^\eta$, it is sufficient to prove the claim when $u=r_m$ for some $m\in \WW^\eta_{x_0}$ (see Definition~\ref{definition:Veta}). But in this case, $uh-h=r_mh-h\in m^{\perp}\subseteq \big(\bigcap_{m'\in\WW^{\eta}_{x_0}}m'\big)^{\perp}=V^\eta$, as desired.

By Lemma~\ref{lemma:Weta}(2), we can write $w=w_{x_0}t_{h_w}$ and $v=v_{x_0}t_{h_v}$ for some $w_{x_0},v_{x_0}\in W_{x_0}^{\eta}$ and some vectors $h_w,h_v\in V$ such that $t_{h_w},t_{h_v}\in T_0$ (where $t_h$ denotes the translation of vector $h$). By assumption, $vwv\inv w\inv\in\ker\pi_{\eta}=\{t_h \ | \ h\in (V^\eta)^\perp\}$ (see Lemma~\ref{lemma:Weta}(4)). A straightforward computation (using the formula $ut_h u\inv=t_{uh}$ for $u\in W_{x_0}$ and $h\in V$ a vector) yields that $vwv\inv w\inv=v_{x_0}w_{x_0}v_{x_0}\inv w_{x_0}\inv\cdot t_h$, where
$$h:=w_{x_0}v_{x_0}\cdot \big((w_{x_0}\inv h_v-h_v) -(v_{x_0}\inv h_w-h_w)\big).$$
Hence $h\in V^\eta$ by the above claim, so that $v_{x_0}w_{x_0}v_{x_0}\inv w_{x_0}\inv=1$ and $h=0$ by Lemma~\ref{lemma:Weta}(2), that is, $v\in\ZZZ_W(w)$.
\end{proof}

We next prove the following consequence of Theorem~\ref{thm:mainthmXi}.

\begin{lemma}\label{lemma:deltasigmablabla}
Let $\eta\in\partial V$ be standard and set $I:=I_\eta$. Let $\sigma,\delta\in\Xi_{\eta}$. Let $K\subseteq I^{\ext}$ be a spherical subset with $\delta(K)=K$, and assume that there exists an element $x\in W^\eta$ of minimal length in $W^\eta_{\sigma(K)}xW^\eta_K$ such that $\delta(x)=x$ and $x\Pi_{K}=\Pi_{\sigma(K)}$. Let $J$ be a component of $K$. Then:
\begin{enumerate}
\item
If $J$ is not of type $A_m$ for some $m\geq 1$, then $\kappa_x\sigma(J)=J$.
\item
If $\delta(J)=J$ and $\delta|_{J}\neq\id$, then $\kappa_x\sigma(J)=J$.
\item
If $J$ is of type $F_4$, then $\kappa_x\sigma|_{J}=\id$.
\end{enumerate}
\end{lemma}
\begin{proof}
Let $I_i^{\ext}$ be the component of $I^{\ext}$ containing $J$. Note that $\kappa_x\sigma|_{K}$ is a diagram automorphism of $W^\eta_K$ commuting with $\delta|_K$ (as $\Xi_{\eta}$ is abelian and $\delta(x)=x$) and stabilising $K\cap I_i^{\ext}$ (as $\sigma$ and $\kappa_x$ stabilise $I_i^{\ext}$). Moreover, $J$ is a component of $K\cap I_i^{\ext}$, and hence $\kappa_x\sigma(J)\cap J\neq\varnothing\implies \kappa_x\sigma(J)=J$. 

If $J$ is of type $D_5$ and is contained in a subset of $I^{\ext}$ of type $E_6$, then $I_i^{\ext}$ is of type $E_n^{(1)}$ for some $n\in\{6,7,8\}$, and hence $J$ is the only component of $K\cap I_i^{\ext}$ of type $D_5$, so that $\kappa_x\sigma(J)=J$ in that case. We may thus assume that $J$ is not of type $D_5$ inside a subset of type $E_6$. 

Lemma~\ref{lemma:C1C2intermediate} (applied to $W:=W^\eta$, $\delta:=\delta$, $J:=\sigma(K)$, $K:=K$, $x:=x\inv$, and $I:=J$) then implies the following:
\begin{enumerate}
\item[(i)]
If $J$ is not of type $A_m$, then $\kappa_{x}(J)=J$.
\item[(ii)]
If $\delta(J)=J$ and $\delta|_{J}\neq\id$, then $\kappa_{x}(J)=J$.
\item[(iii)]
If $J$ is of type $F_4$, then $\kappa_{x}|_J=\id$. 
\end{enumerate}
In particular, if $J$ satisfies the assumptions of one of the statements (1), (2), (3), then $$\sigma(J)\cap J\neq\varnothing\implies \kappa_x\sigma(J)\cap J\neq\varnothing\implies \kappa_x\sigma(J)=J\implies\sigma(J)=J,$$ and it is sufficient to show that $\sigma(J)\cap J\neq\varnothing$ (and $\sigma|_J=\id$ for (3)).

If $J$ is of type $F_4$, then $I_i^{\ext}$ must be of type $F_4^{(1)}$, and hence $\sigma|_{I_i^{\ext}}=\id$, yielding (3). We now prove (1) and (2). Assume for a contradiction that $\sigma(J)\cap J=\varnothing$. 

(1) If $J$ is not of type $A_m$, then one of the following two possibilities occurs:

(1a) $J$ contains a subset of type $D_4$ and $I^{\ext}_i$ is of type $D_n^{(1)}$ for some $n\geq 5$ (as $J$ and $\sigma(J)$ are disjoint subsets of $I_i^{\ext}$ both containing a subset of type $D_4$). In particular, $\Gamma_S$ is of type $D_{\ell}^{(1)}$ and $\{s_{\ell-2},s_{\ell-1},s_{\ell}\}\subseteq I$. The cases (D1), (D2) of Theorem~\ref{thm:mainthmXi} then imply that $\sigma|_{I_i^{\ext}}$ stabilises each of the two subsets of $I_i^{\ext}$ of type $D_4$, a contradiction.

(1b) $J$ contains a double edge and $I_i^{\ext}$ is of type $C_n^{(1)}$ for some $n\geq 4$ (as $J$ and $\sigma(J)$ are disjoint subsets of $I_i^{\ext}$ both containing a double edge). In particular, $\Gamma_S$ is of type $C_{\ell}^{(1)}$ and $s_{\ell}\in I$. The case ($C_{\ell}^{(1)}$) of Theorem~\ref{thm:mainthmXi} then implies that $\sigma|_{I_i^{\ext}}=\id$, again a contradiction. This proves (1).

(2) Assume finally that $\delta(J)=J$ and $\delta|_{J}\neq\id$. Since $\sigma(J)\cap J=\varnothing$ by assumption, (1) implies that $J$ is of type $A_m$ for some $m\geq 2$. As $\sigma|_{I_i^{\ext}}$ and $\delta|_{I_i^{\ext}}$ are distinct nontrivial elements of $\widetilde{W}_{I_i^{\ext}}/W_{I_i^{\ext}}$ that are not inverse to one another, they generate a subgroup of order at least $4$. Moreover, the fact that $\delta(J)=J$ and $\delta|_{J}\neq\id$ rules out the possibility that $I_i^{\ext}$ is of type $A_n^{(1)}$ by Lemma~\ref{lemma:extended_diagram_autom}(1). We conclude that $I_i^{\ext}$ is of type $D_n^{(1)}$ ($n\geq 5$) --- see again Lemma~\ref{lemma:extended_diagram_autom}. But then $\Gamma_S$ is of type $D_{\ell}^{(1)}$ and $\{s_{\ell-2},s_{\ell-1},s_{\ell}\}\subseteq I$, and the cases (D1), (D2) of Theorem~\ref{thm:mainthmXi} imply that $\sigma|_{I_i^{\ext}}$ and $\delta|_{I_i^{\ext}}$ are contained in a common subgroup of $\widetilde{W}_{I_i^{\ext}}/W_{I_i^{\ext}}$ of order $2$, a contradiction.
\end{proof}

Here is finally the key proposition to establish Theorem~\ref{thm:XietaXiw}.

\begin{prop}\label{proprelationXiwXieta}
Let $w\in W$ be of infinite order, and set $\eta:=\eta_w\in\partial V$.
Let $C,D\in\CMin(w)$ and $\sigma\in \Xi_{\eta}$ be such that $\sigma(I_w(C))=I_w(D)$. Then $\sigma\in\Xi_w$.
\end{prop}
\begin{proof}
Assume first that $\eta$ is standard. Since $C^\eta,D^\eta\in\CMin_{\Sigma^\eta}(w_\eta)$ by Proposition~\ref{prop:corresp_Minsets2}, the elements $\pi_{w_\eta}(C^\eta),\pi_{w_\eta}(D^\eta)\in W^\eta\delta_w$ are cyclically reduced and conjugate in $W^\eta$. Lemma~\ref{lemma:241He} (applied to $(W,S):=(W^\eta,S^\eta)$ and $\delta:=\delta_w$) then yields some $x\in W^\eta$ of minimal length in $W^\eta_{I_w(D)}xW^\eta_{I_w(C)}$ and such that $\delta_w(x)=x$ and $x\Pi_{I_w(C)}=\Pi_{I_w(D)}$.

In view of Lemma~\ref{lemma:deltasigmablabla} (applied with $K:=I_w(C)$, so that $\sigma(K)=I_w(D)$), we are in a position to apply Proposition~\ref{prop:useP3} with $(W,S):=(W^\eta_{I_w(C)},I_w(C))$, $\delta:=\delta_w|_{I_w(C)}$, $\sigma:=\kappa_{x}\sigma|_{I_w(C)}$, and $w:=\pi_{w_\eta}(C^\eta)\delta_w\inv$ (note that the conjugacy class of $\pi_{w_\eta}(C^\eta)$ in $W^\eta_{I_w(C)}$ is cuspidal by Proposition~\ref{prop:He07}(1)). We conclude that $\kappa_x\sigma(\pi_{w_\eta}(C^\eta))$ and $\pi_{w_\eta}(C^\eta)$ are conjugate in $W^\eta_{I_w(C)}$. 

In particular, there exists $u\in W^\eta$ such that $\sigma \pi_{w_{\eta}}(C^\eta)  \sigma\inv = u \pi_{w_{\eta}}(C^\eta) u\inv$. Then $u\inv \sigma$ commutes with $\pi_{w_\eta}(C^\eta)$, and hence writing $C^\eta=aC_0^\eta$ for some $a\in W^\eta$, the element $au\inv\sigma a\inv=au\inv\sigma(a)\inv\cdot\sigma$ commutes with $w_\eta$. Lemma~\ref{lemma:centraliser} then implies that $au\inv\sigma(a)\inv\cdot\sigma\in\pi_{\eta}(\ZZZ_W(w))$, so that $\sigma\in\Xi_w$, as desired.

If now $\eta$ is arbitrary, we let $b,I$ be as in Proposition~\ref{prop:WetaSetaSigmaetaAff}(6), so that $\overline{\eta}:=b\inv\eta$ is standard and $S^{\overline{\eta}}=b\inv S^\eta b$. Note that $\overline{\eta}=\eta_{\overline{w}}$ where $\overline{w}:=b\inv wb$, and that $\overline{C}:=b\inv C$ and $\overline{D}:=b\inv D$ belong to $\CMin(\overline{w})$. Moreover, by Remarks~\ref{remark:binveta} and \ref{remark:Xietaabelian}, we have $\Xi_{\overline{\eta}}=b\inv \Xi_{\eta}b$ (and $\Xi_{\overline{w}}=b\inv \Xi_{w}b$), while $I_{\overline{w}}(\overline{C})=b\inv I_w(C)b$ and $I_{\overline{w}}(\overline{D})=b\inv I_w(D)b$ (see Lemma~\ref{lemma:I_w(C)geometric}). We have just proved that $\overline{\sigma}:=b\inv \sigma b\in\Xi_{\overline{w}}$, and hence also $\sigma\in\Xi_w$ in this case.
\end{proof}

Recall from Definition~\ref{definition:quotientgraph} the definition of the graph $\KKK_{\delta_w}^0(I_w)/\Xi_w$, for $w\in W$ with $\Pc(w)=W$. Recall also from Remark~\ref{remark:binveta} that $\Xi_{\eta_w}$ is abelian. In particular, we may define the quotient graph $\KKK_{\delta_w}^0(I_w)/\Xi_{\eta_w}$ as in Definition~\ref{definition:quotientgraph}, which is itself a quotient graph of $\KKK_{\delta_w}^0(I_w)/\Xi_{w}$ as $\Xi_w\subseteq\Xi_{\eta_w}$. Proposition~\ref{proprelationXiwXieta} implies that these last two graphs in fact coincide.

\begin{theorem}\label{thm:XietaXiw}
Let $w\in W$ be of infinite order, and set $\eta:=\eta_w\in\partial V$. Assume that $w_\eta$ is cyclically reduced. Then 
$$\KKK_{\delta_w}^0(I_w)/\Xi_w=\KKK_{\delta_w}^0(I_w)/\Xi_{\eta}.$$
\end{theorem}
\begin{proof}
We have to show that if two vertices $I,J$ of $\KKK_{\delta_w}^0(I_w)$ are in a same $\Xi_{\eta}$-orbit, then they are also in a same $\Xi_{w}$-orbit. But as the vertices of $\KKK_{\delta_w}^0(I_w)$ are of the form $I_w(C)$ with $C\in\CMin(w)$ by Lemma~\ref{lemma:vertexset_IwCMinw}, this follows from Proposition~\ref{proprelationXiwXieta}.
\end{proof}

As an additional consequence of the above results, we obtain the following description of $\Xi_w$ inside $\Xi_\eta$.

\begin{prop}\label{prop:equivalentdefinitionsXiw}
Let $w\in W$ be of infinite order, and set $\eta:=\eta_w\in\partial V$. Assume that $w_\eta$ is cyclically reduced. Let $\sigma\in\Xi_{\eta}$. Then the following assertions are equivalent:
\begin{enumerate}
\item
$\sigma\in\Xi_w$.
\item
There exists $u\in W^\eta$ such that $u\sigma\in\ZZZ_{\pi_\eta(W_\eta)}(w_\eta)$.
\item
$\sigma(I_w)=I_w(D)$ for some $D\in\CMin(w)$.
\item
$\sigma(I_w)$ is a vertex of $\KKK_{\delta_w}^0(I_w)$.
\end{enumerate}
\end{prop}
\begin{proof}
(1)$\Leftrightarrow$(2): If (1) holds, then $\sigma=u\inv v_\eta$ for some $u\in W^\eta$ and $v\in\ZZZ_W(w)$, and hence $u\sigma=v_\eta\in \ZZZ_{\pi_\eta(W_\eta)}(w_\eta)$. Conversely, assume that (2) holds, and let $v\in W_\eta$ be such that $\sigma=v_\eta$. Then $uv\in\ZZZ_W(w)$ by Lemma~\ref{lemma:centraliser}, and hence $\sigma=\delta_{uv}\in\Xi_w$. 

(1)$\Leftrightarrow$(3): If (1) holds, then $\sigma(I_w)=I_w(vC_0)$ for some $v\in\ZZZ_W(w)$ by Lemma~\ref{lemma:IwvC_sigmaIwC}. Conversely, (3)$\Rightarrow$(1) follows from Proposition~\ref{proprelationXiwXieta}. 

(3)$\Leftrightarrow$(4): This follows from Lemma~\ref{lemma:vertexset_IwCMinw}.
\end{proof}

We conclude this subsection with the following description of the centraliser of an infinite order element of $W$, whose relevance follows from the above explicit description of $\Xi_w$. Recall from Remark~\ref{remark:pietaWetagivenbyXieta} that $\pi_\eta(W_\eta)=W^\eta\rtimes \Xi_{\eta}$ for any $\eta\in\partial V$. In particular, there is a natural projection map $\pi_\eta(W_\eta)\to\Xi_\eta$.

\begin{prop}\label{prop:centraliserAFF}
Let $w\in W$ be of infinite order, and set $\eta:=\eta_w\in\partial V$. Then:
\begin{enumerate}
\item
the map $\ZZZ_{\pi_\eta(W_\eta)}(w_\eta)\to \Xi_\eta$ has kernel $\ZZZ_{W^\eta}(w_\eta)$ and image $\Xi_w$.
\item
$\ZZZ_{\pi_\eta(W_\eta)}(w_\eta)=\ZZZ_{W^\eta}(w_\eta)\rtimes \Xi_w$.
\item
there is an exact sequence
\begin{equation*}
1\to T_\eta\rtimes \ZZZ_{W^\eta}(w_\eta)\to\ZZZ_W(w)\to\Xi_w\to 1,
\end{equation*}
where $T_\eta:=\{t_h\in W \ | \ (W_{x_0}\cap W^\eta).h=h\}$. 
\end{enumerate}
\end{prop}
\begin{proof}
The kernel of $\ZZZ_{\pi_\eta(W_\eta)}(w_\eta)\to \Xi_\eta$ is $W^\eta\cap \ZZZ_{\pi_\eta(W_\eta)}(w_\eta)=\ZZZ_{W^\eta}(w_\eta)$ and its image coincides with $\Xi_w$ since $\ZZZ_{\pi_\eta(W_\eta)}(w_\eta)=\pi_\eta(\ZZZ_W(w))$ by Lemma~\ref{lemma:centraliser}. Hence (1) holds, and (2) follows from (1).

By (1), the map $\ZZZ_W(w)\to\Xi_w$, which is the composition of $\pi_\eta$ with the map $\ZZZ_{\pi_\eta(W_\eta)}(w_\eta)\to \Xi_w$, is surjective and has kernel $\ker\pi_\eta\rtimes \ZZZ_{W^\eta}(w_\eta)$ (note that $\ker\pi_\eta\subseteq \ZZZ_W(w)$ by Lemma~\ref{lemma:centraliser}). Since $\ker\pi_\eta=T_\eta$ by Lemma~\ref{lemma:Weta}(4), (3) follows as well.
\end{proof}


\subsection{The standard splitting in \texorpdfstring{$W$}{W} of an element}\label{subsection:TSSIWOAE}
Throughout this subsection, we fix an element $w\in W$ of infinite order, and we set $\eta:=\eta_w\in\partial V$.

Recall from \S\ref{section:TPSOAE} (and Example~\ref{example:PsplittingPwmin}) that we defined the $P_w^{\min}$-splitting of $w$, where $P_w^{\min}=\Fix_W(\Min(w))$. In this subsection, we compare this splitting with the standard splitting of $w$ as an affine isometry of $V$ (in the sense of \cite[3.4]{BMC15}), and we show how it can be used to compute $I_w$ and $\delta_w$ (see Proposition~\ref{prop:deltawandIwcomputationAFF}), by proving an analogue of Lemma~\ref{lemma:awcndeltaeta} in the affine setting.

\begin{definition}
We call the $P_w^{\min}$-splitting $w=w_{\tor}w_{\infty}$ of $w$ (with $w_{\tor}=w_{\tor}(P_w^{\min})$ and $w_{\infty}=w_{\infty}(P_w^{\min})$) the {\bf standard splitting in $W$}\index{Standard splitting} of $w$. 

Recall from \cite[3.4]{BMC15} that $w$ also possesses a {\bf standard splitting} $w=ut_{\mu}$ as an affine isometry of $V$, with $u\in\Isom(V)$ of finite order (the {\bf elliptic part}\index{Elliptic part} of $w$) satisfying $\Fix(u)=\Min(w)$, and $t_{\mu}\in\Isom(V)$ a translation of vector $\mu$ commuting with $u$ (the {\bf translation part}\index{Translation part} of $w$). Note, however, that $u,t_{\mu}$ need not preserve the underlying simplicial structure of $V$, and hence need not belong to $W$.
\end{definition}

\begin{example}
Assume that $W$ is the affine Coxeter group of type $A_2^{(1)}$, with set of simple reflections $S=\{s_0,s_1,s_2\}$, and set $w:=s_0s_1s_2$. The Davis complex of $W$ is the tessellation of the Euclidean plane by congruent equilateral triangles pictured on Figure~\ref{figure:intro} from \S\ref{section:introduction}, and $w$ is a glide reflection along the axis $L$ represented on that picture. In particular, the elliptic part of $w$ is the orthogonal reflection across $L$, and hence does not belong to $w$ (or even to $\Aut(\Sigma)$). On the other hand, $w_{\tor}=1$ and $w_{\infty}=w$ as $\Min(w)=L$ and hence $P_w^{\min}=\Fix_W(L)=\{1\}$.
\end{example}

We first observe that when the elliptic and translation parts of $w$ belong to $W$, then both standard splittings of $w$ coincide. More generally, we show how the standard splitting in $W$ of $w$ can be computed from its standard splitting as an affine isometry of $V$. 
\begin{lemma}\label{lemma:comparisonstandardsplittings}
Let $w=ut_{\mu}$ be the standard splitting of $w$ as an affine isometry of $V$, with elliptic part $u\in\Isom(V)$ and translation part $t_{\mu}$.
\begin{enumerate}
\item
If $t_{\mu},u\in W$, then $w_{\infty}=t_{\mu}$ and $w_{\tor}=u$.
\item
Writing $\Fix_W(\Fix(u))=vW_Iv\inv$ for some $I\subseteq S$ and some $v\in W$ of minimal length in $vW_I$, and letting $w_I\in W_I$ be such that $w_I\inv v\inv wv$ is of minimal length in $W_I v\inv wv$, we have $w_{\tor}=vw_Iv\inv$. 
\end{enumerate}
\end{lemma}
\begin{proof}
(1) Assume that $t_{\mu},u\in W$. Since $\Fix(u)=\Min(w)$, certainly $u\in P_w^{\min}$. On the other hand, if $L$ is a $w$-axis, then $L$ is also a translation axis for $t_{\mu}$. As the walls of $P_w^{\min}$ contain $L$, they are stabilised by $t_{\mu}$, and hence $t_{\mu}$ is $P_w^{\min}$-reduced. Therefore, $w_{\tor}=u$ and $w_{\infty}=t_{\mu}$ by Proposition~\ref{prop:Psplitting}.

(2) As $\Fix(u)=\Min(w)$, we have $P_w^{\min}=vW_Iv\inv$. Moreover, $v\inv wv\in N_W(W_I)$ and $n_I:=w_I\inv v\inv wv\in N_I$, so that the claim follows from Lemma~\ref{Psplittingwini}.
\end{proof}

\begin{corollary}\label{corollary:tmuuinW}
Assume that the elliptic part $u$ of $w$ belongs to $W$. Then:
\begin{enumerate}
\item
$P_w^{\min}=\Pc(u)$.
\item
if $u$ is cyclically reduced, then $w$ is cyclically reduced.
\end{enumerate}
\end{corollary}
\begin{proof}
By Lemma~\ref{lemma:comparisonstandardsplittings}, $u=w_{\tor}\in P_w^{\min}$ and $w_{\infty}$ is a translation of $V$. In particular, $\Pc(u)\subseteq P_w^{\min}$ and $P_w^{\min}=\Fix_W(\Fix(u))$. But as $\Pc(u)$ is a spherical parabolic subgroup, it is the fixer of a point $x\in V$ (and $x\in\Fix(u)$ as $u\in\Pc(u)$), so that $\Pc(u)=\Fix_W(x)\supseteq\Fix_W(\Fix(u))=P_w^{\min}$. This proves (1). 

Since $w_{\infty}$ is a translation, it is straight (see e.g. \cite[Lemma~4.3]{straight}), and hence cyclically reduced by \cite[Lemma~4.1]{straight}. Assume now that $u$ is cyclically reduced. Then $\Pc(u)$ is standard (see e.g. \cite[Proposition~4.2]{CF10}), and hence $P_w^{\min}$ is standard by (1). Since $w_{\infty}$ centralises $P_w^{\min}$ (as it stabilises each wall in $\WW^{\eta}$, and hence each wall of $P_w^{\min}$), Lemma~\ref{lemma:criterioncyclicallyreducedPsplitting}(2) then implies that $w$ is cyclically reduced, yielding (2).
\end{proof}

Note that, while $w_{\infty}$ need not be a translation, $w_{\infty}^n$ (and $w^n$) is always a translation for some $n\geq 1$ (e.g., $n=|W_{x_0}|$), as follows from the decomposition $W=W_{x_0}\ltimes T_0$. The following elementary observation is then a useful criterion to check whether an element is straight.

\begin{lemma}\label{lemma:affinecriterionstraight}
Let $n\in\NN$ be such that $w^n$ is a translation. Then $w$ is straight if and only if $\ell(w^n)=n\ell(w)$. 
\end{lemma}
\begin{proof}
The forwad implication is clear. Conversely, since $w^n$ is straight and $\ell(w^n)=n\ell(w)$, we have $n\ell(w^m)\geq \ell(w^{mn})=m\ell(w^n)=mn\ell(w)$ and hence $\ell(w^m)=m\ell(w)$ for any $m\in\NN$, as desired.
\end{proof}

Our next goal is to show how the standard splitting in $W$ of $w$ can be used to compute $I_w$ and $\delta_w$ (see Proposition~\ref{prop:deltawandIwcomputationAFF} below). To this end, we first need to show that the sets $\Min(w)\subseteq V$ and $$\Min_{V^\eta}(w_\eta):=\Fix_{V^\eta}(w_\eta)\subseteq V_\eta,$$ as well as their subsets of regular points, correspond to one another under the projection map $\pi_{V^\eta}\co V\to V^\eta$. 

\begin{lemma}\label{lemma:corresp_MinsetsAff}
We have $\pi_{V^\eta}(\Min(w))= \Min_{V^\eta}(w_\eta)$.
\end{lemma}
\begin{proof}
 Let $x\in\Min(w)$, and let $L_x\subseteq\pi_{V^\eta}\inv(x^{\eta})$ be the $w$-axis through $x$. Then $wx\in L_x$, so that $w_\eta x^\eta=\pi_{V^\eta}(wx)=x^\eta$ by Lemma~\ref{lemma:Weta}(5), and hence $x^\eta\in\Min_{V^\eta}(w_\eta)$. Conversely, if $w_\eta x^\eta=x^\eta$ for some $x\in V$, then $\pi_{V^\eta}(wy)=w_\eta x^\eta=x^\eta$ for all $y\in \pi_{V^\eta}\inv(x^{\eta})$, and hence $w$ stabilises the nonempty closed convex set $\pi_{V^\eta}\inv(x^{\eta})$. Therefore, $\pi_{V^\eta}\inv(x^{\eta})$ contains a point $y\in\Min(w)$, as desired.
\end{proof}

Recall from Definition~\ref{definition:regularpoints} the definition of $w$-regular points.

\begin{lemma}\label{lemma:corresp_w-regular}
We have $\pi_{V^\eta}(\Reg_V(w))=\Reg_{V^\eta}(w_\eta)$, and any $w$-essential point $x\in\pi_{V^\eta}\inv(\Reg_{V^\eta}(w_\eta))\cap\Min(w)$ is $w$-regular. 
\end{lemma}
\begin{proof}
Note that $\Fix_{W^\eta}(x)=\Fix_{W^\eta}(x^\eta)$ for any $x\in V$. Hence if $x\in\Reg_V(w)$, then $x^\eta\in\Min(w_{\eta})$ by Lemma~\ref{lemma:corresp_MinsetsAff} and 
\begin{align*}
\Fix_{W^\eta}(x^\eta)&=\Fix_{W^\eta}(x)=\Fix_W(x)\cap W^\eta=\Fix_W(\Min(w))\cap W^\eta\\
&=\Fix_{W^\eta}(\Min(w)) =\Fix_{W^\eta}(\pi_{V^\eta}(\Min(w)))=\Fix_{W^\eta}(\Min(w_{\eta}))
\end{align*}
(where the last equality again follows from Lemma~\ref{lemma:corresp_MinsetsAff}), so that $x^{\eta}\in\Reg_{V^\eta}(w_\eta)$. Thus, $\pi_{V^\eta}(\Reg_V(w))\subseteq\Reg_{V^\eta}(w_\eta)$.

Conversely, if $x\in \Min(w)$ is $w$-essential, then $\Fix_W(x)$ fixes the $w$-axis through $x$, and hence $\Fix_W(x)=\Fix_{W^\eta}(x)$. If, moreover, $x^{\eta}\in \Reg_{V^\eta}(w_\eta)$, we then have
\begin{align*}
\Fix_{W}(x)&=\Fix_{W^\eta}(x)=\Fix_{W^\eta}(x^\eta)=\Fix_{W^\eta}(\Min(w_{\eta}))\\
&=\Fix_{W^\eta}(\pi_{V^\eta}(\Min(w)))=\Fix_{W^\eta}(\Min(w))\subseteq\Fix_W(\Min(w)),
\end{align*}
so that $\Fix_W(x)=\Fix_W(\Min(w))$ and $x\in\Reg_V(w)$. This proves the second statement of the lemma.

Finally, note that if $y\in\Min(w_\eta)$, then $y=x^\eta$ for some $x\in\Min(w)$ by Lemma~\ref{lemma:corresp_MinsetsAff}, which we may assume to be $w$-essential (as $\pi_{V^\eta}\inv(y)$ contains the $w$-axis through $x$). Hence $\pi_{V^\eta}(\Reg_V(w))\supseteq\Reg_{V^\eta}(w_\eta)$, thus concluding the proof of the lemma.
\end{proof}

\begin{lemma}\label{lemma:deltawandIwcomputationfromP}
Let $P$ be a $w$-parabolic subgroup, and let $w=w_{\tor}w_{\infty}$ be the $P$-splitting of $w$, with $w_{\tor}:=w_{\tor}(P)$ and $w_{\infty}:=w_{\infty}(P)$. Assume that there exists a point $x\in\Min(w)$ with $P=\Fix_W(x)$ such that $x^\eta\in C_0^\eta$. Then $\delta_w=\pi_{\eta}(w_{\infty})$ and $w_\eta=w_{\tor}\delta_w$.
\end{lemma}
\begin{proof}
Note first that $x$ is $w$-essential, as the walls of $R_x$ (i.e. of $P$) all belong to $\WW^\eta$. Since $x^\eta\in C_0^\eta$, Lemma~\ref{lemma:corresp_wresidues}(3) implies that the restriction of $\pi_{\Sigma^\eta}$ to $R_x$ is a cellular isomorphism onto a $w_\eta$-stable residue $R_x^{\eta}$ of $\Sigma^\eta$ containing $C_0^\eta$ . 

Let $C:=\proj_{R_x}(C_0)$, so that $C^\eta=C_0^\eta$ by Lemma~\ref{lemma:corresp_wresidues}(1). Since $\Stab_W(R_x)=P$, the chambers $C_0$ and $w_{\infty}C_0$ lie on the same side of every wall of $R_x$. Since $w_{\infty}$ stabilises this set of walls (i.e. normalises $P$), and since $C,C_0$ lie on the same side of every wall of $R_x$, so do $w_{\infty}C_0$ and $w_{\infty}C$. In particular, $C=\proj_{R_x}(w_{\infty}C)$. As $x\in \Min(w_{\infty})$ by Lemma~\ref{lemma:waxisawaxis}, Lemma~\ref{lemma:corresp_wresidues}(2) (applied with $x:=w_\infty x$, $y:=x$ and $C:=w_{\infty}C$) then implies that $C^\eta=\pi_{\Sigma^\eta}(w_\infty C)$. Therefore,
$$\pi_{\eta}(w_{\infty})C_0^\eta=\pi_{\eta}(w_{\infty})C^\eta=\pi_{\Sigma^\eta}(w_{\infty}C)=C^\eta=C_0^\eta,$$
and since $w_{\tor}\in P\subseteq W^\eta$, we conclude that $\delta_w=\pi_{\eta}(w_{\infty})$ and $w_{\eta}=w_{\tor}\delta_w$.
\end{proof}

We can now prove an affine analogue of Lemma~\ref{lemma:awcndeltaeta}.

\begin{prop}\label{prop:deltawandIwcomputationAFF}
Let $w=w_{\tor}w_{\infty}$ be the standard splitting in $W$ of $w$. Assume that $w_\eta$ is cyclically reduced. Then the following assertions hold:
\begin{enumerate}
\item
$w_{\infty}\in W_\eta$ and $\delta_w=\pi_{\eta}(w_{\infty})$. In particular, if $\eta$ is standard, then $\delta_w$ is the unique diagram automorphism of $I_{\eta}^{\ext}$ such that $\delta_w(s)=w_{\infty}sw_{\infty}\inv$ for all $s\in I_{\eta}$. 
\item
$w_{\tor}\in W^\eta$ and $w_{\eta}=w_{\tor}\delta_w$. In particular, $I_w$ is the smallest $\delta_w$-invariant subset of $S^\eta$ containing $\supp_{S^\eta}(w_{\tor})$.
\end{enumerate}
\end{prop}
\begin{proof}
Since $C_0^\eta\in\CMin_{\Sigma^\eta}(w_\eta)$, Proposition~\ref{prop:basicprop_finiteCox} and Lemma~\ref{lemma:corresp_w-regular} yield a $w$-regular point $x\in V$ with $x^\eta\in C_0^\eta$. The proposition then follows from Lemma~\ref{lemma:deltawandIwcomputationfromP} (for the second claim in (1), note that $\delta_w$ is determined by its restriction to $I_{\eta}$). 
\end{proof}

The following remark is the affine analogue of Remark~\ref{remark:conjCoxA3IND}.
\begin{remark}\label{remark:conjCoxA3AFF}
Let $w\in W$ be of infinite order, and assume that $w$ is straight. Let $w=w_{\tor}w_{\infty}$ be the standard splitting in $W$ of $w$. Then $w$ is cyclically reduced and $w_{\tor}=1$ by Corollary~\ref{corollary:Psplittingstraight}.  In particular, $w_{\eta}$ is cyclically reduced by Proposition~\ref{prop:corresp_Minsets2}. Moreover, $I_w=\varnothing$ by Proposition~\ref{prop:deltawandIwcomputationAFF}, so that $\OOO_w^{\min}=\Cyc(w)$ by Theorem~\ref{thmintro:graphisomorphism}. We thus recover \cite[Theorem~A(3)]{conjCox} (or \cite[Theorem~A(3)]{HN14}) for $W$ of irreducible affine type.
\end{remark}

As illustrated by Proposition~\ref{prop:deltawandIwcomputationAFF}, it is useful to be able to choose as distinguished representative of the conjugacy class $\OOO_w$ an element $w$ such that $w_\eta$ is cyclically reduced and $\eta$ is standard. This yields to the following definition. 

\begin{definition}
We call $w$ {\bf standard} if $w_{\eta}$ is cyclically reduced and $\eta$ is standard, that is, $$P_w^{\infty}:=\Fix_W([x_0,\eta))=W_{x_0}\cap W^\eta$$ is standard\index{Standard! element of $W$} (equivalenty, $P_w^{\infty}=W_{I_{\eta}}$ --- see Definition~\ref{definition:etastandard}).
\end{definition}

We now show how a standard element $v$ of $\OOO_w$ with $\Cyc(w)=\Cyc_{\min}(v)$ can be obtained from any cyclically reduced $w$.

\begin{lemma}\label{lemma:computation_standard}
Assume that $w\in W$ is cyclically reduced. Write $P_w^{\infty}=a_wW_Ia_w\inv$ for some $I\subseteq S\setminus\{s_0\}$ and some $a_w\in W_{x_0}$ of minimal length in $a_wW_I$. Then $v:=a_w\inv wa_w$ is standard and $w\in\Cyc(v)$. 
\end{lemma}
\begin{proof}
For this proof, we let $\eta$ denote $\eta_v\in\partial V$ instead of $\eta_w$. Note first that $P_{v}^{\infty}=W_I$, as $[x_0,\eta)=a_w\inv [x_0,\eta_w)$. By assumption, $C:=a_w\inv C_0\in\CMin(v)$, and $C_0=\proj_{R_I}(C)$. 

Since $\Stab_W(R_I)=P_v^{\infty}$, the set of walls of $R_I$ is $\WW_{x_0}^\eta$. Since $C_0,C\in R_{x_0}$ are not separated by any wall of $R_I$, we deduce that $C_0^\eta=C^\eta$. Since $C^\eta\in\CMin_{\Sigma^\eta}(v_\eta)$ by Proposition~\ref{prop:corresp_Minsets2}, this implies in particular that $v_\eta$ is cyclically reduced, and hence $v$ is standard. Moreover, $v=\pi_v(C_0)\to \pi_v(C)=w$ by Proposition~\ref{prop:wdecreasingfromCtoD}, yielding the lemma.
\end{proof}

We conclude this subsection with an elementary observation allowing to compute $P_w^{\infty}$ and $I_{\eta}$.

\begin{lemma}\label{lemma:computation_Pwinfty}
If $t\in \mathrm{Isom}(V)$ is a translation in the direction $\eta$, then $P_w^{\infty}=\langle s\in S^W\cap W_{x_0} \ | \ sts=t\rangle$. In particular, if $P_w^{\infty}$ is standard, then $$I_{\eta}=\{s\in S\setminus\{s_0\} \ | \ sts=t\}.$$
\end{lemma}
\begin{proof}
This follows from the fact that if $t\in \mathrm{Isom}(V)$ is a translation in the direction $\eta$, then $[x_0,\eta)$ is contained in a translation axis $L$ of $t$, and hence a reflection $s\in S^W$ belongs to $P_w^{\infty}$ if and only if $s$ fixes $[x_0,\eta)$ if and only if $s$ fixes $L$ if and only if $s\in W_{x_0}$ and $t$ commutes with $s$.
\end{proof}

\begin{remark}\label{remark:whatistranslation}
One can for instance take $t$ in Lemma~\ref{lemma:computation_Pwinfty} to be the translation part of the standard splitting of $w$ as an affine isometry of $V$, or else $t:=w^n$ where $n:=|W_{x_0}|$.
\end{remark}

\begin{remark}\label{remark:directcomputationdeltawIw}
If $w=ut$ for some $u\in W^\eta$ and $t\in W$ normalising $I_{\eta}^{\ext}$ (and $P_w^{\infty}=W_{I_\eta}$ is standard), then $\delta_w\co I_{\eta}^{\ext}\to I_{\eta}^{\ext}: s\mapsto tst\inv$ and $w_\eta=u\delta_w$.

Indeed, if $t\in W$ normalises $I_{\eta}^{\ext}$, then it stabilises the set of walls of $C_0^{\eta}$, and hence $\pi_{\eta}(t)$ stabilises $C_0^\eta$, yielding the claim.
\end{remark}


\subsection{Theorem~\ref{thmintro:affine} and Corollary~\ref{corintro:tightconjugationgraph}, concluded}\label{subsection:TBADACEC}

\begin{theorem}
Theorem~\ref{thmintro:affine} holds.
\end{theorem}
\begin{proof}
Let $w\in W$ with $\Pc(w)=W$ (that is, of infinite order), and let us prove the statements (1)--(7) of Theorem~\ref{thmintro:affine}.

(1) and (3) follow from Proposition~\ref{prop:WetaSetaSigmaetaAff}.

(2) follows from Lemma~\ref{lemma:computation_standard} and Proposition~\ref{prop:WetaSetaSigmaetaAff}(6).

(4) The first assertion follows from Example~\ref{example:PsplittingPwmin} and Proposition~\ref{prop:Psplitting}, and the second from Lemma~\ref{lemma:comparisonstandardsplittings}.

(5) follows from Proposition~\ref{prop:deltawandIwcomputationAFF}.

(6) follows from Lemma~\ref{lemma:centraliser} and Proposition~\ref{prop:centraliserAFF}.

(7) follows from Theorem~\ref{thm:mainthmXi}, Theorem~\ref{thm:XietaXiw} and Proposition~\ref{prop:equivalentdefinitionsXiw}(1)$\Leftrightarrow$(4). 
\end{proof}

Here is the example needed for the proof of Corollary~\ref{corintro:tightconjugationgraph}(3).
\begin{example}\label{example:bigdiameter}
Let $W$ be the affine Coxeter group of type $A_{4n+1}^{(1)}$ for some $n\geq 1$, and denote as usual by $S=\{s_0,s_1,\dots,s_{4n+1}\}$ its set of simple reflections. Let $L\subseteq V$ be the affine line spanned by $\{x_0,x_{2n+1}\}$ (where $x_i$ is as in \S\ref{subsection:SFTROSSACG}), and let $\eta\in\partial V$ be the direction of the geodesic ray based at $x_0$ and containing $x_{2n+1}$. Thus, $\eta$ is standard, and $I_{\eta}=S\setminus\{s_0,s_{2n+1}\}$ (see Figure~\ref{figure:ExampleCorE} for the case $n=2$).

The components of $I_{\eta}$ are $I_1=\{s_1,\dots,s_{2n}\}$ and $I_2=\{s_{2n+2},\dots,s_{4n+1}\}$, and hence $I_1^{\ext}=I_1\cup\{\tau_1\}$ and $I_2^{\ext}=I_2\cup\{\tau_2\}$ are both of type $A_{2n}^{(1)}$ by Theorem~\ref{thmintro:affine}(3). Denoting, as in Definition~\ref{definition:sigmaj}, by $\sigma_i\in\Aut(\Gamma_{I_i}^{\ext})$ ($i=1,2$) the automorphism of $\Gamma_{I_i}^{\ext}$ of order $2n+1$ mapping $\tau_i$ to $s_1$ if $i=1$ and to $s_{2n+2}$ if $i=2$, the group $\Xi_\eta$ is generated by $\sigma_1\inv\sigma_2$ (see Theorem~\ref{thm:mainthmXi}).

\begin{figure}
    \centering
        \includegraphics[trim = 15mm 212mm 87mm 8mm, clip, width=0.8\textwidth]{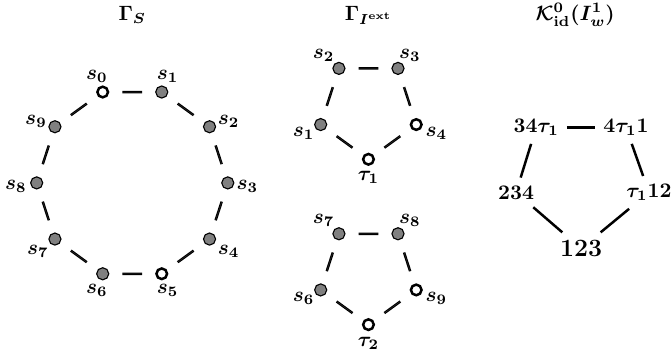}
        \caption{Example~\ref{example:bigdiameter} with $n=2$.}
        \label{figure:ExampleCorE}
\end{figure}

Let $t\in T_0$ be any translation with axis $L$ and such that $\eta_t=\eta$ (see e.g. \cite[Proposition~1.24]{Wei09}). Let also $u\in \Fix_W(L)\subseteq W_{x_0}$ be any cyclically reduced element with $\Pc(u)=W_{I_{\eta}\setminus\{s_{2n},s_{4n+1}\}}$ (for instance, one can take $u=s_1s_2\dots s_{2n-1}\cdot s_{2n+2}s_{2n+3}\dots s_{4n}$), so that $\Fix(u)$ is the affine span of $\{x_0,x_{2n},x_{2n+1},x_{4n+1}\}$. Finally, set $w:=ut$.

Note that $u$ is the elliptic part of $w$ and $t$ is its translation part, and hence $w=tu$ is also the standard splitting in $W$ of $w$ by Lemma~\ref{lemma:comparisonstandardsplittings}. In particular,  $w$ is cyclically reduced by Corollary~\ref{corollary:tmuuinW}. Note also that $\eta_w=\eta$ is standard, and hence $w$ is standard by Proposition~\ref{prop:corresp_Minsets2}. Finally, since $\pi_{\eta}(t)=\id$, Proposition~\ref{prop:deltawandIwcomputationAFF} implies that $\delta_w=\id$ and $I_w=I_{\eta}\setminus\{s_{2n},s_{4n+1}\}$.

The vertex set of the graph $\KKK_{\delta_w}^0(I_w)=\KKK_{\id}^0(I_w)$ coincides with $$\{\sigma_1^i\sigma_2^j(I_w)=\sigma_1^i(I_w\cap I^{\ext}_1)\cup\sigma_2^j(I_w\cap I^{\ext}_2) \ | \ 0\leq i,j\leq 2n\},$$
and hence the quotient graph $\KKK_{\delta_w}^0(I_w)/\Xi_w=\KKK_{\delta_w}^0(I_w)/\Xi_\eta$ (see Theorem~\ref{thmintro:affine}(7)) can be identified with $\KKK_{\id}^0(I^1_w)$, where $I_w^1:=I_w\cap I^{\ext}_1=\{s_1,\dots,s_{2n-1}\}=I_1^{\ext}\setminus\{\tau_1,s_{2n}\}$ corresponds to the class of $I_w$ in $\KKK_{\delta_w}^0(I_w)/\Xi_w$. Its vertex set is then given by
$$\{J_i:=\sigma_1^i(I_w^1) \ | \ 0\leq i\leq 2n\},$$
and it has an edge between $J_i$ and $J_{i+1}$ for each $i\in\{0,\dots,2n\}$ (where $J_{2n+1}:=I_w^1$), as $J_i\stackrel{I_1^{\ext}\setminus\{s_i\}}{\too}J_{i+1}$ for $i\neq 0$ and $J_0\stackrel{I_1^{\ext}\setminus\{\tau_1\}}{\too}J_{1}$ (see Figure~\ref{figure:ExampleCorE} for the case $n=2$, where $s_i$ is simply denoted $i$). Note that the spherical paths in $\KKK_{\id}^0(I^1_w)$ have length at most $1$, that is, the graphs $\overline{\KKK}_{\id}^0(I^1_w)$ and $\KKK_{\id}^0(I^1_w)$ (and hence $\overline{\KKK}_{\delta_w}^0(I_w)/\Xi_w$ and $\KKK_{\delta_w}^0(I_w)/\Xi_w$) coincide. Since $\KKK_{\id}^0(I^1_w)$ is a cycle of length $2n+1$, we conclude that $\overline{\KKK}_{\delta_w}^0(I_w)/\Xi_w$ has $2n+1$ vertices and diameter $n$.
\end{example}

\begin{theorem}
Corollary~\ref{corintro:tightconjugationgraph} holds.
\end{theorem}
\begin{proof}
The existence of the graph isomorphism $\overline{\varphi}_w$ is provided by Theorem~\ref{theorem=R1R2implyCorE}. The statement (1) is then a consequence of Theorem~\ref{thm:Wetafinite}, since it implies that every path in $\KKK_{\delta_w}=\KKK_{\delta_w,W^\eta}$ is spherical. The statement (2) is clear as well, since the number of vertices of $\KKK_{\id,W^\eta}$ is an upper bound for the diameter of $\KKK^t_{\OOO_w}$. The statement (3) follows from Example~\ref{example:bigdiameter}.
\end{proof}


\subsection{Computing \texorpdfstring{$\KKK_{\OOO_w}$}{the structural conjugation graph}: a summary}\label{subsection:CKASAFF}

This subsection is the analogue of \S\ref{subsection:CKASIND} in the affine setting.

\subsubsection{A further algorithm}\label{subsubsection:AFAAff}

Before reviewing the steps to compute $\KKK_{\OOO_w}$ from an infinite order element $w\in W$, we mention one last algorithm, to check whether $w$ is a translation, and which is easily implemented in CHEVIE (see \S\ref{subsubsection:AFA}).

\begin{algo}\label{algo:AffineRootAction}
There is a function $\ARA$ in CHEVIE, taking as input the triple $(W,w,x)$, where $x\in V\approx\RR^{|S|-1}$ is a vector in the basis $\{\alpha_1,\dots,\alpha_{\ell}\}$ of simple roots of $W_{x_0}$, and giving as output the vector $wx\in V$. 

In particular, one can check whether an element $w\in W$ is a translation, that is, whether $\ARA(W,w,\alpha_i)-\alpha_i$ is independent of $i\in\{1,\dots,\ell\}$.
\end{algo}

\subsubsection{Steps to compute $\KKK_{\OOO_w}$}

Let $w\in W$ be of infinite order. Here are the steps to follow in order to compute $\KKK_{\OOO_w}$:

\begin{enumerate}
\item
Up to modifying $w$ inside $\Cyc(w)$, one can assume that $w$ is cyclically reduced (see Algorithm~\ref{algo:cyclicshiftclass} from \S\ref{subsubsection:AFA}). 
\item
Up to further modifying $w$ without changing $\Cyc_{\min}(w)$, we may then assume that $w$ is standard, say $P_w^{\infty}=W_{I}$ for some $I\subseteq S\setminus\{s_0\}$:
\begin{enumerate}
\item
Compute $a_w\in W_{x_0}$ and $I\subseteq S$ such that $P_w^{\infty}=a_wW_Ia_w\inv$ using Lemma~\ref{lemma:computation_Pwinfty} and Remark~\ref{remark:whatistranslation}.
\item
Choose $a_w$ so that it is of minimal length in $a_wW_I$. Then $w':=a_w\inv wa_w$ is standard and $w'\to w$ by Lemma~\ref{lemma:computation_standard}.
\end{enumerate}
\item
Compute the Dynkin diagram associated to $I^{\ext}$, as in Definition~\ref{definition:tauIextetc}.
\item
Compute the standard splitting $w=w_{\tor}w_{\infty}$ in $W$ of $w$, with the help of Lemma~\ref{lemma:comparisonstandardsplittings}.
\item
Compute $\delta_w\in\Aut(W^{\eta_w},I^{\ext})$ and $I_w\subseteq I^{\ext}$ using Proposition~\ref{prop:deltawandIwcomputationAFF}:
\begin{enumerate}
\item
$\delta_w$ is the unique diagram automorphism of $I^{\ext}$ such that $\delta_w(s)=w_{\infty}sw_{\infty}\inv$ for all $s\in I$. 
\item
$I_w$ is the smallest $\delta_w$-invariant subset of $I^{\ext}$ containing $\supp_{I^{\ext}}(w_{\tor})$.
\end{enumerate}
(In some circumstances, it might be easier, instead of performing the steps (4) and (5), to directly compute $\delta_w$ and $I_w$ using Remark~\ref{remark:directcomputationdeltawIw}).
\item
Compute $\KKK_{\delta_w}^0(I_w)$ using Lemma~\ref{lemma:oppositionfinite}.
\item
Compute $\KKK_{\delta_w}^0(I_w)/\Xi_w=\KKK_{\delta_w}^0(I_w)/\Xi_{\eta_w}$ (see Theorem~\ref{thm:XietaXiw}) using Theorem~\ref{thm:mainthmXi}. This yields $\KKK_{\OOO_w}$ by Theorem~\ref{thmintro:graphisomorphism}.
\end{enumerate}


\subsection{Examples}\label{subsection:ExAFF}

\begin{example}\label{example:A53A55}
Consider the Coxeter group $W$ of affine type $D_7^{(1)}$, with set of simple reflections $S=\{s_i \ | \ 0\leq i\leq 7\}$ as on Figure~\ref{figure:ExampleAff1}.

Let $\theta:=\theta_{S\setminus\{s_0\}}$ be the highest root associated to $S\setminus\{s_0\}$: its components in the basis of simple roots are given in Figure~\ref{figure:TableFIN} (see \S\ref{subsubsection:RS}), and hence $\theta=x\alpha_7$, where $x:=s_2s_1s_3s_2s_4s_3s_5s_4s_6s_5\in W$. In particular, $r_{\theta}=xs_7x\inv$. Note that, in the notation of \S\ref{subsection:SFTROSSACG}, the reflections $r_{\theta}$ and $s_0=r_{\delta-\theta}$ have parallel fixed walls, and hence $t:=s_0r_{\theta}$ is a translation. Moreover, as $s\theta=\theta$ (and $s\delta=\delta$) for all $s\in I:=S\setminus\{s_0,s_2\}$, the translation $t$ centralises $W_I$, and hence admits the affine span of $\{x_0,x_2\}$ as a translation axis (where $x_i$ is as in \S\ref{subsection:SFTROSSACG}).

Let $u\in W_I\subseteq W_{x_0}$ be cyclically reduced. Then $w:=ut$ is the standard splitting in $W$ of $w$ by Lemma~\ref{lemma:comparisonstandardsplittings}, and $w$ is cyclically reduced by Corollary~\ref{corollary:tmuuinW}. In particular, $w$ is standard, as $\eta:=\eta_w=\eta_t$ is standard and $w_\eta$ is cyclically reduced by Proposition~\ref{prop:corresp_Minsets2}. Thus,  $P_w^{\infty}=P_t^{\infty}=W_I$ and $I=I_{\eta}$. Moreover, since $\pi_{\eta}(t)=\id$, Proposition~\ref{prop:deltawandIwcomputationAFF} implies that $\delta_w=\id$ and $I_w=\supp(u)\subseteq I$.

Let $I_1=\{s_1\}$ and $I_2=\{s_3,s_4,s_5,s_6,s_7\}$ be the components of $I$. Then $I_1^{\ext}=I_1\cup\{\tau_1=r_{\delta-\alpha_1}\}$ is of type $A_1^{(1)}$, and $I_2^{\ext}=I_2\cup\{\tau_2=r_{\delta-\theta_{I_2}}\}$ is of type $D_5^{(1)}$. The labels in Figure~\ref{figure:TableFIN} (see \S\ref{subsubsection:RS}) allow to compute $\delta-\theta_{I_2}=s_2s_1s_3s_2\alpha_0$, and hence $$\tau_2=s_2s_1s_3s_2s_0s_2s_3s_1s_2.$$

Denote, as in Definition~\ref{definition:sigmaj}, by $\sigma_1$ the diagram automorphism of $\Gamma_{I_1^{\ext}}$ of order $2$ mapping $\tau_1$ to $s_1$, and by $\sigma_2$ the diagram automorphism of $\Gamma_{I_2^{\ext}}$ of order $2$ which permutes nontrivially each of the sets $\{\tau_2,s_3\}$ and $\{s_6,s_7\}$. Then the group $\Xi_{\eta}$ is generated by $\sigma_1\sigma_2$ (see Theorem~\ref{thm:mainthmXi}(D1)).

\smallskip

(1) Choose now $u=s_3s_4s_5s_6$, so that $I_w=\{s_3,s_4,s_5,s_6\}$. The graph $\KKK_{\delta_w}^0(I_w)$ has $4$ vertices, $I_{3456}:=I_w$, $I_{3457}:=\{s_3,s_4,s_5,s_7\}$, $I_{\tau 457}:=\{\tau_2,s_4,s_5,s_7\}$, and $I_{\tau 456}:=\{\tau_2,s_4,s_5,s_6\}$, forming a cycle, in that order (and these are the only edges of $\KKK_{\delta_w}^0(I_w)$). The graph $\KKK_{\delta_w}^0(I_w)/\Xi_w=\KKK_{\delta_w}^0(I_w)/\Xi_{\eta}$, on the other hand, has only two vertices, $[I_w]$ and $[I_{3457}]$.

Set $L:=\{s_3,s_4,s_5,s_6,s_7\}$, so that $I_{3457}=\op_L(I_w)$. Then the second part of Theorem~\ref{thmintro:graphisomorphism} implies that $\varphi_w\inv([I_w])=\Cyc(w)$ and $\varphi_w\inv([I_{3457}])=\Cyc(w')$, where $$w':=w_0(L)ww_0(L)\stackrel{L}{\too}w.$$

\medskip

(2) Choose next $u=s_1s_3s_4s_5s_6$, so that $I_w=\{s_1,s_3,s_4,s_5,s_6\}$. The graph $\KKK_{\delta_w}^0(I_w)$ has again $4$ vertices, $I_{3456}:=I_w$, $I_{3457}:=\{s_1,s_3,s_4,s_5,s_7\}$, $I_{\tau 457}:=\{s_1,\tau_2,s_4,s_5,s_7\}$, and $I_{\tau 456}:=\{s_1,\tau_2,s_4,s_5,s_6\}$, and is pictured on Figure~\ref{figure:ExampleAff1} (with the same notational convention as on Figure~\ref{figure:KKKInd}, and writing simply $\tau$ for $\tau_2$). Since $\sigma_1\sigma_2(I_w)=\{\tau_1,\tau_2,s_4,s_5,s_7\}$ is not one of these vertices, $\Xi_w=\{1\}$ by Theorem~\ref{thmintro:affine}(7). Hence $\KKK_{\delta_w}^0(I_w)/\Xi_w$ coincides with $\KKK_{\delta_w}^0(I_w)$ in this case (as well as with the tight conjugation graph $\overline{\KKK}_{\delta_w}^0(I_w)/\Xi_w$).

\begin{figure}
    \centering
        \includegraphics[trim = 8mm 210mm 81mm 34mm, clip, width=0.8\textwidth]{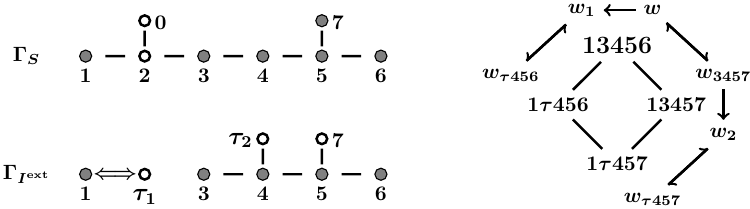}
        \caption{Example~\ref{example:A53A55}}
        \label{figure:ExampleAff1}
\end{figure}

As before, $\varphi_w\inv([I_w])=\Cyc(w)$ and $\varphi_w\inv([I_{3457}])=\Cyc(w_{3457})$, where $$w_{3457}:=w_0(L)ww_0(L)\stackrel{L}{\too}w\quad\textrm{with $L:=\{s_1,s_3,s_4,s_5,s_6,s_7\}$.}$$ 

Set now $L_1:=\{s_1,\tau_2,s_3,s_4,s_5,s_6\}$, so that $I_{\tau 456}=\op_{L_1}(I_{3456})$. Set $$a_1:=s_2s_1s_3s_2s_4s_3s_5s_4s_6s_5,$$ and observe that $a_1(s_0,s_1,s_2,s_3,s_4,s_6)a_1\inv=(\tau_2,s_3,s_4,s_5,s_6,s_1)$, so that $W^{\eta}_{L_1}=a_1W_{K_1}a_1\inv$, where $K_1:=\{s_0,s_1,s_2,s_3,s_4,s_6\}$. Note also that $a_1$ is of minimal length in $a_1W_{K_1}$. Moreover, $a_1\inv ua_1=s_6s_1s_2s_3s_4$ is cyclically reduced, and $a_1\inv ta_1$ is a translation centralising $\{s_1,s_2,s_3,s_4,s_6\}$ (as $t$ centralises $I$). In particular, we conclude as before that $w_1:=a_1\inv wa_1$ is cyclically reduced. The second part of Theorem~\ref{thmintro:graphisomorphism} then implies that $\varphi_w\inv([I_{\tau 456}])=\Cyc(w_{\tau 456})$, where $$w_{\tau 456}:=w_0(K_1)w_1w_0(K_1),$$ and that $w\to w_1\stackrel{K_1}{\too}w_{\tau 456}$. 

Finally, set $L_2:=\{s_1,\tau_2,s_3,s_4,s_5,s_7\}$, so that $I_{\tau 457}=\op_{L_2}(I_{3457})$. Let $a_2,K_2$ be respectively obtained from $a_1,K_1$ by exchanging the roles of $s_6$ and $s_7$, so that $W^{\eta}_{L_2}=a_2W_{K_2}a_2\inv$ and $a_2$ is of minimal length in $a_2W_{K_2}$. The same argument as before then yields that $w_2:=a_2\inv w_{3457}a_2$ is cyclically reduced,  that $\varphi_w\inv([I_{\tau 457}])=\Cyc(w_{\tau 457})$ where $$w_{\tau 457}:=w_0(K_2)w_2w_0(K_2),$$ and that $w_{3457}\to w_2\stackrel{K_2}{\too}w_{\tau 457}$. 
\end{example}

\begin{example}\label{exampleE7tilde}
Consider the Coxeter group $W$ of affine type $E_7^{(1)}$, with set of simple reflections $S=\{s_i \ | \ 0\leq i\leq 7\}$ as on Figure~\ref{figure:ExampleAff2}. 

Set $J_1:=\{s_i \ | \ 1\leq i\leq 5\}$ and $J_2:=\{s_i \ | \ 2\leq i\leq 7\}$, and consider the element $x:=s_0w_0(J_1)w_0(J_2)s_7\in W$. Then $x$ normalises $J:=\{s_2,s_3,s_4,s_5,s_7\}$ (more precisely, conjugation by $x$ permutes $s_2$ and $s_5$, and fixes $s_3,s_4,s_7$), as follows from Lemma~\ref{lemma:oppositionfinite}. One checks (see Algorithm~\ref{algo:AffineRootAction}) that $x^2$ is a translation commuting with $I:=\{s_1\}\cup J=S\setminus\{s_0,s_6\}$. In particular, since $|I_{\eta}|\leq |S|-2$ where $\eta:=\eta_x$, Lemma~\ref{lemma:computation_Pwinfty} implies that $I_{\eta}=I$.

Let $I_1=\{s_1,s_2,s_3,s_4,s_5\}$ and $I_2=\{s_7\}$ be the components of $I$. Then $I_1^{\ext}=I_1\cup\{\tau_1=r_{\delta-\theta_{I_1}}\}$ is of type $D_5^{(1)}$, and $I_2^{\ext}=I_2\cup\{\tau_2=r_{\delta-\alpha_7}\}$ is of type $A_1^{(1)}$. We claim that $xI^{\ext}x\inv=I^{\ext}$. Indeed, since $x\delta=\delta$ and $xs_7x\inv=s_7$, we have $x\tau_2x\inv=\tau_2$. On the other hand, the labels in Figure~\ref{figure:TableFIN} (see \S\ref{subsubsection:RS}) allow to compute $\delta-\theta_{I_1}=y\alpha_7$, where $y:=s_6s_5s_4s_0s_1s_2s_3s_4s_5s_6$. Hence $\tau_1=ys_7y\inv$, and one checks that $xs_1x\inv=\tau_1$ (so that $x\tau_1x\inv=x^2s_1x^{-2}=s_1$), yielding the claim. 

Let $\sigma_{s_1},\sigma_{s_2},\sigma_{s_7}$ be as in Definition~\ref{definition:sigmaj}: with the notational convention taken in Lemma~\ref{lemma:extended_diagram_autom}, $\sigma_{s_1},\sigma_{s_2}\in\Aut(\Gamma_{I_1^{\ext}})$ have cycle decomposition $\sigma_{s_1}=(\tau_1,s_1)(s_2,s_5)$ and $\sigma_{s_2}=(\tau_1,s_2,s_1,s_5)$, while $\sigma_{s_7}\in\Aut(\Gamma_{I_2^{\ext}})$ has cycle decomposition $\sigma_{s_7}=(\tau_2,s_7)$. Then $\delta_x=\pi_{\eta}(x)=\sigma_{s_1}=(\sigma_{s_2}\sigma_{s_7})^2$ by Remark~\ref{remark:directcomputationdeltawIw}, while $\Xi_{\eta}=\langle \sigma_{s_2}\sigma_{s_7}\rangle$ by Theorem~\ref{thm:mainthmXi}(E3).

Let now $u\in W_I$, and set $w:=ux^n$ for some $n\geq 1$. Since $u\in W^{\eta}$, we have $\eta_w=\eta$ and $\delta_w=\delta_x^n=\sigma_{s_1}^n$, while $w_{\eta}=u\delta_x^n$, so that $I_w$ is the smallest $\sigma_{s_1}^n$-invariant subset of $I^{\ext}$ containing $\supp(u)\subseteq I$. We now make various choices of $(n,u)$ such that $w_\eta$ is cyclically reduced.

\smallskip

(1) Choose first $n=1$ and $u=s_4$. Then $\delta_w=\sigma_{s_1}$ and $I_w=\{s_4\}$. The graph $\KKK_{\delta_w}^0(I_w)$ has two vertices, $\{s_4\}$ and $\{s_3\}$, while the graph $\KKK_{\delta_w}^0(I_w)/\Xi_w=\KKK_{\delta_w}^0(I_w)/\Xi_{\eta}$ has only one vertex. In particular, $\OOO_w^{\min}=\Cyc_{\min}(w)$.

\medskip

(2) Choose next $n=1$ and $u=s_4s_7$. Then $\delta_w=\sigma_{s_1}$ and $I_w=\{s_4,s_7\}$. In this case, the graph $\KKK_{\delta_w}^0(I_w)$ has two vertices, $I_w=\{s_4,s_7\}$ and $\op_L(I_w)=\{s_3,s_7\}$ (where $L:=\{s_3,s_4,s_7\}$), and coincides with $\KKK_{\delta_w}^0(I_w)/\Xi_w$. One checks (see Algorithm~\ref{algo:cyclicshiftclass}) that $w$ is cyclically reduced. The second part of Theorem~\ref{thmintro:graphisomorphism} then implies that $\varphi_w\inv(I_w)=\Cyc(w)$ and $\varphi_w\inv(\op_L(I_w))=\Cyc(w')$, where $$w':=w_0(L)ww_0(L)\stackrel{L}{\too}w.$$

\medskip

(3) Choose next $n=2$ and $u=s_1s_4s_5$. Then $\delta_w=\id$ and $I_w=\{s_1,s_4,s_5\}$. The graph $\KKK_{\delta_w}^0(I_w)$ has $8$ vertices, and is pictured on Figure~\ref{figure:ExampleAff2} (with the same notational convention as on Figure~\ref{figure:KKKInd}, and writing simply $\tau$ for $\tau_1$). The graph $\KKK_{\delta_w}^0(I_w)/\Xi_w$, on the other hand, has only two vertices, $[I_w]$ and $[\{s_1,s_3,s_5\}]$, and we have $\{s_1,s_3,s_5\}=\op_L(I_w)$ with $L:=\{s_1,s_3,s_4,s_5\}$.

\begin{figure}
    \centering
        \includegraphics[trim = 8mm 211mm 38mm 30mm, clip, width=\textwidth]{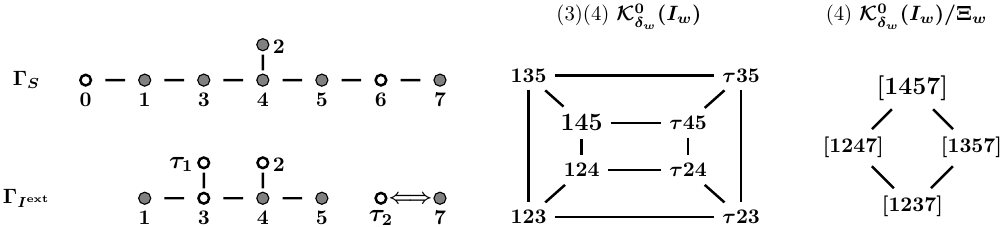}
        \caption{Example~\ref{exampleE7tilde}}
        \label{figure:ExampleAff2}
\end{figure}

Since $x^2$ is a translation, $w$ is cyclically reduced by Corollary~\ref{corollary:tmuuinW}. The second part of Theorem~\ref{thmintro:graphisomorphism} then implies that $\varphi_w\inv([I_w])=\Cyc(w)$ and $\varphi_w\inv([\op_L(I_w)])=\Cyc(w')$, where $$w':=w_0(L)ww_0(L)\stackrel{L}{\too}w.$$

\medskip

(4) Finally, choose $n=2$ and $u=s_1s_4s_5s_7$. Then $\delta_w=\id$ and $I_w=\{s_1,s_4,s_5,s_7\}$. The graph $\KKK_{\delta_w}^0(I_w)$ has $8$ vertices, and is the same as the one pictured on Figure~\ref{figure:ExampleAff2} (adding $s_7$ to each vertex). The graph $\KKK_{\delta_w}^0(I_w)/\Xi_w$, on the other hand, has this time $4$ vertices, with representatives $I_{145}=I_w,I_{135},I_{123},I_{124}$ (where $I_{ijk}:=\{s_i,s_j,s_k,s_7\}$), and is pictured on Figure~\ref{figure:ExampleAff2}. Note also that the tight conjugation graph $\overline{\KKK}_{\delta_w}^0(I_w)/\Xi_w$ is the complete graph on this set of vertices.

As in (3), $w$ is cyclically reduced. Consider the subsets $L_{1345},L_{1245},L_{1235}$ of $I^{\ext}$, where $L_{ijkl}:=\{s_i,s_j,s_k,s_l,s_7\}$, so that $$I_{135}=\op_{L_{1345}}(I_w),\quad I_{124}=\op_{L_{1245}}(I_w),\quad\textrm{and}\quad I_{123}=\op_{L_{1235}}(I_{135}).$$ The second part of Theorem~\ref{thmintro:graphisomorphism} then implies that $\varphi_w\inv([I_w])=\Cyc(w)$, while $\varphi_w\inv([I_{135}])=\Cyc(w_{135})$ and $\varphi_w\inv([I_{124}])=\Cyc(w_{124})$, where $$w_{135}:=w_0(L_{1345})ww_0(L_{1345})\stackrel{L_{1345}}{\too}w\quad\textrm{and}\quad w_{124}:=w_0(L_{1245})ww_0(L_{1245})\stackrel{L_{1245}}{\too}w.$$ Similarly, $\varphi_w\inv([I_{123}])=\Cyc(w_{123})$, where $$w_{123}:=w_0(L_{1235})w_{135}w_0(L_{1235})\stackrel{L_{1235}}{\too}w_{135}.$$
\end{example}


\printindex
\printindex[s]

\newpage
\bibliographystyle{amsalpha} 
\bibliography{these} 

\end{document}